\pdfoutput=1
\documentclass[10pt,oneside,pdftex,dvipsnames,a4paper]{report}

\usepackage{etex}
\reserveinserts{18}

\usepackage{a4wide}
\usepackage{setspace,xspace,makeidx,fancyvrb,relsize}
\usepackage{alltt}

\usepackage{amsmath,amsthm,amssymb}
\usepackage{parallel,amsfonts,accents}

\usepackage{amsxtra,amstext,latexsym,dsfont} 

\normalfont
\usepackage{bm}

\usepackage[comma]{natbib}

\usepackage[symbol]{footmisc}
\usepackage{perpage}
\MakePerPage[2]{footnote}

\usepackage[pdftex]{graphicx}
\usepackage[dvipsnames]{xcolor}

\usepackage{pgf,tikz}
\usetikzlibrary{snakes}
\usetikzlibrary{backgrounds}
\usetikzlibrary{arrows}
\usetikzlibrary{patterns}
\usetikzlibrary{fit}
\usetikzlibrary{calc}

\usepackage{lscape} 
\usepackage{shorttoc}

\usepackage{array,longtable,booktabs,float,rotating}
\usepackage{dcolumn,ctable}

\usepackage[labelfont={bf},ruled,labelsep=period,small]{caption}

\usepackage[flushleft]{threeparttable}

\usepackage[linkcolor=black,colorlinks=true,citecolor=black,
bookmarks=true,bookmarksopen=true,bookmarksnumbered=true,bookmarksopenlevel=0,%
hyperindex=true,pdfpagemode=UseOutlines,linktocpage=true,pdfpagelayout=TwoColumnLeft,
pdftitle=Biostatistics,pdfstartview=FitH,pdffitwindow=true,urlcolor=blue]{hyperref}

\usepackage{ragged2e}

\usepackage{multirow,multicol}

\usepackage{fancyhdr,extramarks}

 \makeatletter
 \newcommand\sub{\@startsection%
     {subsubsection}{5}{0mm}{-1\baselineskip}{.01\baselineskip}%
     {\normalfont\itshape}}
 \renewcommand\subsubsection{\@startsection%
     {subsubsection}{3}{0mm}{-1\baselineskip}{.01\baselineskip}%
     {\normalfont\itshape}}
 \makeatother

\makeatletter
        \newcommand\Appendix[2][?]{%
            \refstepcounter{section}%
            \addcontentsline{toc}{appendix}%
                {\protect\numberline{\appendixname~\thesection}#1}%
            {\raggedleft\bfseries \appendixname\
                \thesection\par \centering#2\par}%
                \sectionmark{#1}%
                \@afterheading
                \addvspace{\baselineskip}}
        \newcommand\sAppendix[1]{%
            \raggedleft\bfseries\appendixname\par
            \@afterheading\addvspace{\baselineskip}}
    \makeatother

\normalsize \makeindex
\newcommand\eg{e.g.\xspace}

\newcolumntype{A}{>{\centering}p{100pt}}
\newlength\savedwidth

\def\coldot{.}%
{\catcode`\.=\active%
    \gdef.{$\egroup\setbox2=\hbox to \dimen0 \bgroup$\coldot}}
\def\rightdots#1{%
    \setbox0=\hbox{$1$}\dimen0=#1\wd0%
    \setbox0=\hbox{$\coldot$}\advance\dimen0 \wd0%
    \setbox2=\hbox to \dimen0 {}%
    \setbox0=\hbox\bgroup\mathcode`\.="8000 $}
\def\endrightdots{$\hfil\egroup\box0\box2}

\newcolumntype{d}[1]{D{.}{.}{#1}}
\newcolumntype{A}{>{\centering}p{100pt}}
\newcolumntype{.}{D{.}{.}{-1}}
\newcolumntype{P}[2]{>{#1\raggedright\arraybackslash}p{#2}}

\DeclareFontFamily{U}{euc}{}
\DeclareFontShape{U}{euc}{m}{n}{<-6>eurm5<6-8>eurm7<8->eurm10}{}%

\newtheorem{thm}{Theorem}
\theoremstyle{plain}      
\theoremstyle{plain}      
\theoremstyle{plain}      
\theoremstyle{definition} 
\theoremstyle{definition} 
\theoremstyle{definition} 
\theoremstyle{plain} \newtheorem{cor}{Corollary}
\theoremstyle{definition} 
\theoremstyle{plain} \newtheorem{pro}{Proposition}
\theoremstyle{definition} 

\newcounter{nctr}
\usepackage[rightbars,color]{changebar}
\cbcolor{blue}

\VerbatimFootnotes

\usepackage{url}
\newcommand\tb{\textbf}
\newcommand\ti{\textit}

\newcommand\ds{\mathds}
\newcommand\bb{\mathbb}

\newcommand\te{\text}
\newcommand\ma[1]{\te{\bf{#1}}}
\newcommand\ca{\mathcal}

\newcommand\op{\operatorname}

\newcommand\as{^\ast}

\newcommand\argmin{\operatornamewithlimits{argmin}}

\newcommand\E{\bb{E}}

\newcommand\hi{\mathit}
\newcommand\iid{\op{iid}}

\newcommand\ind{\te{ind}}

\newcommand\lb{\lbrace}
\newcommand\lt{\left}

\newcommand\p{\bb{P}}           
\newcommand\pr{^\prime}

\newcommand\q{\quad}
\newcommand\qq{\qquad}

\newcommand\rb{\rbrace}
\newcommand\rt{\right}

\newcommand\ssel{\te{SSEL}}

\newcommand\stack{\stackrel} 

\newcommand\tth{^\text{th}}
\newcommand\var{\operatorname{\bb{V}ar}}

\newcommand\R{\ds{R}}  

\newcommand\unif{\op{Unif}}

\newcommand\mvn{\op{MVN}}
\newcommand\poi{\op{Pois}}
\newcommand\dgam{\op{Gam}}

\newcommand\ba{\ma{a}} %
\newcommand\bp{\ma{p}} 
\newcommand\bu{\ma{u}} 
\newcommand\bv{\ma{v}} 
\newcommand\bx{\ma{x}}
\newcommand\by{\ma{y}}
\newcommand\bz{\ma{z}}

\newcommand\bD{\ma{D}} 
\newcommand\bP{\ma{P}} 
\newcommand\bR{\ma{R}} 
\newcommand\bX{\ma{X}}

\newcommand\bone{\bm{1}} 

\newcommand\cB{\ca{B}} 
\newcommand\cD{\ca{D}} 
\newcommand\cE{\ca{E}} 
\newcommand\cI{\ca{I}} 
\newcommand\cP{\ca{P}} 
\newcommand\cY{\ca{Y}} %
\newcommand\cZ{\ca{Z}} %

\newcommand\al{\alpha}

\newcommand\ga{\gamma}

\newcommand\sig{\sigma}
\newcommand\ome{\omega}

\newcommand\Ga{\Gamma}

\newcommand\Sig{\Sigma}
\newcommand\Th{\Theta}

\newcommand\bsig{\bm\sigma}

\newcommand\bth{\bm\theta}
\newcommand\bxi{\bm\xi}
\newcommand\bphi{\bm\phi}

\newcommand\bSig{\bm\Sigma}

\newcommand\bTh{\bm\Theta}

\begin{document}

\setcounter{page}{1}
\sloppy

\begin{center} 
\end{center}
\begin{center}
   \huge
   \tb{Bayesian Decision-theoretic Methods 
    \\ for Parameter Ensembles 
    \\ with Application to 
    \\ Epidemiology}\\   
\end{center}
\vspace{1cm}
\begin{center}
  \includegraphics[width=6cm]{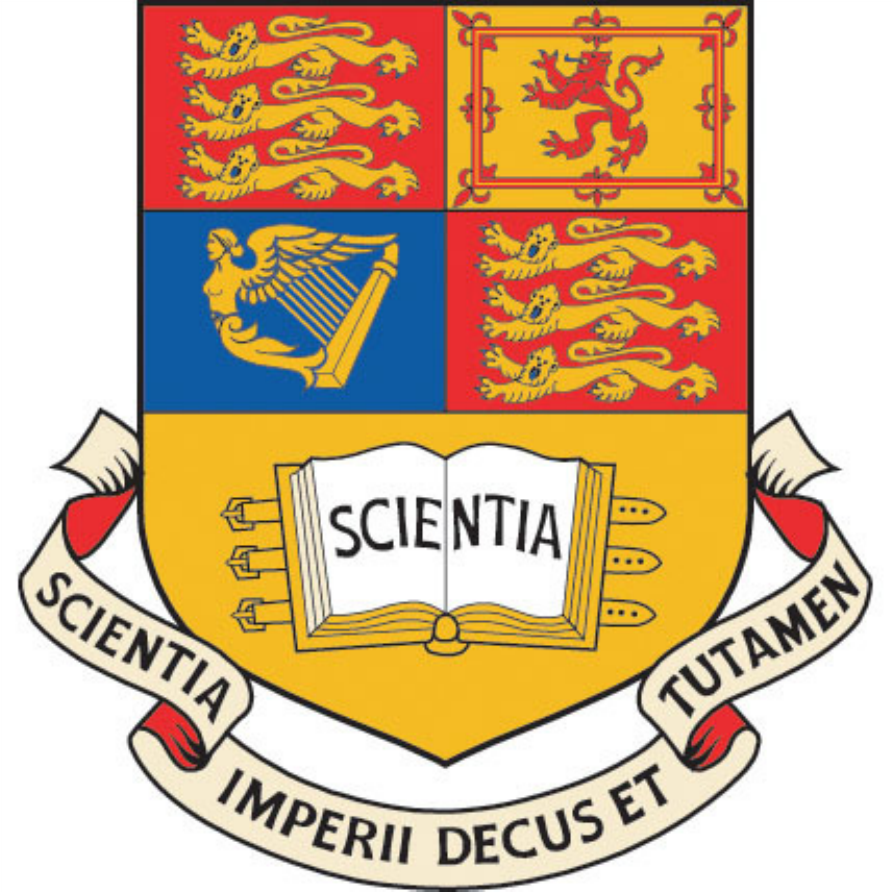}
\end{center}
\vspace{1cm}
\begin{center}
   \Large
   By\\
   \tb{Cedric E. Ginestet}\\
\end{center}
\begin{center}
   \Large
   Under the supervision of \\
   Nicky G. Best, Sylvia Richardson and David J. Briggs
\end{center}
\vspace{.25cm}
\begin{center}  
  \large
  In partial fulfilment of the requirements
  for the degree of \\
  \ti{Doctor of Philosophy.}\\
  \ti{February 2011}
\end{center}
\begin{center}
  \large
  Doctorate Thesis submitted to the \\
  Department of Epidemiology and Biostatistics \\
  Imperial College London.\\
\end{center}
\pagebreak[4]
\addcontentsline{toc}{subsection}{Abstract}
\begin{abstract}
   Parameter ensembles or sets of random effects constitute one of the
   cornerstones of modern statistical practice. This is especially the
   case in Bayesian hierarchical models, where several decision
   theoretic frameworks can be deployed to optimise the estimation of
   parameter ensembles. 
   The reporting of such ensembles in the form of sets of point estimates is an important concern in
   epidemiology, and most particularly in spatial
   epidemiology, where each element in these ensembles represent an epidemiological unit such
   as a hospital or a geographical area of interest. 
   The estimation of these parameter ensembles may substantially vary
   depending on which inferential goals are prioritised by the
   modeller. Since one may wish to satisfy a range of desiderata, 
   it is therefore of interest to investigate whether some sets of point estimates can
   simultaneously meet several inferential objectives. In this
   thesis, we will be especially concerned with identifying 
   ensembles of point estimates that produce good approximations of (i) the true
   empirical quantiles and empirical quartile ratio (QR) 
   and (ii) provide an accurate classification of the ensemble's elements above and
   below a given threshold. 
   For this purpose, we review various decision-theoretic frameworks, which have
   been proposed in the literature in relation to the optimisation of
   different aspects of the empirical distribution of a parameter
   ensemble. This includes the constrained Bayes (CB),
   weighted-rank squared error loss (WRSEL), and triple-goal (GR)
   ensembles of point estimates. In addition, we also consider the set
   of maximum likelihood estimates (MLEs) and the ensemble of posterior means --the latter
   being optimal under the summed squared error loss (SSEL).  
   Firstly, we test the performance of these different sets of point estimates
   as plug-in estimators for the empirical quantiles and empirical QR under a
   range of synthetic scenarios encompassing both spatial and non-spatial
   simulated data sets. Performance evaluation is here conducted using
   the posterior regret, which corresponds to the difference in
   posterior losses between the chosen plug-in estimator and the optimal
   choice under the loss function of interest. The triple-goal plug-in estimator is found
   to outperform its counterparts and produce close-to-optimal
   empirical quantiles and empirical QR. A real data set documenting schizophrenia
   prevalence in an urban area is also used to illustrate the
   implementation of these methods. 
   Secondly, two threshold classification losses (TCLs) --weighted and
   unweighted-- are formulated. The weighted TCL can be used
   to emphasise the estimation of false positives over
   false negatives or the converse. These weighted and unweighted TCLs are
   optimised by a set of posterior quantiles and a set of posterior
   medians, respectively. Under an unweighted classification framework, the SSEL point
   estimates are found to be quasi-optimal for all scenarios
   studied. In addition, the five candidate plug-in estimators are also evaluated
   under the rank classification loss (RCL), which has been previously
   proposed in the literature. The SSEL set of point estimates are
   again found to constitute quasi-optimal plug-in estimators under
   this loss function, approximately on a par with the CB and GR sets of point
   estimates. The threshold and rank classification loss functions are 
   applied to surveillance data reporting methicillin resistant
   \ti{Staphylococcus aureus} (MRSA) prevalence in UK hospitals. This
   application demonstrates that all the studied plug-in
   classifiers under TCL tend to be more liberal than the optimal
   estimator. That is, all studied plug-in estimators tended to classify
   a greater number of hospitals above the risk threshold than the set of
   posterior medians.
   In a concluding chapter, we discuss some possible generalisations of the loss
   functions studied in this thesis, and consider how model
   specification can be tailored to better serve the inferential
   goals considered.
\end{abstract}


\setcounter{page}{3}
\pagebreak[4]
\tableofcontents
\addcontentsline{toc}{subsection}{Table of Content}
\pagebreak[4]

\begin{center}
  \Large \tb{Abbreviations}
\end{center}
\begin{longtable}{>{\RaggedRight}p{125pt}>{\RaggedRight}p{200pt}}
AVL & Absolute value error loss \\
DAG & Directed Acyclic Graph \\
iid & Independent and identically distributed \\
ind & Independent  \\
BHM & Bayesian hierarchical model \\
BYM & Besag, York and Molli\'{e} model\\

CAR  & Conditional autoregressive \\
CB   & Constrained Bayes \\
CDI  & Carstairs' deprivation index \\
CDF  & Cumulative distribution function \\
DoPQ & Difference of posterior quartiles \\

EB & Empirical Bayes \\
EDF & Empirical distribution function \\
G-IG & Gamma-Inverse Gamma model \\
GR  &  Triple-goal ($G$ for empirical distribution and R for rank) \\
ICC  &  Intraclass Correlation \\
IQR  &  Interquartile range \\
IQR-SEL  &  Interquartile range squared error loss \\
ISEL  &  Integrated squared error loss \\

L1  & Laplace-based Besag, York and Molli\'{e} model \\
MCMC& Markov chain Monte Carlo \\
MLE & Maximum likelihood estimate\\
MOR & Median odds ratio \\
MRSA & Methicillin resistant \ti{Staphylococcus aureus} \\

N-N & Normal-Normal model \\

OR & Odds ratio \\
pdf & Probability density function \\
PM  & Posterior mean \\

Q-SEL  & Quantiles squared error loss \\
QR     & Quartile ratio  \\
QR-SEL & Quartile ratio squared error loss \\

RCL  & Rank classification loss \\
RoPQ & Ratio of Posterior Quartiles \\
RR   & Relative risk \\
RSEL   & Rank squared error loss \\

SEL & Squared error loss \\
SBR & Smoothing by roughening algorithm \\
SSEL & Summed squared error loss \\
WRSEL & Weighted rank squared error loss \\

TCL & Threshold classification loss 
\end{longtable}

\pagebreak[4]


\vspace{13cm}
\begin{center}
  \ti{\`{A} mes parents.}
\end{center}
\pagebreak[4]



\setstretch{1} 

\addcontentsline{toc}{subsection}{Acknowledgements}
\begin{center}
    \tb{Acknowledgements}
\end{center}

Writing a thesis is always an arduous process, and the present one has been
a particularly testing one. However, during that long journey, I had
the chance to receive support from distinguished scholars. For my
entry in the world of statistics, I am greatly indebted to Prof.
Nicky Best, my director of studies, who has entrusted me with the
capacity to complete a PhD in biostatistics, and has demonstrated
mountains of patience in the face of my often whimsical behaviour.
\\
 
I also owe a great debt of gratitude to my second supervisor, Prof.
Sylvia Richardson,  who has helped me a great deal throughout the PhD process, and
has been particularly understanding during the write-up stage. This revised thesis,
however, would not be in its current form without the advice and
recommendations of the two examiners --Prof. Angelika van der Linde
and Dr. Maria De Iorio-- and I am also extremely grateful to them. In
particular, I would like to thank Prof. Angelika van der Linde for
suggesting how to simplify the minimisation of the threshold
classification loss (TCL). 
\\

This work was supported by a capacity building scholarship awarded by
the UK Medical Research Council. In terms of technical resources, I
also wish to acknowledge the use of the Imperial College High
Performance Computing (HPC) Service. 
In particular, I would like to thank Simon Burbidge for his help in
using the HPC. In addition, I have also benefited form the help of
James Kirkbride from the Department of Psychiatry in Cambridge
University, who provided the schizophrenia prevalence data set. 
\\

Lastly, my greatest thanks go to Sofie Davis, who has been a
trove of patience and encouragement throughout the last four years. 

\pagebreak[4]	

%
%
\chapter{Introduction}\label{chap:introduction}
An important concern in epidemiology is the reporting of
ensembles of point estimates. In disease mapping, for example, one may
wish to ascertain the levels of risk for cancer in different regions of the British
Isles \citep{Jarup2002}, or evaluate cancer mortality rates in
different administrative areas \citep{Lopez2007}.
In public health, one may be interested in comparing different service
providers such as neonatal clinics \citep{MacNab2004}. 
Estimation of parameter ensembles may also be of interest as
performance indicators, such as when compiling league tables
\citep{Goldstein1996}. Thus, one of the fundamental tasks of the
epidemiologist is the summarising of data in the form of an ensemble of 
summary statistics, which constitutes an approximation of a `parameter
ensemble'.

Such a task, however, is complicated by the variety of goals that such
an ensemble of point estimates has to fulfil. A primary goal,
for instance, may be to produce element-specific point estimates,
which optimally reflect the individual level of risk in each area. Alternatively,
one may be required to select the ensemble of point estimates that
best approximate the histogram of the true parameter ensemble
\citep{Louis1984}. A related but distinct desideratum may be to choose
the set of point estimates that gives a good approximation of the true
heterogeneity in the ensemble. This is especially important from a
public health perspective since unexplained dispersion in the ensemble of
point estimates may indicate the effect of unmeasured covariates.  
Naturally, there does not exist a set of point estimates, which
simultaneously optimise all of these criteria. However, reporting
several ensembles of point estimates corresponding to different
desiderata can yield to some inconsistencies, which generally lead 
epidemiologists to solely report a single set of point
estimates when communicating their results to
the public or to decision makers. There is therefore a need for
considering how certain ensembles of point estimates can satisfy several 
epidemiological goals at once. 

A natural statistical framework within which these issues can be addressed is
Bayesian decision theory. This approach relies on the
formulation of a particular loss function and the fitting of a Bayesian model to the data of
interest. The former formalises one's inferential goals, whereas the
latter permits to derive the joint posterior distribution of the
parameters of interest, which will then be used for the optimisation of the
loss function. Once these two ingredients are specified, standard arguments
in decision theory imply that an optimal set of point estimates
can be obtained by minimising the posterior expected loss function. 
In spatial epidemiology, the use of Bayesian methods, thanks to the now wide
availability of computational resources, has increasingly become the
procedure of choice for the estimation of small-area statistics
\citep{Lawson2009}. This has paralleled an increase in the
amount of multilevel data routinely collected in most fields of
public health and in the social sciences. 
The expansion of Bayesian methods has especially been motivated by an
increased utilisation of hierarchical models, which are characterised
by the use of hierarchical priors
\citep{Best1996,Wakefield2000,Gelman2004}. 
This family of models has the desirable property of borrowing strength
across different areas, which permits to decrease the variability of
each point estimate in the ensemble.

Such Bayesian models are generally used in conjunction with 
summed of squared error loss (SSEL) function, whose optimal
estimator is the set of posterior means of the elements of the
parameter ensemble of interest. The SSEL is widely used in most
applications because it produces estimators, which are easy to handle
computationally, and often readily available from MCMC
summaries. Despite being the most popular
loss function in both spatial and non-spatial epidemiology, this 
particular choice of loss function, however, remains criticised by
several authors \citep{Lehmann1995,Robert1996}.
In particular, the use of the quadratic loss has been challenged by
researchers demonstrating that the posterior means 
tend to overshrink the empirical distribution of the ensemble's elements
\citep{Louis1984}. The empirical variance of the ensemble of
point estimates under SSEL can indeed be shown to represent only a fraction of the true
empirical variance of the unobserved parameters
\citep{Ghosh1992,Richardson2004}.

Due to these limitations and motivated by a need to produce
parameter ensemble estimates that satisfy other epidemiological
desiderata, several authors have suggested the use of
alternative loss functions. Specifically, \citet{Louis1984} and
\citet{Ghosh1992} have introduced a constrained loss function, which
produce sets of point estimates that match both the empirical mean and
empirical variance of the true parameter ensemble. 
Other authors have made use of the flexibility of the ensemble 
distribution to optimise the estimation of certain parts of the
empirical distribution to the detriment of the remaining ones. This
has been done by specifying a set of weights $\phi$, which emphasise
the estimation
of a subset of the elements of the target ensemble \citep{Wright2003,
  Craigmile2006}. A particularly successful foray in simultaneously
satisfying several inferential objectives was achieved by
\citet{Shen1998}, who proposed the use of triple-goal ensembles of 
point estimates. These sets of point estimates constitute a good
approximation of the empirical distribution of the parameter
ensemble. Moreover, their ranks are close to the optimal ranks. Finally, they also
provide good estimates of element-specific levels of risk. 

Two specific goals, however, do not appear to have hitherto been
fully addressed in the literature on Bayesian
decision theory. These are (i) the estimation of the empirical
quantiles and the empirical quartile ratio (QR) of a parameter ensemble of
interest, and (ii) the optimisation of the classification of the
elements of an ensemble above or below a given threshold. The first
objective lies close to a need of evaluating the amount of dispersion
of a parameter ensemble. The estimation of the QR indeed constitutes
a good candidate for quantifying the variability of the elements in
the ensemble, which can then be related to the presence or absence of
unmeasured risk factor. While some epidemiologists have considered the
problem of choosing a specific measure of dispersion for parameter
ensembles \citep{Larsen2000,Larsen2005}, these methods have been formulated
in a frequentist framework and little work has been conducted from a
Bayesian perspective. 

The second goal, which we wish to explore in this thesis, is the
classification of elements in a parameter ensemble. Increased interest 
in performance evaluation and the development of league tables in
education and health has led to the routine gathering of
surveillance data, which permit to trace the evolution of particular
institutions over time. Despite the range of drawbacks associated with
these methods \citep{Goldstein1996}, a combination of factors has made
the use of surveillance data particularly popular.
Excess mortality rates following paediatric cardiac surgery in Bristol
Royal Infirmary, for instance, has raised awareness about the
significance of this type of data \citep{Grigg2003}.
The Shipman inquiry, in addition, has stressed the need for a
closer monitoring of mortality rates in general practices in the UK
\citep{ShipmanInquiry2004}. These factors have been compounded by
the public and political attention to hospital-acquired
infections such as methicillin resistant \ti{Staphylococcus aureus}
(MRSA) or \ti{Clostridium difficile} \citep{Grigg2007,Grigg2009}. 
Such developments are reflected by the recent creation of 
a new journal entitled \ti{Advances in
Disease Surveillance}, published by the International Society for Diseases
Surveillance in 2007. While some statistical work has focused on
the monitoring of disease counts over time
\citep{Spiegelhalter2005,Grigg2009}, few authors have considered the
problem of classifying the elements of a parameter ensemble in low-
and high-risk groups \citep{Richardson2004}. Such classifications,
however, may be particularly useful for the mapping of risk in
spatial epidemiology, where different groups of areas could be
assigned different colours, according to each group's level of risk. 

However, while the estimation of the empirical quantiles and empirical
QR or the classification of the elements of a parameter ensemble can be
conducted optimally, these optimal estimators may not succeed to meet other
inferential goals. One of the central themes of this thesis will 
therefore be to identify sets of point estimates that 
can simultaneously satisfy several inferential
objectives. For this purpose, we will study the behaviour of
various plug-in estimators under the specific loss functions of
interest. Plug-in estimators are computed by applying a particular function to a 
candidate set of point estimates. In order to evaluate the performance
of each of these candidate ensembles in comparison to the optimal
choice of estimator under the loss functions of interest, we
will compute the posterior regret associated with the use of that
specific candidate plug-in estimator. We will compare different
plug-in estimators using spatial and non-spatial simulations, as well
as two real data sets documenting schizophrenia and MRSA prevalence.

This thesis is organised as follows. In chapter \ref{chap:review}, we
describe the general principles of decision theory and present the
specific modelling assumptions commonly made in epidemiology and
spatial epidemiology. Several loss functions, which have been proposed
for the estimation of parameter ensembles in hierarchical models are introduced
with their respective optimal estimators. Estimators resulting from
these loss functions will be
used throughout the thesis as plug-in estimators under other loss
functions of interest. Chapter \ref{chap:mrrr} is specifically dedicated to the optimisation
of the empirical quantiles and empirical QR of a parameter ensemble,
which permit to evaluate the amount of dispersion in the ensemble
distribution. In chapter \ref{chap:clas}, we consider the
issue of classifying the different elements of a parameter ensemble
above or below a given threshold of risk,
which particularly pertains to the analysis of surveillance data. 
Finally, in chapter \ref{chap:discussion}, we consider some possible
extensions and generalisations of the techniques proposed in this
thesis, with special emphasis on the specification of 
tailored Bayesian models, which may better serve 
the target inferential goals.

%
%
\chapter{Loss functions for Parameter Ensembles}\label{chap:review}
\hspace{2cm}
\begin{minipage}[c]{11.5cm}
\small 
\begin{center}\tb{Summary}\end{center}\vspace{-.3cm}
   In this introductory chapter, we briefly
   introduce the general decision-theoretic framework used in Bayesian
   statistics, with special attention to point estimation problems.
   We present three commonly used loss functions: the squared-error
   loss, the absolute
   error loss and the $0/1$-loss. A working definition of Bayesian
   hierarchical models is then provided, including the description of three
   specific families of models characterised by different types of
   prior structure, which will be used throughout the thesis. The
   concept of a parameter ensemble
   in a hierarchical model is also defined with some emphasis on 
   the optimal estimation of the functions of parameter
   ensembles. Issues related to hierarchical shrinkage are reviewed
   with a brief discussion of the Ghosh-Louis theorem. Several
   commonly adopted decision-theoretic approaches to the estimation of 
   a parameter ensemble are also introduced, including the
   constrained Bayes estimator (CB), the triple-goal estimator (GR)
   and the weighted-rank squared-error loss estimator (WRSEL). Finally,
   we discuss plug-in estimators, which consist of functions of an
   ensemble of point estimates. Differences between such plug-in estimators and
   the optimal estimators of various functions of parameter ensembles
   will be the object of most of the thesis at hand. 
   In particular, the performance of optimal and
   plug-in estimators will be studied within the context
   of two inferential objectives relevant to epidemiological practice: 
   (i) the estimation of the dispersion of parameter ensembles and 
   (ii) the classification of the elements of a parameter ensemble
   above or below a given threshold. We conclude with a description of
   the posterior regret, which will be used throughout the thesis as a
   criterion for performance evaluation. 
\end{minipage}

\section{Bayesian Decision Theory}\label{sec:decision theory}
In this section, we briefly introduce the premises of decision theory
with special attention to point estimation. We also consider the
differences between the frequentist and the Bayesian approach to the
problem of point estimation, and describe three classical loss
functions. A discussion of the specific issues arising from the
estimation of a function of the model's parameters is also given, as
this is especially relevant to the thesis at hand. 
Note that, throughout this chapter and the rest of this thesis, we
will not emphasise the niceties of
measure theory, but restrict ourselves to the level of formality and
the notation introduced by \citet{Berger1980} and
\citet{Robert2007}. Unless otherwise specified, all random variables in
this thesis will be assumed to be real-valued.

\subsection{Premises of Decision Theory}\label{sec:premises}
Decision theory formalises the statistical approach to decision making.
Here, decision making should be understood in its broadest
sense. The estimation of a particular parameter, for instance,
constitutes a decision as we opt for a specific value among a set of
possible candidates. The cornerstone of decision theory is the
specification of a utility function. The main strength and appeal of
decision theory is that once a utility
function has been selected, and a set of alternative options has been
delineated, the decision problem is then fully specified and the
optimal decision can be derived. 

A rigorous approach to decision theory is based on the definition of
three spaces. Firstly, let $\Theta$ denotes the space of the true states
of the world. Secondly, the decision space, denoted $\cD$, will comprise
the decisions --sometimes referred to as acts, actions or choices-- that
one can take. A third space encompasses the consequences of a
particular course of action. These are often expressed in monetary terms,
and for that reasons are generally termed rewards. This last space will be denoted $\cZ$.
The true states of the world, the decision space and the space of
consequences are linked together by a loss function defined as follows, 
\begin{equation}
     L: \Theta\times\cD\mapsto \cZ,
\end{equation}
where $\times$ denotes the Cartesian product.
A loss function is thus a criterion for evaluating a possible
procedure or action $\delta\in\cD$, given some true state of the world
$\theta\in\Theta$. This loss function takes values in the space of consequences $\cZ$.
A decision \ti{problem} is therefore fully specified when the above
four ingredients are known: $(\Theta,\cD,\cZ,L)$. Note that, in
axiomatic treatments of decision theory, the aforementioned quadruple
is generally replaced by $(\Theta,\cD,\cZ,\preceq)$, where $\preceq$ is a
total ordering on the space of consequences. Providing several
properties of $\preceq$ hold (e.g. transitivity, asymmetry), the
existence of a loss function $L$ can
be demonstrated. Decision theory originated in the context of game
theory \citep{Neumann1944}, and was adapted to statistics by
\citet{Savage1954}. The formalisation of a decision problem as the
quadruple $(\Theta,\cD,\cZ,\preceq)$ is probably the most accepted definition 
\citep{Fishburn1964,Kreps1988,DeGroot1970,Robert2007}, although some
other authors also include the set of all possible experiments, $\cE$, in
the definition of a decision problem \citep{Raiffa1960}. 

In this thesis, we will be especially interested in parameter
estimation. Since the space of the true states of the world is the set
of values that the parameter of interest can take, we will simply refer to
$\Theta$ as the parameter space. In estimation problems,
the space of decisions, denoted $\cD$, is generally taken to be
identical to the parameter space, $\Theta$. In addition, both spaces are
also usually defined as subsets of the Real line. That is, we have
\begin{equation}
      \cD = \Theta\subseteq\R. 
\end{equation}
For convenience, we will assume in this section that $\theta$ is
univariate and study the multivariate case in section \label{sec:parameter ensemble}.
Moreover, the space of consequences is defined as the positive half of
the Real line $[0,+\infty)$. Thus, when considering point estimation
problems, such as the ones of interest in the present thesis,
a decision problem will be defined as $(\Theta,\Theta,[0,+\infty),L)$, with $L$
satisfying 
\begin{equation}
     L: \Theta\times\Theta\mapsto [0,+\infty).
\end{equation}
Albeit our discussion has centered on the specification of a
particular loss function, historically, the definition of a total
ordering on $\cZ$ has been associated with the choice of an arbitrary \ti{utility}
function \citep[see][]{Neumann1944}. However, the use of an unbounded
utility leads to several problems, as exemplified by the Saint
Petersburg's paradox. (This paradox involves a
gamble, which results in an infinite average gain, thereby leading to
the conclusion that an arbitrarily high entrance fee should be paid to
participate in the game \citep[see][for a full
description]{Robert2007}.) As a result, the codomain of the utility
function is usually taken to be bounded, with
$U:\Th\times\Th\mapsto(-\infty,0]$. This gives the
following relationship with our preceding definition of the loss
function in estimation problems,
\begin{equation}
   U(\theta,\delta) = -L(\theta,\delta).
\end{equation}

In the context of decision theory, point estimation problems can be naturally
conceived as games, where the player attempts to minimise her losses
\citep[see][]{Berger1980}. Such losses, however, are associated
with some level of uncertainty. That is, there exists a space $\cP$ of
probability distributions on $\cZ$. The total ordering on the space of
consequences $\cZ$ can then be transposed onto the probability space
$\cP$, using the expectation operator. An axiomatic construction of a
total ordering $\preceq$ on $\cZ$ leading to the proof of the
existence of a utility or loss function can be found in several
texts \citep[see][]{Savage1954,Fishburn1964,Kreps1988} or see
\citet{Robert2007} for a modern treatment. The purpose of the game
is therefore to minimise one's loss, by selecting the optimal decision,
$\delta\as$, which is defined as follows,
\begin{equation}
   \delta\as := \argmin_{\delta\in\cD} \E[L(\theta,\delta)],
\end{equation}
where $:=$ indicates that the LHS is defined as the RHS.
Here, the two schools of thought in statistics differ in their
handling of the expectation operator. Frequentist and Bayesian
statisticians have traditionally chosen different types of
expectations, as we describe in the next section. 

\subsection{Frequentist and Bayesian Decision Theories}
From a statistical perspective, the specification of the space of the states of the world
$\Theta$ and the decision space $\cD$ are complemented by the
definition of a sample of observations, which we will denote by 
$\by:=\lb y_{1},\ldots,y_{n}\rb$, where $\by\in\cY\subseteq\R^{n}$. We assume that this sample 
was generated from a population distribution
$p(\by|\theta)$, with $\theta\in \Theta$. Our decision $\delta$ is
therefore computed on the basis of this sample, and we will denote
this dependency by $\delta(\by)$. Following \citet{Robert2007}, we
will use the term \ti{estimator} to refer to the
function $\delta(\cdot)$. That is, an estimator
is the following mapping, 
\begin{equation}
      \delta:\cY \mapsto \cD,
\end{equation}
where $\cD$ is taken to be equivalent to the space of states of
the world, $\Theta$. An \ti{estimate}, by contrast, will be a single value in $\cD=\Theta$,
given some observation $\by$.

From a frequentist perspective, the experiment that has generated the
finite sample of observations $\by$ is assumed to be infinitely
repeatable \citep{Robert2007}. Under this assumption of repeatability, the optimal
decision rule $\delta$ can be defined as the rule that minimises the
expectation of the loss function with respect to the (unknown)
population distribution $p(\by|\theta)$. This gives 
\begin{equation}
    R_{F}[\theta,\delta(\by)]:= 
    \int\limits_{\cY} L(\theta,\delta(\by))p(\by|\theta)d\by.
    \label{eq:freq risk}
\end{equation}
The quantity $R_{F}$ is called the \ti{frequentist risk}, whereas 
$p(\by|\theta)$ can be regarded as the
likelihood function of traditional statistics, evaluated at the true
state of the world, $\theta\in\Th$.

In the context of Bayesian theory, by contrast, we specify a prior distribution on
the space of the states of the world, $\Theta$. This distribution, denoted
$p(\theta)$, reflects our uncertainty about the true state of
the world. Given these assumptions and the specification of a
particular prior, it is then possible to integrate
over the parameter space. This gives the
following Bayesian version of equation (\ref{eq:freq risk}),
\begin{equation}
   R_{B}[p(\theta),\delta(\by)] := \iint\limits_{\Th \cY}
    L(\theta,\delta(\by))p(\by|\theta)p(\theta)d\by\,d\theta,
    \label{eq:bayes risk}
\end{equation}
which is usually termed the \ti{Bayes risk}. 
Note that $R_{B}$ takes the prior distribution $p(\theta)$ as an
argument, since its value entirely depends on the choice of prior
distribution on $\theta$. The optimal decision $\delta^{\as}$ that minimises
$R_{B}$ in equation (\ref{eq:bayes risk}) can be
shown (using Fubini's theorem and Bayes' rule, see \citet{Robert2007} for a proof) to
be equivalent to the argument that minimises the \ti{posterior expected loss}, which
is defined as follows, 
\begin{equation}
   \rho\lt[p(\theta|\by),\delta(\by)\rt]:= 
       \int\limits_{\Th} L(\theta,\delta(\by))p(\theta|\by) d\theta,
       \label{eq:posterior expected loss}
\end{equation}
where $p(\theta|\by)$ is the posterior distribution of $\theta$,
obtained via Bayes' rule, 
\begin{equation}
        p(\theta|\by) = \frac{p(\by|\theta)p(\theta)}{\int_{\Theta}
          p(\by|\theta)p(\theta)d\theta}.
\end{equation}
Since minimisation of $\rho$ for all $y$ is equivalent to the minimisation of
$R_{B}$, we will generally use the terms Bayes risk and posterior
(expected) loss interchangeably. The optimal decision under both the
Bayes risk and the posterior loss is termed the Bayes choice,
action or decision. 

For notational convenience, we will use the following shorthand for
the posterior loss,
\begin{equation}
  \rho\lt[\theta,\delta(\by)\rt] := \rho\lt[p(\theta|\by),\delta(\by)\rt],
\end{equation}
where the dependence of $\rho$ on the posterior distribution
$p(\theta|\by)$, and the dependence of the decision $\delta(\by)$ on the
data are made implicit. In addition, since in this thesis, we will
adopt a Bayesian approach to statistical inference, all expectation
operator will be defined with respect to the appropriate posterior
distribution of the parameter of interest, generally denoted
$\theta$, except when otherwise specified. 

The defining condition for the Bayes choice to exist is that 
the Bayes risk is finite. In the preceding discussion, we have assumed
that the prior distribution on $\theta$ is proper. When specifying improper priors,
the Bayes risk will, in general, not be well defined. 
In some cases, however --when an improper prior yields a
proper posterior-- the posterior expected loss will be well defined. The
decision that minimises $\rho$ in that context is then referred to as
the \ti{generalised Bayes} estimator. Several commonly used models in
spatial epidemiology make use of improper prior distributions, which
nonetheless yield proper posteriors. In such cases, we will be using
the generalised Bayes estimator in the place of the Bayes choice. 

We thus readily see why decision theory and Bayesian statistics form
such a successful combination. The use of a prior distribution in Bayesian models
completes the assumptions about the space of the states of the world
in decision theory. By contrast, the frequentist perspective runs into difficulties 
by conditioning on the true parameter $\theta$. As a result, the
decisions that minimise $R_{F}$ are conditional on $\theta$, which is
unknown. In a Bayesian setting, by contrast, the expected loss is evaluated
conditional on $\by$, which does not constitute a problem since such observations
are known \citep[see also][for a discussion]{Robert2007}.

Decision theory puts no constraints on the nature of the loss function
that one may utilise. It should be clear from the foregoing
discussion that the specification of the loss function will completely
determine which decision in $\cD$ is considered to be optimal,
everything else being equal. In practice, the choice of a particular
loss function is dependent on the needs of the decision maker, and which
aspects of the decision problem are of interest. However, to facilitate comparisons between
different estimation problems, there exists a set of loss functions which
are widely accepted and used throughout the statistical sciences. These
classical loss functions are described in the following section. 

\subsection{Classical Loss Functions}\label{sec:classical loss}
There exist three classical loss functions, which constitute the building blocks of 
many other loss functions, and are therefore of key importance to our
development. They are the following: (i) the squared error loss, (ii)
the absolute value error loss and (iii) the $0/1$-loss. We review them
in turn, alongside their corresponding
minimisers. These three loss functions are especially useful for
estimation problems, and will therefore be described with respect to
some known sample of observations, $\by$. 

\sub{Squared Error Loss}
The quadratic loss or squared error loss (SEL) is the most commonly used loss function in
the statistical literature. It is defined as the Euclidean distance
between the true state of the world $\theta$ and the candidate
estimator $\delta(\by)$. Thus, we have
\begin{equation}
    \op{SEL}(\theta,\delta(\by)) := (\theta - \delta(\by))^{2}.
\end{equation}
In a Bayesian context, this loss function can be minimised by
integrating $\op{SEL}(\theta,\delta(\by))$ with respect to the
posterior distribution of $\theta$, and minimising the posterior expected
loss with respect to $\delta(\by)$. Since the SEL is strictly convex,
there exists a unique minimiser to SEL, which is 
the posterior mean: $\E[\theta|\by]$. We will usually denote $\E[\theta|\by]$ by 
$\theta^{\op{SEL}}$, when comparing that particular estimator with
other estimators based on different loss functions.

\sub{Absolute Value Error Loss}
Another classical loss function is the absolute error loss (AVL), which
makes use of the modulus to quantify one's losses. It is
defined as follows, 
\begin{equation}
    \op{AVL}(\theta,\delta(\by)) := |\theta - \delta(\by)|,
\end{equation}
and its minimiser in the context of Bayesian statistics,
is the posterior median, which will be denoted 
$\theta^{\op{AVL}}$. It can be shown that the posterior median is
not the unique minimiser of the posterior expected AVL
\citep[see][]{Berger1980}. In addition, since the posterior median will be
found to be the optimal Bayes choice under other loss functions, it
will be useful to denote this quantity without any reference to a
specific loss function. We will therefore use $\theta^{\op{MED}}$ to
denote the posterior median in this context. This quantity will be
discussed in greater details in chapter \ref{chap:clas}, when considering
the classification of the elements of a parameter ensemble. 

\sub{$0/1$-Loss}
For completion, we also introduce the $0/1$-loss function. It is defined as follows,
\begin{equation}
     L_{0/1}(\theta,\delta(\by)) = 
    \begin{cases}
        0 & \te{if }\delta(\by) = \theta; \\
        1 & \te{if }\delta(\by) \neq \theta.
    \end{cases}
\end{equation}
The optimal minimiser of the $0/1$-loss under the posterior
distribution of $\theta$ is the posterior mode.
In the sequel, we will drop any reference to the data $\by$, when
referring to the estimator $\delta(\by)$, and simply use the notation
$\delta$. We now turn to a more specific aspect of loss function
optimisation, which will be especially relevant to the research
questions addressed in this thesis. 

\subsection{Functions of Parameters}\label{sec:function of parameters}
It will be of interest to consider the estimation of
functions of parameters, say $h(\theta)$, where $h$ is some mapping 
from $\Theta$ to itself, or possibly to a subset of 
the true states of the world. For convenience, we will assume in this section that
$\theta$ is univariate, and study the multivariate case in section
\ref{sec:parameter ensemble}. In such cases, the Bayesian decision problem can be formulated as
follows. Choose the $\delta^{\as}$, which is defined as,
\begin{equation}
     \delta^{\as} := \argmin_{\delta}\rho\lt(h(\theta),\delta\rt),
\end{equation}
where the minimisation is conducted over the decision space, here
defined as $\cD\subseteq\Th$.
When the loss function of interest is the SEL, we then have the
following equality, for any arbitrary function $h(\cdot)$ of the
parameter $\theta$, 
\begin{equation}
     \E\lt[ h(\theta)|\by\rt] =
     \argmin_{\delta}\E\lt[\lt.\lt(h(\theta)-\delta\rt)^{2}\rt|\by\rt]. 
\end{equation}
That is, the optimal estimator $\delta^{\as}\in\cD$ of $h(\theta)$ is the posterior expectation
of $h(\theta)$. Here, the decision problem is fully specified by
defining both the parameter space $\Theta$ and the decision space
$\cD$, as the codomain of $h(\cdot)$. 

If, in addition, $h(\cdot)$ is a linear function, it then follows from the linearity of the
expectation that 
\begin{equation}
   \E\lt[ h(\theta)|\by\rt] = h(\E\lt[\theta|\by\rt]). 
\end{equation}
However, for non-linear $h(\cdot)$, this relationship
does not hold. Let $\rho$ denotes an arbitrary Bayesian expected loss
for some loss function $L$, we then have the following non-equality, 
\begin{equation}
    \delta^{\as} := 
    \argmin_{\delta} \rho\lt(h(\theta),\delta\rt)
    \neq
    h\lt(\argmin_{\delta}\rho\lt(\theta,\delta\rt)\rt) 
    =: h(\theta^{\as}).
    \label{eq:argmin}
\end{equation}
Much of this thesis will be concerned with the differences between
$\delta^{\as}$ and $h(\theta^{\as})$, for particular $\rho$'s. An estimator of
$h(\theta)$ based on an estimator of $\theta$ will be referred to
as a \ti{plug-in estimator}. Thus, in equation (\ref{eq:argmin}),
$h(\theta^{\as})$ is the plug-in estimator of $h(\theta)$. 
Our use of the term plug-in, here, should be
distinguished from its utilisation in reference to the plug-in
principle introduced by \citet{Efron1993} in the context of the
bootstrap. For \citet{Efron1993}, plug-in estimators have desirable asymptotic
properties in the sense that such estimators cannot be asymptotically
improved upon. In the present thesis, however, estimators such as
$h(\theta^{\as})$ are not deemed to be necessarily optimal, except in the sense
that $\theta^{\as}$ is the Bayes action for some $\rho\lt(\theta,\delta\rt)$. 

Specifically, it will be of interest to evaluate
whether commonly encountered optimal Bayes actions, such as the
posterior means and medians can be used to construct quasi-optimal
plug-in estimators. This problem especially arises when the true
parameter of interest is an ensemble of
parameters, as described in the next section. 

\section{Parameter Ensembles and Hierarchical Shrinkage}\label{sec:shrinkage}
The notion of a parameter ensemble will be the object of most of our
discussion in the ensuing two chapters. In this section, we define
this concept for general Bayesian models as well as for 
specific hierarchical models that will be considered throughout 
the thesis. In addition, we present the Ghosh-Louis theorem, which originally
motivated research on the specific inferential problems associated
with parameter ensembles.
\begin{figure}[t]
\centering
\tikzstyle{background rectangle}=[draw=gray!30,fill=gray!10,rounded corners=1ex]
\begin{tikzpicture}[font=\footnotesize,scale=0.75,show background
  rectangle]

    \draw (-6,0) node[draw,circle, minimum size=1.1cm](Y1){$Y_1$};
    \draw (-4,0) node{$\ldots$};
    \draw (-2,0) node[draw,circle, minimum size=1.1cm](Y+k){$Y_{k-1}$};
    \draw (0,0) node[draw,circle, minimum size=1.1cm](Yk){$Y_k$};
    \draw (2,0) node[draw,circle, minimum size=1.1cm](Y-k){$Y_{k+1}$};
    \draw (4,0) node{$\ldots$};
    \draw (6,0) node[draw,circle, minimum size=1.1cm](Yn){$Y_n$};
    \draw (0,6) node[draw,circle, minimum size=1.1cm](xi){$\bxi$};
    \draw (-6,2.5) node[draw,circle, minimum size=1.1cm](th1){$\theta_1$};
    \draw (-4,2.5) node{$\ldots$};
    \draw (-2,2.5) node[draw,circle, minimum size=1.1cm](th+k){$\theta_{k+1}$};
    \draw (0,2.5) node[draw,circle, minimum size=1.1cm](thk){$\theta_k$};
    \draw (2,2.5) node[draw,circle, minimum size=1.1cm](th-k){$\theta_{k-1}$};
    \draw (4,2.5) node{$\ldots$};
    \draw (6,2.5) node[draw,circle, minimum size=1.1cm](thn){$\theta_n$};
    \draw[->] (th1) -- (Y1);
    \draw[->] (th-k) -- (Y-k);
    \draw[->] (thk) -- (Yk);
    \draw[->] (th+k) -- (Y+k);
    \draw[->] (thn) -- (Yn);

    \draw[->] (xi) -- (th1);
    \draw[->] (xi) -- (th-k);
    \draw[->] (xi) -- (thk);
    \draw[->] (xi) -- (th+k);
    \draw[->] (xi) -- (thn);
\end{tikzpicture}
\caption{Directed Acyclic Graph (DAG) for a general hierarchical
  model, with $\by=\lt\lb y_{1},\ldots,y_{n}\rt\rb$ denoting $n$
  observations and $\bth=\lt\lb \theta_{1},\ldots,\theta_{n}\rt\rb$ denoting
  the parameter ensemble of interest. The prior distribution of each
  $\theta_{i}$ is controlled by a vector of hyperparameter $\bm\xi$, which is given
  a hyperprior. Following standard convention, arrows here indicate direct
  dependencies between random variables.}
\label{fig:bhm}
\end{figure}
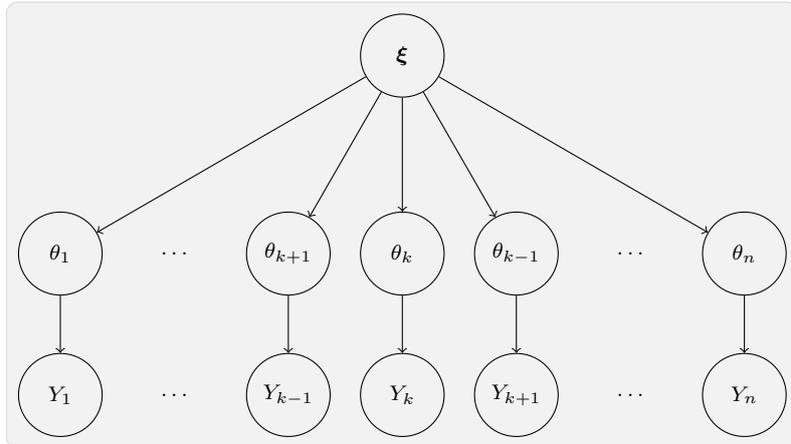

\subsection{Bayesian Hierarchical Models}\label{sec:bhm}
Bayesian hierarchical models (BHMs) can be conceived as Bayesian
extensions of traditional mixed models, sometimes called
multilevel models, where both fixed and random effects are included in a
generalized linear model \citep[see][]{Demidenko2007}. 
In its most basic formulation, a BHM is composed of the following two
layers of random variables,
\begin{equation}
      y_{i}\stack{\ind}{\sim} p(y_{i}|\theta_{i},\bsig_{i}), 
      \qq 
      g(\bth) \sim p(\bth|\bxi), 
     \label{eq:bhm}
\end{equation}
for $i=1,\ldots,n$ and where $g(\cdot)$ is a transformation of $\bth$,
which may be defined as a link function as commonly used
in generalised linear models \citep[see][]{McCullagh1989}. The
probability density function at the first level of the BHM,
$p(y_{i}|\theta_{i},\bsig_{i})$, is the likelihood function. The
joint distribution on the $g(\bth)$, here denoted $p(\bth|\bxi)$, is a 
multivariate prior, which is directly specified as a distribution on
the transformed $\bth$. Different choices of transformations,
$g(\cdot)$, will hence induce different joint prior distributions. 
The specification of a BHM is complete when one
fixes values for the vector of hyperparameters, $\bxi$, or specifies a
hyperprior distribution for that quantity. When the $\theta_{i}$'s are independent and
identically distributed (iid) given $\bxi$, we obtain the
hierarchical structure described in figure \ref{fig:bhm} on page
\pageref{fig:bhm}, which is represented using a directed acyclic graph
(DAG). By extension, we will refer to the prior on such $\theta_{i}$'s
as an iid prior.

In the sequel, we will assume that the $n$ vectors of nuisance parameters
$\bsig_{i}$ are known, albeit in practice, this may not be the case.
Typically, such $\bsig_{i}$'s will include the sampling
variance of each $y_{i}$. In its simplest form, the model in equation
(\ref{eq:bhm}) will be assumed to
be composed of conjugate and proper probability density
functions and each $\theta_{i}$ will be an iid draw from a hierarchical prior. 
These modelling assumptions, however, will be sometimes relaxed
in the present thesis. In particular, we will be interested in
the following three variants.
\begin{description}
   \item[i.] \ti{Conjugate proper iid} priors on the $\theta_{i}$'s, where 
     the link function $g(\cdot)$ is the identity function.
     This simple case will encompass both the Normal-Normal model
     --sometimes referred to as compound Gaussian model-- and the 
     compound Gamma or Gamma-Inverse Gamma model. For the former, we
     have the following hierarchical structure, 
     \begin{equation}
       y_{i} \stack{\ind}{\sim}N(\theta_{i},\sig^{2}_{i}), \qq
       \theta_{i}\stack{\iid}{\sim} N(\mu,\tau^{2}),  
       \label{eq:normal-normal}
     \end{equation}
     with $i=1,\ldots,n$, and where the $\sig^{2}_{i}$'s are known
     variance components. Secondly, for the Gamma-Inverse Gamma model, we
     specify the following likelihood and prior, 
     \begin{equation}
        y_{i} \stack{\ind}{\sim}\op{Gam}(a_{i},\theta_{i}), \qq
       \theta_{i}\stack{\iid}{\sim} \op{Inv-Gam}(\alpha,\beta),        
        \label{eq:gamma-inverse gamma} 
     \end{equation}
     where the Gamma and Inverse-Gamma distributions will be specifically
     described in section \ref{sec:mrrr non-spatial design}.
  \item[ii.] \ti{Non-conjugate proper iid} priors on the $\theta_{i}$'s.
     In standard epidemiological settings, the specification of a
     Poisson likelihood for the observations naturally leads to 
     the modelling of the logarithm of the relative risks at the
     second level of the hierarchy. Such a model is sometimes referred
     to as a log-linear model. In such cases, however, the conjugacy of
     the likelihood with the prior on the $\theta_{i}$'s 
     does not hold anymore. As an example of such non-conjugacy,
     we will study the following model,
     \begin{equation}
           y_i\stack{\ind}{\sim} \poi(\theta_iE_{i}), 
           \qq
           \log\theta_i\stack{\iid}{\sim} N(\al,\tau^2),
     \end{equation}
     for $i=1,\ldots,n$, where the link function $g(\cdot):=\log(\cdot)$.
     Here, the conjugacy of the prior with the Poisson likelihood does not
     hold, and specific sampling schemes need to be adopted in order
     to evaluate the posterior distributions of the $\theta_{i}$'s
     \citep[see][]{Robert2004}. 
  \item[iii.] \ti{Non-conjugate improper non-iid} priors on the $\theta_{i}$'s.
     In this final situation, all assumptions on the prior
     distributions of the $\theta_{i}$'s will be relaxed. This type of
     model will be illustrated in the context of spatial epidemiology,
     where Poisson-based generalized linear models 
     are commonly used to model counts of disease cases in each of a
     set of geographical areas, and the joint prior distribution on the
     $\theta_{i}$'s models the spatial dependence between the regions
     of interest. A popular choice of prior distribution reflecting
     inter-regional spatial dependence is the intrinsic Gaussian conditional
     autoregressive (CAR) prior --an instance of a Markov
     random field. The intrinsic CAR does not constitute a
     proper distribution \citep{Besag1974}. However, \citet{Besag1991}
     have shown that the resulting marginal posterior distributions of the
     $\theta_{i}$'s are proper \citep[see also][]{Besag1995}. 
     It then follows that the posterior expected loss for the
     parameter ensemble is finite, and that we can therefore derive
     the generalised Bayes
     decision, which will be optimal for that decision problem. 
     We describe this particular family of spatial models in more detail in the
     following section.
\end{description}
In this thesis, when studying the properties of such epidemiological
models, we will employ the term \ti{prevalence} to refer to the
rate of a particular condition in the population at risk. That is, the
term prevalence here refers to the number of affected cases divided by the
total number of individuals at risk for that condition. 

\subsection{Spatial Models}\label{sec:spatial models}
Here, we briefly present some conventional modelling assumptions made
in spatial epidemiology, which will be used throughout the thesis
\citep{Wakefield2000}. The
starting point for modelling a non-infectious disease with known
at-risk populations is the Binomial model,
\begin{equation}
  y_{\hi{ij}}|p_{\hi{ij}} \sim \op{Bin}(N_{\hi{ij}}, p_{\hi{ij}}), \qq
  i=1,\ldots,n; \q j=1,\ldots,J,
\end{equation}
where $p_{\hi{ij}}$ and $N_{\hi{ij}}$ represent, respectively, the risk of disease
and the population size in the $i\tth$ area for the $j\tth$ age strata. 
In this thesis, we are especially interested in modelling the
prevalence of rare non-infectious diseases, such as specific types of cancers.
For rare conditions, we approximate this model with a Poisson model
\begin{equation}
  y_{\hi{ij}}|p_{\hi{ij}} \sim \op{Pois}(N_{\hi{ij}}
  p_{\hi{ij}}). 
\end{equation}
Furthermore, we generally make a proportionality assumption, which
states that $ p_{\hi{ij}} = \theta_{i} \times p_{j}$,
where $\theta_{i}$ is the relative risk (RR) for the $i\tth$ area, and $p_j$ is the reference rate
for age strata $j$, which is assumed to be known. Each $\theta_{i}$ is here the ratio
of the age-standardised rate of the disease in area $i$ compared to
the age-standardised reference rate. Using the proportionality
assumption, we may then sum over the risk in each strata in order to obtain 
\begin{equation}
    y_i \stack{\ind}{\sim} \poi(\theta_iE_{i}),
    \label{eq:car likelihood}
\end{equation}
where $y_i=\sum_{j=1}^{J} y_{\hi{ij}}$ and
$E_{\hi{i}}=\sum_{j=1}^{J}N_{\hi{ij}}p_{j}$ are the observed and
expected counts, respectively. 
Equation (\ref{eq:car likelihood}) is the likelihood of the
model. Maximum likelihood, in this context, produces the following point estimate
ensemble, 
\begin{equation}
      \hat\theta^{\op{MLE}}_{i}:=\frac{y_{i}}{E_{i}}, 
      \label{eq:smr}
\end{equation}
for every $i=1,\ldots,n$, which are 
generally referred to as the standardised mortality or morbidity ratios (SMRs). The
$\hat\theta_{i}^{\op{MLE}}$'s, however, tend to be over-dispersed in
comparison to the true RRs. 

In order to provide such a model with more
flexibility, different types of hierarchical priors are commonly specified on the
$\theta_{i}$'s \citep[see][for a review]{Wakefield2000}.
Two spatial BHMs will be implemented in this
thesis. A popular hierarchical prior in spatial
epidemiology is the convolution prior \citep{Besag1991}, which is
formulated as follows,
 \begin{equation}
    \log\theta_i = v_i + u_i, 
      \label{eq:linear theta}
\end{equation}
for every region $i=1,\ldots,n$. Note, however, that this model will be
implemented within the WinBUGS software, which uses a different
representation of the spatial random effects, based on the joint
specification of an intercept with an improper flat prior and a
sum-to-zero constraint on the $u_{i}$ (see Appendix \ref{app:winbugs},
and section \ref{sec:mrrr spatial models}).

Here, $\bv$ is a vector of unstructured random effects with the
following specification,
\begin{equation}
    v_i \stack{\iid}{\sim} N(0,\tau_v^2),
    \label{eq:unstructured}
\end{equation}
and the vector $\bu$ captures the spatial
auto-correlation between neighbouring areas. Priors on each element
of $\bu$ are specified conditionally, such that 
\begin{equation}
       u_i|u_j, \forall \, j\neq i \sim N\lt( \frac{\sum_{j \in \partial_i} u_j}{m_i},
       \frac{\tau_u^2}{m_i}\rt),
       \label{eq:car}
\end{equation}
with $\partial_{i}$ is defined as the set
of indices of the neighbours of the $i\tth$ area. Formally,
$\partial_{i}:=\lb j\sim i: j=1,\ldots,n\rb$, where $i\sim j$ implies
that regions $i$ and $j$ are neighbours, and by convention,
$i\notin \partial_{i}$. Moreover, in equation (\ref{eq:car}), we have
also used $m_{i}:=|\partial_{i}|$ --that is, $m_{i}$ is the total
number of neighbours of the $i\tth$ area. Therefore, each of the
$u_{i}$'s is normally distributed around the mean level of risk
of the neighbouring areas, and its variability is inversely
proportional to its number of neighbours.  

A different version of this model can be formulated using the
Laplace distribution \citep{Besag1991}. As this
density function has heavier tails, this specification is expected
to produce less smoothing of abrupt changes in risk between adjacent
areas. For the CAR Laplace prior, we therefore have
$u_i|u_j, \forall \, j\neq i \sim L(u_{i}|\sum_{j \in \partial_i}
u_j/m_i,\tau_u^2/m_i)$, for every $i=1,\ldots,n$, using the following
definition of the Laplace distribution,
\begin{equation}
     L(x|m_{0},s_{0}) := \frac{1}{2s_{0}}\exp \lt\lb \frac{1}{s_{0}}|m_{0}-x|\rt\rb.
\end{equation}
We will refer to this model as the CAR Laplace or L1 model.
The rest of the specification of this BHM is identical to the one chosen
for the CAR Normal. In these two models, following common
practice in spatial epidemiology, two regions are considered to be neighbours if
they share a common boundary \citep{Clayton1987,Besag1991,Waller1997}.

\subsection{Estimation of Parameter Ensembles}\label{sec:parameter ensemble}
The set of parameters, $\theta_{i}$'s, in the hierarchical model described in
equation (\ref{eq:bhm}) is generally referred to as
a vector of random effects \citep{Demidenko2007}. In this thesis, we will refer to such a vector
as an ensemble of parameters. That is, in a BHM following the general
structure presented in equation (\ref{eq:bhm}), the vector of
parameters,
\begin{equation}
     \bth:=\lb\theta_{1},\ldots,\theta_{n}\rb,
\end{equation}
will be referred to as a \ti{parameter ensemble}.
Several properties of a parameter ensemble may be of
interest. One may, for instance, wish to optimise the estimation of
each of the individual elements in the ensemble. A natural choice in
this context is the sum of quadratic losses for each parameter. This
particular loss function is the summed squared error loss (SSEL)
function that takes the following form, 
\begin{equation}
  \ssel\lt( \bth,\bth^{\op{est}} \rt) =
  \sum_{i=1}^{n}\lt(\theta_i-\theta^{\op{est}}_i\rt)^2,
  \label{eq:ssel}
\end{equation}
 In this context, using the notation that we have adopted in section
\ref{sec:premises}, the decision problem for the estimation of a
parameter ensemble using the SSEL function results in a parameter
space $\bTh$ and a decision space $\cD$ which are both assumed to be subsets of 
$\R^{n}$. The posterior expected loss associated with
the loss function in equation \ref{eq:ssel} can be minimised in a
straightforward manner. Its optimum is attained when selecting the vector of posterior
means as Bayes choice, which will be denoted by 
\begin{equation}
     \hat\bth^{\op{SSEL}} := 
     \lb \hat\theta^{\op{SSEL}}_{1},\ldots,\hat\theta^{\op{SSEL}}_{n}\rb = 
     \lb \E[\theta_{1}|\by],\ldots,\E[\theta_{n}|\by] \rb.
\end{equation}
This type of decision problems, characterised by the estimation of a
set of parameters are sometimes referred to as compound estimation
problems or \ti{compound loss functions} \citep{Ghosh1992}. We will
use of the term \ti{estimator} to refer
to an entire set of point estimates, such as $\hat\bth^{\op{SSEL}}$
with respect to a particular loss function, here SSEL. This indeed
follows from the fact that such a multivariate estimator is optimal
under that posterior loss. Such an estimator, however, also
constitutes an ensemble of single \ti{point estimates}.

Another property of a parameter ensemble, which may be of interest is
the empirical distribution function (EDF) of that ensemble, which will
generally be referred to as the ensemble distribution. The EDF of
$\bth$ is defined as follows, 
\begin{equation}
     F_{n}(t) := \frac{1}{n}\sum_{i=1}^{n} \cI\lb \theta_{i}\leq t\rb, 
     \label{eq:edf}
\end{equation}
for every $t\in\R$, and where $\cI$ is the usual indicator function. 
The term EDF is used here in analogy with its usual application in the
context of iid observations. Note, however, that neither the elements
of a true parameter ensemble nor the realisations of random effects
can (in general) be assumed to be realisations of iid variables. 
A range of different loss functions may be
considered in order to optimise the estimation of the EDF in equation
(\ref{eq:edf}). Previous authors have formalised this decision problem
by using the integrated squared error loss function (ISEL), which
takes the following form \citep[see][]{Shen1998}, 
\begin{equation}
    \op{ISEL}(F_{n},F_{n}^{\op{est}}):= \int (F_{n}(t)-F^{\op{est}}_{n}(t))^{2}dt.
    \label{eq:isel}  
\end{equation}
The posterior expectation of the ISEL can be easily minimised using
Fubini's theorem to invert the ordering of the two integrals (i.e.
the one with respect to $t$, and the one with respect to $\bth$, see
\citet{Shen1998}, for a formal proof). It
then follows that the optimal estimator of
$\E[\op{ISEL}(F_{n},F_{n}^{\op{est}})|\by]$ is the posterior EDF, 
\begin{equation}
      \widehat{F}_{n}(t) := \E[F_{n}(t)|\by] = 
          \frac{1}{n}\sum_{i=1}^{n} \p[\theta_{i}\leq t |\by].
          \label{eq:isel minimiser}
\end{equation}
When there may be ambiguity as to which parameter ensemble the
posterior EDF is based on, we will emphasise the dependence on the
vector $\bth$, by using $\widehat{F}_{\bth}(t)$.

Finally, as discussed in section \ref{sec:function of parameters}, one
is sometimes interested in specific functions of a parameter. For the case of
parameter ensembles, summary functions are often used to quantify
particular properties of the ensemble distribution, such as the variance of the
ensemble, for instance. These functions are generally real-valued, and
the decision problem, in this context, can therefore be formalised as the
standard quadruple: $(\R,\R,[0,+\infty),L)$, for some loss function
$L$. When using the quadratic loss, we may have
\begin{equation}
     \op{SEL}(h(\bth),\delta) = (h(\bth)-\delta)^{2},
\end{equation}
for a function of interest, $h:\bTh \mapsto \Theta$,
with $\bTh\subseteq\R^{n}$ and $\Theta\subseteq\R$. One may, for
instance, wish to estimate the empirical variance
of the parameter ensemble. This gives the following $h(\cdot)$, 
\begin{equation}
      h(\bth) := \frac{1}{n}\sum_{i=1}^{n}(\theta_{i}-\bar\theta)^{2},
\end{equation}
where $\bar{\theta}:=n^{-1}\sum_{i=1}^{n}\theta_{i}$.
Naturally, in this particular case, the optimal Bayes estimator would be
$\E[h(\bth)|\by]$. However, as argued in section \ref{sec:function of parameters},
since $h(\cdot)$ is a non-linear function of $\bth$, the posterior
empirical variance is different from the empirical variance
of the posterior means. It turns out that
this particular observation has led to a key result in the study of
BHMs, as we describe in the next section. 

\subsection{Hierarchical Shrinkage}
Although the vector of posterior means is optimal under SSEL,
their empirical variance is biased in
comparison to the empirical variance of the true parameters of interest. 
The theorem that shows this systematic bias 
was first proved by \citet{Louis1984} for the Gaussian compound case
and further generalised by \citet{Ghosh1992}, who relaxed the
distributional assumptions, but retained the modelling assumptions. In
particular, whereas \citet{Louis1984} proved this result for the case
of conjugate models composed of Normal distributions,
\citet{Ghosh1992} showed that this relationship also holds
for non-conjugate models based on arbitrary probability densities.
\begin{thm}[\tb{Ghosh-Louis Theorem}]\label{thm:gl}
  Let a parameter ensemble $\bth$ controlling the distribution of a
  vector of observations $\by$, in a general two-stage hierarchical
  model as described in equation (\ref{eq:bhm}). If $n\geq 2$, then  
  \begin{equation}
      \E\lt[\lt. \frac{1}{n}\sum_{i=1}^{n} (\theta_i - \bar\theta)^2\rt|\by\rt] \geq
      \frac{1}{n}\sum_{i=1}^{n}\lt[\E[\theta_i|\by] - \frac{1}{n}\sum_{i=1}^{n}\E[\theta_i|\by]\rt]^2,
      \label{gl1}
  \end{equation}
  where $\bar\theta := n^{-1}\sum_{i=1}^{n} \theta_i$;
  with equality holding if and only if 
  all $\lb(\theta_1-\bar\theta), \ldots, (\theta_n-\bar\theta)\rb$ have
  degenerate posteriors.
\end{thm}
A proof of this result is provided in \citet{Ghosh1992}, and a
weighted version is presented in \citet{Frey2003}. The Ghosh-Louis
theorem bears a lot of similarity with the \ti{law of total variance}, 
which posits that $\var(X) = \E\lt[\var\lt(X|Y\rt)\rt] +
\var\lt(\E\lt[X|Y\rt]\rt)$. One
may, however, note that there exists one substantial difference
between this standard law of probability and  the result at hand. 
The Ghosh-Louis theorem differs from the law of total
probability in the sense that the former is conditioning on the data
on both sides of the equation.

The Ghosh-Louis theorem states a general property of
BHMs. Hierarchical shrinkage is the under-dispersion of the empirical
distribution of the posterior means in comparison to the posterior
mean of the empirical variance of the true parameter ensemble. This
should be contrasted to the commonly encountered issue of shrinkage in Bayesian model, where
a single posterior mean is shrank towards its prior mean. 
Although hierarchical shrinkage is here presented as a problem, it is
often seen as a desirable property of BHMs;
most especially when little information is available for each data
point. This theorem has been used both to justify such a modelling
decision and to highlight the limitations of this choice. 
In spatial epidemiological settings, Gelman and Price
\citeyearpar{Gelman1999} have shown that such shrinkage especially affects areas with
low expected counts. Before reviewing the different decision-theoretic
solutions that have been proposed to produce better estimates of the
empirical properties of parameter ensembles, we briefly introduce
several statistical concepts, which will be useful in the sequel.

\section{Ranks, Quantile Functions and Order Statistics}\label{sec:preliminaries}
\subsection{Ranks and Percentile Ranks}\label{sec:ranks}
We here adopt the nomenclature introduced by \citet{Laird1989} on order
statistics and rank percentiles as a guidance for our choice of
notation. Of particular importance to our development is the definition
of a rank. The rank of an element in a parameter ensemble is
defined as follows,
\begin{equation}\label{eq:rank}
    R_{i}(\bth) := \op{rank}(\theta_{i}|\bth) =
   \sum_{j=1}^{n} \cI\lt\lb \theta_{i}\geq \theta_{j}\rt\rb,  
\end{equation}
where the smallest $\theta_{i}$ in $\bth$ is given rank 1 and the largest $\theta_{i}$
is given rank $n$. In equation (\ref{eq:rank}), we have made explicit the fact that the rank of each
$\theta_{i}$ depends on the entire parameter ensemble, $\bth$. In particular, this
notation emphasises the fact that the function $R_{i}(\cdot)$ is
non-linear in its argument. 

In the sequel, we will also make extensive use of percentile ranks
(PRs). These are formally defined as follows, 
\begin{equation}
   \label{eq:percentile}
   P_{i}(\bth) := \frac{R_{i}(\bth)}{n+1},
\end{equation}
Quite confusingly, percentile ranks are sometimes referred to as
``percentiles'', which should not be confused with the notion of percentiles
discussed in section \ref{sec:quantile function} in reference to quantiles.
In general terms, the percentile rank of an element in a parameter ensemble is the
percentage of elements in the corresponding EDF, which are lower or
equal to the value of that rank. Rank percentiles are especially useful
in statistics when communicating ranking statistics to
practitioners. Percentile ranks are empirical quantities, in the sense
that they depend on the size of the parameter ensemble, $n$. 
However, it can be shown that percentile ranks rapidly converge to
asymptotic quantities, which are independent of $n$ \citep{Lockwood2004}. 

\subsection{Quantile Function}\label{sec:quantile function}
The quantile function, denoted $Q(p)$, of a continuous random variable $X$ is
formally defined as the inverse of the cumulative distribution
function (CDF), $F(x)$. Formally, since the function
$F:\R\mapsto[0,1]$ is continuous and strictly monotonic, we can define 
\begin{equation}
   Q(p) := F^{-1}(p), 
\end{equation}
for some real number $p\in [0,1]$. In general, the inverse function of the CDF
of most random variables does not exist in closed-form. Among the rare 
exceptions are the uniform distribution, $\op{unif}(x|a,b)$, and the exponential
distribution, $\op{exp}(x|\lambda)$, for which the quantile functions
are $Q(p|a,b)=(1-p)a + pb$ and $Q(p|\lambda)=-\log(1-p)/\lambda$, respectively \citep[see][]{Gilchrist2000}. 
For discrete random variables, whose CDF may only be weakly monotonic,
the quantile distribution function (QDF) is defined as 
\begin{equation}
    Q(p) := \inf\lt\lb x \in \R: F(x)\geq p \rt\rb,
    \label{eq:qdf}
\end{equation}
for every $p\in[0,1]$. Note that the latter definition holds for any
arbitrary CDF, whether continuous or discrete. Like the CDF, the QDF is monotone
non-decreasing in its argument. It is continuous for continuous random variables, and
discrete for discrete random variables. However, whereas the $F(x)$ is
right-continuous, its inverse is, by convention, continuous from the
left. This last property is a consequence of the use of the infimum in equation
\ref{eq:qdf}. When it is deemed necessary for clarity, we will specify the
random variable for which a quantile is computed by a subscript, such
as in $Q_{X}(p)$, for the random variable $X$. More generally, we will 
distinguish between the \ti{theoretical} QDF and the \ti{empirical}
QDF, by denoting the former by $Q(p)$ and the latter by $Q_{n}(p)$.
For some parameter ensemble $\bth$ of size $n$, we define the
empirical QDF as follows,
\begin{equation}
    Q_{n}(p) := \min\lt\lb \theta_{1},\ldots,\theta_{n}:
                F_{n}(\theta_{i})\geq p\rt\rb,
\end{equation}
where $F_{n}$ is the EDF of the parameter ensemble as defined in
equation (\ref{eq:edf}). Note that our utilisation of the EDF, in this
context, corresponds to a slight abuse of the concept. As aforementioned,
the EDF of a particular random variable,
assumes that several iid realisations of that random variable are
available. In our case, given the hierarchical nature of the models
under scrutiny, and the possible spatial structure linking the different
geographical units of interest, such an iid assumption is not
satisfied. Our use of the term empirical distribution
should therefore be understood solely in reference to our
chosen mode of construction for the discrete distribution of an
ensemble of realisations.

When we wish to emphasise the vector of
parameters with respect to which the QDF and EDF are defined, we will
use the notation, $Q_{\bX}$ and $F_{\bX}$, respectively. Sometimes, it
will be useful to allow the function $Q(\cdot)$ to accept multivariate
arguments, such as $\bp:=\lb p_{1},\ldots,p_{k}\rb$. In such cases,
the QDF will be vector-valued, such that 
\begin{equation}
    Q(\bp) = Q(\lb p_{1},\ldots,p_{k}\rb) = \lb
    Q(p_{1}),\ldots,Q(p_{k})\rb.
   \label{eq:qdf vector}
\end{equation}

In contrast to the sample mean, for which the relationship
$\E[g(X)]=g(\E[X])$ only holds for linear function $g(\cdot)$, the
quantile function satisfies 
\begin{equation}
     Q_{h(X)}(p) = h[Q(p)], 
\end{equation}
for every monotonic non-decreasing function $h(\cdot)$. 
In particular, we have the following useful transformation 
\begin{equation}
    Q_{\log(X)}(p) = \log\lt[ Q_{X}(p)\rt],
    \label{eq:log quantile}
\end{equation}
which follows from the monotonicity of the logarithm. 
Moreover, we can recover the mean of $X$ from its quantile function, by
integrating the quantile function over its domain. That is, 
\begin{equation}
     \int_{[0,1]}Q_{X}(p) dp = \E[X],
\end{equation}
which is a standard property of the quantile function
\citep[see][]{Gilchrist2000}.
When considering the quantile function of the standard cumulative
normal distribution, denoted $\Phi(x)$, the quantile function
specialises to the \ti{probit}, defined as $\op{probit}(p):=
\Phi^{-1}(p)$. The probit is widely used in econometrics as a link
function for the generalised linear model when modelling binary
response variables. 

\subsection{Quantiles, Quartiles and Percentiles}\label{sec:quantiles}
In this thesis, we will be mainly concerned with the
quantiles of parameter ensembles, these quantities will therefore be defined
with respect to the EDF of a vector $\bth$. In such cases, the
$p\tth$ empirical \ti{quantile} of the ensemble $\bth$ is formally defined as
\begin{equation}
     \theta_{(p)} := Q_{\bth}(p),
     \label{eq:quantile}
\end{equation}
where the $.50\tth$ quantile is the empirical median. In the sequel, we will also make use
of the $.25\tth$, $.50\tth$ and $.75\tth$ quantiles, which are referred to as the
first, second and third empirical \ti{quartiles}, respectively. Of particular interest is the
difference between the third and first quartiles, which produces the
empirical interquartile range (IQR). For some parameter ensemble
$\bth$, we have
\begin{equation}
    \op{IQR}(\bth) := \theta_{(.75)}-\theta_{(.25)}. 
    \label{eq:iqr}
\end{equation}
In some texts, the term quantile is used to classify different types
of divisions of the domain of a probability distribution. More
precisely, the $k\tth$ $q$-quantile is defined as
$Q_{\bth}(k/q)$. In that nomenclature, the \ti{percentiles} therefore correspond
to the 100-quantiles, which would imply that the values taken by the quantile
function are, in fact, percentiles. Other authors, however, have used
the term quantile in a more straightforward manner, where the $p\tth$
quantile is simply defined as $Q_{\bth}(p)$, as in equation
(\ref{eq:quantile}) \citep[see, for instance,][]{Koenker1978}. In the
sequel, we adopt this particular definition of quantiles and only
use the term percentiles to refer to percentile ranks, as introduced
in section \ref{sec:ranks}. 

There exist different techniques to derive the quantile function and
quantiles of a particular CDF \citep[see][for recent
advances]{Steinbrecher2008}. One of the most popular methods has been
the algorithm AS 241 \citep{Wichura1988}, which permits to compute the
empirical $p\tth$ quantile of any finite parameter ensemble very
efficiently. We will make use of this standard computational technique in chapter
\ref{chap:mrrr}. 

\section{Loss Functions for Parameter Ensembles}\label{sec:loss for
  ensemble}
In this section, we review three important approaches, which have been
proposed to address the issue of hierarchical shrinkage that we have 
highlighted in the previous section. These three decision-theoretic
perspectives will be specifically considered in chapters \ref{chap:mrrr} and
\ref{chap:clas}, and compared with other methods through the
construction of plug-in estimators. Here, the parameter
space satisfies $\bTh\subseteq \R^{n}$, and $\cD=\bTh$, as before. 

\subsection{Constrained Bayes}
In order to address the problem associated with hierarchical shrinkage,
described in section \ref{sec:shrinkage}, 
\citet{Louis1984} developed a particular set of point estimates that
automatically correct for hierarchical shrinkage. 
This approach is a constrained minimisation problem, where 
$\ssel(\bth,\bth^{\op{est}}) = \sum_{i=1}^{n}( \theta_i - \theta^{\op{est}}_i)^2$
is minimised subject to the following two constraints: (i) the mean of the
ensemble of point estimates is equal to the mean of the true ensemble, 
\begin{equation}
     \bar\theta^{\op{est}} :=
     \frac{1}{n}\sum_{i=1}^{n}\theta^{\op{est}}_{i} =
     \frac{1}{n}\sum_{i=1}^{n}\theta_{i}=:\bar\theta,   
     \label{eq:constrained loss eq1}
\end{equation}
and (ii) the variance of the ensemble of point estimates is equal to the
variance of the true ensemble, 
\begin{equation}
      \frac{1}{n}\sum_{i=1}^{n}\lt(\theta^{\op{est}}_{i} -
      \bar\theta^{\op{est}}\rt)^{2} = 
      \frac{1}{n}\sum_{i=1}^{n}\lt(\theta_{i} -
      \bar\theta\rt)^{2}.
     \label{eq:constrained loss eq2}
\end{equation}
Based on these constraints, \citet{Louis1984} derived the optimal
Bayes estimator minimising the corresponding constrained SSEL
function. Since this method is developed within a Bayesian framework,
it is generally referred to as the constrained Bayes (CB) or constrained
empirical Bayes (EB) estimator, albeit its use is not restricted to models
estimated using EB techniques \citep[see][]{Rao2003}. 
This result was further generalised by \citet{Ghosh1992} for any
distribution belonging to the exponential family. 
\begin{thm}[Ghosh-Louis Estimator]
  The minimiser of the SSEL under the constraints in equation
  (\ref{eq:constrained loss eq1}) and (\ref{eq:constrained loss eq2}) 
  is
  \begin{equation}
    \hat\theta^{\op{CB}}_i :=  \ome \hat\theta^{\op{SSEL}}_i +
    (1-\ome)\hat{\bar\theta}^{\op{SSEL}}, \label{eq:cb}
  \end{equation}
  where $\hat{\theta}^{\op{SSEL}}_{i}:=\E[\theta_i|\by]$, for every $i=1,\ldots,n$,
  and $\hat{\bar\theta}^{\op{SSEL}}:= 1/n\sum_{i=1}^{n}\hat{\theta}_{i}^{\op{SSEL}}$ 
  is the mean of the empirical distribution of posterior means. The
  weight $\ome$ is defined as
  \begin{equation}
    \ome = \lt[ 1+ \frac{n^{-1}\sum_{i=1}^{n}\var\lt[\theta_i|\by\rt]
    }{n^{-1}\sum_{i=1}^n\lt(\hat\theta^{\op{SSEL}}_i -
      \hat{\bar\theta}^{\op{SSEL}}\rt)^2} \rt]^{1/2},
  \end{equation}
  where $\var[\theta_{i}|\by]$ is the posterior variance of the
  $i\tth$ parameter in the ensemble. 
\end{thm}
\begin{proof}
  The minimisation of the posterior expected loss is a constrained
  minimisation problem, which can be solved using Lagrange
  multipliers \citep[see][page 221]{Rao2003}. 
\end{proof}
The role played by the weight $\omega$ is more explicitly demonstrated by
transforming equation (\ref{eq:cb}), in order to obtain the following
\begin{equation}
    \hat\theta^{\op{CB}}_i =  
    \hat{\bar\theta}^{\op{SSEL}} + 
    \ome\lt(\hat\theta^{\op{SSEL}}_i - \hat{\bar\theta}^{\op{SSEL}}\rt). 
\end{equation}
In addition, note that the expression in equation
(\ref{eq:cb}) resembles a convex combination of the corresponding SSEL
estimator of the $i\tth$ element with respect to the mean of the
ensemble. This is not the case, however, since the weight $\ome$ may
take values greater than unity.  Note also that the weights do not
depend on the index $i$. This is a constant quantity, which is identical for all
elements in the parameter ensemble. In particular, for the Gaussian compound case,
described in equation (\ref{eq:compound gaussian}), the CB estimates have
a direct interpretation in terms of hierarchical shrinkage. Let
\begin{equation}
     y_i \stack{\ind}{\sim} N(\theta_i, \sig^2),\qq \theta_i
     \stack{\iid}{\sim} N(\mu, \tau^2),
\end{equation}
where we assume that $\sig^2$, $\tau^2$ and $\mu$ are known. Note that
here, for convenience, we have assumed $\sig^{2}$ to be constant over
all elements in the parameter ensemble, which differs from the model specification described in
equation (\ref{eq:compound gaussian}). The use of the Bayes' rule under quadratic
loss yields the conventional posterior mean, $\hat\theta^{\op{SSEL}}_i = \ga y_i + (1-\ga) \mu$, 
where $\ga$ is the countershrinkage parameter, defined as 
\begin{equation}
    \ga := \frac{\tau^2}{\sig^2 + \tau^2}. 
\end{equation}
In this context, it can be shown that the Ghosh-Louis constrained point
estimates bears some important structural similarities with the formulae of the
posterior means for this model. We have
\begin{equation}
   \hat\theta^{\op{CB}}_i \doteq \ga^{1/2} y_i + (1-\ga^{1/2})\mu,
\end{equation}
where $0\leq \ga \leq 1$ \citep[see][p.212]{Rao2003}, for every
$i=1,\ldots,n$, and where $\doteq$ means that the equality is an approximation.
This is a particularly illuminating result, which intuitively
illustrates how this estimator controls for the under-dispersion of the
posterior means. The coutershrinkage parameter $\ga$ is here given more weight by
taking its square root, as opposed to the use of $\ga$ in computing
the $\hat\theta_{i}^{\op{SSEL}}$'s. This produces a set of point estimates which are
therefore more dispersed than the posterior means.

This set of point estimates has very desirable asymptotic properties. As an ensemble, the mean and
the variance of the CB point estimates converge almost surely to
the mean and variance, respectively, of the true ensemble
distribution \citep{Ghosh1994}. Furthermore, the CB estimator has a readily
interpretable formulation as a shrinkage estimator, as we have seen for
the compound Gaussian case. However, the CB approach also suffers from
an important limitation: Its performance will be greatly influenced by
the functional form of the true ensemble distribution. In particular, 
since the CB estimator only match the first two moments of the true ensemble,
the empirical distribution of the CB point estimates may provide a poor
approximation of the ensemble distribution, when the distribution of interest is substantially
skewed. The next approach attempts to address this limitation by
directly optimising the estimation of the EDF of the parameter
ensemble. 

\subsection{Triple-Goal}\label{sec:gr}
The triple-goal estimator of a parameter ensemble was
introduced by \citet{Shen1998}. It constitutes a natural extension of
the CB approach. Here, however, instead of solely
constraining the minimisation exercise with respect to the first two
moments of the ensemble distribution, we consider three successive
goals, which are optimised in turn \citep[see also][]{Shen2000}. 
The set of point estimates resulting from these successive minimisations are
generally referred to as the GR point estimates, where $G$ denotes the EDF and
$R$ refers to the ranks. For consistency, we will therefore adhere to
this acronym in the sequel \citep{Shen1998}. 
Note, however, that the GR point estimates are not optimal for these three
goals, but will produce very good performance under each of these
distinct objectives. These three consecutive steps can be described as
follows. 
\begin{description}
\item[i.] Firstly, minimise the ISEL function, as introduced in equation
  (\ref{eq:isel}), in order to obtain an estimate 
  $\widehat{F}_{n}(t)$ of the ensemble EDF, $F_{n}(t)$, for some
  ensemble of interest, $\bth=\lb\theta_{1},\ldots,\theta_{n}\rb$. As we have seen
  in equation (\ref{eq:isel minimiser}), the optimal estimate is the
  posterior EDF defined as follows,
 \begin{equation}
      \widehat{F}_{n}(t) := \E[F_{n}(t)|\by] 
      = \frac{1}{n}\sum_{i=1}^{n} \p[\theta_{i}\leq t |\by]. 
  \end{equation}
\item[ii.] Secondly, minimise the ranks squared error loss (RSEL)
  in order to optimise the estimation of the ranks of the parameter ensemble $\bth$. Let
  $\bR^{\op{est}}:=\bR^{\op{est}}(\bth)$ and $\bR:=\bR(\bth)$ denote
  the vector of candidate ranks and the vector of true ranks, respectively. We wish
  to minimise the following loss function, 
  \begin{equation}
    \op{RSEL}(\bR,\bR^{\op{est}}) = \frac{1}{n} \sum_{i=1}^{n}\lt( R_{i}- R^{\op{est}}_{i}\rt)^{2}. 
  \end{equation}
  The vector of optimisers $\bar{\bR}$ of the posterior expected
  RSEL is composed of the following elements,
  \begin{equation}        
       \bar{R}_{i}:= \E[R_{i}|\by] = \sum_{j=1}^{n}\p(\theta_{i}\geq\theta_{j}|\by).
  \end{equation}
  The $\bar{R}_{i}$'s are not generally integers. However, one can
  easily transform the $\bar{R}_{i}$'s into integers by ranking
  them such that 
  \begin{equation}
       \widehat{R}_{i} := \op{rank}(\bar{R}_{i}|\bar{\bR}),
  \end{equation}
  for every $i=1,\ldots,n$. The $\widehat{R}_{i}$'s are then used as
  optimal estimates of the true ranks, under RSEL. 
\item[iii.] Finally, we generate an ensemble of points estimates,
  conditional on the optimal estimate of the ensemble
  EDF, $\widehat{F}_{n}$, and the optimal estimates of the true
  ranks, $\widehat{R}_{i}$'s. This is done by setting 
  \begin{equation}
    \hat\theta^{\op{GR}}_{i}:=
    \widehat{F}_{n}^{-1}\lt(\frac{2\widehat{R}_{i} -1}{2n}\rt),
    \label{eq:gr estimates}
  \end{equation}
  for every $i=1,\ldots,n$. The $-1$ in the numerator of
  equation (\ref{eq:gr estimates}) arises from the minimisation of the
  posterior expected ISEL \citep[see][for this derivation]{Shen1998}.
\end{description}
Despite the seeming complexity of the successive minimisations
involved in producing the triple-goal estimators, the computation is
relatively straightforward \citep[more details can be found in][]{Shen1998}. 
We note, however, that one of the limitations of the
triple-goal technique is that it heavily relies on the quality of the
prior distribution --that is, the joint prior distribution for the
$\theta_{i}$'s, which we denoted by $p(\bth|\bxi)$ in equation
(\ref{eq:bhm}). Specific non-parametric methods have been proposed
to obtain an EB estimate of the prior distribution, which
permits to attenuate these limitations. The smoothing by roughening (SBR)
algorithm introduced by \citet{Shen1999} is an example of such a
technique. In such cases, the joint prior distribution, $p(\bth|\bxi)$,
is estimated using the SBR algorithm. Estimation of the parameter
ensemble is then conducted using standard Bayesian inference based on 
this SBR prior. \citet{Shen1999} demonstrated good performance of this
method using simulated data and rapid convergence of the SBR
algorithm. 

\subsection{Weighted and Weighted Ranks Loss Functions}\label{sec:wrsel}
The standard SSEL framework can also be extended by the 
inclusion of a vector of weights within the quadratic loss. 
These weights, denoted $\phi(\theta_i)$, may be specified as a function of the
unknown parameters of interest, with $\phi:\R\mapsto\R$. The weights may
therefore be made to vary with each of the elements in $\bth$. In
addition, since loss functions are assumed to be strictly positive,
all weights are also constrained to be positive: they satisfy
$\phi_i\geq 0,\,\forall\,i=1,\ldots,n$.
The resulting weighted squared error loss (WSEL) is then
\begin{equation}
    \op{WSEL}(\bphi,\bth,\bth^{\op{est}}) =
    \sum_{i=1}^n\phi(\theta_i)(\theta_i-\theta^{\op{est}}_i)^2,
    \label{eq:wsel}
\end{equation}
with $\phi(\theta_i)\geq 0$ for all $i=1,\ldots,n$.
The set of optimal point estimates, denoted $\theta_i^{\op{WSEL}}$'s, that
minimise the posterior expected WSEL take the following form
\begin{equation}
        \hat\theta_i^{\op{WSEL}}
        :=\E[\phi(\theta_i)\theta_i|\by]/\E[\phi(\theta_i)|\by],
\end{equation}
for every $i=1,\ldots,n$.
Naturally, if the weights do not depend on the parameters of interest,
then the optimal estimators reduce to the posterior means. That is, when the
weights are independent of $\bth$, they can be extracted from the
expectations and cancel out. In the ensuing discussion, we will
therefore implicitly assume that the weights are functions of the $\bth$, and
they will be simply denoted by $\phi_{i}$'s. 

The weighted ranks squared error loss (WRSEL) is a component-wise
loss function that spans the entire set of order statistics of a
given function \citep{Wright2003}. The vector of weights
$\bphi$ is here dependent on the ranks of each of the
$\theta_i$'s. Taking $\hat\theta^{\op{WRSEL}}_i$
to be the Bayes decision under WRSEL, we therefore have a compound
loss function of the form,
\begin{equation}
    \op{WRSEL}(\bphi,\bth,\bth^{\op{est}})
        =  \sum_{i=1}^{n}\sum^{n}_{j=1} \phi_j
        \cI\lb R_{j}=i\rb(\theta_i-\theta^{\op{est}}_i)^2,
\end{equation}
where as before $R_{i}:=R_{i}(\bth)$ is the rank of the $i\tth$
element in the parameter ensemble. For notational convenience, we write
\begin{equation}
    \op{WRSEL}(\bphi,\bth,\bth^{\op{est}})=
          \sum_{i=1}^{n}\phi_{R_{i}}(\theta_i-\theta^{\op{est}}_i)^2,
    \label{eq wrsel}
\end{equation}
where $\phi_{R_i}:=\sum_{j=1}^{n} \phi_{j}\cI\lb R_{j}=i\rb$ 
identifies the weight that is assigned to the $i\tth$
rank. The Bayes action minimising the posterior expected WRSEL is an
ensemble of point estimates $\hat\bth^{\op{WRSEL}}$, whose elements are
defined as 
\begin{equation}
    \hat\theta_i^{\op{WRSEL}} := \frac{\sum^{n}_{j=1} \phi_j 
     \E[\theta_i|R_{j}=i,\by]\p(R_{j}=i|\by)}{\sum^{n}_{j=1} \phi_j
    \p(R_{j}=i|\by)}, 
\end{equation}
where $\p(R_{j}=i|\by)$ is the posterior probability
that the $i\tth$ element has rank $j$. Each estimator $\hat\theta_i^{\op{WRSEL}}$ can therefore
be seen as a weighted average of conditional posterior means of
$\theta_i$, given that this element has rank $j$.
Note that we can further simplify the formulae provided by
\citet{Wright2003} by using the law of total expectation,
$\int\E(x|y)p(y)dy=\E(x)$, such that 
\begin{equation}
   \hat\theta^{\op{WRSEL}}_i = \frac{\E\big[\theta_i\phi_{R_{i}}|\by\big]}
                     {\E\big[\phi_{R_{i}}|\by\big]},
\end{equation}
for every $i=1,\ldots,n$, which is equivalent to the 
minimiser of the standard WSEL function presented in equation (\ref{eq:wsel}).

In addition, \citet{Wright2003} noted that the WRSEL and the SSEL are equivalent up to a multiplicative
constant, which is the vector of weights $\bphi$. That is,
\begin{equation}
    \op{WRSEL}=\bphi\op{SSEL}.
\end{equation}
Moreover, the WRSEL trivially reduces to the SSEL when $\bphi=\bone_n$. It therefore
follows that the WRSEL is a generalization of the SSEL. This family of loss
functions allows the targeting of a range of different inferential goals by
adjusting the shape of the vector of weights,
$\bphi$. \citet{Wright2003} and \citet{Craigmile2006} proposed a 
vector of weights consisting of a bowl-shaped function of the rank of each
of the $\theta_{i}$'s, which emphasises estimation of the extreme
quantiles of the ensemble distribution. \citet{Wright2003} makes use
of the following specification for the $\phi_{i}$'s, 
\begin{equation}
  \phi_i := \exp\lt\lb a_{1}\lt(i - \frac{n+1}{2}\rt)\rt\rb
            + \exp\lt\lb - a_{2}\lt(i - \frac{n+1}{2}\rt)\rt\rb,
  \label{eq:wrsel weights}
\end{equation}
for every $i=1,\ldots,n$, where both $a_{1}$ and $a_{2}$ are real
numbers greater than 0. Different choices of $a_{1}$ and $a_{2}$ permit to adjust the
amount of emphasis put on the extremes of
the EDF. In the sequel, when considering spatial models, these two parameters will be
fixed to $a_{1}=a_{2}=0.5$, which is a symmetric version of the
specification used by \citet{Wright2003}. When considering non-spatial
models, however, we 
will see that the shape of the ensemble of WRSEL point estimates is
highly sensitive to the choice of $a_{1}$ and $a_{2}$, as will be
described in section \ref{sec:plug-in q-sel}. For such non-spatial
models, we will choose the following `softer' specification: $a_{1}=a_{2}=0.05$.
Such symmetric specifications of the $\phi_{i}$'s emphasises the estimation
of the extrema over the estimation of other elements occupying the
middle ranks of the parameter ensemble. Therefore, the ensemble
distribution of WRSEL point estimators resulting from such a
specification of the $\phi_{i}$'s will be expected to be more
dispersed than the empirical of the posterior means. 

\section{Research Questions}\label{sec:research question}
The main focus of this thesis is the estimation of
\ti{functions} of parameter ensembles. Given a particular BHM, we wish to
estimate a function $h(\cdot)$ of the parameter ensemble $\bth$. As we have
seen in section \ref{sec:function of parameters}, the optimal estimator of $h(\bth)$
for some loss function $L$ is the following, 
\begin{equation}
     \delta^{\as} := \argmin_{\delta}
     \E\lt[\lt. L(h(\bth),\delta)\rt|\by\rt].
     \label{eq:optimal}
\end{equation}
Depending on the nature of the function $h(\cdot)$, the computation of
the optimiser $\delta^{\as}$ could be relatively expensive. 
By contrast, computation of the ensemble of optimal point estimates,
$\hat\bth^{L\pr}$, for a variety of commonly used loss functions $L\pr$,
is generally straightforward. Moreover, in the interest of
consistency, one generally wishes to report a single set of point
estimates that may simultaneously satisfy several inferential
desiderata. Thus, we are interested in evaluating the performance of
$h(\hat\bth^{L\pr})$ as an alternative to the optimal
estimator, $\delta^{\as}$. For instance, one may evaluate the performance of a
plug-in estimator based on the ensemble of posterior means, when this
particular estimator --i.e. $h(\hat\bth^{\op{SSEL}})$-- is used in the
place of $\delta^{\as}$.

A natural way of comparing the performance of different sets of
estimators, is to utilise the posterior regret \citep{Shen1998}. The
posterior regret expresses the penalty paid for
using a sub-optimal estimator. That is, it is defined as the
difference between the posterior loss associated with using the Bayes
choice and the posterior loss incurred for using another
estimator. Formally, the posterior regret for using $\delta\pr$ under
$L$ is 
\begin{equation}
   \op{regret}(L,\delta\pr) := \E[L(\theta,\delta\pr)|\by]
      - \min_{\delta}\E[L(\theta,\delta)|\by].
      \label{eq:regret}
\end{equation}
As for loss functions, $\op{regret}(L,\delta\pr)\geq 0$, with equality
holding if and only if $\delta\pr=\argmin\E[L(\theta,\delta)|\by]$.
One of the first utilisations of this concept was \citet{Savage1954},
although not expressed in terms of posterior
expectations. \citet{Shen1998} specifically used the posterior regret
when comparing ensembles of point estimators. 
This concept has also been used in more applied settings such
as in computing credibility premiums in mathematical finance
\citep{Gomez-Deniz2006}.

For our purpose, we can specialise the definition in equation
(\ref{eq:regret}) to the estimation of functions of parameter ensembles in BHMs. 
We will be interested in considering the posterior regret associated
with using a plug-in estimator, $h(\hat\bth^{L\pr})$, for some loss
functions $L\pr$. We will therefore compute the following posterior
regret, 
\begin{equation}
      \op{regret}(L,h(\hat\bth^{L\pr})) = \E[L(\bth,h(\hat\bth^{L\pr}))|\by]
      - \min_{\bm\delta}\E[L(\bth,\bm\delta)|\by],
      \label{eq:regret research}  
\end{equation}
where the optimal estimator $\bm\delta^{\as}$ under $L$ may be
vector-valued. In general, we will refer to the quantity in equation
(\ref{eq:regret research}) as the posterior regret based on $L(\bth,h(\hat\bth^{L\pr}))$.
When comparing different plug-in estimators, the classical loss function, denoted
$L\pr$ in equation (\ref{eq:regret research}), will be taken to be 
one of the loss functions reviewed in section \ref{sec:loss for
  ensemble}. In particular, following the research conducted in estimating parameter
ensembles in BHMs, the $L\pr$ function of interest will be taken to be
the SSEL, WRSEL, constrained Bayes or triple-goal loss. Moreover, we will also
evaluate the posterior regret for the ensemble of MLEs, as these
estimators represent a useful benchmark with respect to which the
performance of other plug-in estimators can be compared. For
clarity, since optimal estimators under $L\pr$ are vectors of
point estimates, we will denote them as $\hat{\bth}^{L\pr}$.
In order to facilitate comparison between experimental conditions
where the overall posterior expected loss may vary, it will be useful
to express the posterior regret in equation (\ref{eq:regret research})
as a percentage of the minimal loss that can be achieved using the
optimal minimiser. That is, we will generally also report the
following quantity, 
\begin{equation}
     \op{PercRegret}(L,h(\hat\bth^{L\pr}) := 
     \frac{100 \times \op{regret}(L,h(\hat\bth^{L\pr}))
     }{\min_{\bm\delta}\E[L(\bth,\bm\delta)|\by]},  
    \label{eq:PercRegret}
\end{equation}
which we term the \ti{percentage regret}. Both the posterior 
and percentage regrets will be used extensively in the ensuing
chapters. In order to emphasise the distinction between these two quantities, 
the posterior regret in equation (\ref{eq:regret research})
will sometimes be referred to as the absolute posterior regret. 

Specifically, in chapters \ref{chap:mrrr} and \ref{chap:clas}, we will study two aspects of parameter
ensembles in BHMs. Firstly, motivated by current practice in both
epidemiology and spatial epidemiology, we will consider the optimal determination of the
heterogeneity or dispersion of a parameter ensemble. In this case, the
function $h(\cdot)$ will either be the quantile function, as described
in section \ref{sec:quantiles} or a ratio of quartiles. 
Secondly, following an increased demand for the classification of
epidemiological units in low-risk and high-risk groups --for
surveillance purposes, for instance-- we will concentrate our attention
on optimising such classification. Here, the function $h(\cdot)$ will
be defined as an indicator function with respect to some cut-off point
of interest, denoted by $C$. Specifically, we will use two variants of the posterior
regret introduced in equation (\ref{eq:regret research}) for the
following choices of $L$ and $h(\cdot)$:
\begin{description}
  \item[i.] In chapter \ref{chap:mrrr}, we will optimise the
      estimation of the heterogeneity of a parameter ensemble. Our main focus will be on 
      defining $h(\cdot)$ as a ratio of some of the quartiles of the parameter ensemble
      $\bth$, with $L$ being the quadratic loss.
  \item[ii.] In chapter \ref{chap:clas}, we will treat the problem of
       classifying the elements of a parameter ensemble below and
       above a particular threshold. Here, the function $h(\cdot)$
       will be defined as the indicator function with respect to some
       cut-off point $C$, and the loss function of interest $L$
       will penalise misclassifications.
\end{description}


%
%
\chapter{Empirical Quantiles and Quartile Ratio Losses}\label{chap:mrrr}
\hspace{2cm}
\begin{minipage}[c]{11.5cm}
\small 
\begin{center}\tb{Summary}\end{center}\vspace{-.3cm}
The amount of heterogeneity present in a parameter ensemble is crucial in
both epidemiology and spatial epidemiology. A highly heterogeneous
ensemble of RRs, for instance, may indicate the presence of a hidden
risk factor not included in the model of interest. In modern BHMs,
however, the quantification of the heterogeneity of a parameter
ensemble is rendered difficult by the utilisation of link functions and
the fact that the joint posterior distribution of the parameter
ensemble is generally not available in closed-form. In this chapter,
we consider the estimation of such heterogeneity through the
computation of the empirical quantiles and empirical quartile ratio (QR)
of a parameter ensemble. We construct a simple loss
function to optimise the estimation of these quantities, which we term
respectively the quantile squared error loss (Q-SEL) and the QR squared
error loss (QR-SEL) functions. We compare the performance of the
optimal estimators under these two losses with the use of plug-in
estimators based on ensembles of point estimates obtained under
standard estimation procedures using the posterior regret. These
candidate plug-in estimators include the set of MLEs and the ensembles
of point estimates under the SSEL, WRSEL, CB and triple-goal
functions. In addition, we also consider the ratio of posterior
quartiles (RoPQ), as a possible candidate for measuring the
heterogeneity of a parameter ensemble. We evaluate the performance of
these plug-in estimators using non-spatial and spatially structured
simulated data. In these two experimental studies, we found that the
RoPQ and GR plug-in estimators tend to outperform other plug-in
estimators, under a variety of simulation scenarios. It was 
observed that the performance of the WRSEL function is highly
sensitive to the size of the parameter ensemble. We also noted that the CAR Normal
model tends to yield smaller posterior losses when using the optimal
estimator under both the Q-SEL and QR-SEL. In addition, we computed these
optimal and plug-in quantities for a real data set, describing the prevalence of cases of
schizophrenia in an urban population. On this basis, we draw some
recommendations on the type of QR summary statistics to use in practice. 
\end{minipage}

\section{Introduction}\label{sec:mrrr intro}
In mixed effects models, which can be regarded as the frequentist version of BHMs,
the main focus of the analysis has historically been on estimating
fixed effects --that is, the regression coefficients of
particular covariates. However, since the values of the fixed effect
parameters are highly dependent on the values taken by the random
effects, statisticians' attention has progressively turned to
the systematic study of random effects \citep{Demidenko2007}. Such a
vector of random effects can be described as a parameter ensemble as
introduced in section \ref{sec:bhm}. A standard way of quantifying
the amount of heterogeneity present in a parameter ensemble is to
use the intraclass coefficient (ICC). The ICC can be computed whenever 
quantitative measurements are made on units organised into
groups. This statistic constitute a useful complement
to the analysis of variance (ANOVA), for instance, where individual
units are grouped into $N$ groups, with $j=1,\ldots,N$. The nested
observations are then modelled as
\begin{equation}
    y_{ij} = \mu + \alpha_{j} + \epsilon_{ij},
\end{equation}
where $\alpha_{j}\stack{\iid}{\sim} N(0,\sig^{2}_{\alpha})$ and
$\epsilon_{ij}\stack{\iid}{\sim}N(0,\sig^{2}_{\epsilon})$. Here, the
$\alpha_{j}$'s are independent of the $\epsilon_{ij}$'s,
and we assume that there is an identical number of units in each
group, i.e. $i=1,\ldots,n$. For this model, the ICC is defined as follows, 
\begin{equation}
    \op{ICC} := \frac{\sig^{2}_{\alpha}}{\sig^{2}_{\alpha}+\sig^{2}_{\epsilon}}. 
\end{equation}
The ICC is a useful measure of parameter ensemble heterogeneity for
general linear models that contains random effects. However, there does
not exist an equivalent measure for generalised linear models.
Indeed, in generalised linear models, the random effects of interest 
affect the response through a link function, which transforms the
distribution of the original random effects. 
Thus, the variance parameter controlling the variability of these
random effects only provides a distorted image of the
dispersion of the parameter ensemble of interest.

Generalised linear models are commonly used in epidemiology and
spatial epidemiology. In this context, the
heterogeneity of a set of random effects may be indicative
of the effect of hidden covariates that have not been included in the
model. An accurate estimation of the dispersion of the random effects
is therefore crucial for the detection of unexplained variability,
which may then lead to the inclusion of new candidate risk factors. 
Most Bayesian models used in modern epidemiological practice, however,
are hierarchical. In such models, the set of parameters of interest
will be a non-linear function of one or more random effects as well as 
unit-specific covariates. The units may here be administrative
areas within a region of interest such as hospitals, schools or other
aggregates. For general BHMs, therefore, there does not exist a 
unique parameter, which controls for the variability of the
unit-specific levels of risk. The variance parameters of several random
effects can be added together, but this only constitutes a heuristic solution for
quantifying the heterogeneity of the resulting empirical distribution
of the parameters of interest. For such generalized linear models
commonly used in epidemiology, there is therefore a need for
a statistically principled way of estimating the heterogeneity of 
parameter ensembles.

One of the first formal attempts at estimating the amount of heterogeneity
present in a set of random effects was conducted by
\citet{Larsen2000}. These authors studied a mixed effects logistic regression model,
and considered the variance of the ensemble distribution of random effects, as a
potential candidate for a measure of dispersion. Formally, they
considered $N$ groups of individuals with $n$ subjects in each
group. The probability of the $i\tth$ individual in the $j\tth$
cluster to develop the disease of interest was modelled as 
\begin{equation}
   \p\lt[Y_{ij}=1| \alpha_{j},\bx_{i}\rt] = \frac{
     \exp\lt( \alpha_{j} + \sum_{k=1}^{K}\beta_{k}x_{ik}\rt)}{
     1+\exp\lt( \alpha_{j} + \sum_{k=1}^{K}\beta_{k}x_{ik}\rt)},
\end{equation}
where $i=1,\ldots,n$ and $j=1,\ldots,N$, and where the random effects
are assumed to be independent and identically distributed such that 
$\alpha_{j}\stack{\iid}{\sim}N(0,\sig^{2}_{\alpha})$, for every $j$. The
$\beta_{k}$'s constitute a set of fixed effects associated with
individual-specific $K$-dimensional vectors of risk factors denoted by
$\bx_{i}$, for every $i$. Although the variance parameter, $\sig^{2}_{\alpha}$,
controls the distribution of the random effects, it only affects the
probability of developing the disease of interest through the logistic map and
the exact relationship between the variances at the first and
second-level of the model hierarchy is therefore difficult to infer
from the sole consideration of $\sig^{2}_{\alpha}$. 

\citet{Larsen2000} addressed this problem by assessing the amount of
variability in the $\alpha_{j}$'s by computing the median odds ratio
(MOR). For two randomly chosen
$\alpha_{j}$'s drawn from the empirical distribution of the parameter
ensemble, \citet{Larsen2000} considered the absolute
difference between these two parameters. Since the $\alpha_{j}$'s 
are iid normal variates centered at zero with 
variance $\sig_{\alpha}^{2}$, as described above, it can be shown that the difference
between two such random variates is $\alpha_{j}-\alpha_{k}\sim N(0,2\sig_{\alpha}^{2})$ for any pair of
indices satisfying $j\neq k$ \citep{Spiegelhalter2004}.
Moreover, the absolute difference between two randomly chosen random
effects is a truncated Gaussian, sometimes called a half-Gaussian,
with mean zero and variance $2\sig_{\alpha}^{2}$. \citet{Larsen2000} suggested
taking the median of that distribution. For the aforementioned specification,
the median of such a truncated Gaussian is $\Phi^{-1}(0.75)\times
\sqrt{2}\sig_{\alpha}^{2}\doteq 1.09\sig_{\alpha}^{2}$. When the $\alpha_{j}$'s are
ORs on the logarithmic scale, the dispersion of the parameter ensemble becomes
$\exp(1.09\sig_{\alpha}^{2})$ \citep[see also][]{Larsen2005}. The MOR has been met with
singular success in the literature, and has been adopted by several
authors in the medical and epidemiological sciences. In particular, it has been used in cancer
epidemiology \citep[e.g.][]{Elmore2002} and in the study of
cardiovascular diseases \citep[e.g.][]{Sundquist2004}, as
well as in the monitoring of respiratory health
\citep[e.g.][]{Basagana2004}. 

In a Bayesian setting, however, the MOR will only be available in closed-form if 
the posterior distribution of the random effect ensemble is
also available in closed-form. This will rarely be the case in
practice. For more realistic models, the computation of
the MOR will typically require a substantial amount of computational power in order to
be estimated. Specifically, the calculation of the MOR is based
on the distribution of the differences between two random
variates. The exhaustive evaluation of such a distribution, however, will require
the computation of $\binom{n}{2}$ differences for a parameter ensemble 
of size $n$. If one wishes to estimate this particular quantity
within a Bayesian context, one could specify a quadratic loss on the
MOR. The optimal estimator under that loss function is the 
posterior mean. It then follows that we would need to compute the
median of the distribution of the absolute differences of any two
randomly chosen parameters at each
iteration of the MCMC algorithm used for fitting the model. Such an 
estimation scheme becomes rapidly unwieldy as $n$ grows (for
$n=2000$, about $2\times 10^{6}$ differences, $|\alpha_{j}-\alpha_{k}|$,
need to be computed at \ti{each} MCMC iteration). The use of the MOR for Bayesian models
thus appears to be impractical. 

A less computationally intensive method for quantifying the dispersion
of a parameter ensemble is the use of the quartile ratio (QR). This
quantity is defined as the ratio of the third to the first quartile of
a parameter ensemble. The use of quartiles for measuring heterogeneity
comes here with no surprise as these summary statistics have been shown to be especially
robust to deviation from normality \citep{Huber2009}. That is, the utilisation of quartiles as
summary statistics of the data is less sensitive to small departures from 
general modelling assumptions. This robustness of the quartiles is therefore
especially significant when the field of application is the 
generalised linear model, as in the present case. 
We will see that, for models with no covariates, the logratio of the
third to the first quartile of the empirical distribution of the RRs
is equivalent to the IQR of the random effects on their natural
scale. Like the IQR, the ratio of the third to first quartiles of the RRs
can thus be regarded as a non-parametric measure of dispersion. 

As any function of a parameter ensemble, the QR is amenable to
standard optimisation procedures using a decision-theoretic
approach. In this chapter, we will construct a specific 
loss function to quantify the estimation of this quantity. In
addition, it will be of interest to optimise the estimation of
the quantiles of the ensemble distribution. Specifically,
we will study the computation of the optimal first and
third quartiles of the ensemble distribution of the parameters of
interest. As described in section \ref{sec:research question}, one of our goals in this
thesis is to compare the use of optimal estimators with the
performance of plug-in estimators, based on commonly used ensembles of
point estimates. We will conduct these comparisons using the posterior
regret, as introduced in chapter \ref{chap:review}. The
performance of the various plug-in estimators will be
assessed by constructing two sets of data simulations. Different
experimental factors deemed to influence the estimators' performance
will be manipulated in order to precisely evaluate their effects. In
particular, the synthetic data sets will be constructed using both a
non-spatial and a spatial
structure. The non-spatial simulations used in the sequel 
reproduce in part some of the simulation
experiments conducted by \citet{Shen1998}, whereas the spatially
structured risk surfaces will be based on the simulation study
implemented by \citet{Richardson2004}.

In addition, we will also be evaluating the impact of using different
types of plug-in estimators for the estimation of the dispersion of parameter ensembles
in a real data set. We will re-analyse a published data set describing 
schizophrenia prevalence in an urban setting. Delineating the geographical distribution of psychoses plays an
important role in generating hypotheses about the aetiology of mental
disorders, including the role of socioeconomic status \citep{Giggs1987, Hare1956,
Hollingshead1958}, urbanicity \citep{Lewis1992}, migration 
\citep{Cantor-Graae2003}, ethnicity \citep{Harrison1997}
and, more recently, of family history interactions that may betray
effect modification by genes \citep{Krabbendam2005,Os2003}.
A raised prevalence of schizophrenia in urban areas was first observed
over 75 years ago \citep{Faris1939}, and has been consistently
observed since \citep{Giggs1973,Giggs1987,Hafner1969,Hare1956,
Lewis1992,Loffler1999}. \citet{Faris1939} demonstrated that the highest rates of
schizophrenia occurred in inner-city tracts of Chicago with a
centripetal gradient. More recently, a dose-response relationship
between urban birth and the prevalence of schizophrenia has been
reported in Sweden \citep{Lewis1992}, the Netherlands
\citep{Marcelis1998} and Denmark \citep{Pedersen2001}, suggesting that
a component of the aetiology of schizophrenia may be associated with the
urban environment. This specific data set is especially
relevant to the task at hand, since it is characterised by low
expected counts, which tend to lead to high levels of shrinkage in
commonly used ensembles of point estimates \citep{Gelman1999}; thereby providing a good
testing ground for comparing the use of optimal dispersion estimators
with the construction of plug-in estimators.

This chapter is organised as follows. We will first describe our decision-theoretic approach to the
estimation of both empirical quantiles and empirical QR in section 
\ref{sec:dispersion}. The construction of the non-spatial simulations
as well as the comparison of the optimal estimators with some plug-in
estimators of interest will then be reported in 
section \ref{sec:mrrr non-spatial sim}. The specification of the
spatial simulation scenarios and the effects of the different
experimental factors considered will be described in section \ref{sec:mrrr spatial sim}.
Finally, the analysis of the schizophrenia prevalence data set will be given in
section \ref{sec:real data}, and some conclusions on the respective
performance of the different studied estimators will be provided in
section \ref{sec:mrrr conclusion}.

\section{Estimation of Empirical Quantiles and Quartile Ratio}\label{sec:dispersion}
We here introduce two specific loss functions, which formalise our
approach to the estimation of parameter ensembles' quantiles and
QR. Both of these loss functions are straightforward adaptations of the
quadratic loss, which provides easily derivable optimal
minimisers. In addition, we discuss the computation of the respective posterior
regrets for these loss functions, as this will be used for
evaluating the performance of the plug-in estimators of the quantities of
interest. 

\subsection{Empirical Quantiles}\label{sec:edf quantiles}
The empirical quantiles of a parameter ensemble are especially relevant when
one wishes to quantify the dispersion of the ensemble distribution of a parameter ensemble.
As defined in section \ref{sec:quantiles}, the
$p\tth$ empirical quantile of an $n$-dimensional parameter ensemble
$\bth$ is 
\begin{equation}
       \theta_{(p)}:=Q_{\bth}(p),
\end{equation}
where $Q_{\bth}(\cdot)$ is the empirical QDF of $\bth$, and
$p\in[0,1]$. The quantity $\theta_{(p)}$ is a non-linear function of
$\bth$. Optimal estimation of this quantity can be formalised by
quantifying our loss using an SEL function on $\theta_{(p)}$, 
which takes the following form, 
\begin{equation}
     \op{SEL}\lt(Q_{\bth}(p),\delta\rt) = \lt(\theta_{(p)} - \delta\rt)^{2}.
\end{equation}
We saw in section \ref{sec:function of parameters} that the optimal
estimator for such a posterior expected SEL is the posterior
mean of the QDF of $\bth$ evaluated at $p$. Note that because
$\theta_{(p)}$ is a function of the entire parameter ensemble $\bth$,
it follows that the posterior expected loss depends on an $n$-dimensional
joint posterior distribution. When more than one
quantile is of interest, one may use a quantiles squared error loss
(Q-SEL) defined for some $k$-dimensional vector $\bp\in[0,1]^{k}$ as 
\begin{equation}
     \op{Q-SEL}_{\bm{p}}\lt(\bth,\bm{\delta}\rt) :=
     \sum_{j=1}^{k}\op{SEL}\lt(Q_{\bth}(p_{j}),\delta_{j}\rt) = 
     \sum_{j=1}^{k}\lt(\theta_{(p_{j})} - \delta_{j}\rt)^{2},
     \label{eq:qsel}
\end{equation}
where we have emphasised the dependence of the Q-SEL on the full
parameter ensemble $\bth$. Note, however, that $\bth$ and
$\bm\delta$ will generally not have the same dimension. Here,
$\bth$ is $n$-dimensional whereas $\bm\delta$ is $k$-dimensional.
Thus, the posterior expected Q-SEL is minimised by a $k$-dimensional vector
$\bth^{\op{Q-SEL}}$, which has the following elements, 
\begin{equation}
   \theta_{(p_{j})}^{\op{Q-SEL}}:= \E\lt[\lt.Q_{\bth}(p_{j})\rt|\by\rt] =
   \E\lt[\lt.\theta_{(p_{j})}\rt|\by\rt]. 
   \label{eq:qsel minimiser}
\end{equation}
for every $j=1,\ldots,k$, where the expectation is taken with respect
to the joint posterior distribution, $p(\theta_{1},\ldots,\theta_{n}|\by)$. 

Parameter ensemble quantiles can also be estimated using plug-in
estimators, which are based on empirical distributions of point
estimates, as described in section \ref{sec:research question}.
For any loss function $L\pr$, we define the empirical quantiles of an
ensemble of point estimates denoted $\hat\bth^{L\pr}$, as follows, 
\begin{equation}
      \hat\theta^{L\pr}_{(p)} :=: Q_{\hat\bth^{L\pr}}(p):= \min\lt\lb
      \hat\theta_{1}^{L\pr},\ldots,\hat\theta_{n}^{L\pr}:
       F_{\hat\bth^{L\pr}}(\hat\theta_{i}^{L\pr})\geq p\rt\rb,
       \label{eq:emp quantiles}
\end{equation}
for any $p\in[0,1]$, where the EDF of the ensemble $\hat\bth^{L\pr}$ is
defined as 
\begin{equation}
       F_{\hat\bth^{L\pr}}(t) := \frac{1}{n}\sum_{i=1}^{n} \cI\lb \hat\theta^{L\pr}_{i}\leq t\rb.
\end{equation}
Note that equation (\ref{eq:emp quantiles}) is solely the empirical
version of the general definition of a quantile reported in equation
(\ref{eq:qdf}) on page \pageref{eq:qdf}.
The performance of such plug-in estimators of the posterior quantiles
will be evaluated and compared to the optimal Bayes choice under
Q-SEL. In particular, we will consider ensembles of point estimates
based on the loss functions described in section 
\ref{sec:loss for ensemble}. This will include the SSEL, WRSEL, CB and
GR ensembles of point estimates, as well as the MLEs. 

\subsection{Empirical Quartile Ratio}\label{sec:qr}
A natural candidate for the quantification of the dispersion of a parameter
ensemble is the quartile ratio (QR). This quantity
is defined as the ratio of the third to the first empirical quartile of
the vector of parameters of interest, $\bth$. For a given BHM, of the type
described in equation (\ref{eq:bhm}), the QR is defined as follows,
\begin{equation}
    \op{QR}(\bth) := \frac{Q_{\bth}(.75)}{Q_{\bth}(.25)} =
    \frac{\theta_{(.75)}}{\theta_{(.25)}},
\end{equation}
where $\theta_{(.25)}$ and $\theta_{(.75)}$ denote the first and third
empirical quartiles, respectively. For generalized linear models, we will compute
this quantity on the scale of the data --that is, as a ratio of ORs for
a binomial likelihood, or as a ratio of RRs for a Poisson likelihood. When considering a
log-linear model (e.g. Poisson likelihood combined with a normal
prior on the log-intensities), the QR is related to the
IQR introduced in section \ref{sec:quantiles}. If we take the
logarithm of the QR, we obtain the IQR of the parameter ensemble on
the prior scale, i.e. on the scale of the normal prior in the case of
a log-linear model. We have
\begin{equation}
   \log\lt(\frac{\theta_{(.75)}}{\theta_{(.25)}}\rt) = 
   \log\theta_{(.75)} - \log\theta_{(.25)} =
   \op{IQR}(g(\bth)), 
\end{equation}
where $g(\cdot)$ is the log link function. 

The decision problem associated with the estimation of the QR can be
formulated using a quadratic loss taking the QR as an argument. For
clarity, we will refer to that particular quadratic loss as the QR
squared error loss (QR-SEL). This loss function is defined as follows,
\begin{equation}
     \op{QR-SEL}(\bth,\delta) := 
     \op{SEL}(\op{QR}(\bth),\delta) =
     \lt(\frac{\theta_{(.75)}}{\theta_{(.25)}} - \delta \rt)^{2}. 
    \label{eq:qrsel}
\end{equation}
The $\op{QR}(\cdot)$ function is a non-linear mapping of the
parameter ensemble $\bth$. Equation (\ref{eq:qrsel}) is of the form
$L(h(\bth),\delta)$, where $L$ is the
SEL and $h(\cdot)$ is the QR. The minimiser of the corresponding
posterior expected loss is therefore the following posterior quantity, 
\begin{equation}
     \op{QR}(\bth|\by) :=
     \E\lt[\lt.\frac{\theta_{(.75)}}{\theta_{(.25)}}\rt|\by\rt],
    \label{eq:qrsel minimiser}
\end{equation}
which will be referred to as the posterior empirical QR, where recall
that $\theta_{(.25)}$ and $\theta_{(.75)}$ denotes the \ti{empirical}
first and third quartiles of $\bth=\lb \theta_{1},\ldots,\theta_{n}\rb$.

It is of interest to compare the performance of this optimal estimator
with some plug-in estimators, $\op{QR}(\hat\bth^{L\pr})$. As before, the
ensemble of point estimates $\hat\bth^{L\pr}$ has been obtained as
the optimiser of $L\pr$, where $L\pr$ represents a commonly
used compound loss function. This therefore gives the following plug-in
QR estimator, 
\begin{equation}
      \op{QR}(\hat\bth^{L\pr}) :=  \frac{\hat\theta_{(.75)}^{L\pr}}{\hat\theta_{(.25)}^{L\pr}},
\end{equation}
where both numerator and denominator are here defined as in equation
(\ref{eq:emp quantiles}). Note that $\op{QR}(\hat\bth^{L\pr})$ only
depends on the data through the vector of point estimates,
$\hat\bth^{L\pr}$. An alternative solution to the problem of estimating the QR of a
parameter ensemble is the direct utilisation of the Bayes choice for
each quartiles. These quantities may have been estimated using the Q-SEL function described in
equation (\ref{eq:qsel}). In such a case, once
the minimiser of the posterior expected Q-SEL has been obtained, we
can estimate the QR as the ratio of posterior empirical quartiles (RoPQ). That
is, we have
\begin{equation}
      \op{RoPQ}(\bth|\by) := 
      \frac{\E[Q_{\bth}(.75)|\by]}{\E[Q_{\bth}(.25)|\by]}
      = \frac{\theta_{(.75)}^{\op{Q-SEL}}}{\theta_{(.25)}^{\op{Q-SEL}}},
      \label{eq:ropq}
\end{equation}
where the numerator and denominator are components of the
optimal estimator under Q-SEL, as described in equation (\ref{eq:qsel
minimiser}). Note that we have here drawn an explicit distinction
between posterior empirical quartiles and quartiles of an ensemble of
point estimates, which are respectively denoted by
$\theta^{\op{Q-SEL}}_{(p)}$ and $\hat\theta_{(p)}^{L\pr}$, by using a
hat to qualify the latter.

The QR is a useful indicator of the dispersion of a parameter
ensemble, when the elements of that ensemble are constrained to take
positive values. For parameter ensembles 
with Normally distributed elements, however, one needs to use a different
measure of dispersion. In such contexts, we will optimise the
estimation of the IQR, as introduced in section
\ref{sec:quantiles}. Such optimisation can be conducted using a
quadratic loss on the IQR such that 
\begin{equation}
     \op{IQR-SEL}(\bth,\delta) := 
     \op{SEL}(\op{IQR}(\bth),\delta) =
     \lt( (\theta_{(.75)}- \theta_{(.25)}) - \delta \rt)^{2}. 
     \label{eq:iqrsel}
\end{equation}
The optimal estimator of the posterior expected IQR-SEL is 
the posterior mean IQR, 
\begin{equation}
         \op{IQR}(\bth|\by):=\E[\theta_{(.75)}- \theta_{(.25)}|\by]. 
\end{equation}
As for the posterior QR, plug-in estimators for this quantity can be
defined by simply using the quartiles of the empirical distribution of
the corresponding point estimates. Alternatively, one can also use two
posterior empirical quartiles to
estimate the IQR, as was done for the QR-SEL in equation
(\ref{eq:ropq}). The latter estimator will be referred to as 
the difference of posterior empirical quartiles (DoPQ), defined as follows, 
\begin{equation}
       \op{DoPQ}(\bth|\by):=\E[\theta_{(.75)}|\by] -
       \E[\theta_{(.25)}|\by].      
       \label{eq:dopq}
\end{equation}
In the sequel, when dealing with ensembles of real-valued
parameters taking both negative and positive values, we will
automatically replace the QR-SEL with the IQR-SEL.

\subsection{Performance Evaluation}\label{sec:mrrr evaluation}
The performance of these different quantities will be evaluated using
simulated data, as well as compared in the context of real data
examples. In the sequel, the posterior expected loss associated with the use of the
optimal estimators under the Q-SEL and the QR-SEL 
functions will be compared to the posterior penalties incurred when
using various plug-in estimators under these loss functions. In particular, we will be interested in
assessing Q-SEL for $\bp:=\lb .25,.75\rb$, as this leads to a direct
comparison with the optimisation of the QR-SEL, which is related to
the estimation of these specific quartiles. 
As described in section \ref{sec:research question}, we will therefore
compute the following posterior regrets. 

In the notation of section \ref{sec:research question}, the function $h(\cdot)$, which takes an
ensemble of point estimates as an argument, becomes the empirical quantile
function or the empirical quartile ratio depending on whether we wish to evaluate
the posterior regret of the Q-SEL or the one of the QR-SEL function. 
Specifically, when the empirical quantiles are of interest, then
equation (\ref{eq:regret research}) becomes 
\begin{equation}
    \op{regret}\lt(\op{Q-SEL}_{\bp},Q_{\hat\bth^{L\pr}}(\bp)\rt) = 
    \E\lt[\lt.\op{Q-SEL}_{\bp}\lt(\bth,Q_{\hat\bth^{L\pr}}(\bp)\rt)
    \rt|\by\rt] - \min_{\bm\delta}
    \E\lt[\lt.\op{Q-SEL}_{\bp}(\bth,\bm\delta)\rt|\by\rt],  
    \notag
\end{equation}
where the optimal estimator, $\bm\delta$, minimising Q-SEL is the corresponding
vector of posterior empirical quartiles, as described in equation
(\ref{eq:qsel minimiser}). In the sequel, $\bp$ will be taken to
be $\lb .25,.75\rb$. In addition, we will use 
\begin{equation}
    \op{regret}\lt(\op{QR-SEL},\op{QR}(\hat\bth^{L\pr})\rt) = 
    \E\lt[\lt.\op{QR-SEL}\lt(\bth,\op{QR}(\hat\bth^{L\pr})\rt)
    \rt|\by\rt] - \min_{\delta}
    \E\lt[\lt.\op{QR-SEL}(\bth,\delta)\rt|\by\rt],
    \notag
\end{equation}
when the focus is on the empirical quartile ratio, with the optimal estimator
minimising QR-SEL being here the posterior empirical QR, as described in
equation (\ref{eq:qrsel minimiser}). A similar posterior regret
can be defined for the IQR-SEL function introduced in equation
(\ref{eq:iqrsel}). For every pair of $L$ and $L\pr$ functions, where
$L$ will be either the Q-SEL or QR-SEL functions and $L\pr$ will be either
the SSEL, WRSEL, CB or GR loss functions. In addition, we will also
consider the use of the RoPQ and the DoPQ, as
introduced in equation (\ref{eq:ropq}) and (\ref{eq:dopq}), when evaluating the QR-SEL
and IQR-SEL functions, respectively.

\section{Non-spatial Simulations}\label{sec:mrrr non-spatial sim}
\subsection{Design}
The proposed quantile and QR estimators were evaluated using synthetic datasets. 
The main objective of these simulations was to assess the influence of
various factors on quantile estimation, and more specifically on
the estimation of the QR of the posterior ensemble. Three factors were of particular
interest. Firstly, (i) the size of the
parameter ensemble was hypothesised to have a significant impact on the quality
of the estimation, since as the number of variables in the ensemble
goes to infinity, the quantiles of that ensemble distribution become
better identified. In addition, (ii) we varied the level of heterogeneity of the sampling
variances of the elements in the parameter ensemble. Finally, (iii) we were also concerned
with assessing the sensitivity of empirical quantile and QR estimation on 
distributional assumptions. Two different hierarchical models were
therefore compared. In this set of simulations, the models used to
generate the data was also used to estimate the different posterior quantities
of interest. That is, the generative and fitted models were
identical. 

\subsection{Generative Models}\label{sec:mrrr non-spatial design}
The data were simulated using two different BHMs, which were partially
described in section \ref{sec:bhm}. We here give a full description of
these models and our specific choices of hyperparameters.
We tested for the effect of skewness on the estimation of the
parameter ensemble quantiles and QR by comparing a compound Gaussian
and a compound Gamma model. In the first case, the prior distribution
on the $\theta_{i}$'s is symmetric, whereas in the latter case, the
prior on the $\theta_{i}$'s is skewed. For convenience, conjugacy
between the likelihood and the prior was respected for these two
models \citep{Bernardo1994}. Non-conjugate models will be studied in
section \ref{sec:mrrr spatial sim}, for the case of spatial priors. 
These two simulation models and the specific values given to 
the hyperparameters can be regarded as a replication the first part of
the simulation design conducted by \citet{Shen1998}. The models of
interest in this section can be described as follows. 
\begin{description}
   \item[\tb{i.}] For the compound Gaussian or Normal-Normal (N-N) model, we had 
     the following likelihood and prior,
     \begin{equation}
       y_{i} \stack{\ind}{\sim}N(\theta_{i},\sig^{2}_{i}), \qq
       \theta_{i}\stack{\iid}{\sim} N(\mu_{0},\tau_{0}^{2}),  
       \label{eq:compound gaussian}
     \end{equation}
      respectively, for every $i=1,\ldots,n$. Standard application of Bayes' rule gives
      $\theta_{i}|y_{i}\stack{\ind}{\sim}
      N(\hat\theta^{\op{SSEL}}_{i},\tau_{0}^{2}\gamma_{i})$, with
      $\hat\theta^{\op{SSEL}}_{i}:=\gamma_{i}\mu_{0} + (1-\gamma_{i})y_{i}$, 
      and where $\gamma_{i}:=\sig^{2}_{i}/(\sig^{2}_{i}+\tau_{0}^{2})$ are the
      shrinkage parameters. The MLEs for this model are given by
      $\hat\theta_{i}^{\op{MLE}}=y_{i}$ for every $i=1,\ldots,n$. In every 
      simulation, the hyperparameters controlling the prior distribution of the
      $\theta_{i}$'s were given the following specification:
      $\mu_{0}=0$, and $\tau_{0}^{2}=1$, as in \citet{Shen1998}. Note
      that this simulation setup is somewhat artificial since $\tau^{2}_{0}$ is
      fixed to a particular value, instead of being estimated from the 
      data. However, this choice of specification has
      the advantage of directly reproducing the simulation study of
      \citet{Shen1998}. 
   \item[\tb{ii.}] For the compound Gamma or Gamma-Inverse Gamma
     (G-IG) model, the likelihood function and prior take the following form, 
     \begin{equation}
        y_{i} \stack{\ind}{\sim}\op{Gam}(a_{i},\theta_{i}), \qq
       \theta_{i}\stack{\iid}{\sim} \op{Inv-Gam}(\alpha_{0},\beta_{0}),        
        \label{eq:compound gamma}
     \end{equation}
     for every $i=1,\ldots,n$. The Gamma and Inverse-Gamma
     distributions were respectively given the following specifications: 
     \begin{equation}
          \op{Gam}(a_{i},\theta_{i}) = \frac{1}{\theta_{i}^{a_{i}}\Ga(a_{i})}y_{i}^{a_{i}-1}e^{-y_{i}/\theta_{i}},
     \end{equation}
     with $a_{i},\theta_{i}>0$ for every $i=1,\ldots,n$, and
     \begin{equation}
          \op{Inv-Gam}(\alpha_{0},\beta_{0}) =
          \frac{\beta_{0}^{\alpha_{0}}}{\Ga(\alpha_{0})}\theta_{i}^{-\alpha_{0}-1}e^{-\beta_{0}/\theta_{i}},
     \end{equation}
     with $\alpha_{0},\beta_{0}>0$. From the conjugacy of the prior, we obtain
     the following posterior distribution, $\theta_{i}|y_{i}\stack{\ind}{\sim}
     \op{Inv-Gam}(a_{i}+\alpha_{0},y_{i}+\beta_{0})$, and the MLEs are
     $\hat\theta_{i}^{\op{MLE}}=y_{i}/a_{i}$ for every $i=1,\ldots,n$.
     All simulations were based on the following choices of the
     hyperparameters: $\alpha_{0}=4$, and $\beta_{0}=3$, as in \citet{Shen1998}.
\end{description}
Detailed descriptive statistics of the synthetic data generated by these models 
are reported in tables \ref{tab Y_table} and \ref{tab Sig_table} in
appendix \ref{app:non-spatial}, for the simulated observations, and
variance parameters, respectively. 

\subsection{Simulation Scenarios}\label{sec:app non-spatial scenario}
In addition to the different models tested, two other factors were
manipulated when producing these synthetic data sets.
Firstly, in order to evaluate the effect of the size of
the ensemble distribution on the quality of the classification
estimates, we chose three different sizes of parameter ensemble.
Specifically, in the BHMs of interest, the vector of observations and
the vector of parameters are both of size $n$. In these simulations, we chose
$n$ to take the following values:
\begin{equation}
          n\in\lt\lb 100,200,1000 \rt\rb. 
\end{equation}
The second experimental factor under consideration was the amount of variability
associated with each data point, i.e. the sampling variances of the
$y_{i}$'s. This aspect of the synthetic data sets was made to vary by
selecting different ratios of the largest to the smallest (RLS, in the
sequel) sampling variances. Formally,
\begin{equation}
        \op{RLS}(\bsig):= \frac{\sig^{2}_{(n)}}{\sig^{2}_{(1)}},
        \qq\te{and}\qq \op{RLS}(\ba):= \frac{a_{(n)}}{a_{(1)}},   
\end{equation}
for the compound Gaussian and compound Gamma models,
respectively. Thus, different choices of the vectors
$\bsig=\lt\lb\sig_{1},\ldots,\sig_{n}\rt\rb$ 
and $\ba=\lt\lb a_{1} \ldots,a_{n}\rt\rb$ produced different levels
of that ratio. We generated the scaling parameters similarly in the two models
in order to obtain three comparable levels of shrinkage. The sampling
variances in both models were produced using 
\begin{equation}
  \log(\sig_{i}^{2}) \stack{\iid}{\sim} \unif(-C_{l},C_{l}),
  \qq\te{and}\qq
  \log(a_{i}) \stack{\iid}{\sim} \unif(-C_{l},C_{l}),
\end{equation}
for the compound Gaussian and compound Gamma models, respectively,
and where $l=1,\ldots,3$. Different choices of the parameters
$C_{l}\in\lt\lb .01,1.5,2.3\rt\rb$
approximately generated three levels of RLS, such that 
both $\op{RLS}(\bsig)$ and $\op{RLS}(\ba)$ took the
following values $\lt\lb 1,20,100\rt\rb$, where small values of RLS
represented scenarios with small variability in sampling variances.

  In spite of these modifications of $n$ and RLS, the empirical mean of
 the ensemble distributions of the true $\theta_{i}$'s was kept
 constant throughout all simulations, in order to ensure that the
 simulation results were comparable. 
 For the compound Gaussian model, the empirical ensemble mean was fixed to 0, and for
 the compound Gamma model, the empirical ensemble mean was fixed to 1. 
 These values were kept constant through adequate choices of the
 hyperparameters of the prior distribution in each model. 
 The combination of all these factors, taking into account the
2 different models, the 3 different sizes of the parameter ensemble
and the 3 different levels of RLS resulted in a total of 1,800 different synthetic
 data sets with 100 replicates for each combination of the experimental
factors. 

\subsection{Fitted Models}
As aforementioned, the fitted models were identical to the generative
models. For both the compound Gaussian and compound Gamma models, 
no burn-in was necessary since each of the $\theta_{i}$'s
are conditionally independent given the fixed hyperparameters. 
The posterior distributions were directly available in closed-form
for the two hierarchical models as both models were given conjugate
priors. Nonetheless, in order to compute the different quantities
of interest, we generated 2,000 draws from the joint posterior
distribution of the $\theta_{i}$'s in each model for each simulation
scenario. The various plug-in estimators were computed on the basis of
these joint posterior distributions. For this set of non-spatial
simulations, the WRSEL function was specified using $a_{1}=a_{2}=0.05$
(see section \ref{sec:wrsel} on page \pageref{sec:wrsel}).

\subsection{Plug-in Estimators under Q-SEL}\label{sec:plug-in q-sel}
\begin{table}[t]
\footnotesize
\caption{
Posterior regrets based on $\op{Q-SEL}_{\bp}(\bth,Q_{\hat{\bth}^{L\pr}}(\bp))$ with $\bp:=\lb.25,.75\rb$,
for five plug-in estimators and with the posterior expected loss of the optimal estimator in the first column.
Results are presented for the compound Gaussian model in equation (\ref{eq:normal-normal})
and the compound Gamma model in equation (\ref{eq:gamma-inverse gamma}),
for 3 different levels of RLS, and 3 different values of $n$, averaged over 100 replicate data 
sets. Entries were scaled by a factor of $10^3$. In parentheses,
the posterior regrets are expressed as percentage of the posterior loss
under the optimal estimator.
\label{tab:qsel_table}} 
\centering
\begin{threeparttable}
\begin{tabular}{>{\RaggedRight}p{60pt}>{\RaggedLeft}p{25pt}|>{\RaggedLeft}p{25pt}@{}>{\RaggedLeft}p{27pt}>{\RaggedLeft}p{25pt}@{}>{\RaggedLeft}p{27pt}>{\RaggedLeft}p{25pt}@{}>{\RaggedLeft}p{35pt}>{\RaggedLeft}p{25pt}@{}>{\RaggedLeft}p{23pt}>{\RaggedLeft}p{20pt}@{}>{\RaggedLeft}p{10pt}}\hline
\multicolumn{1}{c}{\itshape Scenarios}&
\multicolumn{1}{c}{}&
\multicolumn{10}{c}{\itshape Posterior regrets\tnote{a}}
\tabularnewline \cline{3-12}
\multicolumn{1}{>{\RaggedRight}p{60pt}}{}&\multicolumn{1}{c}{Q-SEL}&
\multicolumn{2}{c}{MLE}&\multicolumn{2}{c}{SSEL}&\multicolumn{2}{c}{WRSEL}&\multicolumn{2}{c}{CB}&
\multicolumn{2}{c}{GR}\tabularnewline
\hline
{\itshape $\op{RLS}\doteq1$}&&&&&&&&&&&\tabularnewline
\normalfont   N-N, $n=100$ &    $24.9$ &    $182.8$ &    ($  734$) &    $ 79.3$ &    ($ 318$) &    $  22.6$ &    ($    91$) &    $ 8.5$ &    ($  34$) &    $0.3$ &    ($1$)\tabularnewline
\normalfont   N-N, $n=200$ &    $12.7$ &    $176.4$ &    ($ 1389$) &    $ 81.7$ &    ($ 643$) &    $ 355.8$ &    ($  2802$) &    $ 5.3$ &    ($  42$) &    $0.0$ &    ($0$)\tabularnewline
\normalfont   N-N, $n=1000$ &    $ 2.6$ &    $158.4$ &    ($ 6177$) &    $ 80.4$ &    ($3134$) &    $3126.8$ &    ($121905$) &    $ 1.3$ &    ($  51$) &    $0.0$ &    ($1$)\tabularnewline
\normalfont   G-IG, $n=100$ &    $ 8.8$ &    $170.6$ &    ($ 1938$) &    $ 59.7$ &    ($ 678$) &    $  57.6$ &    ($   654$) &    $ 6.0$ &    ($  69$) &    $0.2$ &    ($2$)\tabularnewline
\normalfont   G-IG, $n=200$ &    $ 4.4$ &    $147.7$ &    ($ 3365$) &    $ 62.9$ &    ($1432$) &    $ 182.7$ &    ($  4164$) &    $ 6.2$ &    ($ 141$) &    $0.1$ &    ($2$)\tabularnewline
\normalfont   G-IG, $n=1000$ &    $ 0.9$ &    $137.2$ &    ($15705$) &    $ 62.7$ &    ($7179$) &    $2193.1$ &    ($251039$) &    $ 5.9$ &    ($ 680$) &    $0.0$ &    ($4$)\tabularnewline
\hline
{\itshape $\op{RLS}\doteq20$}&&&&&&&&&&&\tabularnewline
\normalfont   N-N, $n=100$ &    $24.9$ &    $251.0$ &    ($ 1009$) &    $ 95.7$ &    ($ 385$) &    $  12.7$ &    ($    51$) &    $ 8.8$ &    ($  35$) &    $0.6$ &    ($3$)\tabularnewline
\normalfont   N-N, $n=200$ &    $12.4$ &    $233.0$ &    ($ 1886$) &    $106.4$ &    ($ 861$) &    $ 241.6$ &    ($  1955$) &    $ 7.2$ &    ($  58$) &    $0.1$ &    ($1$)\tabularnewline
\normalfont   N-N, $n=1000$ &    $ 2.5$ &    $210.7$ &    ($ 8327$) &    $ 97.8$ &    ($3865$) &    $2506.6$ &    ($ 99044$) &    $ 3.3$ &    ($ 131$) &    $0.0$ &    ($1$)\tabularnewline
\normalfont   G-IG, $n=100$ &    $ 8.5$ &    $186.9$ &    ($ 2201$) &    $ 79.0$ &    ($ 930$) &    $  73.3$ &    ($   864$) &    $13.5$ &    ($ 159$) &    $0.2$ &    ($3$)\tabularnewline
\normalfont   G-IG, $n=200$ &    $ 4.3$ &    $163.9$ &    ($ 3812$) &    $ 76.5$ &    ($1779$) &    $ 185.9$ &    ($  4324$) &    $17.1$ &    ($ 398$) &    $0.1$ &    ($2$)\tabularnewline
\normalfont   G-IG, $n=1000$ &    $ 0.8$ &    $156.2$ &    ($18600$) &    $ 77.0$ &    ($9167$) &    $1979.9$ &    ($235817$) &    $18.2$ &    ($2166$) &    $0.1$ &    ($6$)\tabularnewline
\hline
{\itshape $\op{RLS}\doteq100$}&&&&&&&&&&&\tabularnewline
\normalfont   N-N, $n=100$ &    $23.7$ &    $302.9$ &    ($ 1277$) &    $111.9$ &    ($ 472$) &    $   6.2$ &    ($    26$) &    $14.6$ &    ($  62$) &    $0.0$ &    ($0$)\tabularnewline
\normalfont   N-N, $n=200$ &    $12.3$ &    $254.9$ &    ($ 2077$) &    $137.3$ &    ($1118$) &    $ 131.4$ &    ($  1070$) &    $18.0$ &    ($ 147$) &    $0.0$ &    ($0$)\tabularnewline
\normalfont   N-N, $n=1000$ &    $ 2.5$ &    $253.8$ &    ($10161$) &    $135.2$ &    ($5413$) &    $1903.4$ &    ($ 76214$) &    $16.1$ &    ($ 643$) &    $0.0$ &    ($1$)\tabularnewline
\normalfont   G-IG, $n=100$ &    $ 8.0$ &    $208.4$ &    ($ 2592$) &    $ 76.4$ &    ($ 950$) &    $  66.3$ &    ($   825$) &    $16.5$ &    ($ 205$) &    $0.5$ &    ($6$)\tabularnewline
\normalfont   G-IG, $n=200$ &    $ 4.0$ &    $190.2$ &    ($ 4698$) &    $ 79.7$ &    ($1968$) &    $ 169.0$ &    ($  4174$) &    $17.1$ &    ($ 422$) &    $0.1$ &    ($2$)\tabularnewline
\normalfont   G-IG, $n=1000$ &    $ 0.8$ &    $190.4$ &    ($23096$) &    $ 80.3$ &    ($9744$) &    $1622.9$ &    ($196836$) &    $25.5$ &    ($3090$) &    $0.1$ &    ($8$)\tabularnewline
\hline
\end{tabular}
\begin{tablenotes}
   \item[a] Entries for the posterior regrets have been truncated to the closest first
     digit after the decimal point, and entries for the percentage
     regrets have been truncated to the closest integer. 
     For some entries, percentage regrets are smaller than 1 percentage point. 
\end{tablenotes}
\end{threeparttable}
\end{table}

The results of these simulations for the Q-SEL function are presented in table
\ref{tab:qsel_table} on page \pageref{tab:qsel_table} for both the
compound Gaussian (N-N) and the compound Gamma (G-IG) models. For the
posterior regret of each plug-in estimator, we have reported in 
parentheses the posterior regret as a percentage of the posterior loss under the
optimal estimator. This quantity is the \ti{percentage regret}, as
described in section \ref{sec:research question} on page
\pageref{sec:research question}. Specifically, for any ensemble of point estimates,
$\hat\bth^{L\pr}$, we have computed the following percentage regret,
\begin{equation}
     \frac{100 \times \op{regret}\lt(\op{Q-SEL}_{\bp},Q_{\hat\bth^{L\pr}}(\bp)\rt)
     }{\min_{\bm\delta}\E\lt[\lt.\op{Q-SEL}_{\bp}(\bth,\bm\delta)\rt|\by\rt]},
     \label{eq:percentage regret}
\end{equation}
where the optimal estimator is the vector of posterior empirical quartiles
for $\bp:=\lb .25,.75\rb$. In the sequel, most simulation results will be presented in
tabulated format, where the first column will correspond to the
denominator in formulae (\ref{eq:percentage regret}), which is here
posterior expected $\op{Q-SEL}_{\bp}$, and the remaining columns will
provide both the posterior regret and the percentage regret in
parentheses expressed with respect to $Q_{\hat\bth^{L\pr}}(\bp)$, for
different choices of $L\pr$. 

Overall, the lowest posterior regret for estimating the first and
third quartiles was achieved by using the GR plug-in
estimator. That is, estimates of the posterior empirical quantiles
using the ensemble of GR point estimates was almost optimal under both
models and all simulation scenarios. By comparing the columns of table 
\ref{tab:qsel_table} corresponding to the posterior regrets of each
plug-in estimator, we can observe that the performance of the GR was
followed by the ones of the CB, SSEL and MLE ensembles in order of increasing percentage 
regret for every condition where $n>100$. 
The quartiles derived from the WRSEL ensemble outperformed the ones
based on the SSEL and MLE sets of point estimates when $n=100$, under both types of models and
under all levels of the RLS. The WRSEL quartile estimator also surpassed
the performance of the CB ensemble of point estimates when
$\op{RLS}\doteq 100$ and $n=100$. However, the performance of the WRSEL estimator
was very sensitive to the size of the parameter ensemble, and
deteriorated rapidly as $n$ increased, under both types of
models. Figure \ref{fig:hist wrsel} on page \pageref{fig:hist wrsel}
illustrates the impact of an
increase of the size of the parameter ensemble on the behaviour of the
WRSEL ensemble of point estimates. In panel (b), one can observe that
the WRSEL ensemble distribution becomes increasingly bimodal
as $n$ increases. This particular behaviour can be explained in
terms of the specification of the $\phi_{i}$'s in the WRSEL estimation
procedure. Each of the $\phi_{i}$'s is a function of both the
parameters $a_{1}$ and $a_{2}$ and the size of the ensemble
distribution, $n$. We here recall the definition of these WRSEL weights, 
\begin{equation}
   \phi_i := \exp\lt\lb a_{1}\lt(i - \frac{n+1}{2}\rt)\rt\rb
            + \exp\lt\lb - a_{2}\lt(i - \frac{n+1}{2}\rt)\rt\rb,
  \label{eq:phi annoying}
\end{equation}
for $i=1,\ldots,n$, and where for the present set of non-spatial
simulations, we chose $a_{1}=a_{2}=0.05$. It can therefore be seen
that as $n$ increases, the range of the
$\phi_{i}$'s tends to increase at an exponential rate. Since 
these weights control the counterbalancing of
hierarchical shrinkage in the ensemble
distribution, it follows that an increase in the size of the ensemble
yields an accentuation of such countershrinkage. As a result, we
obtain the phenomenon visible in panel (b) of figure \ref{fig:hist
  wrsel}, where the units in the tails of the WRSEL
ensemble distribution are pushed away from the center of the
distribution, thus creating a bimodal distribution. This effect
explains why the performance of the WRSEL tends to rapidly deteriorate as $n$
increases.

\begin{figure}[t]
 \centering
 \includegraphics[width=14cm]{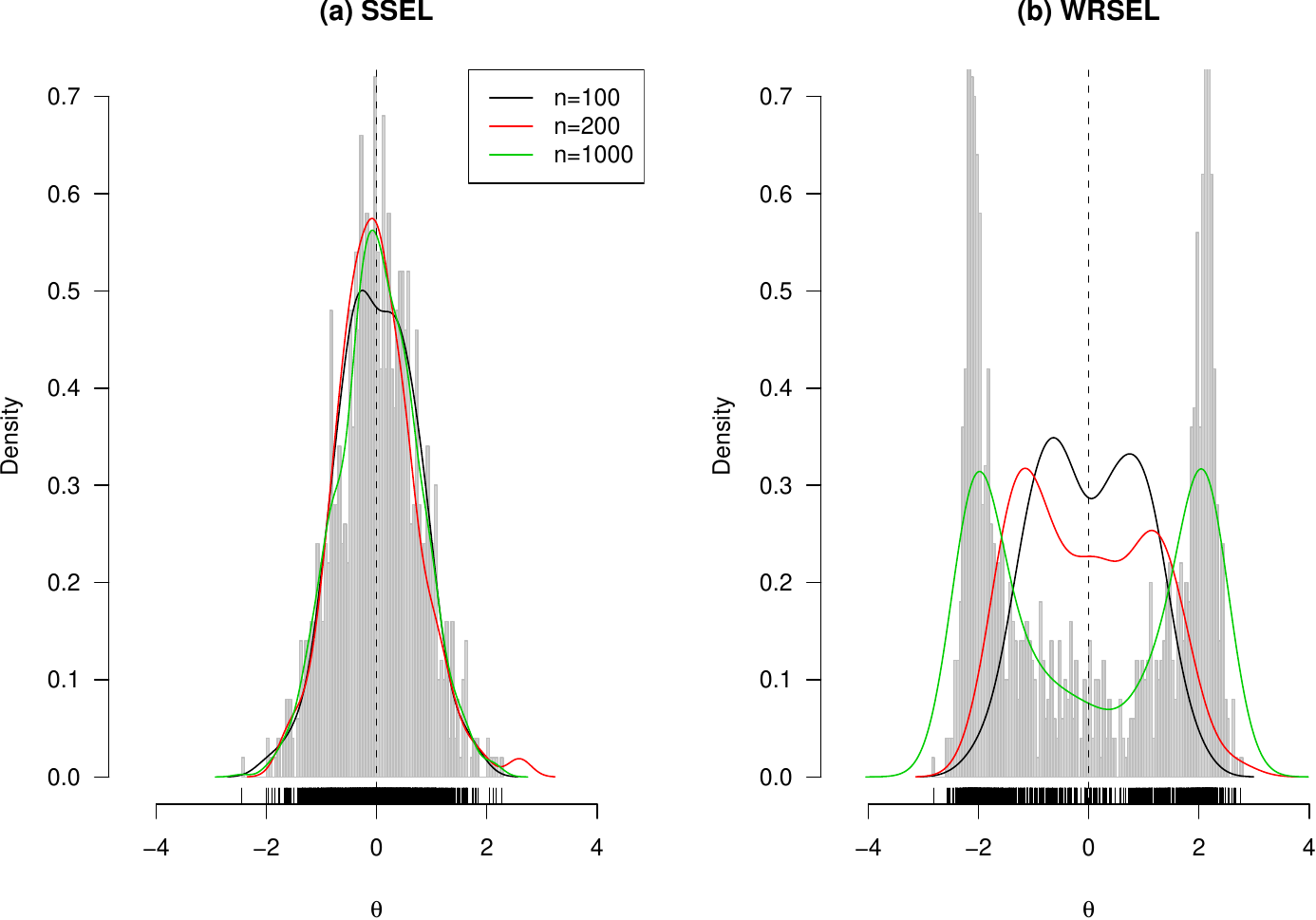}\\
 \caption{Effect of the change in the size of the parameter ensemble
   on the performance of the WRSEL point estimates. Panel (b) shows
   a histogram of the ensemble of WRSEL point estimates for $n=1000$
   with superimposed lines representing the same set of point estimates
   under $n=100$, $200$ and $1000$. We have also provided the SSEL
   parameter ensemble in panel (a) for comparison. Results are here
   shown for a randomly chosen data set among the 100 replicates. The
   superimposed curves were produced using a Gaussian kernel over
   512 equal-sized bins.}
   \label{fig:hist wrsel}
\end{figure}

The relative performance --that is, in terms of percentage regrets--
of all plug-in estimators, except perhaps for the GR
quartile estimator, appeared to worsen under the
compound Gamma model, in comparison to the compound Normal model. The
percentage regrets of most plug-in estimators, indicated in 
parentheses in table \ref{tab:qsel_table}, can be observed to be
higher under the compound Gamma model. The sole exception to that
trend was the performance of the GR plug-in
quartiles for which no systematic trend could be identified. Overall,
however, that estimator also tended to do slightly worse under the
G-IG model, although its performance was still very close to optimal. 
A comparison of the shape of the ensemble distributions of the
different point estimates studied in this simulation experiment is
reported in figure \ref{fig:hist norm-gam} on page \pageref{fig:hist
  norm-gam}. It can be observed that the empirical distribution of point estimates
remains centered at zero under the N-N model, thereby matching the
central tendency of the true ensemble distribution. Under the G-IG
model, however, the skewness of the true ensemble distribution made it
more difficult for plug-in estimators to match the values of the
quantiles of the true parameter ensemble.

Increasing the magnitude of the RLS was found to detrimentally affect the
performance of the MLE, SSEL and CB plug-in estimators. The percentage
regrets for these three families of plug-in estimators was found to increase
systematically with the magnitude of the RLS. This effect may be explained in terms of the
added variability associated with a high RLS. For the MLEs, this
difference in percentage regrets is directly attributable to
differences in sampling variance. For the SSEL and CB plug-in
estimators, one may explain the detrimental effect of
the RLS experimental factor in terms of hierarchical shrinkage. That
is, units with very high 
sampling variance are subject to a larger degree of hierarchical
shrinkage than ensemble units with low sampling variance. These changes
may then modify the ordering of the units and the overall shape of the
ensemble of SSEL and CB point estimates, thereby leading to poorer
performance than under low RLS. No similar trend could be observed for the
GR and WRSEL quartile estimators. The GR quartile estimators seemed to
be robust to changes in RLS levels, although the scale of the 
differences in performance was too small to reach any definite conclusions.
For the WRSEL plug-in estimator, an increase in RLS was found to
decrease percentage regret when considering the N-N model, but not
necessarily under the G-IG model. 

Finally, increasing the size of the parameter ensemble resulted in a
systematic increase of percentage regret for all plug-in
estimators, except for the quartiles based on the triple-goal
function. As noted earlier, this relative deterioration of the performance of
the plug-in estimators was particularly acute for the WRSEL
estimators. The absolute value of the posterior loss under the optimal
Q-SEL estimator, however, tended to decrease as $n$ increased. Thus,
while the posterior empirical quartiles are better estimated when the size
of the parameter ensemble of interest is large, the relative
performance of the plug-in estimators tends to decrease when their
performance is quantified in terms of percentage regret. 
One should note, however, that for ensembles of point estimates, which
optimised some aspects of the ensemble empirical distribution, the
absolute posterior regrets (cf. first column of table
\ref{tab:qsel_table}) tended to diminish with $n$, even if this
was not true for the corresponding percentage posterior regrets 
(cf. values in parentheses in the remaining columns \ref{tab:qsel_table}). This
trend was especially noticeable for the CB and GR plug-in estimators,
whose absolute posterior regrets can be seen to be lower under 
larger parameter ensembles. 
\begin{figure}[t]
 \centering
 \includegraphics[width=14cm]{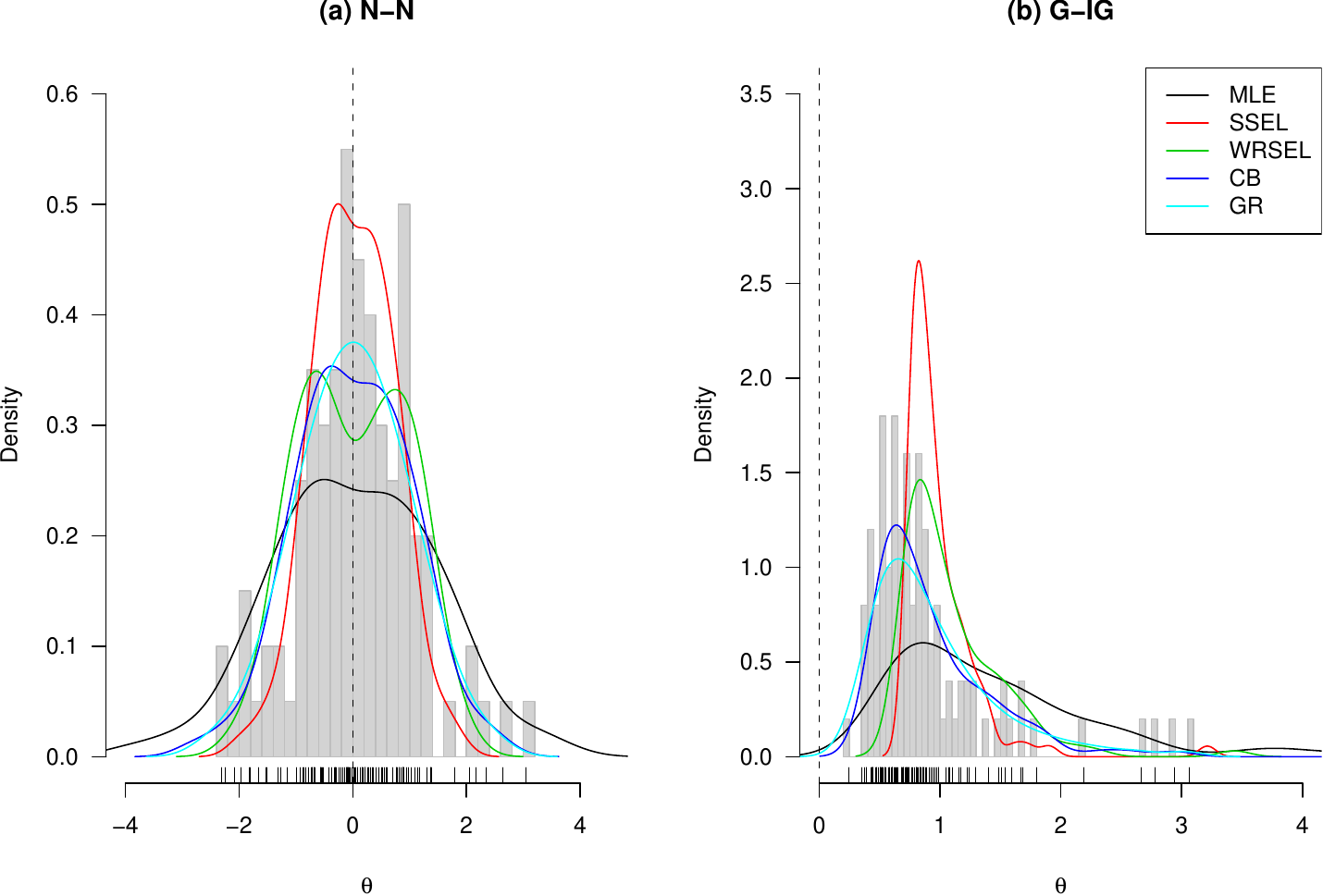}\\
 \caption{Histograms of a randomly chosen simulated parameter ensemble
  under the compound Gaussian and compound Gamma models in panel (a)
  and (b), respectively, for $n=100$ and $\op{RLS}\doteq 1$. For each model, 
  ensemble distributions of point estimates have been superimposed for five
  different estimation schemes.}
 \label{fig:hist norm-gam}
\end{figure}

\subsection{Plug-in Estimators under QR-SEL}\label{sec:non-spatial qrsel}
\begin{table}[t]
 \footnotesize
 \caption{
Posterior regrets based on $\op{QR-SEL}(\bth,\op{QR}(\bth^{L\pr}))$,
for five plug-in estimators,
for the compound Gaussian model in equation (\ref{eq:normal-normal})
and the compound Gamma model in equation (\ref{eq:gamma-inverse gamma}),
for 3 different levels of RLS,
and 3 different values of $n$, averaged over 100 replicate data
sets. The posterior expected loss of the optimal estimator is given in the
first column. 
Entries were scaled by a factor of $10^3$. In parentheses,
the posterior regrets are expressed as percentage of the posterior loss
under the optimal estimator.
\label{tab:qrsel_table}}
 \centering
 \begin{threeparttable}
 \begin{tabular}{>{\RaggedRight}p{58pt}>{\RaggedLeft}p{17pt}|>{\RaggedLeft}p{22pt}@{}>{\RaggedLeft}p{35pt}>{\RaggedLeft}p{18pt}@{}>{\RaggedLeft}p{29pt}>{\RaggedLeft}p{18pt}@{}>{\RaggedLeft}p{33pt}>{\RaggedLeft}p{15pt}@{}>{\RaggedLeft}p{25pt}>{\RaggedLeft}p{8pt}@{}>{\RaggedLeft}p{15pt}>{\RaggedLeft}p{8pt}@{}>{\RaggedLeft}p{10pt}}\hline
\multicolumn{1}{c}{\itshape Scenarios}&
\multicolumn{1}{c}{}&
\multicolumn{12}{c}{\itshape Posterior regrets\tnote{a}}
\tabularnewline \cline{3-14}
\multicolumn{1}{>{\RaggedRight}p{58pt}}{}&\multicolumn{1}{c}{QR-SEL}&
\multicolumn{2}{c}{MLE}&\multicolumn{2}{c}{SSEL}&\multicolumn{2}{c}{WRSEL}&\multicolumn{2}{c}{CB}&
\multicolumn{2}{c}{GR}&\multicolumn{2}{c}{RoPQ}\tabularnewline
\hline
{\itshape $\op{RLS}\doteq1$}&&&&&&&&&&&&&\tabularnewline
\normalfont   N-N, $n=100$ &    $21$ &    $  329$ &    ($   1565$) &    $151$ &    ($  717$) &    $  39$ &    ($   184$) &    $  9$ &    ($  45$) &    $0$ &    ($ 1$) &    $0$ &    ($2$)\tabularnewline
\normalfont   N-N, $n=200$ &    $11$ &    $  325$ &    ($   3001$) &    $158$ &    ($ 1455$) &    $ 705$ &    ($  6503$) &    $  4$ &    ($  38$) &    $0$ &    ($ 0$) &    $0$ &    ($1$)\tabularnewline
\normalfont   N-N, $n=1000$ &    $ 2$ &    $  313$ &    ($  14063$) &    $160$ &    ($ 7198$) &    $6258$ &    ($281565$) &    $  2$ &    ($  76$) &    $0$ &    ($ 1$) &    $0$ &    ($0$)\tabularnewline
\normalfont   G-IG, $n=100$ &    $27$ &    $18641$ &    ($  69283$) &    $467$ &    ($ 1736$) &    $ 160$ &    ($   593$) &    $ 61$ &    ($ 226$) &    $0$ &    ($ 2$) &    $0$ &    ($0$)\tabularnewline
\normalfont   G-IG, $n=200$ &    $14$ &    $12501$ &    ($  92484$) &    $488$ &    ($ 3608$) &    $  16$ &    ($   116$) &    $ 44$ &    ($ 328$) &    $0$ &    ($ 2$) &    $0$ &    ($0$)\tabularnewline
\normalfont   G-IG, $n=1000$ &    $ 3$ &    $12692$ &    ($ 471628$) &    $486$ &    ($18078$) &    $2098$ &    ($ 77969$) &    $ 50$ &    ($1871$) &    $0$ &    ($ 4$) &    $0$ &    ($0$)\tabularnewline
\hline
{\itshape $\op{RLS}\doteq20$}&&&&&&&&&&&&&\tabularnewline
\normalfont   N-N, $n=100$ &    $21$ &    $  448$ &    ($   2139$) &    $181$ &    ($  866$) &    $  20$ &    ($    97$) &    $ 10$ &    ($  48$) &    $0$ &    ($ 1$) &    $0$ &    ($2$)\tabularnewline
\normalfont   N-N, $n=200$ &    $10$ &    $  436$ &    ($   4237$) &    $206$ &    ($ 1998$) &    $ 478$ &    ($  4643$) &    $  9$ &    ($  91$) &    $0$ &    ($ 0$) &    $0$ &    ($1$)\tabularnewline
\normalfont   N-N, $n=1000$ &    $ 2$ &    $  412$ &    ($  19081$) &    $195$ &    ($ 9012$) &    $5016$ &    ($232022$) &    $  6$ &    ($ 278$) &    $0$ &    ($ 1$) &    $0$ &    ($0$)\tabularnewline
\normalfont   G-IG, $n=100$ &    $27$ &    $22062$ &    ($  83039$) &    $572$ &    ($ 2154$) &    $ 296$ &    ($  1113$) &    $105$ &    ($ 397$) &    $0$ &    ($ 1$) &    $0$ &    ($0$)\tabularnewline
\normalfont   G-IG, $n=200$ &    $13$ &    $19019$ &    ($ 142416$) &    $565$ &    ($ 4228$) &    $  63$ &    ($   475$) &    $153$ &    ($1148$) &    $0$ &    ($ 1$) &    $0$ &    ($0$)\tabularnewline
\normalfont   G-IG, $n=1000$ &    $ 3$ &    $17738$ &    ($ 675792$) &    $566$ &    ($21555$) &    $ 850$ &    ($ 32365$) &    $176$ &    ($6721$) &    $0$ &    ($11$) &    $0$ &    ($0$)\tabularnewline
\hline
{\itshape $\op{RLS}\doteq100$}&&&&&&&&&&&&&\tabularnewline
\normalfont   N-N, $n=100$ &    $20$ &    $  522$ &    ($   2660$) &    $213$ &    ($ 1087$) &    $   7$ &    ($    34$) &    $ 18$ &    ($  93$) &    $0$ &    ($ 0$) &    $0$ &    ($2$)\tabularnewline
\normalfont   N-N, $n=200$ &    $10$ &    $  463$ &    ($   4578$) &    $270$ &    ($ 2664$) &    $ 259$ &    ($  2555$) &    $ 32$ &    ($ 320$) &    $0$ &    ($ 0$) &    $0$ &    ($1$)\tabularnewline
\normalfont   N-N, $n=1000$ &    $ 2$ &    $  500$ &    ($  24301$) &    $270$ &    ($13108$) &    $3804$ &    ($184974$) &    $ 32$ &    ($1542$) &    $0$ &    ($ 1$) &    $0$ &    ($0$)\tabularnewline
\normalfont   G-IG, $n=100$ &    $26$ &    $65739$ &    ($ 256477$) &    $563$ &    ($ 2197$) &    $ 301$ &    ($  1176$) &    $148$ &    ($ 579$) &    $1$ &    ($ 3$) &    $0$ &    ($0$)\tabularnewline
\normalfont   G-IG, $n=200$ &    $13$ &    $36096$ &    ($ 284722$) &    $572$ &    ($ 4510$) &    $  97$ &    ($   763$) &    $147$ &    ($1160$) &    $0$ &    ($ 1$) &    $0$ &    ($0$)\tabularnewline
\normalfont   G-IG, $n=1000$ &    $ 3$ &    $29789$ &    ($1140841$) &    $586$ &    ($22424$) &    $ 569$ &    ($ 21800$) &    $230$ &    ($8806$) &    $1$ &    ($24$) &    $0$ &    ($0$)\tabularnewline
\hline
\end{tabular}
\begin{tablenotes}
   \item[a] Entries for both the posterior and percentage regrets have been truncated to
     the closest integer. For some entries, percentage regrets are smaller than 1 percentage point. 
\end{tablenotes}
\end{threeparttable}
\end{table}

The simulation results for the QR-SEL function are reported in table
\ref{tab:qrsel_table}, on page \pageref{tab:qrsel_table}. For the
compound Gaussian model, we computed the plug-in estimators under the
IQR-SEL function. For this loss, the optimal estimator is the
posterior IQR, and the RoPQ is replaced by the DoPQ, as
described in equation (\ref{eq:dopq}). We considered the effect of the different 
experimental factors on the performance of the plug-in estimators, in turn.

Overall, the ordering of the different plug-in estimators in terms of
percentage regret under QR-SEL was found to be identical to the
ordering observed under the Q-SEL function. The empirical QR of the
MLEs tended to exhibit the largest percentage regret across all conditions and
scenarios, except under the N-N model with $n=1000$, for which the
WRSEL plug-in estimator performed worse. As was previously noticed
under the Q-SEL, the WRSEL plug-in estimator of the empirical QR was found to be
very sensitive to the size of the parameter ensemble. That is, the performance
of the WRSEL plug-in estimator severely deteriorated as $n$ grew
large. We can explain this effect in terms of the modification of the
$\phi_{i}$'s in the WRSEL function following an increase in $n$. 
Ignoring the WRSEL-based estimator, the SSEL plug-in estimator
exhibited the second worst percentage regret after the MLE-based
QR. This was followed by the CB plug-in estimator. The triple-goal and
RoPQ were found to almost match the performance of the optimal
estimator under the QR-SEL function over all simulated scenarios. No
systematic difference in percentage regret could be identified between
the GR and the RoPQ plug-in estimators. 

As for the Q-SEL function, increasing the level of heterogeneity of
the sampling variances had a detrimental effect on the percentage
regret exhibited by the MLE, SSEL and CB plug-in estimators. The
performance of these three families of QR estimators systematically
decreased as RLS increased (cf. table \ref{tab:qrsel_table} on page
\pageref{tab:qrsel_table}). The plug-in estimator under the WRSEL
function, by contrast, tended to behave in the opposite direction,
whereby an increase in RLS yielded an increase in performance, albeit
this increase was restricted to the N-N model. Both the GR and
DoPQ/RoPQ plug-in estimators appeared to be robust to changes in RLS
levels, although the GR estimator of the QR exhibited a substantial
increase in percentage regret for very large parameter ensembles
--i.e. for $\op{RLS}\doteq 100$ and $n=1000$, where the percentage
regret for using the GR reached $24\%$ of the posterior loss under the
optimal estimator. 

The choice of statistical model had a strong effect on the performance
of the MLE, SSEL and CB plug-in estimators. Consider the values of the
percentage regrets in parentheses in columns three, five and nine of
table \ref{tab:qrsel_table} on page \pageref{tab:qrsel_table}. For these three families
of ensembles of point estimates, the use of a compound Gamma model
yielded a larger percentage regret. For the WRSEL
estimator of the QR, no systematic trend could be identified, as the
performance of that estimator under the different statistical models seemed to
be also dependent on the choice of RLS levels and the number of units in the
parameter ensemble. In addition, the behaviour of the triple-goal
plug-in estimator appeared to deteriorate under the G-IG model. The
DoPQ/RoPQ plug-in estimators, by
contrast, exhibited slightly lower percentage regret under the compound Gamma
model. 

Increasing the size of the parameter ensemble tended to produce
worse plug-in estimators with higher posterior percentage
regrets. This trend can be observed for the MLE, SSEL, WRSEL and CB
plug-in estimators. The triple-goal and DoPQ/RoPQ-based QR estimators,
by contrast, appeared to be robust to an increase in the size of the
parameter ensemble. No systematic effect of the size of $n$ could be
noticed for these two families of estimators. 
In summary, the GR and DoPQ/RoPQ plug-in estimators exhibited the best
performance under both the Q-SEL and QR-SEL functions, over
all the scenarios considered. We now present a set of spatial
simulations to evaluate the behaviour of the different families of
plug-in estimators of interest under more realistic modelling assumptions.

\section{Spatial Simulations}\label{sec:mrrr spatial sim}
In this simulation study, four experimental factors were of
interest. Firstly, we manipulated the spatial structure of the
true parameter ensemble. Secondly, we considered the effect of different
modelling assumptions. Thirdly, we modified the level of heterogeneity
of the parameter ensemble, and finally, we assessed the
influence of changing the overall level of the expected counts, in
each scenario.

\subsection{Design}
The general framework for data generation followed the one
implemented by \citet{Richardson2004}. 
The data were generated using a subset of expected counts from the 
Thames Cancer Registry (TCR) for lung cancer in West
Sussex, which contains $n=166$ wards. The expected
counts, denoted $E_{i}$'s for $i=1,\ldots,n$,
were adjusted for age only, and correspond to the number of
cases of lung cancer occurring among males during the period 
1989--2003. The overall level of expected counts for lung cancer over
the entire region varied from $\min(E)=7.97$ to 
$\max(E)=114.35$ with a mean of $42.44$. 
The level of the expected counts was also made to vary across the
simulations, in order to assess its influence on the performances of
the different plug-in estimators. Specifically, we will
study the effect of three different levels of expected counts over the
performance of the plug-in estimators of interest. 

Four different types of spatial patterns were simulated. The main
objective of these simulations was to create a set of varying levels
of difficulty for a smoothing technique, such as the CAR prior. 
The scenarios were thus constructed in order to include a scenario with
only a single isolated area of elevated risk (simulation SC1, the
hardest scenario for a smoothing procedure) and a situation
containing a relatively smooth risk surface with clear spatial
structure (simulation SC3). The spatial simulations were constructed
either by randomly selecting a subset of the areas in the region and labelling these
areas as elevated-risk areas (SC1 and SC2), or by 
directly modelling the generation of the RRs ascribed to
each area (SC3 and SC4). The construction of the spatial scenarios
occurred only once, and the overall level of the
expected counts was kept constant throughout the four spatial scenarios.

\begin{figure}[t]
\centering
\includegraphics[height=11cm]{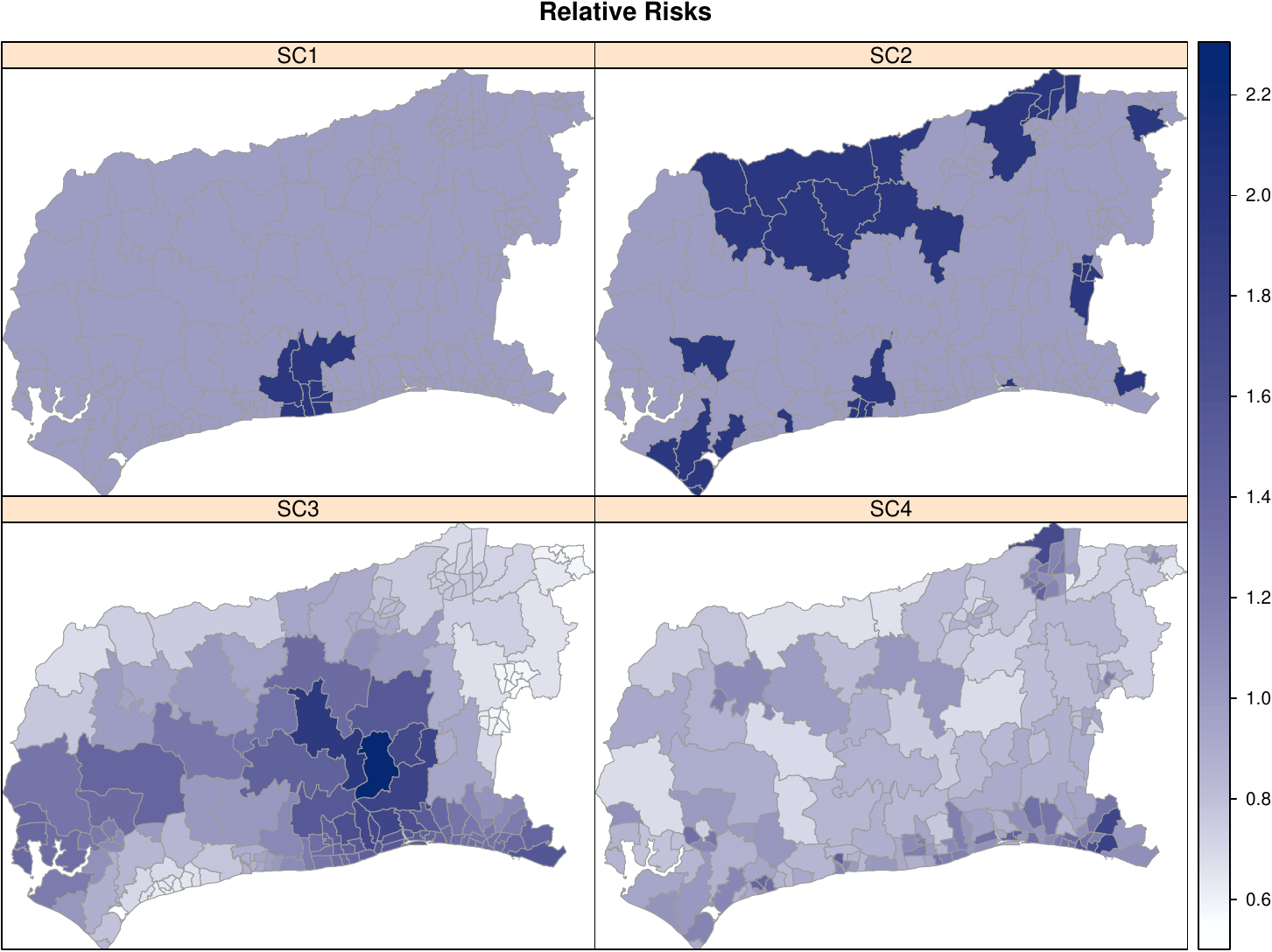}
\caption{True RRs based on expected counts ($\op{SF}=1.0$) 
     for lung cancer among males between 1989 and 2003,
     for 166 wards located in West Sussex, UK. The values of
     the RRs were generated using the protocols described in section
     \ref{sec:spatial scenarios}, and are respectively denoted by 
     SC1 (one isolated cluster), SC2 (five isolated clusters and five isolated areas),
     SC3 (spatial pattern generated using the Mat\`{e}rn function),
     SC4 (spatial pattern using a hidden covariate). RRs in all four
     scenarios were here produced with a medium  level of variability.}
     \label{fig:RR}
\end{figure}
\begin{figure}[t]
\centering
\includegraphics[height=11cm]{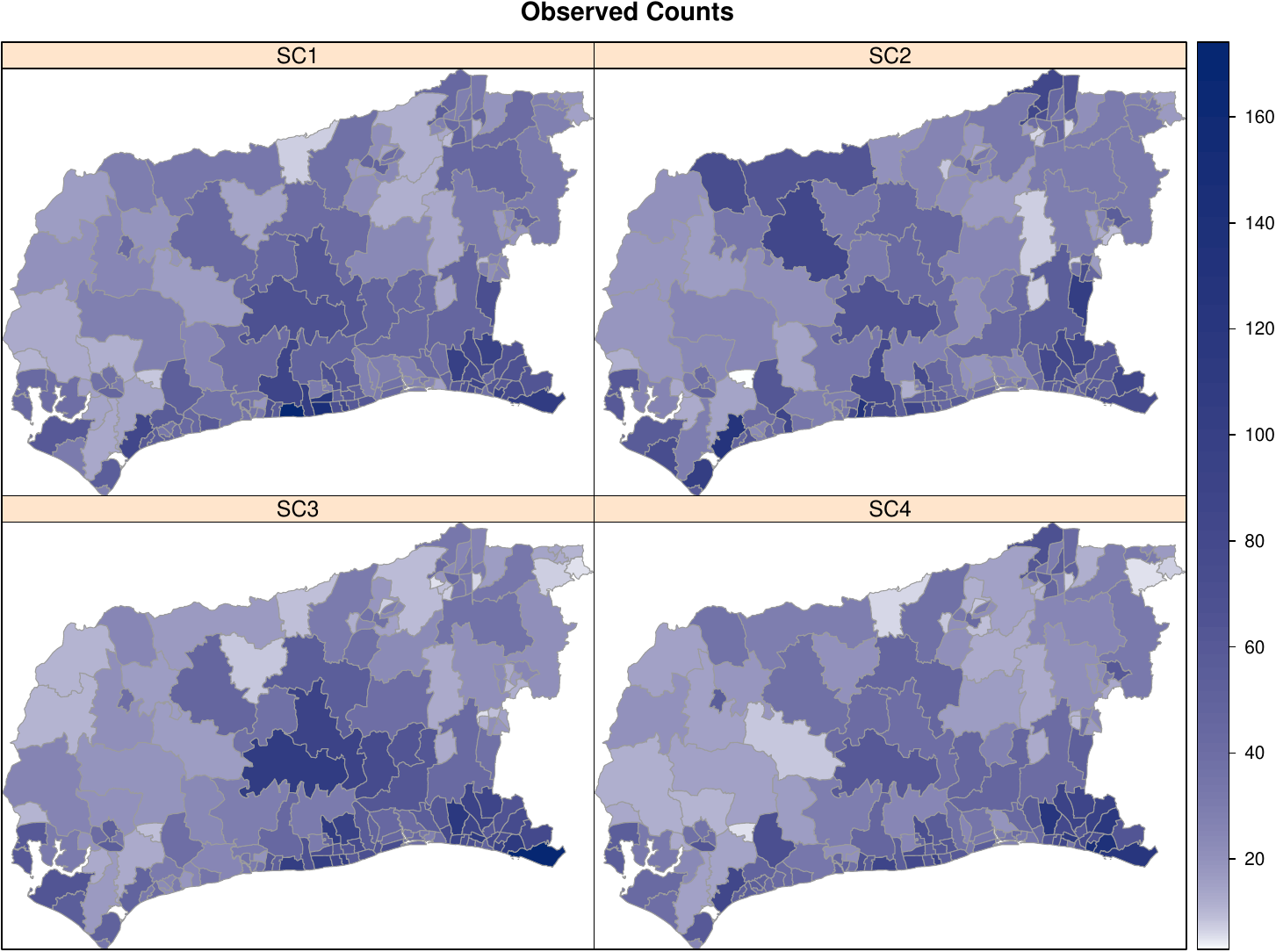}
\caption{
     Simulated observed counts based on expected
     counts ($\op{SF}=1.0$) for lung cancer among males between 1989 and 2003,
     for 166 wards located in West Sussex, UK. The values of
     the $y_{i}$'s were generated following the protocols described in section
     \ref{sec:spatial scenarios}, with a medium level of variability in the RRs.}
     \label{fig:OBS}
\end{figure}
\begin{figure}[t]
\centering
\includegraphics[height=11cm]{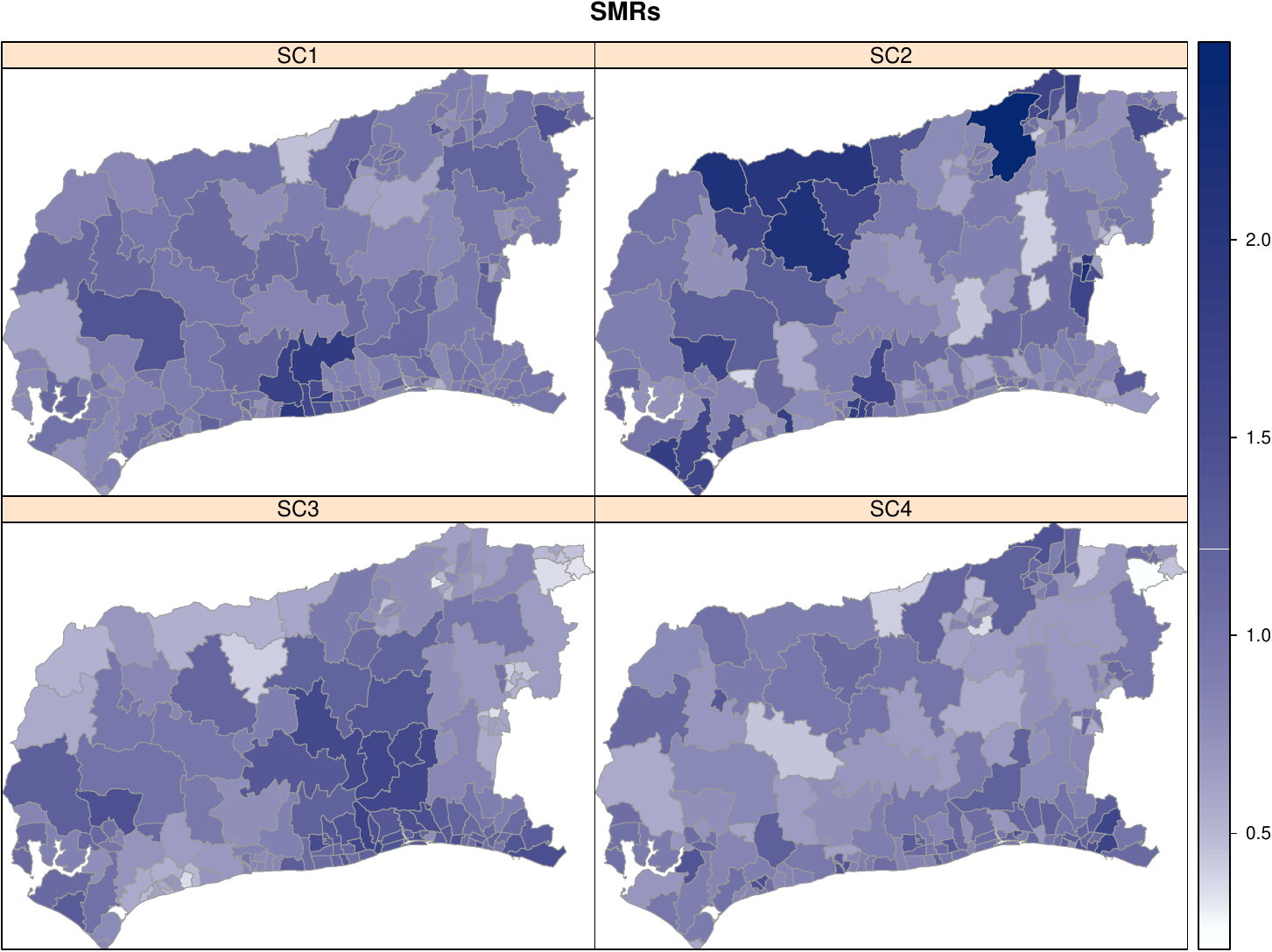}
\caption{
     Simulated SMRs based on expected
     counts ($\op{SF}=1.0$) for lung cancer among males between 1989 and 2003,
     for 166 wards located in West Sussex, UK. The values of
     the SMRs were generated following the protocols described in section
     \ref{sec:spatial scenarios}, with a medium level of variability
     in the RRs.}
     \label{fig:SMR}
\end{figure}

\subsection{Spatial Structure Scenarios}\label{sec:spatial scenarios}
Four different spatial structures were constructed, which respectively
correspond to a situation with one isolated foyer of elevated risk
(SC1), five isolated foyers and five isolated areas of elevated risk (SC2), a
spatially distributed risk pattern (SC3), and the effect of a spatially
structured risk factor (SC4). For each scenario, we describe the
generation of the true RRs. In each case, three different levels of
risk variability across the region was generated by manipulating
scenario-specific parameters. Scenarios SC1 and SC2 replicate the
first two spatial simulations considered by \citet{Richardson2004}.
\begin{description}
\item[Simulation 1 (SC1).] In the first scenario, a single isolated cluster of areas
    with elevated risk was generated by choosing an area randomly, and
    selecting its first-degree neighbours. The cluster was chosen such
    that the total sum of the expected counts of the areas
    contained in the cluster was approximately equal to $5\%$ of
    the sum of the expected counts over the entire region. The remaining areas in the
    region will be referred to as background areas. The set of 
    indices of the wards in the cluster will be denoted by $\cE$,
    whereas the background areas will be denoted by $\cB:=\lb i=1,\dots,n: i\notin
    \cE\rb$. The level of risk in each individual ward was then
    fixed using the following rule, 
    \begin{equation}
       \theta_{i} :=
        \begin{cases}
         1       & \te{if } i \in \cB, \\
         \op{LR}  & \te{if } i \in \cE;
        \end{cases}
        \label{eq:elevated risk rule}
    \end{equation}
    where LR stands for level of risk. Three different levels of
    elevated risk were chosen with $\op{LR}=\lb 1.5, 2.0, 3.0\rb$, thereby
    creating different levels of RR variability in the simulated data sets. 
\item[Simulation 2 (SC2).] In the second scenario, we created a situation of
    mixed heterogeneity by combining both the simulation of clustered foyers of risks 
    with the presence of isolated single areas with elevated
    risks. Firstly, five single areas were categorised as elevated risk
    areas. These areas corresponded to the 
    $10\tth$, $20\tth$, $50\tth$, $75\tth$ and $90\tth$
    percentiles of the empirical distribution of the expected
    counts. The indices of these five wards
    will be denoted by $I$. In addition, we chose five non-contiguous
    clusters, which did not comprise the individual areas of elevated
    risk. These foyers, denoted $C_{j}$'s with $j=1,\ldots,5$, were
    constructed using the method described for the first scenario,
    ensuring that the cumulated expected counts in each cluster did not
    exceed $5\%$ of the total sum of expected counts in the region. 
    Variability in levels of risk was here controlled by varying LR.
    As in the first scenario, for every $i$, we defined the
    $\theta_{i}$'s according to the rule described in equation
    (\ref{eq:elevated risk rule}), except that the set of indices of
    the elevated risk areas was here defined as 
    \begin{equation}
        \cE := I \cup \bigcup_{j=1}^{5}C_{j}.
    \end{equation}
    Note that some buffering was required in simulations SC1 and SC2
    in order to ensure
    that the randomly selected areas and clusters, respectively, were not
    adjacent or overlapping with each other. This was done by selecting a
    buffer of background areas around each of the cluster in
    SC1, and around each cluster and each of the individual areas 
    in SC2. Note that in the latter case, the neighbours of the neighbours
    of the regions included in any one cluster were also excluded from further random
    selection, in order to ensure non-adjacent clusters. 
\item[Simulation 3 (SC3).] In a third scenario, a spatially structured
    risk surface was generated using the Mat\`{e}rn function. 
    Spatial structure was specified through the variance/covariance
    matrix, $\bSig$, of the $\log\theta_{i}$'s. 
    We computed the following symmetric matrix of intercentroidal
    distances, denoted $\bD$, whose entries take the following values,
    \begin{equation}
          D_\hi{ij}=||\bx_i-\bx_j||_{2},
    \end{equation}
for every pair of indices $(i,j)$, with 
$||\cdot||_{2}$ denoting the L2-norm and where each $\bx_i$ represents a set of
2-dimensional coordinates locating the centroid of the $i\tth$ region in $\R^{2}$. 
By convention, $D_\hi{ii}:= 0$ for every $i$. We then computed the mean
intercentroidal distance between all pairs of neighbours, defined as follows,
\begin{equation}
    \bar{w}=\frac{1}{N_{\partial}}\sum_{i<j}^n D_{ij}\cI\lb i\sim j\rb,
\end{equation}
where $N_{\partial}$ is the total number of neighbouring relationships
in the entire region. We used $\bar{w}$ to select a particular level
of risk decay, as a function of the intercentroidal distance. Our
objective was here to obtain high to moderate correlations when the
intercentroidal distance approached $\bar{w}$. Spatial
autocorrelations were then made to drop off to
a low level once the distances were equal to twice the
value of the mean intercentroidal distance. We used the
Mat\`{e}rn function to specify this form of spatial structure in the 
 entries of the variance/covariance matrix, such that
\begin{equation}
    \Sig_\hi{ij}=
    \frac{1}{2^{\nu-1}\Ga(\nu)}(2\sqrt{\nu}D_\hi{ij}\phi)^\nu
    K_\nu(2\sqrt{\nu}D_\hi{ij}\phi),
\end{equation}
choosing $\phi=3000$ and $\nu=40$, and where $K_\nu(\cdot)$ denotes
    the modified Bessel function of the third kind of order $\nu$. The
    resulting matrix $\bSig$ was then used to draw sets of realisations
    from a multivariate normal distribution,
    \begin{equation}
         \log(\bth)\sim \mvn(\bm{0},\sig_{\op{SC3}}^2\bSig),
    \end{equation}
    with parameter $\sig^2_{\op{SC3}}$ controlling for the overall marginal
    variance of the spatial structure. Three different values of
    $\sig_{\op{SC3}}$ among $\lb 0.1,0.2,0.3\rb$ were
    selected in order to vary the amount of variability of the RRs.
\item[Simulation 4 (SC4).] The last scenario 
    produced a risk surface characterised by a high
    heterogeneity, which was generated using a hidden covariate. 
    We used the Carstairs' deprivation index (CDI) for each area in the
    region of interest in order to create spatial variation in levels
    of risk \citep{Carstairs1989,Carstairs1989a,Carstairs2000}.
    The logRRs were here generated using a linear combination of the
    ward-specific CDIs. Formally, we had
    \begin{equation}
         \log(\theta_{i}) = \alpha + \beta C_{i} + v_{i},
    \end{equation}
    where $v_{i}\sim N(0,\sig_{\op{SC4}}^{2})$, and $C_{i}$ indicates the level
    of social deprivation in the $i\tth$ ward. The
    intercept $\alpha$ was assumed to be null throughout the
    simulations. The regression coefficient, $\beta$, took values
    in the set $\lb 0.2,0.3,0.4\rb$ in order to produce different levels of
    variability in the parameter ensemble, while the standard
    deviation, $\sig_{\op{SC4}}$, was fixed to $0.1$ in this scenario.
    The set of values for $\beta$  allowed the generation of scenarios
    with low, medium and high RR variability.
\end{description}

The simulated observed counts in each scenario were simulated under
the constraint $\sum_{i=1}^{n} y_{i}=\sum_{i=1}^{n}E_{i}$
\citep[see][for an example of synthetic simulations under this
constraint]{Wright2003}. 
This condition was respected by generating the full vector of observed
counts, $\by$, from a multinomial distribution parametrised such
that the number of trials corresponded to the total number of expected
counts $\sum_{i=1}^{n}E_{i}$, denoted $N_{E}$, and with the following vector of
probability masses 
\begin{equation}
  \bm{\pi} :=  \lt\lb\frac{E_{1}\theta_{1}}{\sum_{i=1}^{n} E_{i}\theta_{i}}, \ldots,
          \frac{E_{n}\theta_{n}}{\sum_{i=1}^{n} E_{i}\theta_{i}}, \rt\rb,
\end{equation}
and using 
\begin{equation}
  \by \sim \op{Multinomial}\lt(N_{E},\bm{\pi}\rt),
\end{equation}
for each of the scenarios and under the three levels of variability. 
For every combination of scenario and variability, we produced 100
replicates, thereby totalling 1,200 different data sets. 
Since each of these data sets was fitted with two different
models, the complete spatial simulations generated 2,400 sets of
joint posterior distributions of parameter ensembles. In addition, we
also produced further simulations in order to evaluate the effect of
different scaling of the expected counts.  We employed three
different values of the scaling factor (SF), such that
\begin{equation}
     \op{SF}\in \lb0.1,1.0,2.0\rb. 
\end{equation}
That is, for each combination of the aforementioned experimental factors,
we additionally varied the level of the expected counts 
\begin{equation}
      y_i\stack{\ind}{\sim} \poi(\theta_iE_{i}\op{SF}), 
\end{equation}
for every $i=1,\ldots,n$, for a given vector of true $\theta_{i}$'s,
whose generation depended on the specific spatial scenario considered. 
The main simulation results in this section will be
described for $\op{SF}=1.0$. In section \ref{sec:mrrr sf}, however, we
will briefly consider the consequences of scaling up or scaling down
the expected counts for a smaller subset of simulated data, choosing
SF to be either 0.1 or 2.0. 

Typical examples of the different maps produced by this set of
simulated scenarios are
illustrated in Figures \ref{fig:RR}, \ref{fig:OBS}, and \ref{fig:SMR},
representing the true RRs, observed counts and standard mortality ratios,
respectively. These maps were produced for a medium level of
variability with $\op{LR}=2$ for SC1 and SC2,
$\sig_{\op{SC3}}^{2}=\sig_{\op{SC4}}^{2}=0.2$ and $\beta=0.3$, for
SC3 and SC4. The full details of the simulated data sets are given in tables 
\ref{tab:E_table}, \ref{tab:RR_table} and \ref{tab:Y_table}, on pages 
\pageref{tab:E_table}, \pageref{tab:RR_table} and
\pageref{tab:Y_table}, respectively, in appendix \ref{app:spatial}.
These tables also include descriptive statistics for simulated data based
on different scaling of the expected counts.

\subsection{Fitted Models}\label{sec:mrrr spatial models}
The synthetic data sets were modelled using the two spatial BHMs
described in section \ref{sec:spatial models}. These models are
characterised by the combined use of spatially
structured and unstructured random effects. This particular
combination of random effects was originally introduced by \citet{Besag1991}, and is
therefore sometimes referred to as the BYM model \citep[see
also][]{Best2005}. The CAR Normal and CAR Laplace models were fitted using WinBUGS
1.4 \citep{Lunn2000} with 4,000 iterations including 2,000
burn-in. In this software, the $u_{i}$'s in the CAR priors are
constrained to sum to zero \citep{Spiegelhalter2000}. The
non-stationarity of the model, however, is induced by specifying a
flat prior on $\alpha$, which is the intercept for the $\log\theta_{i}$'s in
equation (\ref{eq:linear theta}). That is, we have 
\begin{equation}
     \alpha \sim \op{Unif}(-\infty,+\infty).
\end{equation}
This particular formulation of the CAR distribution was suggested by
\citet{Besag1995}. The CAR Normal and CAR Laplace models are fully
specified once the hyperparameters $\tau_{v}^{2}$ and $\tau^{2}_{u}$
in equations (\ref{eq:unstructured}) and (\ref{eq:car}), respectively,
are given hyperpriors. In our case, these two parameters were given
the following `vague' specification, 
\begin{equation}
    \tau^{-2}_{v} \sim \op{Gam}(a_{0},b_{0}), \qq
    \tau^{-2}_{u} \sim \op{Gam}(a_{0},b_{0}), 
\end{equation}
where we chose $a_{0}:=.5$ and $b_{0}:=.0005$, which constitutes a common
specification for BHMs \citep{Gelman2006}. Note, however, that
different choices of hyperparameters may yield different posterior
distributions of the parameters of interest.
The WinBUGS codes used for these two models is reported in appendix
\ref{app:winbugs}. We now consider in greater detail the various aspects of a parameter
ensemble in a BHM that may be of interest, and how estimation of
these aspects may be optimised.

For each simulated scenario, we computed the posterior regrets under
the Q-SEL and QR-SEL functions, as described in section \ref{sec:mrrr
  evaluation}. In both cases, we estimated these loss functions with
respect to the vector of RRs, $\bth$. For the Q-SEL function, this
gave
\begin{equation}
    \op{regret}\lt(\op{Q-SEL}_{\bp},Q_{\hat{\bth}^{L\pr}}(\bp)\rt).
\end{equation}
for different plug-in estimators, $\hat\bth^{L\pr}$, and where the
optimal estimator under this loss function is the vector of posterior empirical
$\bp$-quantiles as reported in equation (\ref{eq:qsel minimiser}).
Similarly, the following posterior regret was calculated to quantify
the estimation of the QR using ensembles of sub-optimal point
estimates,
\begin{equation}
   \op{regret}\lt(\op{QR-SEL},\op{QR}(\hat\bth^{L\pr})\rt), 
 \end{equation}
defined with respect to $\op{QR-SEL}(\bth,\delta)$, as described in
section \ref{sec:mrrr evaluation}. Note that we have here chosen to
compute both the empirical quantiles and the QR on the scale of the
RR. However, we could have chosen to evaluate these quantities with
respect to the logRRs. In section \ref{sec:mrrr conclusion}, we
discuss the implications of such a change of parametrisation on the
decision-theoretic problems of interest. 

All the results presented in the following sections are expressed as posterior
expected losses, which are functions of the joint posterior
distribution of the parameters under scrutiny. Since these evaluative
criteria are highly dependent on the type of hierarchical models used
for the evaluation, we have also computed the mean squared error of
the QR with respect to the true QR, based on the simulated RRs. 
For the BYM model, we found that the optimal QR estimate (i.e. the
posterior mean of the empirical QR) was on average at a $0.082$ square
distance from the true QR. That discrepancy was slightly higher for
the Laplace model, at $0.087$, when averaged over all simulated scenarios.

\subsection{Plug-in Estimators under Q-SEL}
\begin{table}[t]
 \footnotesize
 \caption{
Posterior regrets based on $\op{Q-SEL}_{\bp}(\bth,Q_{\hat{\bth}^{L\pr}}(\bp))$ with $\bp:=\lb.25,.75\rb$,
for five plug-in estimators, and with the posterior expected loss of the optimal estimator in the first column.
Results are presented for three different levels of variability, with $\op{SF}=1.0$ 
and for four spatial scenarios: an isolated cluster (SC1), a set of isolated clusters and isolated areas (SC2),
highly structured spatial heterogeneity (SC3), and a risk surface generated by a hidden covariate (SC4). 
Entries, averaged over 100 replicate data sets in each condition, are scaled by a factor of $10^3$
with posterior regrets expressed as percentage of the posterior loss
under the optimal estimator in parentheses.
\label{tab:spatial_qsel_table}} 
 \centering
 \begin{threeparttable}
 \begin{tabular}{>{\RaggedRight}p{50pt}>{\RaggedLeft}p{25pt}|>{\RaggedLeft}p{25pt}@{}>{\RaggedLeft}p{25pt}>{\RaggedLeft}p{25pt}@{}>{\RaggedLeft}p{25pt}>{\RaggedLeft}p{25pt}@{}>{\RaggedLeft}p{30pt}>{\RaggedLeft}p{25pt}@{}>{\RaggedLeft}p{20pt}>{\RaggedLeft}p{25pt}@{}>{\RaggedLeft}p{15pt}}\hline
\multicolumn{1}{c}{\itshape Scenarios}&
\multicolumn{1}{c}{}&
\multicolumn{10}{c}{\itshape Posterior regrets\tnote{a}}
\tabularnewline \cline{3-12}
\multicolumn{1}{>{\RaggedRight}p{50pt}}{}&\multicolumn{1}{c}{Q-SEL}&
\multicolumn{2}{c}{MLE}&\multicolumn{2}{c}{SSEL}&\multicolumn{2}{c}{WRSEL}&\multicolumn{2}{c}{CB}&\multicolumn{2}{c}{GR}\tabularnewline
\hline
{\itshape  Low Variab.~}&&&&&&&&&&&\tabularnewline
\normalfont   BYM-SC1 &    $0.6$ &    $ 6.9$ &    ($1091$) &    $ 2.0$ &    ($315$) &    $ 22.0$ &    ($ 3477$) &    $ 0.3$ &    ($ 55$) &    $0.1$ &    ($11$)\tabularnewline
\normalfont   BYM-SC2 &    $0.8$ &    $ 4.1$ &    ($ 511$) &    $ 3.3$ &    ($416$) &    $ 93.8$ &    ($11651$) &    $ 1.1$ &    ($142$) &    $0.1$ &    ($ 7$)\tabularnewline
\normalfont   BYM-SC3 &    $1.1$ &    $ 3.4$ &    ($ 300$) &    $ 0.7$ &    ($ 63$) &    $151.4$ &    ($13427$) &    $ 0.7$ &    ($ 58$) &    $0.1$ &    ($ 5$)\tabularnewline
\normalfont   BYM-SC4 &    $0.9$ &    $ 5.5$ &    ($ 606$) &    $ 1.7$ &    ($183$) &    $140.6$ &    ($15476$) &    $ 0.4$ &    ($ 48$) &    $0.1$ &    ($ 7$)\tabularnewline
\normalfont   L1-SC1 &    $0.7$ &    $ 7.2$ &    ($1058$) &    $ 2.1$ &    ($316$) &    $ 21.1$ &    ($ 3104$) &    $ 0.3$ &    ($ 41$) &    $0.1$ &    ($10$)\tabularnewline
\normalfont   L1-SC2 &    $0.8$ &    $ 4.1$ &    ($ 504$) &    $ 3.7$ &    ($458$) &    $ 99.3$ &    ($12136$) &    $ 1.0$ &    ($124$) &    $0.1$ &    ($ 8$)\tabularnewline
\normalfont   L1-SC3 &    $1.2$ &    $ 4.5$ &    ($ 375$) &    $ 1.3$ &    ($108$) &    $206.2$ &    ($17237$) &    $ 0.6$ &    ($ 54$) &    $0.1$ &    ($ 5$)\tabularnewline
\normalfont   L1-SC4 &    $1.0$ &    $ 5.6$ &    ($ 578$) &    $ 2.2$ &    ($232$) &    $156.9$ &    ($16341$) &    $ 0.4$ &    ($ 42$) &    $0.1$ &    ($ 7$)\tabularnewline
\hline
{\itshape  Med.~ Variab.~}&&&&&&&&&&&\tabularnewline
\normalfont   BYM-SC1 &    $0.6$ &    $ 1.6$ &    ($ 246$) &    $ 3.3$ &    ($510$) &    $ 65.9$ &    ($10179$) &    $ 1.3$ &    ($195$) &    $0.1$ &    ($10$)\tabularnewline
\normalfont   BYM-SC2 &    $1.3$ &    $ 2.7$ &    ($ 200$) &    $ 9.9$ &    ($738$) &    $235.5$ &    ($17548$) &    $ 7.3$ &    ($542$) &    $0.1$ &    ($ 5$)\tabularnewline
\normalfont   BYM-SC3 &    $1.5$ &    $ 2.8$ &    ($ 181$) &    $ 1.0$ &    ($ 65$) &    $215.0$ &    ($13906$) &    $ 0.8$ &    ($ 52$) &    $0.1$ &    ($ 4$)\tabularnewline
\normalfont   BYM-SC4 &    $1.2$ &    $ 4.0$ &    ($ 325$) &    $ 1.2$ &    ($ 98$) &    $230.9$ &    ($18609$) &    $ 0.5$ &    ($ 43$) &    $0.1$ &    ($ 5$)\tabularnewline
\normalfont   L1-SC1 &    $0.7$ &    $ 1.5$ &    ($ 223$) &    $ 3.7$ &    ($531$) &    $ 71.1$ &    ($10299$) &    $ 1.3$ &    ($188$) &    $0.1$ &    ($10$)\tabularnewline
\normalfont   L1-SC2 &    $1.4$ &    $ 3.9$ &    ($ 270$) &    $10.6$ &    ($738$) &    $210.7$ &    ($14650$) &    $ 7.4$ &    ($517$) &    $0.1$ &    ($ 4$)\tabularnewline
\normalfont   L1-SC3 &    $1.7$ &    $ 3.5$ &    ($ 208$) &    $ 1.5$ &    ($ 88$) &    $321.7$ &    ($19389$) &    $ 1.0$ &    ($ 59$) &    $0.1$ &    ($ 4$)\tabularnewline
\normalfont   L1-SC4 &    $1.3$ &    $ 4.4$ &    ($ 338$) &    $ 1.7$ &    ($129$) &    $262.5$ &    ($20067$) &    $ 0.7$ &    ($ 51$) &    $0.1$ &    ($ 5$)\tabularnewline
\hline
{\itshape  High Variab.~}&&&&&&&&&&&\tabularnewline
\normalfont   BYM-SC1 &    $0.8$ &    $ 0.6$ &    ($  85$) &    $ 4.3$ &    ($566$) &    $115.2$ &    ($15209$) &    $ 2.6$ &    ($342$) &    $0.1$ &    ($10$)\tabularnewline
\normalfont   BYM-SC2 &    $2.2$ &    $ 7.7$ &    ($ 353$) &    $18.4$ &    ($840$) &    $578.4$ &    ($26420$) &    $16.8$ &    ($767$) &    $0.1$ &    ($ 3$)\tabularnewline
\normalfont   BYM-SC3 &    $1.8$ &    $ 2.8$ &    ($ 158$) &    $ 1.0$ &    ($ 57$) &    $253.8$ &    ($14156$) &    $ 1.1$ &    ($ 62$) &    $0.1$ &    ($ 4$)\tabularnewline
\normalfont   BYM-SC4 &    $1.6$ &    $ 3.4$ &    ($ 214$) &    $ 1.0$ &    ($ 61$) &    $300.6$ &    ($18872$) &    $ 0.8$ &    ($ 48$) &    $0.1$ &    ($ 4$)\tabularnewline
\normalfont   L1-SC1 &    $0.8$ &    $ 1.0$ &    ($ 120$) &    $ 4.4$ &    ($537$) &    $124.0$ &    ($15255$) &    $ 2.6$ &    ($316$) &    $0.1$ &    ($ 9$)\tabularnewline
\normalfont   L1-SC2 &    $2.4$ &    $10.8$ &    ($ 447$) &    $19.1$ &    ($794$) &    $619.5$ &    ($25727$) &    $17.1$ &    ($711$) &    $0.1$ &    ($ 3$)\tabularnewline
\normalfont   L1-SC3 &    $2.0$ &    $ 3.9$ &    ($ 196$) &    $ 1.3$ &    ($ 64$) &    $357.2$ &    ($18064$) &    $ 1.2$ &    ($ 61$) &    $0.1$ &    ($ 5$)\tabularnewline
\normalfont   L1-SC4 &    $1.7$ &    $ 3.7$ &    ($ 224$) &    $ 1.4$ &    ($ 82$) &    $346.5$ &    ($20771$) &    $ 1.0$ &    ($ 57$) &    $0.1$ &    ($ 4$)\tabularnewline
\hline
\end{tabular}
\begin{tablenotes}
   \item[a] Entries for the posterior regrets have been here truncated
     to the closest first digit after the decimal point, whereas
     entries for the percentage regrets have been truncated to the
     closest integer. 
\end{tablenotes}
\end{threeparttable}
\end{table}

The results of the spatial data simulations for the estimation of
the parameter ensemble's quantiles are presented in table
\ref{tab:spatial_qsel_table}, on page \pageref{tab:spatial_qsel_table}, for the CAR normal (denoted BYM
in the table) and the CAR Laplace (denoted L1 in the table). As for the
non-spatial simulations, we reported the percentage posterior regrets
in parentheses for each plug-in estimator. Note that for these
spatial simulations, the parameters controlling the size of the
$\phi_{i}$'s in the WRSEL function were given the following
specification, $a_{1}=a_{2}=0.5$, which is a symmetric version of
the original specification used by \citet{Wright2003} for similar
spatial models. 

Overall, the GR plug-in estimator of the first and third quartiles
 outperformed the four other plug-in
estimators across all spatial scenarios and levels of the different
experimental parameters. The quartiles derived from the ensemble of CB
point estimates tended to exhibit the second best performance on most
scenarios, followed by the quartiles of the SSEL and WRSEL ensemble
distributions, with the WRSEL displaying the worse performance, as can
be observed in column six of table \ref{tab:spatial_qsel_table}, on page
\pageref{tab:spatial_qsel_table}. The MLE-based
empirical quartiles were poor when we considered SC3 and SC4 as
spatial scenarios. However, the MLE plug-in estimates outperformed
the SSEL, WRSEL and CB ensemble quartiles under scenarios SC1 and
SC2, for simulations with medium to high heterogeneity of the true RRs.

The most important experimental factor was the use of different
scenarios, as can be observed in table \ref{tab:spatial_qsel_table} 
by comparing the SC1/SC2 lines with the SC3/SC4 lines of each
section of the table. In particular, SC1 and SC2 tended to produce
very different outcomes from the ones obtained for the SC3 and SC4
synthetic data sets. One should note that the SC1 and SC2 spatially
structured scenarios are somewhat artificial
because the true distributions of the level of risks in these
scenarios is discrete: each RR can only take
one of two values. This produced some particularly
counter-intuitive results when comparing the performances of different
plug-in estimators. 
\begin{figure}[t]
 \centering
  \includegraphics[width=14cm]{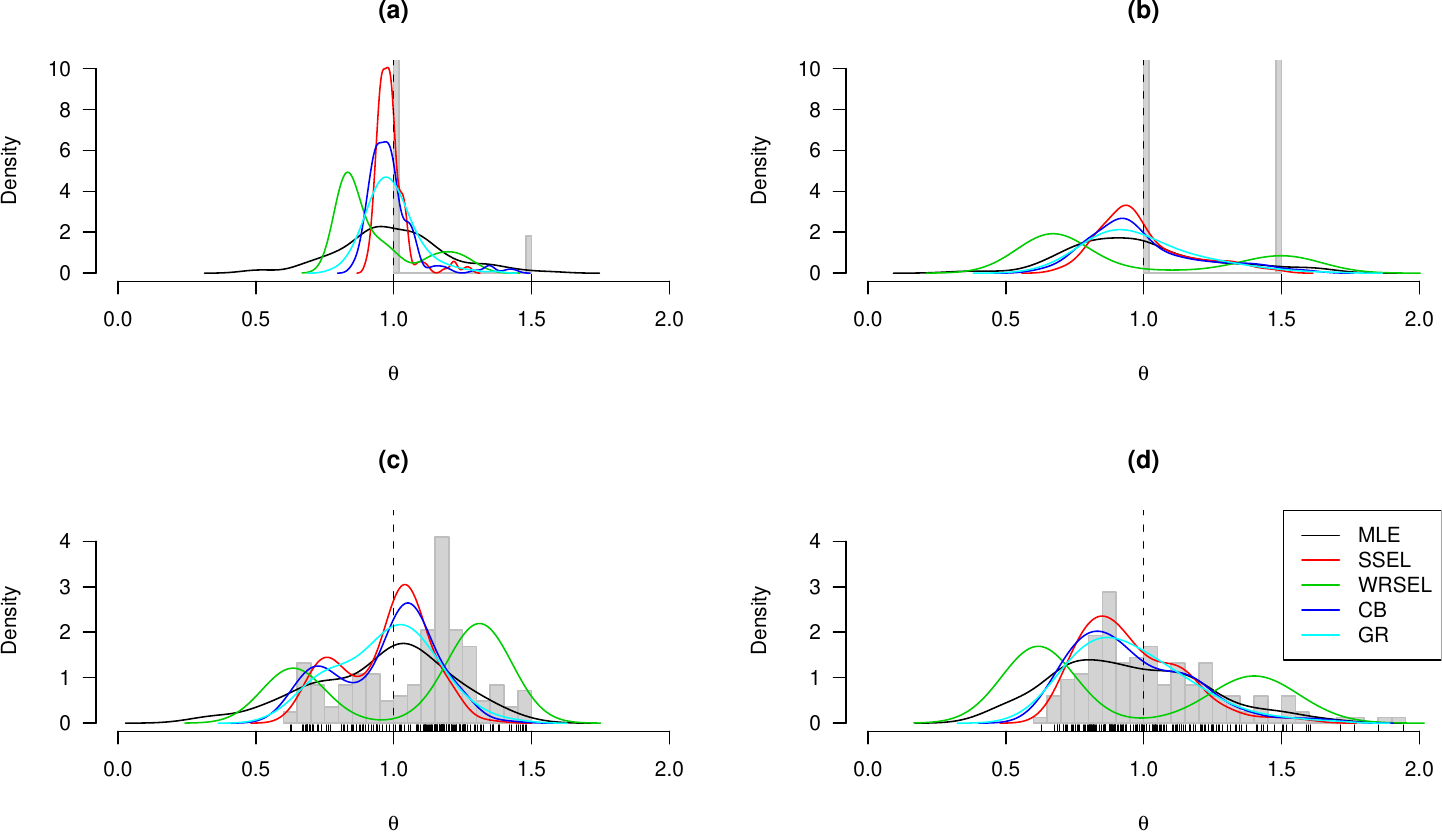}\\
  \caption{Histograms of simulated parameter ensembles,
  under scenarios SC1, SC2, SC3 and SC4 in panels (a),
  (b), (c) and (d), respectively for a low variability specification. 
  Distributions of point estimate ensembles have been superimposed using
  different coloured lines, and are based on posterior distributions
  produced under the CAR Normal model.
  For SC1 and SC2, in panels (a) and (b), the true
  parameter ensemble of the RRs is a discrete distribution on solely
  two values. 
  \label{fig:hist spatial LR1}} 
\end{figure}

In figure \ref{fig:hist spatial LR1} 
on page \pageref{fig:hist spatial LR1}, the typical
ensemble distributions for different choices of point estimates are
reported under the four spatial scenarios. 
Despite the discrete nature of the true distributions in scenarios SC1 and SC2, 
all ensembles of point estimators behaved similarly in terms of level
of shrinkage. That is, the WRSEL and SEL functions exerted a
moderate to high level of shrinkage towards the prior mean, whereas
the CB and GR point estimators resulted in an ensemble distribution
closer to the true ensemble of the RRs. The MLEs, by contrast, tended to be
over-dispersed. In effect, the ordering of the different families of
point estimators in terms of level of shrinkage was
maintained. Therefore, the estimators retained
their typical properties. However, it is these properties that have
changed in their level of desirability. While the over-dispersion of
the MLEs is generally a disadvantage in terms of the estimation of the
quantiles of the true ensemble distribution; for SC1 and SC2, it has
now become one of the desirable properties of a point estimator
because the true RRs  are artificially over-dispersed in these two
scenarios. These observations account for the superiority of the
quartiles based on the ensemble distribution of the MLEs in comparison to 
the SSEL, WRSEL and CB plug-in estimators for these scenarios. 

The level of variability of the true ensemble of RRs was also found to
have a substantial effect on the ordering of the different plug-in
estimators in terms of percentage regret, although this effect was
mediated by the choice of spatial scenarios. 
For the discrete two-category spatial scenarios --SC1 and SC2-- the quartiles of
the ensemble distributions of the SSEL, WRSEL and CB estimators were
detrimentally affected by an increase in the heterogeneity of the true
parameter ensemble. The MLE-based empirical quartiles, by contrast,
appeared to improve their performance as the true RRs became more dispersed,
although this trend was mainly restricted to the SC1 scenario.
For SC3 and SC4, however, the effect of increasing the variability
of the RRs tended to result in a decrease in percentage regret for most
plug-in estimators, although systematic trends were difficult to
isolate. Under SC4, the MLE, SSEL and WRSEL plug-in estimators were
positively affected by an increase in the heterogeneity of the RR
ensemble. 

The choice of the CAR prior --that is, the contrast between the BYM model with its robust
version based on the Laplace distribution-- mainly affected the 
SSEL, WRSEL and CB plug-in estimators, whose performance was worsened
when using the CAR Laplace prior, under the SC3 and SC4 simulation
scenarios. We note, however, that this trend was not completely
verified for the CB quartile estimator for which the use of the BYM model
produced worse percentage regret under low variability. The GR plug-in estimators, by
contrast, appeared to be only marginally affected by the choice of
prior. The use of the Laplace prior slightly increased
the posterior expected loss (cf. first column of table
\ref{tab:spatial_qsel_table}, on page \pageref{tab:spatial_qsel_table}) 
associated with the optimal estimator under the
Q-SEL function, thereby indicating that the use of such a model leads
to estimators of empirical quartiles, which are inferior to the ones obtained when specifying 
a CAR Normal prior. 

\subsection{Plug-in Estimators under QR-SEL}
\begin{table}[t]
 \footnotesize
 \caption{
Posterior regrets based on $\op{QR-SEL}(\bth,\op{QR}(\bth^{L\pr}))$ for six plug-in estimators,
and with the posterior expected loss of the optimal estimator in the first column.
Results are presented for three different levels of variability, with $\op{SF}=1.0$ 
and for four spatial scenarios: an isolated cluster (SC1), a set of isolated clusters and isolated areas (SC2),
highly structured spatial heterogeneity (SC3), and a risk surface
generated by a hidden covariate (SC4). Entries, averaged over 100
replicate data sets in each condition, are scaled by a factor of $10^3$
with posterior regrets expressed as percentage of the posterior loss under the optimal estimator in parentheses.
\label{tab:spatial_qrsel_table}} 
 \centering
 \begin{threeparttable}
 \begin{tabular}{>{\RaggedRight}p{50pt}>{\RaggedLeft}p{25pt}|>{\RaggedLeft}p{25pt}@{}>{\RaggedLeft}p{27pt}>{\RaggedLeft}p{20pt}@{}>{\RaggedLeft}p{27pt}>{\RaggedLeft}p{25pt}@{}>{\RaggedLeft}p{30pt}>{\RaggedLeft}p{15pt}@{}>{\RaggedLeft}p{20pt}>{\RaggedLeft}p{10pt}@{}>{\RaggedLeft}p{15pt}>{\RaggedLeft}p{10pt}@{}>{\RaggedLeft}p{12pt}}\hline
\multicolumn{1}{c}{\itshape Scenarios}&
\multicolumn{1}{c}{}&
\multicolumn{12}{c}{\itshape Posterior regrets\tnote{a}}
\tabularnewline \cline{3-14}
\multicolumn{1}{>{\RaggedRight}p{50pt}}{}&\multicolumn{1}{c}{QR-SEL}&
\multicolumn{2}{c}{MLE}&\multicolumn{2}{c}{SSEL}&\multicolumn{2}{c}{WRSEL}&\multicolumn{2}{c}{CB}&\multicolumn{2}{c}{GR}&\multicolumn{2}{c}{RoPQ}
\tabularnewline
\hline
{\itshape  Low Variab.~}&&&&&&&&&&&&&\tabularnewline
\normalfont   BYM-SC1 &    $1$ &    $18$ &    ($2583$) &    $ 5$ &    ($ 679$) &    $  58$ &    ($ 8187$) &    $ 1$ &    ($ 106$) &    $0$ &    ($1$) &    $0$ &    ($0$)\tabularnewline
\normalfont   BYM-SC2 &    $1$ &    $13$ &    ($1152$) &    $ 9$ &    ($ 849$) &    $ 489$ &    ($44527$) &    $ 2$ &    ($ 195$) &    $0$ &    ($0$) &    $0$ &    ($0$)\tabularnewline
\normalfont   BYM-SC3 &    $3$ &    $14$ &    ($ 521$) &    $ 2$ &    ($  86$) &    $1355$ &    ($50724$) &    $ 2$ &    ($  89$) &    $0$ &    ($1$) &    $0$ &    ($0$)\tabularnewline
\normalfont   BYM-SC4 &    $1$ &    $20$ &    ($1349$) &    $ 5$ &    ($ 371$) &    $ 871$ &    ($58847$) &    $ 1$ &    ($  57$) &    $0$ &    ($0$) &    $0$ &    ($0$)\tabularnewline
\normalfont   L1-SC1 &    $1$ &    $19$ &    ($2295$) &    $ 5$ &    ($ 636$) &    $  61$ &    ($ 7454$) &    $ 1$ &    ($  73$) &    $0$ &    ($2$) &    $0$ &    ($0$)\tabularnewline
\normalfont   L1-SC2 &    $1$ &    $12$ &    ($1016$) &    $11$ &    ($ 924$) &    $ 537$ &    ($46600$) &    $ 2$ &    ($ 167$) &    $0$ &    ($0$) &    $0$ &    ($0$)\tabularnewline
\normalfont   L1-SC3 &    $3$ &    $19$ &    ($ 686$) &    $ 5$ &    ($ 173$) &    $2059$ &    ($72936$) &    $ 2$ &    ($  62$) &    $0$ &    ($1$) &    $0$ &    ($0$)\tabularnewline
\normalfont   L1-SC4 &    $2$ &    $18$ &    ($1150$) &    $ 8$ &    ($ 496$) &    $1021$ &    ($65581$) &    $ 1$ &    ($  51$) &    $0$ &    ($1$) &    $0$ &    ($0$)\tabularnewline
\hline
{\itshape  Med.~ Variab.~}&&&&&&&&&&&&&\tabularnewline
\normalfont   BYM-SC1 &    $1$ &    $ 3$ &    ($ 451$) &    $ 9$ &    ($1327$) &    $  72$ &    ($10596$) &    $ 4$ &    ($ 585$) &    $0$ &    ($1$) &    $0$ &    ($0$)\tabularnewline
\normalfont   BYM-SC2 &    $3$ &    $ 2$ &    ($  86$) &    $30$ &    ($1149$) &    $ 453$ &    ($17301$) &    $19$ &    ($ 711$) &    $0$ &    ($0$) &    $0$ &    ($0$)\tabularnewline
\normalfont   BYM-SC3 &    $6$ &    $16$ &    ($ 284$) &    $ 5$ &    ($  81$) &    $3093$ &    ($54655$) &    $ 3$ &    ($  61$) &    $0$ &    ($3$) &    $0$ &    ($0$)\tabularnewline
\normalfont   BYM-SC4 &    $3$ &    $18$ &    ($ 640$) &    $ 5$ &    ($ 193$) &    $2181$ &    ($77181$) &    $ 1$ &    ($  38$) &    $0$ &    ($1$) &    $0$ &    ($0$)\tabularnewline
\normalfont   L1-SC1 &    $1$ &    $ 3$ &    ($ 344$) &    $10$ &    ($1271$) &    $ 100$ &    ($12706$) &    $ 4$ &    ($ 514$) &    $0$ &    ($1$) &    $0$ &    ($0$)\tabularnewline
\normalfont   L1-SC2 &    $3$ &    $ 3$ &    ($ 103$) &    $32$ &    ($1128$) &    $ 760$ &    ($26507$) &    $19$ &    ($ 650$) &    $0$ &    ($0$) &    $0$ &    ($0$)\tabularnewline
\normalfont   L1-SC3 &    $6$ &    $21$ &    ($ 341$) &    $ 8$ &    ($ 132$) &    $5694$ &    ($93879$) &    $ 4$ &    ($  65$) &    $0$ &    ($2$) &    $0$ &    ($0$)\tabularnewline
\normalfont   L1-SC4 &    $3$ &    $17$ &    ($ 574$) &    $ 8$ &    ($ 252$) &    $2569$ &    ($86403$) &    $ 1$ &    ($  47$) &    $0$ &    ($1$) &    $0$ &    ($0$)\tabularnewline
\hline
{\itshape  High Variab.~}&&&&&&&&&&&&&\tabularnewline
\normalfont   BYM-SC1 &    $1$ &    $ 1$ &    ($  99$) &    $13$ &    ($1409$) &    $ 149$ &    ($16321$) &    $10$ &    ($1043$) &    $0$ &    ($2$) &    $0$ &    ($0$)\tabularnewline
\normalfont   BYM-SC2 &    $6$ &    $16$ &    ($ 288$) &    $63$ &    ($1128$) &    $  76$ &    ($ 1357$) &    $52$ &    ($ 927$) &    $0$ &    ($1$) &    $0$ &    ($0$)\tabularnewline
\normalfont   BYM-SC3 &    $8$ &    $18$ &    ($ 215$) &    $ 5$ &    ($  65$) &    $4858$ &    ($58570$) &    $10$ &    ($ 122$) &    $0$ &    ($1$) &    $0$ &    ($0$)\tabularnewline
\normalfont   BYM-SC4 &    $5$ &    $19$ &    ($ 398$) &    $ 5$ &    ($ 108$) &    $3930$ &    ($83209$) &    $ 2$ &    ($  34$) &    $0$ &    ($1$) &    $0$ &    ($0$)\tabularnewline
\normalfont   L1-SC1 &    $1$ &    $ 2$ &    ($ 172$) &    $13$ &    ($1211$) &    $ 175$ &    ($16350$) &    $ 9$ &    ($ 873$) &    $0$ &    ($1$) &    $0$ &    ($0$)\tabularnewline
\normalfont   L1-SC2 &    $6$ &    $24$ &    ($ 379$) &    $67$ &    ($1077$) &    $ 103$ &    ($ 1658$) &    $54$ &    ($ 870$) &    $0$ &    ($1$) &    $0$ &    ($0$)\tabularnewline
\normalfont   L1-SC3 &    $9$ &    $26$ &    ($ 280$) &    $ 7$ &    ($  77$) &    $7609$ &    ($82689$) &    $ 6$ &    ($  70$) &    $0$ &    ($1$) &    $0$ &    ($0$)\tabularnewline
\normalfont   L1-SC4 &    $5$ &    $17$ &    ($ 354$) &    $ 8$ &    ($ 155$) &    $4772$ &    ($97158$) &    $ 2$ &    ($  40$) &    $0$ &    ($1$) &    $0$ &    ($0$)\tabularnewline
\hline
\end{tabular}
\begin{tablenotes}
   \item[a] Entries for both posterior and percentage regrets have
     been here truncated to the closest integer.
     For some entries, percentage regrets are smaller than 1 percentage point. 
\end{tablenotes}
\end{threeparttable}
\end{table}

The posterior regrets associated with the QR-SEL function for
different plug-in estimators and different levels of experimental
factors are presented in table \ref{tab:spatial_qrsel_table}, on page
\pageref{tab:spatial_qrsel_table}.
The results for the QR-SEL function were found to follow the pattern
described for the Q-SEL function. As in the preceding section, the
most important experimental factor was the choice of simulation
scenario. Although the over-dispersion of the MLEs was found to be
disadvantageous when estimating the QR under SC1 and SC2 with a
low level of variability in the true RRs, this property became
advantageous as the variability of the RRs increased. The converse was
true for the QRs of the SSEL, WRSEL and CB ensemble
distributions. That is, these three families of plug-in estimators
produced QRs associated with greater percentage regret under SC1 and
SC2, as the level of variability increased. 
As for the Q-SEL, the triple-goal plug-in estimator was found to
outperform all the other ensembles of point estimates across all
scenarios and experimental conditions. For the QR-SEL, we also evaluated the posterior
regret associated with the use of the RoPQ. This particular choice of
estimator yielded a quasi-optimal performance with percentage regrets
lower than one percentage point under all scenarios considered. 

As with quantile estimation, the use of the CAR Laplace prior resulted in higher
posterior expected losses when considering the
optimal estimator of the QR, and also tended to produce substantially
larger percentage posterior regrets for the SSEL, WRSEL and CB plug-in estimators. 
In summary, this spatial simulation study has shown that the different plug-in
estimators of interest tend to behave similarly under both the Q-SEL
and QR-SEL functions. As for the non-spatial study, the triple-goal
estimators were found to exhibit the lowest amount of percentage
regret across all the simulation scenarios considered. 

\subsection{Consequences of Scaling the Expected Counts}\label{sec:mrrr sf}
%
\begin{table}[htbp]
 \footnotesize
 \caption{
Posterior regrets based on $\op{Q-SEL}_{\bp}(\bth,Q_{\hat{\bth}^{L\pr}}(\bp))$ with $\bp:=\lb.25,.75\rb$,
for five plug-in estimators, and with the posterior expected loss of
the optimal estimator reported in the first column.
Results are presented for three different levels of variability,
and for four spatial scenarios: an isolated cluster (SC1), a set of isolated clusters and isolated areas (SC2),
highly structured spatial heterogeneity (SC3), and a risk surface generated by a hidden covariate (SC4).
Entries, averaged over 100 replicate data sets in each condition, are scaled by a factor of $10^3$
with posterior regrets expressed as percentage of the posterior loss
under the optimal estimator in parentheses.
\label{tab:spatial_qsel_table_SF}} 
 \centering
 \begin{threeparttable}
 \begin{tabular}{>{\RaggedRight}p{50pt}>{\RaggedLeft}p{25pt}|>{\RaggedLeft}p{25pt}@{}>{\RaggedLeft}p{27pt}>{\RaggedLeft}p{25pt}@{}>{\RaggedLeft}p{25pt}>{\RaggedLeft}p{25pt}@{}>{\RaggedLeft}p{25pt}>{\RaggedLeft}p{25pt}@{}>{\RaggedLeft}p{20pt}>{\RaggedLeft}p{25pt}@{}>{\RaggedLeft}p{15pt}}\hline
\multicolumn{1}{c}{\itshape Scenarios}&
\multicolumn{1}{c}{}&
\multicolumn{10}{c}{\itshape Posterior regrets\tnote{a}}
\tabularnewline \cline{3-12}
\multicolumn{1}{>{\RaggedRight}p{50pt}}{}&\multicolumn{1}{c}{Q-SEL}&
\multicolumn{2}{c}{MLE}&\multicolumn{2}{c}{SSEL}&\multicolumn{2}{c}{WRSEL}&\multicolumn{2}{c}{CB}&\multicolumn{2}{c}{GR}\tabularnewline
\hline
\multicolumn{5}{l}{\itshape $\op{SF}=0.1$ Low Variab.~}&\tabularnewline
\normalfont   BYM-SC1 &    $ 4.2$ &    $216.6$ &    ($5163$) &    $ 2.8$ &    ($ 66$) &    $ 1.9$ &    ($ 46$) &    $ 0.3$ &    ($  7$) &    $0.2$ &    ($ 4$)\tabularnewline
\normalfont   BYM-SC2 &    $ 4.1$ &    $231.5$ &    ($5614$) &    $ 3.2$ &    ($ 78$) &    $ 2.5$ &    ($ 60$) &    $ 0.3$ &    ($  8$) &    $0.1$ &    ($ 1$)\tabularnewline
\normalfont   BYM-SC3 &    $ 6.3$ &    $138.8$ &    ($2197$) &    $ 3.0$ &    ($ 47$) &    $ 3.0$ &    ($ 47$) &    $ 1.7$ &    ($ 27$) &    $0.1$ &    ($ 1$)\tabularnewline
\normalfont   BYM-SC4 &    $ 5.5$ &    $211.6$ &    ($3827$) &    $ 6.0$ &    ($109$) &    $ 4.9$ &    ($ 88$) &    $ 1.2$ &    ($ 21$) &    $0.1$ &    ($ 2$)\tabularnewline
\normalfont   L1-SC1 &    $ 4.3$ &    $222.3$ &    ($5136$) &    $ 2.7$ &    ($ 61$) &    $ 2.0$ &    ($ 45$) &    $ 0.3$ &    ($  6$) &    $0.2$ &    ($ 4$)\tabularnewline
\normalfont   L1-SC2 &    $ 4.1$ &    $243.9$ &    ($5912$) &    $ 2.9$ &    ($ 70$) &    $ 2.3$ &    ($ 55$) &    $ 0.3$ &    ($  7$) &    $0.2$ &    ($ 5$)\tabularnewline
\normalfont   L1-SC3 &    $ 8.6$ &    $200.1$ &    ($2335$) &    $ 7.6$ &    ($ 89$) &    $ 6.7$ &    ($ 78$) &    $ 0.9$ &    ($ 11$) &    $0.3$ &    ($ 4$)\tabularnewline
\normalfont   L1-SC4 &    $ 6.0$ &    $235.7$ &    ($3900$) &    $ 7.3$ &    ($121$) &    $ 6.1$ &    ($101$) &    $ 0.5$ &    ($  8$) &    $0.2$ &    ($ 3$)\tabularnewline
\hline
\multicolumn{5}{l}{\itshape $\op{SF}=0.1$ Med.~ Variab.~}&\tabularnewline
\normalfont   BYM-SC1 &    $ 5.7$ &    $189.2$ &    ($3332$) &    $ 6.4$ &    ($114$) &    $ 4.7$ &    ($ 82$) &    $ 1.7$ &    ($ 30$) &    $0.1$ &    ($ 1$)\tabularnewline
\normalfont   BYM-SC2 &    $ 8.6$ &    $166.6$ &    ($1941$) &    $13.9$ &    ($162$) &    $12.3$ &    ($144$) &    $ 1.5$ &    ($ 18$) &    $0.3$ &    ($ 3$)\tabularnewline
\normalfont   BYM-SC3 &    $ 7.4$ &    $ 92.9$ &    ($1261$) &    $ 4.5$ &    ($ 61$) &    $ 3.5$ &    ($ 48$) &    $ 3.2$ &    ($ 43$) &    $0.1$ &    ($ 1$)\tabularnewline
\normalfont   BYM-SC4 &    $ 7.4$ &    $155.8$ &    ($2115$) &    $14.0$ &    ($190$) &    $10.1$ &    ($137$) &    $ 2.8$ &    ($ 38$) &    $0.1$ &    ($ 2$)\tabularnewline
\normalfont   L1-SC1 &    $ 4.8$ &    $231.2$ &    ($4810$) &    $ 4.9$ &    ($102$) &    $ 3.9$ &    ($ 80$) &    $ 0.5$ &    ($  9$) &    $0.2$ &    ($ 4$)\tabularnewline
\normalfont   L1-SC2 &    $ 6.7$ &    $171.0$ &    ($2559$) &    $18.7$ &    ($280$) &    $15.3$ &    ($229$) &    $ 1.1$ &    ($ 17$) &    $0.1$ &    ($ 2$)\tabularnewline
\normalfont   L1-SC3 &    $ 7.3$ &    $154.7$ &    ($2115$) &    $23.0$ &    ($314$) &    $20.1$ &    ($275$) &    $ 3.5$ &    ($ 48$) &    $0.1$ &    ($ 2$)\tabularnewline
\normalfont   L1-SC4 &    $ 8.0$ &    $175.4$ &    ($2184$) &    $17.1$ &    ($213$) &    $14.2$ &    ($177$) &    $ 1.5$ &    ($ 19$) &    $0.3$ &    ($ 4$)\tabularnewline
\hline
\multicolumn{5}{l}{\itshape $\op{SF}=0.1$ High Variab.~}&\tabularnewline
\normalfont   BYM-SC1 &    $ 6.1$ &    $103.9$ &    ($1706$) &    $17.7$ &    ($291$) &    $14.5$ &    ($238$) &    $ 6.4$ &    ($104$) &    $0.1$ &    ($ 1$)\tabularnewline
\normalfont   BYM-SC2 &    $ 8.3$ &    $ 78.2$ &    ($ 939$) &    $26.0$ &    ($312$) &    $24.0$ &    ($288$) &    $ 7.9$ &    ($ 95$) &    $0.1$ &    ($ 1$)\tabularnewline
\normalfont   BYM-SC3 &    $ 7.5$ &    $ 73.6$ &    ($ 986$) &    $ 6.6$ &    ($ 89$) &    $ 7.3$ &    ($ 98$) &    $ 6.7$ &    ($ 90$) &    $0.1$ &    ($ 1$)\tabularnewline
\normalfont   BYM-SC4 &    $ 7.6$ &    $103.9$ &    ($1364$) &    $11.8$ &    ($155$) &    $ 9.2$ &    ($121$) &    $ 3.3$ &    ($ 43$) &    $0.1$ &    ($ 1$)\tabularnewline
\normalfont   L1-SC1 &    $ 6.7$ &    $ 89.3$ &    ($1329$) &    $27.7$ &    ($413$) &    $23.2$ &    ($345$) &    $ 6.0$ &    ($ 89$) &    $0.1$ &    ($ 1$)\tabularnewline
\normalfont   L1-SC2 &    $ 8.4$ &    $ 75.8$ &    ($ 898$) &    $38.9$ &    ($462$) &    $39.9$ &    ($473$) &    $10.2$ &    ($121$) &    $0.1$ &    ($ 1$)\tabularnewline
\normalfont   L1-SC3 &    $10.1$ &    $109.5$ &    ($1079$) &    $17.7$ &    ($175$) &    $18.1$ &    ($178$) &    $ 5.8$ &    ($ 58$) &    $0.3$ &    ($ 3$)\tabularnewline
\normalfont   L1-SC4 &    $ 7.3$ &    $116.2$ &    ($1589$) &    $32.0$ &    ($438$) &    $28.0$ &    ($383$) &    $ 5.0$ &    ($ 68$) &    $0.1$ &    ($ 1$)\tabularnewline
\hline
\multicolumn{5}{l}{\itshape $\op{SF}=2$ Low Variab.~}&\tabularnewline
\normalfont   BYM-SC1 &    $ 0.3$ &    $  1.6$ &    ($ 540$) &    $ 1.3$ &    ($429$) &    $ 1.0$ &    ($344$) &    $ 0.5$ &    ($159$) &    $0.1$ &    ($25$)\tabularnewline
\normalfont   BYM-SC2 &    $ 0.5$ &    $  1.3$ &    ($ 260$) &    $ 2.7$ &    ($516$) &    $ 1.6$ &    ($302$) &    $ 1.4$ &    ($279$) &    $0.1$ &    ($13$)\tabularnewline
\normalfont   BYM-SC3 &    $ 0.7$ &    $  1.3$ &    ($ 191$) &    $ 0.3$ &    ($ 44$) &    $ 0.2$ &    ($ 26$) &    $ 0.3$ &    ($ 44$) &    $0.1$ &    ($ 9$)\tabularnewline
\normalfont   BYM-SC4 &    $ 0.6$ &    $  1.5$ &    ($ 249$) &    $ 0.7$ &    ($112$) &    $ 0.5$ &    ($ 89$) &    $ 0.2$ &    ($ 39$) &    $0.0$ &    ($ 8$)\tabularnewline
\normalfont   L1-SC1 &    $ 0.3$ &    $  1.6$ &    ($ 461$) &    $ 1.4$ &    ($419$) &    $ 1.0$ &    ($307$) &    $ 0.5$ &    ($143$) &    $0.1$ &    ($16$)\tabularnewline
\normalfont   L1-SC2 &    $ 0.5$ &    $  1.3$ &    ($ 231$) &    $ 2.6$ &    ($467$) &    $ 2.0$ &    ($373$) &    $ 1.1$ &    ($195$) &    $0.1$ &    ($ 9$)\tabularnewline
\normalfont   L1-SC3 &    $ 0.7$ &    $  1.7$ &    ($ 250$) &    $ 0.5$ &    ($ 66$) &    $ 0.4$ &    ($ 54$) &    $ 0.4$ &    ($ 56$) &    $0.1$ &    ($ 9$)\tabularnewline
\normalfont   L1-SC4 &    $ 0.6$ &    $  1.5$ &    ($ 237$) &    $ 0.7$ &    ($108$) &    $ 1.0$ &    ($152$) &    $ 0.1$ &    ($ 13$) &    $0.1$ &    ($15$)\tabularnewline
\hline
\multicolumn{5}{l}{\itshape $\op{SF}=2$ Med.~ Variab.~}&\tabularnewline
\normalfont   BYM-SC1 &    $ 0.3$ &    $  0.2$ &    ($  74$) &    $ 1.8$ &    ($552$) &    $ 1.5$ &    ($464$) &    $ 1.1$ &    ($327$) &    $0.1$ &    ($24$)\tabularnewline
\normalfont   BYM-SC2 &    $ 1.2$ &    $  3.1$ &    ($ 270$) &    $ 8.8$ &    ($755$) &    $ 6.3$ &    ($541$) &    $ 7.7$ &    ($664$) &    $0.1$ &    ($ 5$)\tabularnewline
\normalfont   BYM-SC3 &    $ 1.1$ &    $  1.5$ &    ($ 129$) &    $ 0.5$ &    ($ 41$) &    $ 0.3$ &    ($ 25$) &    $ 0.5$ &    ($ 48$) &    $0.1$ &    ($ 5$)\tabularnewline
\normalfont   BYM-SC4 &    $ 0.8$ &    $  1.1$ &    ($ 142$) &    $ 0.3$ &    ($ 38$) &    $ 0.2$ &    ($ 25$) &    $ 0.3$ &    ($ 32$) &    $0.1$ &    ($ 7$)\tabularnewline
\normalfont   L1-SC1 &    $ 0.4$ &    $  0.7$ &    ($ 197$) &    $ 1.9$ &    ($518$) &    $ 1.6$ &    ($445$) &    $ 1.0$ &    ($274$) &    $0.1$ &    ($26$)\tabularnewline
\normalfont   L1-SC2 &    $ 1.2$ &    $  4.3$ &    ($ 346$) &    $ 7.7$ &    ($623$) &    $ 7.1$ &    ($576$) &    $ 6.4$ &    ($521$) &    $0.0$ &    ($ 4$)\tabularnewline
\normalfont   L1-SC3 &    $ 1.2$ &    $  2.1$ &    ($ 175$) &    $ 0.7$ &    ($ 62$) &    $ 0.6$ &    ($ 49$) &    $ 0.8$ &    ($ 69$) &    $0.1$ &    ($ 5$)\tabularnewline
\normalfont   L1-SC4 &    $ 0.9$ &    $  1.2$ &    ($ 136$) &    $ 0.5$ &    ($ 59$) &    $ 0.5$ &    ($ 61$) &    $ 0.3$ &    ($ 32$) &    $0.1$ &    ($ 6$)\tabularnewline
\hline
\multicolumn{5}{l}{\itshape $\op{SF}=2$ High Variab.~}&\tabularnewline
\normalfont   BYM-SC1 &    $ 0.4$ &    $  0.3$ &    ($  73$) &    $ 2.3$ &    ($618$) &    $ 1.9$ &    ($489$) &    $ 1.7$ &    ($449$) &    $0.1$ &    ($19$)\tabularnewline
\normalfont   BYM-SC2 &    $ 1.7$ &    $  8.7$ &    ($ 505$) &    $12.4$ &    ($717$) &    $10.7$ &    ($621$) &    $12.1$ &    ($701$) &    $0.1$ &    ($ 4$)\tabularnewline
\normalfont   BYM-SC3 &    $ 1.3$ &    $  0.8$ &    ($  61$) &    $ 0.5$ &    ($ 40$) &    $ 0.6$ &    ($ 45$) &    $ 0.6$ &    ($ 46$) &    $0.1$ &    ($ 9$)\tabularnewline
\normalfont   BYM-SC4 &    $ 1.1$ &    $  1.6$ &    ($ 150$) &    $ 0.5$ &    ($ 43$) &    $ 0.3$ &    ($ 32$) &    $ 0.4$ &    ($ 33$) &    $0.1$ &    ($ 6$)\tabularnewline
\normalfont   L1-SC1 &    $ 0.7$ &    $  0.7$ &    ($ 100$) &    $ 2.2$ &    ($333$) &    $ 1.5$ &    ($232$) &    $ 1.6$ &    ($241$) &    $0.1$ &    ($12$)\tabularnewline
\normalfont   L1-SC2 &    $ 1.9$ &    $  9.2$ &    ($ 485$) &    $12.8$ &    ($676$) &    $11.5$ &    ($610$) &    $12.5$ &    ($660$) &    $0.1$ &    ($ 5$)\tabularnewline
\normalfont   L1-SC3 &    $ 1.5$ &    $  0.8$ &    ($  56$) &    $ 0.8$ &    ($ 56$) &    $ 0.6$ &    ($ 40$) &    $ 0.7$ &    ($ 47$) &    $0.2$ &    ($12$)\tabularnewline
\normalfont   L1-SC4 &    $ 1.1$ &    $  1.7$ &    ($ 149$) &    $ 0.3$ &    ($ 31$) &    $ 0.5$ &    ($ 44$) &    $ 0.4$ &    ($ 32$) &    $0.1$ &    ($ 5$)\tabularnewline
\hline
\end{tabular}
\begin{tablenotes}
   \item[a] Entries for the posterior regrets have been truncated to
     the closest first
     digit after the decimal point, and entries for the percentage
     regrets have been truncated to the closest integer. 
\end{tablenotes}
\end{threeparttable}
\end{table}

%
\begin{table}[htbp]
 \footnotesize
 \caption{
Posterior regrets based on $\op{QR-SEL}(\bth,\op{QR}(\bth^{L\pr}))$ for six plug-in estimators,
and with the posterior expected loss of the optimal estimator in the first column.
Results are presented for three different levels of variability,
and for four spatial scenarios: an isolated cluster (SC1), a set of isolated clusters and isolated areas (SC2),
highly structured spatial heterogeneity (SC3), and a risk surface generated by a hidden covariate (SC4).
Entries, averaged over 100 replicate data sets in each condition, are scaled by a factor of $10^3$
with posterior regrets expressed as percentage of the posterior loss under the optimal estimator in parentheses.\label{tab:spatial_qrsel_table_SF}} 
 \centering
 \begin{threeparttable}
 \begin{tabular}{>{\RaggedRight}p{50pt}>{\RaggedLeft}p{25pt}|>{\RaggedLeft}p{25pt}@{}>{\RaggedLeft}p{32pt}>{\RaggedLeft}p{20pt}@{}>{\RaggedLeft}p{27pt}>{\RaggedLeft}p{25pt}@{}>{\RaggedLeft}p{27pt}>{\RaggedLeft}p{15pt}@{}>{\RaggedLeft}p{20pt}>{\RaggedLeft}p{10pt}@{}>{\RaggedLeft}p{15pt}>{\RaggedLeft}p{10pt}@{}>{\RaggedLeft}p{12pt}}\hline
\multicolumn{1}{c}{\itshape Scenarios}&
\multicolumn{1}{c}{}&
\multicolumn{12}{c}{\itshape Posterior regrets\tnote{a}}
\tabularnewline \cline{3-14}
\multicolumn{1}{>{\RaggedRight}p{50pt}}{}&\multicolumn{1}{c}{QR-SEL}&
\multicolumn{2}{c}{MLE}&\multicolumn{2}{c}{SSEL}&\multicolumn{2}{c}{WRSEL}&\multicolumn{2}{c}{CB}&\multicolumn{2}{c}{GR}&\multicolumn{2}{c}{RoPQ}
\tabularnewline
\hline
\multicolumn{5}{l}{\itshape $\op{SF}=0.1$ Low Variab.~}\tabularnewline
\normalfont   BYM-SC1 &    $ 3$ &    $ 945$ &    ($28427$) &    $  7$ &    ($ 202$) &    $  5$ &    ($ 138$) &    $ 0$ &    ($  14$) &    $0$ &    ($ 9$) &    $0$ &    ($0$)\tabularnewline
\normalfont   BYM-SC2 &    $ 3$ &    $ 992$ &    ($30596$) &    $  8$ &    ($ 253$) &    $  6$ &    ($ 196$) &    $ 1$ &    ($  23$) &    $0$ &    ($ 4$) &    $0$ &    ($0$)\tabularnewline
\normalfont   BYM-SC3 &    $11$ &    $ 783$ &    ($ 6819$) &    $ 10$ &    ($  83$) &    $  9$ &    ($  74$) &    $ 4$ &    ($  38$) &    $0$ &    ($ 1$) &    $0$ &    ($0$)\tabularnewline
\normalfont   BYM-SC4 &    $ 6$ &    $1005$ &    ($15697$) &    $ 17$ &    ($ 264$) &    $ 14$ &    ($ 216$) &    $ 1$ &    ($  17$) &    $0$ &    ($ 3$) &    $0$ &    ($0$)\tabularnewline
\normalfont   L1-SC1 &    $ 4$ &    $ 964$ &    ($25898$) &    $  6$ &    ($ 175$) &    $  5$ &    ($ 129$) &    $ 0$ &    ($  11$) &    $0$ &    ($10$) &    $0$ &    ($0$)\tabularnewline
\normalfont   L1-SC2 &    $ 3$ &    $1037$ &    ($29816$) &    $  7$ &    ($ 198$) &    $  5$ &    ($ 156$) &    $ 0$ &    ($  13$) &    $0$ &    ($11$) &    $0$ &    ($0$)\tabularnewline
\normalfont   L1-SC3 &    $17$ &    $1047$ &    ($ 6098$) &    $ 25$ &    ($ 143$) &    $ 20$ &    ($ 119$) &    $ 2$ &    ($  11$) &    $1$ &    ($ 8$) &    $0$ &    ($0$)\tabularnewline
\normalfont   L1-SC4 &    $ 9$ &    $1070$ &    ($12516$) &    $ 21$ &    ($ 243$) &    $ 17$ &    ($ 201$) &    $ 1$ &    ($  10$) &    $1$ &    ($ 6$) &    $0$ &    ($0$)\tabularnewline
\hline
\multicolumn{5}{l}{\itshape $\op{SF}=0.1$ Med.~ Variab.~}\tabularnewline
\normalfont   BYM-SC1 &    $ 7$ &    $ 869$ &    ($12769$) &    $ 17$ &    ($ 254$) &    $ 13$ &    ($ 187$) &    $ 3$ &    ($  42$) &    $0$ &    ($ 1$) &    $0$ &    ($0$)\tabularnewline
\normalfont   BYM-SC2 &    $18$ &    $ 957$ &    ($ 5237$) &    $ 51$ &    ($ 279$) &    $ 44$ &    ($ 243$) &    $ 2$ &    ($  13$) &    $1$ &    ($ 7$) &    $0$ &    ($0$)\tabularnewline
\normalfont   BYM-SC3 &    $18$ &    $ 584$ &    ($ 3219$) &    $ 18$ &    ($ 102$) &    $ 15$ &    ($  84$) &    $10$ &    ($  57$) &    $0$ &    ($ 1$) &    $0$ &    ($0$)\tabularnewline
\normalfont   BYM-SC4 &    $13$ &    $ 855$ &    ($ 6678$) &    $ 46$ &    ($ 361$) &    $ 34$ &    ($ 265$) &    $ 5$ &    ($  42$) &    $0$ &    ($ 2$) &    $0$ &    ($0$)\tabularnewline
\normalfont   L1-SC1 &    $ 5$ &    $ 988$ &    ($18505$) &    $ 13$ &    ($ 251$) &    $ 11$ &    ($ 197$) &    $ 1$ &    ($  19$) &    $0$ &    ($ 8$) &    $0$ &    ($0$)\tabularnewline
\normalfont   L1-SC2 &    $13$ &    $ 957$ &    ($ 7170$) &    $ 70$ &    ($ 524$) &    $ 56$ &    ($ 416$) &    $ 3$ &    ($  21$) &    $0$ &    ($ 3$) &    $0$ &    ($0$)\tabularnewline
\normalfont   L1-SC3 &    $16$ &    $ 832$ &    ($ 5207$) &    $ 97$ &    ($ 605$) &    $ 79$ &    ($ 493$) &    $13$ &    ($  83$) &    $0$ &    ($ 3$) &    $0$ &    ($0$)\tabularnewline
\normalfont   L1-SC4 &    $16$ &    $ 894$ &    ($ 5756$) &    $ 58$ &    ($ 375$) &    $ 46$ &    ($ 299$) &    $ 4$ &    ($  23$) &    $1$ &    ($ 8$) &    $0$ &    ($0$)\tabularnewline
\hline
\multicolumn{5}{l}{\itshape $\op{SF}=0.1$ High Variab.~}\tabularnewline
\normalfont   BYM-SC1 &    $ 9$ &    $ 564$ &    ($ 6583$) &    $ 60$ &    ($ 697$) &    $ 49$ &    ($ 575$) &    $18$ &    ($ 209$) &    $0$ &    ($ 2$) &    $0$ &    ($0$)\tabularnewline
\normalfont   BYM-SC2 &    $25$ &    $ 629$ &    ($ 2545$) &    $144$ &    ($ 581$) &    $141$ &    ($ 571$) &    $20$ &    ($  83$) &    $0$ &    ($ 1$) &    $0$ &    ($0$)\tabularnewline
\normalfont   BYM-SC3 &    $23$ &    $ 528$ &    ($ 2276$) &    $ 26$ &    ($ 111$) &    $ 30$ &    ($ 128$) &    $33$ &    ($ 144$) &    $0$ &    ($ 1$) &    $0$ &    ($0$)\tabularnewline
\normalfont   BYM-SC4 &    $17$ &    $ 736$ &    ($ 4381$) &    $ 52$ &    ($ 312$) &    $ 41$ &    ($ 245$) &    $ 2$ &    ($   9$) &    $0$ &    ($ 1$) &    $0$ &    ($0$)\tabularnewline
\normalfont   L1-SC1 &    $12$ &    $ 455$ &    ($ 3806$) &    $101$ &    ($ 844$) &    $ 82$ &    ($ 689$) &    $17$ &    ($ 144$) &    $0$ &    ($ 2$) &    $0$ &    ($0$)\tabularnewline
\normalfont   L1-SC2 &    $27$ &    $ 603$ &    ($ 2203$) &    $220$ &    ($ 802$) &    $221$ &    ($ 808$) &    $34$ &    ($ 125$) &    $1$ &    ($ 2$) &    $0$ &    ($0$)\tabularnewline
\normalfont   L1-SC3 &    $30$ &    $ 748$ &    ($ 2533$) &    $ 86$ &    ($ 291$) &    $ 84$ &    ($ 285$) &    $18$ &    ($  62$) &    $1$ &    ($ 5$) &    $0$ &    ($1$)\tabularnewline
\normalfont   L1-SC4 &    $16$ &    $ 713$ &    ($ 4374$) &    $144$ &    ($ 882$) &    $123$ &    ($ 752$) &    $15$ &    ($  94$) &    $0$ &    ($ 1$) &    $0$ &    ($0$)\tabularnewline
\hline
\multicolumn{5}{l}{\itshape $\op{SF}=2$ Low Variab.~}\tabularnewline
\normalfont   BYM-SC1 &    $ 0$ &    $   4$ &    ($ 1541$) &    $  3$ &    ($1158$) &    $  2$ &    ($ 936$) &    $ 1$ &    ($ 363$) &    $0$ &    ($52$) &    $0$ &    ($0$)\tabularnewline
\normalfont   BYM-SC2 &    $ 1$ &    $   3$ &    ($  373$) &    $  7$ &    ($ 889$) &    $  4$ &    ($ 566$) &    $ 2$ &    ($ 303$) &    $0$ &    ($27$) &    $0$ &    ($0$)\tabularnewline
\normalfont   BYM-SC3 &    $ 1$ &    $   4$ &    ($  313$) &    $  1$ &    ($  50$) &    $  0$ &    ($  32$) &    $ 1$ &    ($  50$) &    $0$ &    ($ 6$) &    $0$ &    ($0$)\tabularnewline
\normalfont   BYM-SC4 &    $ 1$ &    $   6$ &    ($  519$) &    $  2$ &    ($ 180$) &    $  2$ &    ($ 145$) &    $ 0$ &    ($  20$) &    $0$ &    ($10$) &    $0$ &    ($0$)\tabularnewline
\normalfont   L1-SC1 &    $ 0$ &    $   4$ &    ($ 1135$) &    $  3$ &    ($1017$) &    $  3$ &    ($ 751$) &    $ 1$ &    ($ 296$) &    $0$ &    ($24$) &    $0$ &    ($0$)\tabularnewline
\normalfont   L1-SC2 &    $ 1$ &    $   3$ &    ($  318$) &    $  7$ &    ($ 872$) &    $  6$ &    ($ 722$) &    $ 2$ &    ($ 252$) &    $0$ &    ($14$) &    $0$ &    ($0$)\tabularnewline
\normalfont   L1-SC3 &    $ 1$ &    $   6$ &    ($  432$) &    $  2$ &    ($ 103$) &    $  1$ &    ($  81$) &    $ 1$ &    ($  86$) &    $0$ &    ($11$) &    $0$ &    ($0$)\tabularnewline
\normalfont   L1-SC4 &    $ 1$ &    $   5$ &    ($  446$) &    $  2$ &    ($ 212$) &    $  3$ &    ($ 280$) &    $ 0$ &    ($  12$) &    $0$ &    ($20$) &    $0$ &    ($0$)\tabularnewline
\hline
\multicolumn{5}{l}{\itshape $\op{SF}=2$ Med.~ Variab.~}\tabularnewline
\normalfont   BYM-SC1 &    $ 0$ &    $   1$ &    ($  176$) &    $  5$ &    ($1505$) &    $  4$ &    ($1280$) &    $ 3$ &    ($ 832$) &    $0$ &    ($35$) &    $0$ &    ($0$)\tabularnewline
\normalfont   BYM-SC2 &    $ 2$ &    $   4$ &    ($  190$) &    $ 23$ &    ($1005$) &    $ 19$ &    ($ 848$) &    $16$ &    ($ 717$) &    $0$ &    ($ 6$) &    $0$ &    ($0$)\tabularnewline
\normalfont   BYM-SC3 &    $ 4$ &    $   5$ &    ($  127$) &    $  2$ &    ($  44$) &    $  1$ &    ($  23$) &    $ 2$ &    ($  61$) &    $0$ &    ($ 5$) &    $0$ &    ($0$)\tabularnewline
\normalfont   BYM-SC4 &    $ 2$ &    $   5$ &    ($  244$) &    $  1$ &    ($  49$) &    $  1$ &    ($  35$) &    $ 0$ &    ($  26$) &    $0$ &    ($ 6$) &    $0$ &    ($0$)\tabularnewline
\normalfont   L1-SC1 &    $ 0$ &    $   2$ &    ($  389$) &    $  5$ &    ($1260$) &    $  4$ &    ($1096$) &    $ 2$ &    ($ 613$) &    $0$ &    ($51$) &    $0$ &    ($0$)\tabularnewline
\normalfont   L1-SC2 &    $ 2$ &    $   7$ &    ($  268$) &    $ 21$ &    ($ 847$) &    $ 21$ &    ($ 864$) &    $13$ &    ($ 552$) &    $0$ &    ($ 5$) &    $0$ &    ($0$)\tabularnewline
\normalfont   L1-SC3 &    $ 4$ &    $   7$ &    ($  161$) &    $  2$ &    ($  47$) &    $  2$ &    ($  56$) &    $ 3$ &    ($  66$) &    $0$ &    ($ 6$) &    $0$ &    ($0$)\tabularnewline
\normalfont   L1-SC4 &    $ 2$ &    $   4$ &    ($  205$) &    $  2$ &    ($  95$) &    $  2$ &    ($  95$) &    $ 1$ &    ($  30$) &    $0$ &    ($ 5$) &    $0$ &    ($0$)\tabularnewline
\hline
\multicolumn{5}{l}{\itshape $\op{SF}=2$ High Variab.~}\tabularnewline
\normalfont   BYM-SC1 &    $ 0$ &    $   0$ &    ($  118$) &    $  6$ &    ($1527$) &    $  5$ &    ($1229$) &    $ 4$ &    ($1063$) &    $0$ &    ($41$) &    $0$ &    ($0$)\tabularnewline
\normalfont   BYM-SC2 &    $ 4$ &    $  19$ &    ($  495$) &    $ 34$ &    ($ 900$) &    $ 33$ &    ($ 871$) &    $29$ &    ($ 754$) &    $0$ &    ($ 6$) &    $0$ &    ($0$)\tabularnewline
\normalfont   BYM-SC3 &    $ 6$ &    $   3$ &    ($   56$) &    $  2$ &    ($  32$) &    $  2$ &    ($  37$) &    $ 3$ &    ($  49$) &    $1$ &    ($10$) &    $0$ &    ($0$)\tabularnewline
\normalfont   BYM-SC4 &    $ 3$ &    $  10$ &    ($  284$) &    $  2$ &    ($  45$) &    $  1$ &    ($  34$) &    $ 1$ &    ($  20$) &    $0$ &    ($ 8$) &    $0$ &    ($0$)\tabularnewline
\normalfont   L1-SC1 &    $ 1$ &    $   1$ &    ($  123$) &    $  6$ &    ($ 520$) &    $  4$ &    ($ 368$) &    $ 4$ &    ($ 349$) &    $0$ &    ($10$) &    $0$ &    ($0$)\tabularnewline
\normalfont   L1-SC2 &    $ 4$ &    $  20$ &    ($  468$) &    $ 36$ &    ($ 834$) &    $ 36$ &    ($ 835$) &    $30$ &    ($ 700$) &    $0$ &    ($ 6$) &    $0$ &    ($0$)\tabularnewline
\normalfont   L1-SC3 &    $ 7$ &    $   4$ &    ($   65$) &    $  4$ &    ($  60$) &    $  3$ &    ($  43$) &    $ 3$ &    ($  50$) &    $1$ &    ($10$) &    $0$ &    ($0$)\tabularnewline
\normalfont   L1-SC4 &    $ 4$ &    $   9$ &    ($  251$) &    $  1$ &    ($  37$) &    $  1$ &    ($  32$) &    $ 1$ &    ($  36$) &    $0$ &    ($ 4$) &    $0$ &    ($0$)\tabularnewline
\hline
\end{tabular}
\begin{tablenotes}
   \item[a] Entries for both the posterior and percentage regrets have been truncated to the closest integer.
     For some entries, percentage regrets are smaller than 1 percentage point. 
\end{tablenotes}
\end{threeparttable}
\end{table}

Simulation results for two different scalings of the expected counts are reported
in table \ref{tab:spatial_qsel_table_SF} on page
\pageref{tab:spatial_qsel_table_SF} for the Q-SEL function and 
in table \ref{tab:spatial_qrsel_table_SF} on page
\pageref{tab:spatial_qrsel_table_SF} for the QR-SEL function. As in
the previous sections, percentage regrets with respect to the posterior loss
under the optimal estimator is reported in parentheses in the two tables.

For both the Q-SEL and QR-SEL functions, the use of smaller expected
counts resulted in a substantial increase of the posterior losses
associated with the use of the optimal estimators under all simulated
scenarios. Similarly, the absolute posterior regrets associated with the use of the MLE,
SSEL, WRSEL and CB plug-in estimators tended to be larger when
applying a scaling factor of $0.1$, in comparison to 
the use of a scaling factor of $\op{SF}=2.0$.
In terms of percentage regret, however, this trend 
was reversed for the GR plug-in estimators under
both the Q-SEL and QR-SEL functions. For this family of estimators,
the percentage regrets tended to increase as the SF became larger. This
somewhat counter-intuitive result, however, may be explained in terms
of the comparative decrease of the posterior losses based on the optimal
estimators for both the Q-SEL and QR-SEL functions. That is, although the percentage
regrets associated with the GR plug-in estimators at $\op{SF}=2.0$ was
larger than the ones at $\op{SF}=0.1$, these percentages tended to
correspond to smaller fractions of the corresponding posterior losses.

\begin{figure}[t]
\centering
\includegraphics[width=14cm]{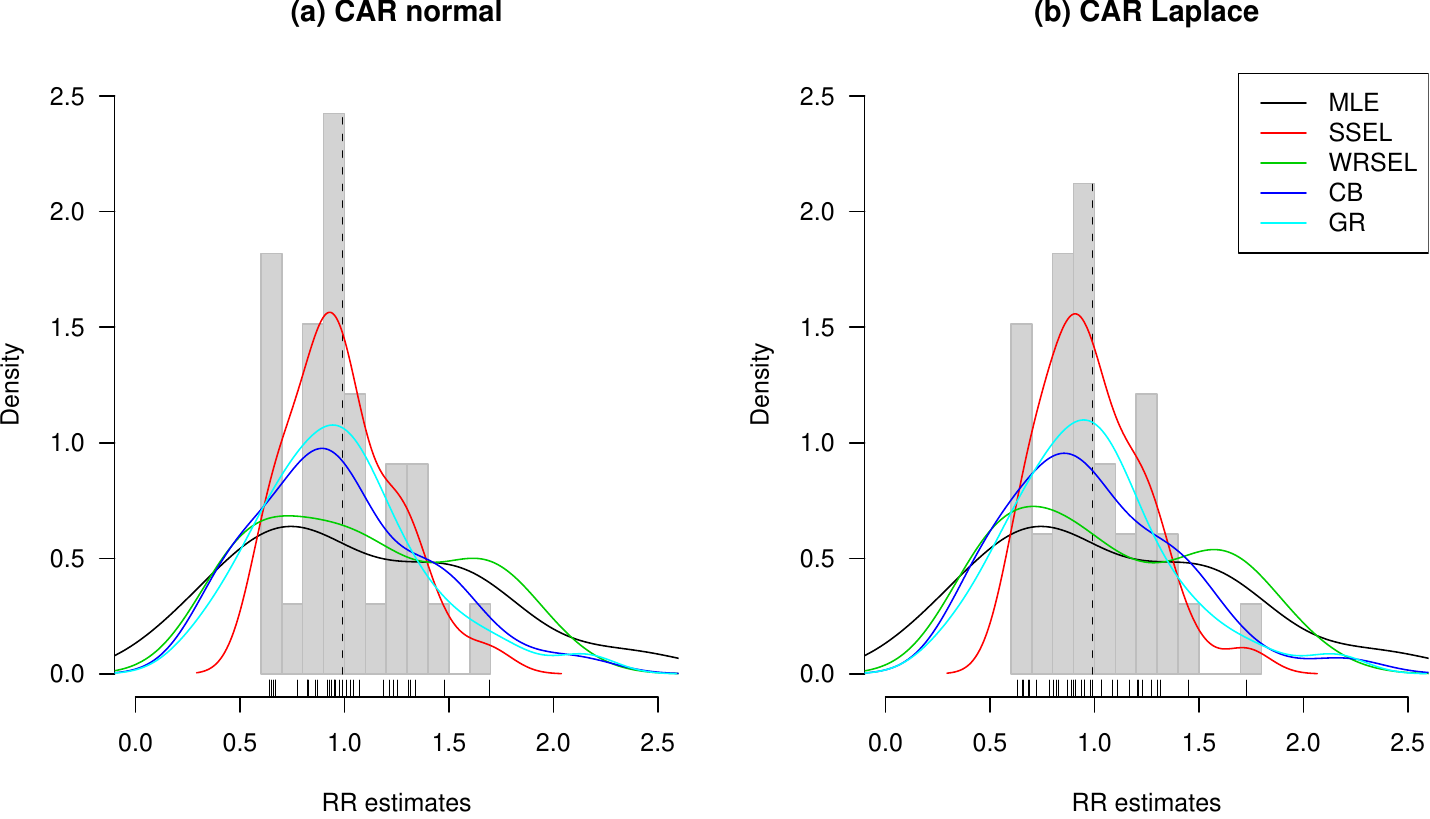} \\
\caption{Histograms of the ensembles of point estimates under SSEL for
  the schizophrenia prevalence data set under the CAR Normal and CAR
  Laplace models in panels (a)
  and (b), respectively. Smoothed empirical distributions of the five
  families of point estimates of interest have been superimposed. 
  \label{fig:schizo histo}}
\end{figure}

\section{Urban Distribution of Schizophrenia Cases}\label{sec:real data}
In this section, we consider the use of different plug-in estimators
for the estimation of the quantiles and QR of a parameter
ensemble for a real data set, which describes the
spatial distribution of cases of schizophrenia in an urban population. This particular data
set was found to be especially relevant to the problem at hand, as it
is characterised by very low expected counts, thereby inducing a high
level of hierarchical shrinkage. 

\subsection{Data Description}
We here re-analysed prevalence data on first-onset cases of schizophrenia collected
from the Southeast London study area of the Aetiology and Ethnicity in
Schizophrenia and Other Psychoses (AESOP) study
\citep{Kirkbride2006}. This is a large,
population-based, epidemiological case-control study of first-episode
psychoses conducted in Southeast London, Nottingham and Bristol. It
was designed to investigate differential rates of psychoses across
different cities and ethnic groups in the UK, based upon a
comprehensive survey of all incident cases of first-episode
psychoses. The data has been shown to be of a very high quality, and
the AESOP study is the largest ever study of first episode psychoses
conducted in the UK to date \citep[see][for a detailed description of
AESOP]{Kirkbride2006}. In the analyses presented here, we solely use
a subset of the AESOP data base, covering prevalence data from the
Southeast London study centre only.

All individuals aged 16-64 years living in the study area having had
contact with mental health services for a first episode
of any probable psychosis, non-psychotic mania or bipolar
disorder were selected. Initial inclusion criteria were broad, based
upon those used in the WHO 10-country study \citep{Jablensky1992}. 
Ascertainment bias was minimised by incorporating a wide variety of
psychiatric services into the study
design. Here, cases were excluded if an address at first presentation
was not obtained. The study took place over 24 months (September 1997-
August 1999). Subjects who passed the screen underwent the Schedules for Clinical
Assessment in Neuropsychiatry (SCAN); a modified Personal and
Psychiatric History Schedule, and a schedule developed to record
sociodemographic data. International Classification of
Diseases (10th Edition) (ICD-10) diagnoses were made by consensus
agreement from a panel of clinicians. 
All cases of broadly defined schizophrenia,
corresponding to ICD-10 F20 to F29 diagnoses were included.
Full details of the methodology used in this study have been provided by
\citet{Kirkbride2006}.
\begin{figure}[t]
\centering
\includegraphics[width=14.5cm]{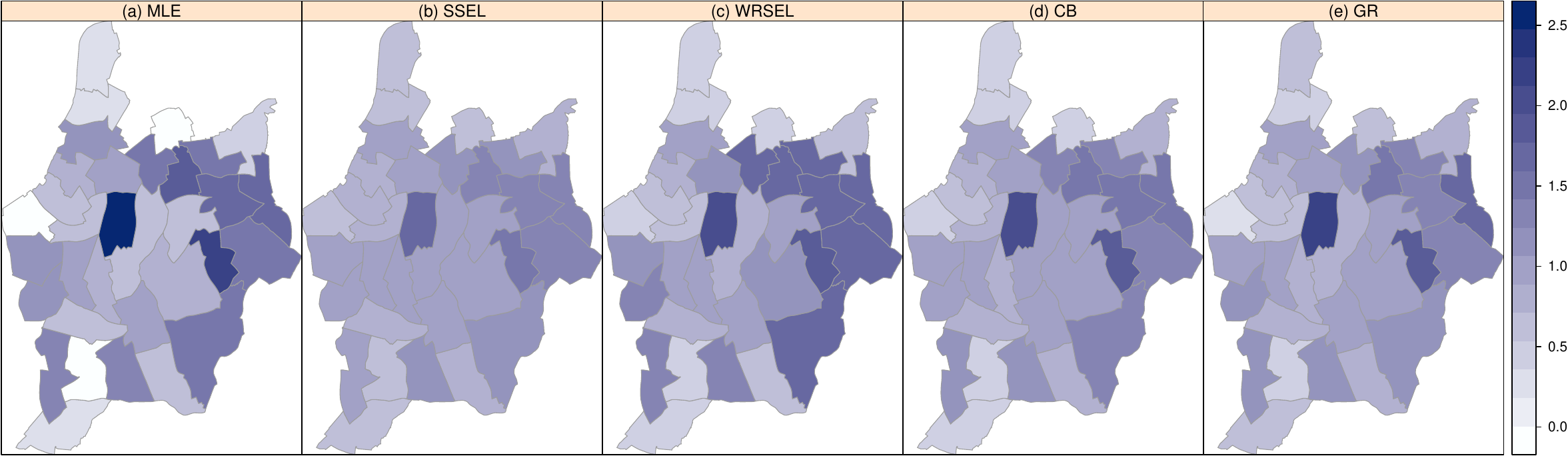}
\caption{Maps of point estimates denoting the RRs for schizophrenia in
  33 boroughs of Southeast London. The ensembles of the MLEs, posterior
  means, WRSEL estimates, CB estimates, and triple-goal estimates are
  plotted in panels (a) to (e), respectively. 
  \label{fig:map normal}}
\end{figure}

\subsection{Models Used and Performance Evaluation}
The expected counts in this data set ranged from 2.68 to 6.86
with a mean of 4.48 and a standard deviation of 0.91. The observed
counts ranged from 0 to 17 with a mean of 
4.48 and a standard deviation of 3.43. For three areas, the observed
counts were null. We fitted the data with the spatial CAR Normal and
Laplace models described in section \ref{sec:mrrr spatial models},
using the same specification for the hyperparameters as in the
spatial simulation study. The BYM model considered here constituted one of the five
different models utilised in the original analysis of this data set by 
\citet{Kirkbride2007}. As before, the models were fitted using
WinBUGS 1.4 \citep{Lunn2000}. Our results are based on 50,000 MCMC samples
collected after discarding 10,000 burn-in realisations from the joint posterior.
The $\phi_{i}$'s used in the computation of the WRSEL plug-in
estimators were here specified using $a_{1}=a_{2}=0.5$, as previously
done for the spatially structured simulations. 

We compared the values taken by the different plug-in
estimators by expressing them as departure from the value of the optimal
estimator. Our strategy, here, is similar in spirit to the one used when
computing the posterior regrets in equation (\ref{eq:regret
  research}), except that we were more specifically interested in
evaluating the direction of the departure from the optimal estimate. 
Thus, for every loss function of interest $L\pr$, we computed
\begin{equation}
     Q_{p}(\hat\bth^{L\pr}) - \E[Q_{p}(\bth)|\by], 
     \label{eq:qsel departure}
\end{equation}
separately for five different empirical quantiles, taking $p\in \lb .05, .25, .50,
.75, .95\rb$. In addition, we derived the departure of the empirical QR plug-in
estimators from the posterior mean empirical QR. Thus, we calculated
\begin{equation}
     \op{QR}(\hat\bth^{L\pr}) - \E[\op{QR}(\bth)|\by],
     \label{eq:qrsel departure}
\end{equation}
for the five ensembles of point estimates under scrutiny. The RoPQ was
also derived for this data set and compared to the posterior mean
empirical QR in a similar fashion. 

\subsection{Results}
\begin{sidewaystable}[htbp]
 \footnotesize
 \caption{Means and credible intervals ($2.5\tth$ and $97.5\tth$
          percentiles of posterior distributions) of the
          posterior empirical quantiles and posterior empirical QR 
         for the schizophrenia prevalence data set. Values of plug-in
          estimators are expressed as departures from the 
          posterior mean empirical quantiles and the posterior mean empirical QR,
          which are reported in the first column, as described in equations (\ref{eq:qsel
            departure}) and (\ref{eq:qrsel departure}). For each
          plug-in estimator, the percentage departures from the value
          of the optimal estimates --i.e. posterior means-- have been
          reported in parentheses.
          \label{tab:schizo}} 
\centering
\begin{threeparttable}
\begin{tabular}{>{\RaggedRight}p{5pt}>{\RaggedRight}p{25pt}|>{\centering}p{80pt}|
>{\RaggedLeft}p{28pt}@{}>{\RaggedLeft}p{20pt}>{\RaggedLeft}p{28pt}@{}>{\RaggedLeft}p{20pt}>{\RaggedLeft}p{28pt}@{}>{\RaggedLeft}p{20pt}
>{\RaggedLeft}p{28pt}@{}>{\RaggedLeft}p{20pt}>{\RaggedLeft}p{28pt}@{}>{\RaggedLeft}p{20pt}>{\RaggedLeft}p{28pt}@{}>{\RaggedLeft}p{20pt}}
\hline
\multicolumn{2}{c}{}& \multicolumn{1}{c}{\ti{Posterior Means}}&
\multicolumn{12}{c}{\itshape Departures from Posterior Mean\tnote{a}}\tabularnewline 
\cmidrule(r){3-3} \cmidrule(l){4-15} 
\multicolumn{2}{>{\RaggedRight}p{30pt}}{\itshape Models} \tabularnewline
 \multicolumn{1}{>{\RaggedRight}p{10pt}}{}
&\multicolumn{1}{>{\RaggedRight}p{15pt}}{}
&\multicolumn{1}{c}{Estimates (CIs)}
&\multicolumn{2}{c}{MLE}
&\multicolumn{2}{c}{SSEL}
&\multicolumn{2}{c}{WRSEL}
&\multicolumn{2}{c}{CB}
&\multicolumn{2}{c}{GR}
&\multicolumn{2}{c}{RoPQ}
\tabularnewline
\hline
\tabularnewline
\multicolumn{4}{l}{CAR Normal}\tabularnewline
& $\theta_{(.05)}$& 0.473 (0.22--0.87) & $-0.47$ & ($100$)&$+0.16$
&($ 35$)& $+0.14$&($ 30$)&$-0.04$&($ 10$)&$-0.07$&($ 17$)& --&
--\tabularnewline
& $\theta_{(.25)}$& 0.714 (0.52--0.96) & $-0.14$ &($20$) & $+0.10$
&($15$) & $+0.10$&($14$) &$+0.01$&($ 1$) &$+0.01$& ($ 1$) &--&
--\tabularnewline
& $\theta_{(.50)}$& 0.932 (0.77--1.10) & $-0.03$ &($4$)  & $+0.02$
&($2$)  & $+0.00$&($0$)  &$-0.00$&($0$)  &$+0.02$& ($2$)
&--&-- \tabularnewline
& $\theta_{(.75)}$& 1.211 (0.99--1.45) & $+0.32$ & ($27$) &$-0.02$ &
($ 1$) &$-0.08$&($ 7$) &$+0.11$&($ 9$) &$-0.02$& ($ 2$) &--&
-- \tabularnewline
& $\theta_{(.95)}$& 1.735 (1.12--2.35) & $+0.51$ & ($30$) & $-0.26$
& ($15$) &$-0.34$& ($20$) & $+0.03$& ($ 2$) & $+0.05$& ($ 3$)
& -- & -- \tabularnewline
& QR             & 1.761 (1.11--2.50) & $+0.92$ &($52$) & $-0.31$ &($18$) &
$-0.38$& ($21$) &$+0.06$& ($ 4$) &$-0.11$& ($ 6$) &$-0.06$ & ($3$)
\tabularnewline
\hline
\tabularnewline
\multicolumn{4}{l}{CAR Laplace}\tabularnewline
& $\theta_{(.05)}$& 0.486 (0.21--0.91)&$-0.48$&($100$)&$+0.16$&($
35$)&  $+0.14$&($ 30$)&  $-0.04$&($  9$)&  $-0.09$& ($ 19$) & -- &
--\tabularnewline
& $\theta_{(.25)}$& 0.719 (0.52--0.97)&$-0.14$&($20$) &$+0.09$& ($13$)
& $+0.08$& ($12$) & $-0.01$&($ 2$) &  $+0.01$& ($ 1$) & -- &
--\tabularnewline
& $\theta_{(.50)}$& 0.932 (0.77--1.10)&$-0.03$&($4$)  &$+0.01$& ($2$)
& $+0.00$& ($0$) & $-0.01$&($1$)  &  $+0.02$& ($2$) &--&--
\tabularnewline
& $\theta_{(.75)}$& 1.204 (0.98--1.45)&$+0.33$&($28$) &$-0.02$&($ 2$)
&  $-0.09$&($ 8$) &  $+0.09$& ($ 8$) & $-0.02$& ($ 2$) & -- &
--\tabularnewline
& $\theta_{(.95)}$& 1.726 (1.09--2.35)&$+0.52$&($30$) &$-0.27$&($16$)
&  $-0.37$&($22$) &  $+0.01$&($ 1$) &  $+0.05$& ($ 3$) & --& --
\tabularnewline
& QR             & 1.744 (1.07--2.49)&$+0.94$&($54$) &$-0.29$&($17$) &
$-0.37$& ($21$) & $+0.11$& ($ 6$) & $-0.11$&($ 7$) & $-0.07$& ($4$)
\tabularnewline
\hline
\end{tabular}
\begin{tablenotes}
   \item[a] Entries for the departures from the optimal estimates
     have been truncated to the closest second digit after the decimal
     point, while entries for the percentage departures have been
     truncated to the closest integer.
\end{tablenotes}
\end{threeparttable}
\end{sidewaystable}

%
\begin{table}[t]
 \footnotesize
 \caption{
Posterior regrets based on $\op{QR-SEL}(\bth,\op{QR}(\bth^{L\pr}))$, 
and $\op{Q-SEL}(\bp,\bth,\bth^{L\pr}_{(\bp)})$ with $\bp:=\lb.25,.75\rb$,
for the schizophrenia prevalence data set. 
The posterior expected loss of the optimal estimator is given in the
first column. In parentheses, posterior regrets are expressed as
percentage of the posterior loss under the optimal estimator.
\label{tab:schizo_regret}}
 \centering
 \begin{threeparttable}
 \begin{tabular}{>{\RaggedRight}p{58pt}>{\RaggedRight}p{17pt}|>{\RaggedLeft}p{22pt}@{}>{\RaggedLeft}p{20pt}>{\RaggedLeft}p{18pt}@{}>{\RaggedLeft}p{18pt}>{\RaggedLeft}p{18pt}@{}>{\RaggedLeft}p{20pt}>{\RaggedLeft}p{18pt}@{}>{\RaggedLeft}p{18pt}>{\RaggedLeft}p{18pt}@{}>{\RaggedLeft}p{8pt}>{\RaggedLeft}p{18pt}@{}>{\RaggedLeft}p{8pt}}\hline
\multicolumn{1}{c}{\itshape Models}&
\multicolumn{1}{c}{}&
\multicolumn{12}{c}{\itshape Posterior regrets\tnote{a}}
\tabularnewline \cline{3-14}
\multicolumn{1}{>{\RaggedRight}p{58pt}}{}&\multicolumn{1}{c}{Post.~ Loss}&
\multicolumn{2}{c}{MLE}&\multicolumn{2}{c}{SSEL}&\multicolumn{2}{c}{WRSEL}&\multicolumn{2}{c}{CB}&\multicolumn{2}{c}{GR}&\multicolumn{2}{c}{RoPQ}
\tabularnewline
\hline
\tabularnewline
{\tb{Q-SEL}}&&&&&&&&&&&&&\tabularnewline
\normalfont CAR Normal &
$0.04$&$0.13$&($342$)&$0.01$&($24$)&$0.15$&($392$)&$0.01$&($25$)&$0.00$&($1$)&--&--\tabularnewline
\normalfont CAR Laplace & 
$0.04$&$0.13$&($341$)&$0.01$&($32$)&$0.16$&($434$)&$0.01$&($32$)&$0.00$&($1$)&--&--\tabularnewline
\hline
\tabularnewline
{\tb{QR-SEL}}&&&&&&&&&&&&&\tabularnewline
\normalfont CAR Normal &
$0.17$&$0.86$&($498$)&$0.10$&($55$)&$0.42$&($245$)&$0.00$&($3$)&$0.01$&($7$)&$0.00$&($3$)\tabularnewline
\normalfont CAR Laplace &
$0.18$&$0.89$&($492$)&$0.09$&($47$)&$0.49$&($273$)&$0.01$&($7$)&$0.01$&($8$)&$0.00$&($3$)\tabularnewline
\hline
\end{tabular}
\begin{tablenotes}
   \item[a] Entries for the posterior regrets have been truncated to
     the closest second digit after the decimal point. Entries for the
     percentage regrets have been truncated to the closest integer.
\end{tablenotes}
\end{threeparttable}
\end{table}

In figure \ref{fig:schizo histo} on page \pageref{fig:schizo histo},
the ensemble distributions of the different families of point
estimates have been plotted. In this figure, the histograms represent
the ensembles of the posterior means under the CAR Normal and CAR
Laplace models. The empirical distributions of
other ensembles of point estimates have been superimposed as densities. 
As expected, the empirical distribution of the MLEs was found to be more dispersed than
the ensemble distributions of the four other sets of point
estimates. In particular, the
ensemble distributions of the SSEL and WRSEL estimators were similar, and were found
to be especially sensitive to hierarchical shrinkage. 
The ensemble distributions of the CB and GR point estimates
struck a balance between the over-dispersion of the MLEs and the
shrinkage of the SSEL and WRSEL ensembles. Overall, the CB and GR
ensemble distributions behaved similarly, albeit the
empirical distribution of the GR point estimates tended to be
smoother. Changes in the modelling assumptions did not
modify the behaviour of the different families of point estimates. As
can be observed from panels (a) and (b) of figure \ref{fig:schizo
  histo}, the ensemble distributions obtained under the 
CAR Normal and CAR Laplace priors did not markedly differ. 
The maps of the ensembles of point estimates evaluated in this section
have also been reported in figure \ref{fig:map normal} on page \pageref{fig:map
  normal}, for the CAR Normal model. 

A summary of the analysis conducted for this data set is provided in
table \ref{tab:schizo} on page \pageref{tab:schizo}, where we report
several summary statistics for the estimation of five different
empirical quantiles and the empirical QR of the parameter ensembles of interest. As expected, the
empirical quantiles of the MLE point estimates
were over-dispersed relative to the optimal estimators under both modelling assumptions:
the MLE plug-in estimates of $\theta_{(.05)}$ and $\theta_{(.25)}$
were found to be smaller than the optimal empirical quantiles, whereas 
the MLE plug-in estimates of $\theta_{(.75)}$ and $\theta_{(.95)}$ were
found to be larger than the optimal empirical quantiles. By contrast, the ensemble of
posterior means followed exactly the opposite pattern: they were
under-dispersed relative to the optimal estimators due to hierarchical shrinkage. 
That is, the SSEL-based estimates of $\theta_{(.05)}$ and $\theta_{(.25)}$
were larger than the corresponding posterior empirical quantiles 
and the SSEL-based estimates of $\theta_{(.75)}$ and
$\theta_{(.95)}$ were smaller than their posterior empirical counterparts. 
The plug-in empirical quantiles obtained under
the WRSEL, CB and GR functions were more difficult to classify, but we
note that the CB plug-in estimators appeared to perform best overall.

Regarding the posterior QR, table \ref{tab:schizo} shows that the
MLE-based plug-in estimator tended to over-estimate the posterior mean 
of the empirical QR, whereas most other plug-in estimators were likely to under-estimate its
value. The RoPQ also tended to be smaller than the posterior mean of
the empirical QR. Overall, the ordering of the different plug-in estimators with
respect to their departure from the optimal QR followed the ordering
reported in our non-spatial and spatial simulation studies. 
The sole exception to this parallel was the CB-based empirical QR, which was
observed to be slightly larger than the posterior mean of the empirical QR. 
These results were found to be robust to a change in modelling
assumptions, albeit the value of the posterior mean of the empirical QR was found to be
slightly smaller when using the CAR Laplace prior. Most importantly, however,
one can observe that the credible intervals for the posterior
empirical quantiles and the empirical QR were found to be larger under the CAR Laplace
model, thereby suggesting that the estimation of these
quantities was noisier for this model than for the CAR Normal. 

In addition, in table \ref{tab:schizo_regret} on page
\pageref{tab:schizo_regret}, we have reported the posterior regrets
for the plug-in estimators of interest under the Q-SEL and QR-SEL. 
For the estimation of the empirical quantiles under Q-SEL, the ordering of the
plug-in estimators in terms of percentage regrets reproduced the one
found in the simulation studies. In particular, the choice of 
the MLEs and WRSEL plug-in estimators yielded the largest
posterior regrets, whereas the triple-goal plug-in empirical quantiles were
found to be quasi-optimal. For this data set, however, the SSEL plug-in empirical
quantiles were found to be almost identical to the ones based on the CB
ensemble. Under the QR-SEL function, the CB plug-in
estimator slightly outperformed the triple-goal plug-in estimator
when estimating the QR. This result should be contrasted with our
synthetic data simulations, which indicated that the triple-goal
plug-in QR was generally better than the CB estimator in terms of posterior
regret. We note, however, that in the case of the schizophrenia
prevalence data set, the small size of the parameter ensemble may have detrimentally
affected the performance of the triple-goal estimator, which
requires large parameter ensembles for a good estimation of the EDF.
Since the data set under scrutiny only had only 33 regions, it
may be speculated that this negatively affected the performance of that
particular plug-in estimator. The behaviour of the remaining QR
plug-in estimators appeared to be in agreement with the conclusions
of our synthetic data simulations. In particular, the RoPQ was found
to outperform all of its counterparts.   

Overall, this re-analysis of Kirkbride et al.'s (2007) schizophrenia prevalence data set
has helped to characterise the heterogeneity of the
empirical distribution of the RRs for schizophrenia in the region of
interest. Optimal estimation under the QR-SEL function has yielded a
QR of approximately 1.7, which implies that the areas in the upper
quartile of risk exhibit $70\%$ higher risk for schizophrenia than the
areas in the lower quartile of risk. This high level of heterogeneity
in the empirical distribution of the RRs in this region suggests the
presence of unmeasured risk factors which may account for the
between-area variability.

\section{Conclusion}\label{sec:mrrr conclusion}
One of the consistent findings of this chapter has been the good performance
of the empirical quantiles and empirical QR based on the ensemble of 
triple-goal point estimates. The GR plug-in estimators behaved
well throughout the set of simulations under both spatial and
non-spatial scenarios. This can be explained in
terms of the optimality criteria satisfied by the triple-goal point
estimates described in section \ref{sec:gr}. The first goal that this
ensemble of point estimates is
optimising is the estimation of the EDF of the parameters of
interest. Since both the empirical quantiles and the empirical QR are
properties of the
EDF, it follows that the ensembles of point estimates optimising the
estimation of the EDF would also be expected to do well when
estimating these empirical properties. These particular findings therefore
highlight another potential benefit of choosing to report the ensemble
of triple-goal point estimates when conducting research in
epidemiology and spatial epidemiology. 
In addition, we have also noted the good performance of the RoPQ and DoPQ under
the QR-SEL and IQR-SEL functions, respectively. These plug-in
estimators were found to consistently outperform all other plug-in
estimators in our simulation studies. It can therefore be concluded
that when the posterior quartiles of a
parameter ensemble have been collected, one can readily obtain a
robust estimation of the dispersion of that ensemble by using either the RoPQ
or the DoPQ, depending on the nature of the models investigated. 

This chapter has also highlighted several limitations associated with
the use of the WRSEL function for the estimation of parameter
ensembles in both spatial and non-spatial models. Through the
manipulation of the size of the parameter ensemble in our non-spatial
simulations, we have first observed that the range of the WRSEL weights
increased exponentially with the size of the parameter
ensemble. Naturally, this is an aspect of the WRSEL function that
would need to be calibrated according to the specificities of the
estimation problem under consideration. Moreover, we
have also noted that the WRSEL plug-in estimators performed poorly under the compound
Gamma model. This is due to the fact that we have specified a vector of symmetric
weights on the ranks of the units in the ensemble of interest. 
Both of these problems can be addressed by 
choosing different values of the parameters $a_{1}$ and $a_{2}$
controlling the range of the WRSEL weights, despite the fact that the
true empirical distribution of the parameters of interest was skewed. 
Asymmetric choices may therefore 
help to deal with skewed parameter ensembles, whereas assigning very
low values to $a_{1}$ and $a_{2}$ would counterbalance the effect of
a large $n$. However, since there does not currently exist an
automated way of estimating there parameters from the properties of
the data, these issues ultimately limit the applicability of the WRSEL
function as an automated estimation procedure, since it necessitates a
substantial amount of preliminary tuning. 

A second important conclusion of this chapter is the apparent
superiority of the CAR Normal over the CAR Laplace model for the
estimation of both the empirical quantiles and the empirical QR of a parameter
ensemble. Our spatial simulation study has shown that the utilisation
of the CAR Laplace model tended to increase the posterior expected
loss associated with the optimal estimators of the ensemble's
quantiles and QR in comparison to the use of the CAR Normal
model. The CAR Laplace also yielded larger posterior regrets when
using plug-in estimators. Finally, our analysis of
the schizophrenia prevalence data set showed that the posterior distributions
of the empirical quantiles and empirical QR were less concentrated
around their means when using the CAR Laplace model. In practice, we
therefore recommend the use of the CAR Normal model when
the dispersion of the RRs is of special interest. However, one should
note that we used identical specifications for the variance
hyperparameters for both the Normal and Laplace CAR priors, which may have
contributed to the comparatively poor performance of the Laplace prior
model. In particular, note that the variance parameter, $\tau^{2}_{u}$, for both the CAR
Normal and CAR Laplace priors was given a Gamma prior, such
that 
\begin{equation}
      \tau^{2}_{u} \sim \op{Gam}(0.5,0.0005).
\end{equation}
However, identical choices of variance parameters in a Laplace and a
Normal distributions do not result in identical amount of
variability. While $X\sim N(0,1)$ has a variance of $1.0$, a Laplace
random variate with distribution $L(0,1)$ has a variance of
$2.0$. Thus, a fair comparison of the CAR Normal and CAR Laplace
performances would require the use of comparative hyperprior
distributions on the variance parameter of the ensemble priors,
$\tau_{u}^{2}$.

In both the spatially-structured simulations and in the analysis of
the schizophrenia prevalence data set, we have computed the Q-SEL on
the scale of the RRs. That is, both the Q-SEL and the QR-SEL were
computed using $\bth$ as the parameter of interest. However, one
should note that, since the true empirical distribution of the RRs tends to
be skewed, the estimation of low and high quantiles will not be
penalised symmetrically. A possible solution to this problem is to compute
the Q-SEL posterior regret on the logscale. In our notation, this may be
formulated with respect to the  following ensemble of plug-in estimators,
\begin{equation}
    \widehat{\log\bth}^{L\pr} := \lt\lb \widehat{\log\theta}^{L\pr}_{1},\ldots,
    \widehat{\log\theta}^{L\pr}_{n}\rt\rb.
\end{equation}
That is, for some loss function $L\pr$, each point estimate could be
computed on the basis of the joint posterior distribution of the RRs
on the logscale, i.e. $p(\log\bth|\by)$. 
It then follows that the corresponding optimal estimator under the
Q-SEL function is the set of posterior empirical $\bp$-quantiles of the ensemble 
of the $\log\theta_{i}$'s. This gives the following posterior regret, 
\begin{equation}
    \op{regret}\lt(\op{Q-SEL}_{\bp},Q_{\widehat{\log\bth}^{L\pr}}(\bp)\rt), 
\end{equation}
which is defined with respect to $\op{Q-SEL}_{\bp}(\log\bth,\bm\delta)$
with $\bp:=\lb .25,.75\rb$. Note that the optimal estimator,
$\bm\delta=\E[Q_{\log\bth}(\bp)|\by]$ will here differ from the
optimal estimator of the Q-SEL on the RR scale, which we considered in
the rest of this chapter. Thus, the estimation of the Q-SEL function
is not invariant under reparametrisation and such a change of scale
produces two different decision-theoretic problems, which cannot be
directly compared. In practice, preference for one approach over another will
be determined by the specific needs of the modeller.

The good performance of the triple-goal plug-in estimators under both
the Q-SEL and QR-SEL functions indicate that this ensemble of point
estimates could be useful in epidemiological practice. As highlighted
in chapter \ref{chap:introduction}, the choice of specific point estimates in
epidemiology and most particularly in spatial epidemiology is made
arduous by the wide set of desiderata that such an ensemble of point
estimates should fulfil. Here, we have seen that, in addition to
providing good estimation of the EDF, the ranks and the set of 
levels of risk; the GR point estimates also provide a good
approximation of the amount of
heterogeneity in the ensemble. Therefore, the triple-goal point
estimates constitute a good candidate for both the mapping of
area-specific levels of risks and the reporting of summary statistics
about the overall risk heterogeneity in the region of interest. 
There are other inferential objectives, however, for which the GR
point estimates are far from optimal. These limitations will be
highlighted in the next chapter where we study different classification
loss functions.

%
%
\chapter{Threshold and Rank Classification Losses}\label{chap:clas}

\hspace{2cm}
\begin{minipage}[c]{11.5cm}
\small 
\begin{center}\tb{Summary}\end{center}\vspace{-.3cm}
In this chapter, we study the problem of the classification of a
subset of parameters in an ensemble given a particular threshold. This
problem can be envisaged from two different perspectives as either (i) a
rank-based classification problem or (ii) a
threshold-based classification problem. In the first case, the total amount of
data points that one wishes to classify is known. We are answering the
question: Which are the parameters situated 
in the top 10\% of the ensemble? By contrast,
in the threshold-based classification, we are only given a particular
threshold of risk, and therefore need to determine both the total
number of elements above the pre-specified threshold, and the identity of these elements. 
We review some previous research by \citet{Lin2006}, who have
investigated this problem from the ranking perspective. We adopt a
similar approach when one is concerned with threshold-based
classification and derive the minimiser in that case. We then evaluate the
corresponding optimal estimators of these two families of loss functions 
using spatial and non-spatial synthetic data for both weighted and
unweighted classification losses. Of special interest, we find that
a decision rule originally proposed by \citet{Richardson2004}, in the
context of spatial epidemiology, can be shown to be equivalent to the
specification of a weighted threshold-based
classification loss. Overall, our experimental simulations support the
use of the set of posterior means as a plug-in estimator under
balanced, but not necessarily under weighted, threshold-based classification. 
Finally, we illustrate the use of our decision-theoretic 
frameworks for a surveillance data set describing prevalence of
MRSA in UK hospitals. The good performance of the SSEL point estimates
on both types of classification paradigms indicate that this may be a
good choice of classifiers in practice when one wishes to deprecate
the occurrence of false alarms. 
\end{minipage}

\section{Introduction}\label{sec:clas intro}
The problem of the optimal classification of a
set of data points into several clusters has 
occupied statisticians and applied mathematicians for several decades
\citep[see][for a review]{Gordon1999}. As is true for all statistical methods, a
classification is, above all, a summary of the data at hand. When
clustering, the statistician is searching for an optimal
partition of the parameter space into a --generally, known or pre-specified-- number of
classes. The essential ingredient underlying all classifications 
is the minimisation of some distance
function, which generally takes the form of a similarity or dissimilarity metric
\citep{Gordon1999}. Optimal classification will then result from a
trade-off between the level of similarity of the within-cluster elements and
the level of dissimilarity of the between-cluster elements. 
In a decision-theoretic framework, such distance functions naturally arise
through the specification of a loss function for the problem at hand. 
The task of computing the optimal partition of the parameter space then
becomes a matter of minimising the chosen loss function. 

In spatial epidemiology, the issue of classifying areas according to
their levels of risk has been previously investigated by
\citet{Richardson2004}. These authors have shown that areas can be
classified according to the joint posterior distribution of the parameter
ensemble of interest. In particular, a taxonomy can be created by
selecting a decision rule $D(\alpha,C_{\al})$ for that purpose, where
$C_{\al}$ is a particular threshold, above and below which we wish to
classify the areas in the region of interest. The parameter $\alpha$,
in this decision rule, is the cut-off point associated with $C_{\al}$,
which determines the amount of probability mass necessary for an area
to be allocated to the above-threshold category. Thus, an area $i$
with level of risk denoted by $\theta_{i}$ 
will be assigned above the threshold $C_{\al}$ if $\p[\theta_{i}>
C_{\al}|\by]>\al$. \citet{Richardson2004} have therefore
provided a general framework for the classification of areas,
according to their levels of risk. This research has attracted a lot of
interest in the spatial
epidemiological literature, and has been applied to a wide range of
epidemiological settings, including healthcare use in France
\citep{Chaix2005}, disparities in the prevalence of psychotic disorders
in London \citep{Kirkbride2007}, as well as the study of the prevalence of bladder
cancer \citep{Lopez-Abente2006} and breast cancer \citep{Pollan2007},
in Spain. 
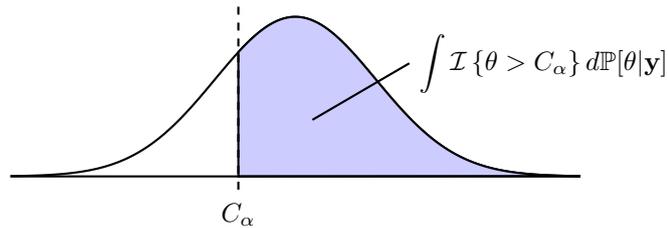
\begin{figure}[t]
  \centering
  \begin{tikzpicture}[scale=.75]
  \draw[white] (-8,0) -- (8,0); 
  \draw[thick,black] plot[smooth] file{gaussian_table};
  \filldraw[thick,fill=blue!20] plot file{gaussian_table-1_5} --
         (-1,0) -- (-1,2.196956) -- cycle;
  \draw[thick,black] (-5,0)--(5,0);   
  \draw[thick,black,dashed](-1,3)--(-1,-.3)node[anchor=north]{$C_{\al}$}; 
  \draw[thick,black](0.3,1) -- (2,2) 
   node[anchor=west]{$\displaystyle
      \int \cI\lt\lb \theta> C_{\al}\rt\rb d\p[\theta|\by] $}; 
  \end{tikzpicture}
  \caption{Posterior probability that parameter $\theta$ is larger than a
    given  threshold $C_{\al}$, which is here $\p[\theta >C_{\al}|\by]$.
    In the framework proposed by \citet{Richardson2004}, the
    element in the ensemble corresponding to parameter
    $\theta$ is considered to be greater than $C_{\al}$ if
    $\p[\theta>C_{\al}|\by]>\al$, for some $\alpha\in[0,1]$.
   \label{fig:intro classification}}   
\end{figure}

Central to the decision framework proposed by \citet{Richardson2004}
is the choice of $C_{\al}$ and $\al$. We have illustrated this general
classification problem in figure \ref{fig:intro classification} on
page \pageref{fig:intro classification}. From this figure, 
it is clear that the choice of the risk threshold $C_{\al}$, apart
from its dependence on $\alpha$, is otherwise arbitrary. Therefore, $C_{\alpha}$ may be
chosen on the basis of previous research, and will ultimately depend on the
subjective evaluation of the spatial epidemiologists or applied
statisticians working on that particular classification problem.
This decision will generally be informed by previous epidemiological
studies of identical or comparable risk surfaces. 
The choice of $\al$, however, is statistically
problematic in the sense that its value will directly determine the
sensitivity and the specificity of the classification procedure. In
selecting a value for $\al$, one is faced with 
issues which are reminiscent of the ones encountered by a statistician
of a frequentist persuasion when deciding which level of significance to adopt. 

Moreover, it should be noted that, in order to obtain good
discrimination using their classification procedure,
\citet{Richardson2004} made use of a value of $C_{\al}$ substantially different from the desired
threshold. (That is, different statistical models were given different decision
rules with different associated $C_{\al}$.)
This is another undesirable consequence of the use of a
particular cut-off point $\al$ for optimal classification: it
requires the choice of a threshold, which usually differs from the true
threshold of interest. We may, for instance, have a target threshold,
which we will denote by $C$ and a decision rule
$D(C_{\al},\al)$ with $C_{\al}\neq C$. This particular dependence
on $\al$ justifies our choice of notation, as we have indexed the
threshold $C_{\al}$ in order to emphasise that its value may vary for
different choices of $\al$. \citet{Richardson2004} therefore 
adopted different decision rules for different models,
depending on the spread of the ensemble distribution estimated by these
different models. They studied the BYM model and the
spatial mixture (MIX) model, introduced by \citet{Green2002}, and chose
$D_{\op{BYM}}(1.0,0.8)$ and $D_{\op{MIX}}(1.5,0.05)$, respectively, as decision rules
for these two models, in order to produce a classification of areas into two
clusters: non-elevated and elevated risk areas. 
In this chapter, we will address the two main quandaries associated
with the classification method proposed by \citet{Richardson2004},
which are (i) the dependence of $C_{\al}$ in $D(\cdot,\cdot)$ on the
choice of $\al$, and (ii) the dependence of the entire decision rule
on the choice of model. It should be noted that these limitations were
discussed by \citet{Richardson2004}, who emphasised the difficulties
associated with the ``\ti{calibration}'' of these decision rules, when
considering different models, and in particular with the use of the MIX model. 

Our approach to this classification question in spatial epidemiology
will draw extensively from the work of \citet{Lin2006}, who
investigated rank-based classification problems in BHMs.
Although this perspective on the classification of
a subset of elements in an ensemble substantially differs from the
problem at hand, we will see that one can easily introduce similar
loss families in the context of \ti{threshold-based classification} --that
is, classification based on a particular cut-off point expressed on
the scale of the parameters of interest, as opposed to Lin et al.'s
(2006) \ti{rank-based classification} problems, where we are given a
proportion (\eg top $20\%$) of elements in the 
ensemble, which we need to identify. Despite the inherent similarities
between rank-based and threshold-based classifications, we will see
that the optimal minimisers of these losses are substantially
different. Moreover, experimentation with simulated data will show
that different plug-in estimators are quasi-optimal under these two
decision-theoretic paradigms. One of the clear advantages of our
proposed classification framework is that it 
does not require any calibration, as the resulting estimates are
optimal for a specific choice of threshold $C$, and thus do not
necessitate any subsequent tuning from the epidemiologist
or applied statistician conducting the classification. 


The chapter is organised as follows. In section
\ref{sec:classification loss}, we introduce both the threshold-based
and rank-based classification frameworks, and describe their
respective optimal estimators. In section \ref{sec:clas data
  non-spatial}, we illustrate the use of these techniques with a 
non-spatial set of synthetic simulations. In section 
\ref{sec:clas data spatial}, the performance of the plug-in estimators
of interest is evaluated within the context of a spatially structured
set of data simulations. In section \ref{sec:mrsa}, these methods are
illustrated with a real data set describing MRSA prevalence in UK
hospitals. Finally, we close the chapter by discussing the broader
implications of these results
for epidemiologists and spatial epidemiologists under the light of
Bayesian decision theory, especially with respect to the choice of a
particular set of point estimates for a given set of data. 

\section{Classification Losses}\label{sec:classification loss}
In this section, we present our general decision-theoretic
framework for the classification of elements of a parameter ensemble
above or below a given threshold. We also introduce the rank-based
classification framework introduced by \citet{Lin2006}. For both types
of classification schemes, we will consider possible plug-in
estimators that can be used in the place of the optimal estimators
for these particular loss functions. Some links with the concepts of
posterior sensitivity and specificity will also be described for both families
of loss functions. 

\subsection{Threshold Classification Loss}
We first describe our proposed family of loss functions, following
the same structure advanced by \citet{Lin2006}. Here, we are given a
particular cut-off point $C$. The loss associated with
misclassification either above or below the threshold of interest will
be formulated as follows. Following standard statistical terminology, we
will express such misclassifications in terms of false positives (FPs)
and false negatives (FNs). These concepts are formally described as
\begin{equation}\label{eq:abba1}
    \op{FP}(C,\theta,\delta) := \cI\lt\lb \theta \leq C, \delta > C \rt\rb,
\end{equation}
and
\begin{equation}\label{eq:abba2}
    \op{FN}(C,\theta,\delta) :=  \cI\lt\lb  \theta > C, \delta \leq C \rt\rb,
\end{equation}
which correspond to the occurrence of a false positive
misclassification (type I error) and a false negative
misclassification (type II error), respectively. In equation
(\ref{eq:abba1}), our decision --denoted $\delta$-- is above threshold, whereas the true
value of the parameter --denoted $\theta$-- is below threshold, and the reverse situation
can be observed in equation (\ref{eq:abba2}).

For the decision problem to be fully specified, we need to choose
a loss function based on the sets of unit-specific FPs and FNs.
We here assume that $C\in\R$ if the parameters of interest are real numbers, or
$C\in\R^{+}$ if the parameters of interest are strictly positive such as
when considering RRs in the context of spatial epidemiology. Following
the decision framework introduced by \citet{Lin2006}, we therefore
formalise this problem using the threshold classification loss (TCL),
defined as follows. We first introduce the weighted version of the TCL
function, and will then specialise our definition to the case of
unweighted threshold classification. Let $0\leq p\leq 1$. The
$p$-weighted threshold classification loss ($\op{TCL}_{p}$) function
is then defined as
\begin{equation}
     \op{TCL}_{p}(C,\bth,\bm\delta) := \frac{1}{n} \sum_{i=1}^{n}
     p\op{FP}(C,\theta_{i},\delta_{i}) + (1-p)\op{FN}(C,\theta_{i},\delta_{i}).     
     \label{eq:ptcl}     
\end{equation}
One of the advantages of the choice of $\op{TCL}_{p}$ for quantifying
the misclassifications of the elements of a parameter ensemble is that
it is normalised, in the sense that $\op{TCL}_{p}(C,\bth,\bm\delta)\in [0,1]$
for any choice of $C$ and $p$. Moreover, $\op{TCL}_{p}$ attains zero
if the classification of each element is correct and $1.0$ if none of
the elements are correctly classified. As was previously done for the loss
families described in the preceding 
chapters, we will drop references to the arguments
$(C,\bth,\bm\delta)$ controlling $\op{TCL}_{p}(C,\bth,\bm\delta)$, 
when their specification is obvious from the context. 
Our main result, in this section, is the following minimisation. 
\begin{pro}\label{pro:tcl}
  For some parameter ensemble $\bth$, and given a real-valued
  threshold $C\in\R$ and $p\in [0,1]$, we have the following optimal estimator under
  weighted TCL, 
  \begin{equation}
    \bth^{\op{TCL}}_{(1-p)} = \argmin_{\bm\delta}
    \E\lt[\op{TCL}_{p}(C,\bth,\bm\delta)|\by\rt],
    \label{eq:tcl minimiser}
  \end{equation}
  where $\bth^{\op{TCL}}_{(1-p)}$ is the vector of posterior $(1-p)$-quantiles defined
  as 
  \begin{equation}
      \bth^{\op{TCL}}_{(1-p)}:=\lt\lb Q_{\theta_{1}|\by}(1-p),\ldots,Q_{\theta_{n}|\by}(1-p)\rt\rb,
  \end{equation}
  where $Q_{\theta_{i}|\by}(1-p)$ denotes the posterior $(1-p)$-quantile of
  the $i\tth$ element, $\theta_{i}$, in the parameter ensemble. Moreover,
  $\bth^{\op{TCL}}_{(1-p)}$ is not unique. That is, there exists more
  than one minimiser in equation (\ref{eq:tcl minimiser}). 
\end{pro}
In our notation, we have emphasised the distinction between the
\ti{posterior empirical quantile} of the ensemble and the individual
\ti{posterior quantiles}, as follows. In chapter \ref{chap:mrrr}, the
posterior quantile of the empirical distribution of the $\theta_{i}$'s was denoted
$\E[Q_{\bth}(p)|\by]$. This quantity
minimises the posterior expected Q-SEL function for
some given $p$. By contrast, in this chapter, we are interested in 
the \ti{posterior quantile} of a single $\theta_{i}$, which we have
denoted by $Q_{\theta_{i}|\by}(p)$ in proposition \ref{pro:tcl}, and
which is formally defined as 
\begin{equation}
     Q_{\theta_{i}|\by}(p) := 
     \inf\lt\lb x \in \Theta_{i}: F_{\theta_{i}|\by}(x)\geq p \rt\rb,
\end{equation}
where $\Theta_{i}$ is the domain of $\theta_{i}$.
The proof of proposition \ref{pro:tcl} makes use of a strategy similar
to the one reported by \citet{Berger1980}, who showed that the
posterior median is the optimal 
estimator under AVL, as reported in section \ref{sec:classical loss}. 
We prove the result by exhaustion in three cases.
The full proof is reported in section \ref{sec:clas proof}, at the end
of this chapter. Note that the fact that $\op{TCL}_{p}$ is minimised
by $\bth^{\op{TCL}}_{(1-p)}$ and not $\bth^{\op{TCL}}_{(p)}$ is solely
a consequence of our choice of definition for the TCL function. If the
weighting of the FPs and FNs had been $(1-p)$ and $p$, respectively,
then the optimal minimiser of that function would indeed be a vector
of posterior $p$-quantiles. 

We now specialise this result to the unweighted TCL family, which is
defined analogously to equation (\ref{eq:ptcl}), as follows,
\begin{equation}
     \op{TCL}(C,\bth,\bm\delta) := \frac{1}{n} \sum_{i=1}^{n}
     \op{FP}(C,\theta_{i},\delta_{i}) + \op{FN}(C,\theta_{i},\delta_{i}).
     \label{eq:tcl}
\end{equation}
The minimiser of this loss function can be shown to be trivially
equivalent to the minimiser of $\op{TCL}_{0.5}$. The following
corollary of proposition \ref{pro:tcl} formalises this relationship
between the weighted and unweighted TCL. 
\begin{cor}\label{cor:tcl}
  For some parameter ensemble $\bth$ and $C\in\R$, we have
  \begin{equation}
      \bth^{\op{MED}} = 
      \argmin_{\bm\delta} \E\lt[\op{TCL}(C,\bth,\bm\delta)|\by\rt],
  \end{equation} 
  where 
  \begin{equation}
       \bth^{\op{MED}} := \bth^{\op{TCL}}_{(0.5)} =\lt\lb
        Q_{\theta_{1}|\by}(0.5),\ldots,Q_{\theta_{n}|\by}(0.5)\rt\rb, 
  \end{equation}
  is the vector of posterior medians, and this optimal estimator is not unique. 
\end{cor}
As noted earlier, this corollary immediately follows from
proposition \ref{pro:tcl} by noting that 
\begin{equation}
      \argmin_{\bm\delta} \E[\op{TCL}(C,\bth,\bm\delta)|\by] =
      \argmin_{\bm\delta} \E[\op{TCL}_{0.5}(C,\bth,\bm\delta)|\by],
\end{equation}
for every $C$. Note that although both the classical AVL loss
described in section
\ref{sec:classical loss} and the unweighted TCL presented here have
identical minimisers, the relationship between these two estimation
schemes is not trivial. Indeed, the former is a problem of
estimation, whereas the latter is a problem of
classification. However, it can be noted that for both the AVL and the
TCL functions, these estimators are not unique. 

Our main focus in this chapter will be on the unweighted TCL, except
in section \ref{sec:clas weighted tcl}, where we consider the
relationship between the decision rule originally proposed by
\citet{Richardson2004} and the weighted TCL. We also supplement this
discussion with a set of simulations evaluating the performance of
plug-in estimators under weighted TCL. The rest of this chapter, however, will
focus on the unweighted TCL, and except otherwise specified, the TCL
function will therefore refer to the unweighted version presented in
equation (\ref{eq:tcl}). 

A graphical interpretation of the posterior TCL is illustrated 
in figure \ref{fig:FN-FP} on page \pageref{fig:FN-FP}. The posterior
expected loss under this function takes the following form,
\begin{equation}
    \E\lt[\op{TCL}(C,\bth,\bm\delta)|\by\rt]
    = \frac{1}{n}\sum_{i=1}^{n} \int\limits_{-\infty}^{C} d\p[\theta_{i}|\by]
    \cI\lt\lb \delta_{i} > C\rt\rb
    + \int\limits_{C}^{+\infty}d\p[\theta_{i}|\by]
    \cI\lt\lb \delta_{i} \leq C\rt\rb,
    \label{eq:posterior tcl}
\end{equation}
whose formulae is derived using $\cI\lt\lb \theta \leq C,\delta > C\rt\rb = \cI\lt\lb \theta \leq
C\rt\rb\cI\lt\lb\delta > C\rt\rb$. It is of special importance to note
that when using the posterior TCL, any classification --correct or
incorrect-- will incur a penalty. The \ti{size} of that penalty,
however, varies substantially depending on whether or not the
classification is correct. A true positive, as in diagram (a) in
figure \ref{fig:FN-FP}, can be distinguished from a false positive, as in diagram (c), 
by the fact that the former will only incur a small penalty
proportional to the posterior probability of the parameter to be
below the chosen cut-off point $C$. By contrast, a false positive, as
in panel (c), will necessarily incur a larger penalty, because more mass is located
below the threshold than above the threshold. We can make similar observations when comparing
diagrams (b) and (d). In addition, note that although the TCL attains zero
when the classification of each element in the ensemble is correct, this
it not the case for the posterior expected TCL, which is necessarily
greater than zero. Moreover, this
representation also clarifies the arguments used in the proof of
proposition \ref{pro:tcl}, and provides an intuitive
justification for the use of the median as an optimal estimator under 
unweighted TCL.
\begin{figure}[htbp]
  \centering
  \tikzstyle{background rectangle}=[draw=gray!30,fill=gray!20,rounded
  corners=1ex]
  \begin{tikzpicture}[scale=.75]
  \draw[white] (-8,0) -- (8,0); 
  \draw (0,4.0) node[anchor=north]{\tb{(a)}};   
  \draw[thick,black] plot[smooth] file{gaussian_table};
  \filldraw[thick,fill=blue!20] plot file{gaussian_table-5_-1} -- 
                     (-1,2.196956)--(-1,0) -- cycle;
  \draw[thick,black] (-5,0)--(5,0);   
  \draw[thick,black,dashed] (-.1,3)--(-.1,-.3)node[anchor=north]{$\delta_{i}$}; 
  \draw[thick,black,dashed](-1,3)--(-1,-.3)node[anchor=north]{$C$};  
  \draw[thick,black](-1.3,1) -- (-2,2) 
   node[anchor=east]{$\displaystyle
      \cI\lt\lb\delta_{i} > C\rt\rb 
      \int\limits_{-\infty}^{C}d\p[\theta_{i}|\by] $}; 
  \end{tikzpicture}\\
  \vspace{.3cm}
  \begin{tikzpicture}[scale=.75]
  \draw[white] (-8,0) -- (8,0); 
  \draw (0,4.0) node[anchor=north]{\tb{(b)}};   
  \draw[thick,black] plot[smooth] file{gaussian_table};
  \filldraw[thick,fill=blue!20] plot file{gaussian_table1_5} --
         (1,0) -- (1,2.196956) -- cycle;
  \draw[thick,black] (-5,0)--(5,0);   
  \draw[thick,black,dashed] (.1,3)--(.1,-.3) 
        node[anchor=north]{$\delta_{i}$}; 
  \draw[thick,black,dashed](1,3)--(1,-.3)node[anchor=north]{$C$}; 
  \draw[thick,black](1.3,1) -- (2,2) 
   node[anchor=west]{$\displaystyle
      \cI\lt\lb\delta_{i} \leq C\rt\rb 
      \int\limits_{C}^{+\infty}d\p[\theta_{i}|\by] $}; 
  \end{tikzpicture}\\
  \vspace{.3cm}
  \begin{tikzpicture}[scale=.75]
  \draw[white] (-8,0) -- (8,0); 
  \draw (0,4.0) node[anchor=north]{\tb{(c)}};   
  \draw[thick,black] plot[smooth] file{gaussian_table};
  \filldraw[thick,fill=blue!20] plot file{gaussian_table-5_1} --
         (1,2.196956) -- (1,0) -- cycle;
  \draw[thick,black] (-5,0)--(5,0);   
  \draw[thick,black,dashed] (1.7,3)--(1.7,-.3) 
        node[anchor=north]{$\delta_{i}$}; 
  \draw[thick,black,dashed](1,3)--(1,-.3)node[anchor=north]{$C$}; 
  \draw[thick,black](-0.3,1) -- (-2,2) 
   node[anchor=east]{$\displaystyle
      \cI\lt\lb\delta_{i} > C\rt\rb 
      \int\limits_{-\infty}^{C}d\p[\theta_{i}|\by] $}; 
  \end{tikzpicture}\\
  \vspace{.3cm}
  \begin{tikzpicture}[scale=.75]
  \draw[white] (-8,0) -- (8,0); 
  \draw (0,4.0) node[anchor=north]{\tb{(d)}};   
  \draw[thick,black] plot[smooth] file{gaussian_table};
  \filldraw[thick,fill=blue!20] plot file{gaussian_table-1_5} --
         (-1,0) -- (-1,2.196956) -- cycle;
  \draw[thick,black] (-5,0)--(5,0);   
  \draw[thick,black,dashed] (-1.7,3)--(-1.7,-.3) 
        node[anchor=north]{$\delta_{i}$}; 
  \draw[thick,black,dashed](-1,3)--(-1,-.3)node[anchor=north]{$C$}; 
  \draw[thick,black](0.3,1) -- (2,2) 
   node[anchor=west]{$\displaystyle
      \cI\lt\lb\delta_{i} \leq C\rt\rb 
      \int\limits_{C}^{+\infty}d\p[\theta_{i}|\by] $}; 
  \end{tikzpicture}
  \caption{Illustration of the components of the posterior expected
    TCL based on the posterior distribution of $\theta_{i}$. The
    choice of a point estimate $\delta_{i}$ for the quantity
    $\theta_{i}$ results in either a correct classification (a-b)
    or a misclassification (c-d), with (a) and (c) representing
    $\E[\op{FP}(C,\theta_{i},\delta_{i})|\by]$
    and (b) and (d) denoting $\E[\op{FN}(C,\theta_{i},\delta_{i})|\by]$.}
    \label{fig:FN-FP}
\end{figure}
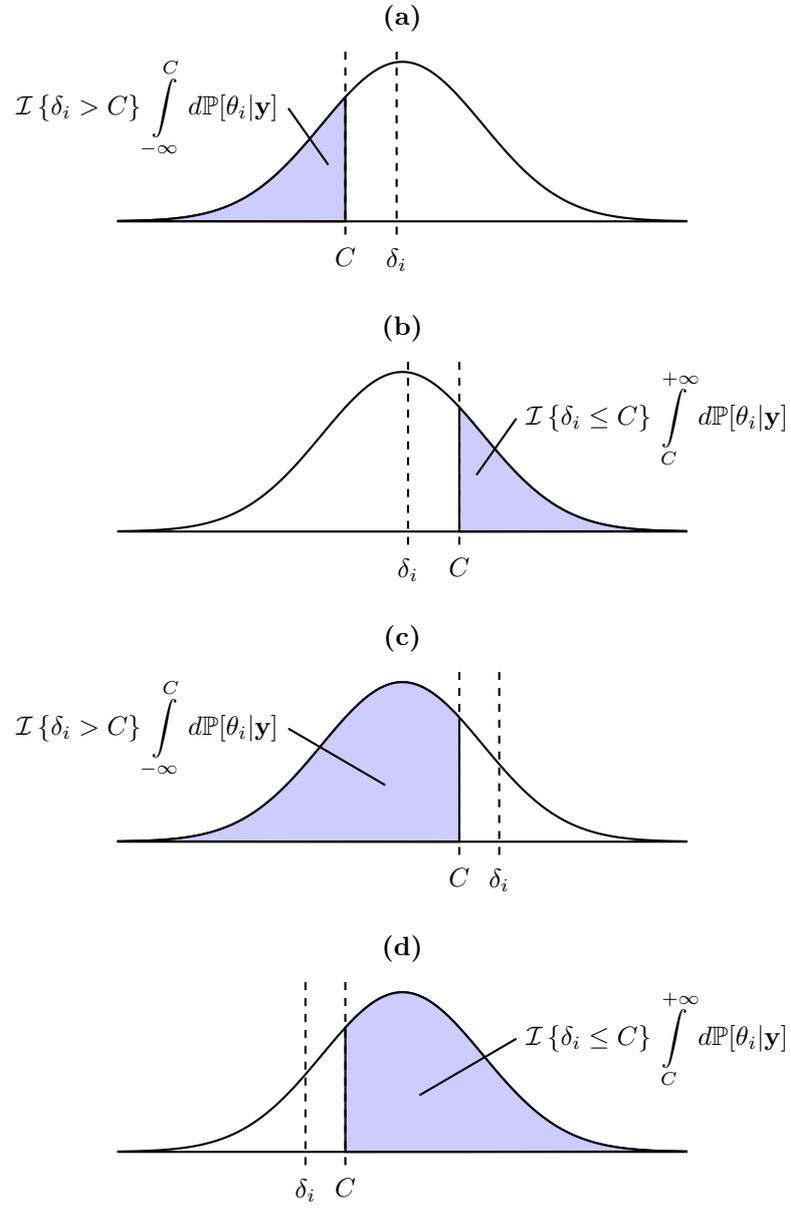

\subsection{Rank Classification Loss}\label{sec:rcl}
We here briefly present a different family of classification loss
functions based on ranks, which was originally introduced by \citet{Lin2006}. In this
chapter, we will be interested in evaluating the performance of
various plug-in estimators for rank-based classification.

\citet{Lin2006} were concerned with the identification 
of the elements of a parameter ensemble that represent the proportion of
elements with the highest level of risk. In an epidemiological setting, one
may, for instance, wish to estimate the ten percent of hospitals that
have the highest RRs for a particular condition. 
Such a classification is therefore based on a particular rank percentile
cut-off denoted $\ga\in [0,1]$, which determines a group of areas of
high-risk. That is, we wish to identify
the areas whose percentile rank is above the cut-off
point $\ga$. Specific false positive and false negative functions
dependent on the percentile ranks of the parameter ensemble can
be defined following the convention introduced in equations
(\ref{eq:abba1}) and (\ref{eq:abba2}). This gives
\begin{equation}
    \op{FP}(\gamma,P,\delta) := \cI\lt\lb P \leq \gamma, \delta > \gamma \rt\rb,
\end{equation}
and
\begin{equation}
    \op{FN}(\gamma,P,\delta) :=  \cI\lt\lb  P > \gamma, \delta \leq \gamma \rt\rb,
\end{equation}
where the percentile rank, $P$, is a function of the ranks of $\bth$, as formally defined in equation
\ref{eq:percentile}. In this chapter, the word percentile will refer
to rank percentile, except specified otherwise. 
Note that we have not made an explicit notational distinction
between threshold-based FP and FN functions and their equivalent
in a percentile rank-based setting. However, which one we
are referring to should be obvious from the context. 
Of particular interest to our development and related to these
families of loss functions is the following unweighted rank
classification loss (RCL) function, 
\begin{equation}\label{eq:rcl}
  \op{RCL}(\gamma,\bth,\bm\delta) := 
        \frac{1}{n} \sum_{i=1}^{n} \op{FP}(\ga,P_{i}(\bth),\delta_{i})
        + \op{FN}(\ga,P_{i}(\bth),\delta_{i}),
\end{equation}
which is the rank-based analog of the unweighted TCL described in
equation (\ref{eq:tcl}). \citet{Lin2006} showed that the RCL function
is minimised by the
following estimator. We report this result here and refer the reader to the
original paper for a proof. 
\begin{pro}[\tb{Lin et al., 2006}]\label{pro:rcl}
   For any $0\leq\ga\leq 1$, the set of percentile ranks
   $\widehat{\bP}=\lb\widehat{P}_{1},\ldots,\widehat{P}_{n}\rb$, 
   whose elements are defined as 
   \begin{equation}
        \widehat{P}_{i} := \frac{1}{(n+1)}
        R_{i}\Big(\p[P_{1}(\bth)\geq \ga|\by],\ldots,
        \p[P_{n}(\bth)\geq \ga|\by]\Big),    
         \label{eq:rank minimiser}
   \end{equation}
   satisfies
   \begin{equation}
        \widehat{\bP} =  \argmin_{\bm\delta}
           \E\lt[\op{RCL}(\gamma,\bth,\bm\delta)|\by\rt].
   \end{equation}    
   Moreover, the minimiser $\widehat{\bP}$ is not unique.
\end{pro}

Associated with the optimal classifier $\widehat{\bP}$, one can
also derive the optimal rank $\widehat{\bR}$. In order to gain a
greater understanding of the mechanisms underlying the computation of
this optimal minimiser, one can expand the formulae in equation
(\ref{eq:rank minimiser}), using the definition of rank in equation
(\ref{eq:rank}) on page \pageref{eq:rank}, in order to obtain the following,
\begin{equation} 
    \widehat{P}_{i} = \frac{1}{(n+1)}\sum_{j=1}^{n}
    \cI\Big\lb \p\lt[P_{i}(\bth) \geq \ga|\by\rt] \geq
    \p\lt[P_{j}(\bth) \geq \ga|\by\rt] \Big\rb,
    \label{eq:minimiser2}
\end{equation}
for every $i=1,\ldots,n$. It can be seen from this expansion that the
percentiles, which will
optimally classify the areas above a certain threshold, are
the percentiles of the posterior probabilities of the true percentiles
being above the $\gamma$ cut-off. We have already encountered this
method of \ti{double ranking} when introducing the triple-goal
estimator of \citet{Shen1998}. 
Proposition \ref{pro:rcl} also states that the minimiser
$\widehat{P}_{i}$ is not unique. This follows from the fact that the
percentile ranks on each side of $\ga$ form two subsets and the
elements in these subsets can be permuted while maintaining the
optimality of the resulting percentile ranks. This non-uniqueness is
especially illustrated by the asymptotic equivalence between the
minimiser $\widehat{\bP}(\ga|\by)$ and another popular classifier 
in the literature: the Normand's estimator \citep[see][]{normand1997}. 

Albeit the TCL and RCL functions have a similar structure, they nonetheless
differ in two important ways. Firstly, the RCL can be simplified due
to the symmetry between the FPs and FNs, while no
such symmetry can be employed to simplify the formulation of the TCL function.
From equation (\ref{eq:rcl}), the RCL can be simplified as follows,
\begin{equation}
\frac{1}{n} \sum_{i=1}^{n} \op{FP}(\ga,P_{i},\delta_{i}) +
  \op{FN}(\ga,P_{i},\delta_{i}) = \frac{2}{n} \sum_{i=1}^{n} 
  \op{FP}(\ga,P_{i},\delta_{i}), 
  \label{eq:rank 01}
\end{equation}
where we have used the relation,
\begin{equation}
     \sum_{i=1}^{n}\op{FP}(\ga,P_{i},\delta_{i}) = 
     \sum_{i=1}^{n}\op{FN}(\ga,P_{i},\delta_{i}),
\end{equation}
which follows from the fact that any number of FPs in the context of rank
classification is necessarily compensated by an identical number of FNs.
By contrast, the equivalent relationship does not hold for TCL. That is, the quantities
\begin{equation}
     \sum_{i=1}^{n}\op{FP}(C,\theta_{i},\delta_{i}) \qq\te{and}\qq
     \sum_{i=1}^{n}\op{FN}(C,\theta_{i},\delta_{i}),
\end{equation}
need not be equal. This follows from the fact that the total number of data points that
should be classified above a particular threshold $C$ is unknown
when evaluating the TCL function. However, this is
not the case for the RCL function, where we know \ti{a
priori} the total number of percentiles, which should be classified
above the target percentile of interest, $\ga$. That is, 
the RCL function is optimised with respect to a vector
$\bm\delta$, whose `size' is a known quantity --where, by
size, we mean the total number of $\delta_{i}$'s above $\gamma$.
By contrast, for TCL, one cannot derive 
the number of elements in the latent categories. Thus, both the
allocation vector $\bm\delta$, and the total number of
elements above $C$ are unknown. In this sense, the minimisation procedure
for the TCL function is therefore more arduous than the optimisation
of the RCL function. 

\subsection{Posterior Sensitivity and Specificity}\label{sec:sense and sensitivity}
Our chosen decision-theoretic framework for classification has the added benefit of being
readily comparable to conventional measures of
classification errors widely used in the context of test theory.
For our purpose, we will define the Bayesian sensitivity 
of a classification estimator $\bm\delta$, also
referred to as the posterior true positive rate (TPR), as follows
\begin{equation}
    \op{TPR}(C,\bth,\bm\delta)
    := \frac{\sum_{i=1}^{n} \E[\op{TP}(C,\theta_{i},\delta_{i})|\by]}
    {\sum_{i=1}^{n} \p[\theta_{i} > C |\by]},
    \label{eq:tpr}
\end{equation}
where the expectations are taken with respect to the joint posterior
distribution of $\bth$. Similarly, the Bayesian specificity, or
posterior true negative rate (TNR), will be defined as 
\begin{equation}
    \op{TNR}(C,\bth,\bm\delta)
    := \frac{\sum_{i=1}^{n}\E[\op{TN}(C,\theta_{i},\delta_{i})|\by]}
    {\sum_{i=1}^{n} \p[\theta_{i} \leq C |\by]},
    \label{eq:tnr}
\end{equation}
where in both definitions, we have used $\op{TP}(C,\theta_{i},\delta_{i}) := 
\cI\lt\lb \theta_{i} > C , \delta_{i} > C \rt\rb$
and $\op{TN}(C,\theta_{i},\delta_{i}) := \cI\lt\lb  \theta_{i} \leq C,
\delta_{i} \leq C\rt\rb$. It then follows that we can formulate the
relationship between the posterior expected TCL and the Bayesian
sensitivity and specificity as
\begin{equation}
  \E[\op{TCL}(C,\bth,\bm\delta)|\by]     
    = \frac{1}{n}\op{FPR}(C,\bth,\bm\delta)\sum_{i=1}^{n} \p[\theta_{i} \leq C |\by]
    + \frac{1}{n}\op{FNR}(C,\bth,\bm\delta)\sum_{i=1}^{n} \p[\theta_{i} > C |\by]. 
    \notag
\end{equation}
where $\op{FPR}(C,\bth,\bm\delta) := 1- \op{TNR}(C,\bth,\bm\delta)$
and $\op{FNR}(C,\bth,\bm\delta) := 1- \op{TPR}(C,\bth,\bm\delta)$.

In the RCL framework, the definitions of the TPR and TNR in equations (\ref{eq:tpr}) and
(\ref{eq:tnr}) can be adapted to the context of percentile rank-based
classification in the following manner:
\begin{equation}
    \op{TPR}(\ga,\bth,\bm\delta)
    := \frac{\sum_{i=1}^{n} \E[\op{TP}(\ga,P_{i}(\bth),\delta_{i})|\by]}
    {\sum_{i=1}^{n} \p[P_{i}(\bth) > \ga |\by]},
    \label{eq:rank tpr}  
\end{equation}
and analogously for $\op{TNR}(\ga,\bth,\bm\delta)$. One can easily
derive the relationship between the posterior RCL and the
percentile rank-based definitions of Bayesian sensitivity and
specificity. These are
\begin{equation}
  \E[\op{RCL}(\ga,\bth,\bm\delta)|\by] = 
     \frac{2}{n}\op{FNR}(\ga,\bth,\bm\delta)\sum_{i=1}^{n}\p[P_{i}(\bth)
     > \ga|\by],
\end{equation}
where $\op{FNR}(\ga,\bth,\bm\delta):=1-\op{TPR}(\ga,\bth,\bm\delta)$,  
and equivalently, 
\begin{equation}
  \E[\op{RCL}(\ga,\bth,\bm\delta)|\by] = 
     \frac{2}{n}\op{FPR}(\ga,\bth,\bm\delta)\sum_{i=1}^{n}\p[P_{i}(\bth) \leq \ga|\by],
\end{equation}
with $\op{FPR}(\ga,\bth,\bm\delta):=1-\op{TNR}(\ga,\bth,\bm\delta)$,  
which follows from equation (\ref{eq:rank 01}) on page \pageref{eq:rank 01}. 

\subsection{Performance Evaluation}\label{sec:clas evaluation}

The posterior regret will be used to compare the optimal classifiers
under the TCL and RCL functions with different plug-in classifiers, denoted
$h(\hat\bth^{L\pr})$, for some loss function of interest
$L\pr$. The posterior regret was introduced in section
\ref{sec:research question}. For the unweighted TCL, this quantity takes the
following form,
\begin{equation}
    \op{regret}(\op{TCL},h(\hat\bth^{L\pr})) = \E[\op{TCL}(C,\bth,h(\hat\bth^{L\pr}))|\by]
    - \min_{\bm\delta}\E[\op{TCL}(C,\bth,\bm\delta)|\by],
\end{equation}
where here $h(\cdot)$ is simply the identity function returning the
original vector of point estimates, and the optimal estimator under
TCL is the vector of posterior medians of the parameter ensemble of
interest. Similarly, the posterior regret under RCL takes the following form, 
\begin{equation}
    \op{regret}(\op{RCL},h(\hat\bth^{L\pr})) = \E[\op{RCL}(\gamma,\bth,\bP(\hat\bth^{L\pr}))|\by]
    - \min_{\bm\delta}\E[\op{RCL}(\gamma,\bth,\bm\delta)|\by],
\end{equation}
where $\bP(\cdot)$ is a multivariate percentile rank function
returning the vector of percentile rank corresponding to the ensemble
of point estimates $\hat\bth^{L\pr}$, and the optimal estimator under
the RCL is the minimiser reported in proposition \ref{pro:rcl}. In
this chapter, we will be interested in comparing the optimal
estimators under the TCL and RCL functions with different plug-in estimators
based on various ensembles of point estimates including the set of MLEs, the
posterior means and the WRSEL, CB and GR ensembles, as described in
chapter \ref{chap:review}. 

Computationally, the posterior regrets of the TCL and RCL functions
for several plug-in estimators can be
calculated at a lesser cost by noting that both functions simply require
the computation of the vector of penalties for misclassification
once. Consider the posterior TCL first. Let us denote the allocation vector
for some classification estimator $\hat{\bth}^{L\pr}$ above and below
the chosen threshold $C$ by
\begin{equation}
  \bz_{C}^{>}(\hat{\bth}^{L\pr}) := \lt\lb \cI\lb \hat\theta^{L\pr}_{1}>
  C\rb ,\ldots, \cI\lb \hat\theta^{L\pr}_{n}>  C\rb \rt\rb,
  \label{eq:bz}
\end{equation}
and $\bz_{C}^{\leq}$, respectively. In addition, we define the following vector of
posterior probabilities, 
\begin{equation}
  \bp_{C}^{>}(\bth) := \lt\lb \p[\theta_{1}>
  C|\by] ,\ldots, \p[\theta_{n}>  C|\by] \rt\rb,  
  \label{eq:bp}
\end{equation}
with $\bp^{\leq}_{C}$ being defined similarly for the $n$-dimensional vector of posterior
probabilities of each $\theta_{i}$ being below $C$. Then, the
posterior TCL used as a criterion for evaluation in this section can
simply be computed as 
\begin{equation}
  \E[\op{TCL}(C,\bth,\bm\delta)|\by] = \frac{1}{n}\lt\langle \bz^{\leq}_{C},\bp^{>}_{C} \rt\rangle + 
  \frac{1}{n}\lt\langle \bz^{>}_{C},\bp^{\leq}_{C} \rt\rangle,
  \label{eq:computational threshold}
\end{equation}
where $\lt\langle\cdot,\cdot\rt\rangle$ stands for the dot product. It
is therefore sufficient to compute the vectors of penalties
--$\bp^{>}_{C}$ and $\bp^{\geq}_{C}$-- once and then apply it to any
vector of plug-in classifiers, $\hat\bth^{L\pr}$. This formulae has also the added
advantage of providing some insights into
the generating process of the threshold-based loss
quantities. Equation (\ref{eq:computational threshold}) shows that 
the posterior TCL is the sum of two dot products,
where the allocations of the elements and posterior probabilities of
these elements being above or below a threshold have been inverted. 
An equivalent computational formulae can be derived for the posterior
expected RCL. For the percentile rank-based performance, we evaluated the
percentile-based classifiers, denoted $\bP(\hat\bth^{L\pr}):=\lb
P_{1}(\hat\bth^{L\pr}),\ldots,P_{n}(\hat\bth^{L\pr})\rb$, using the
following simplified Bayesian expected loss, 
\begin{equation}
    \E[\op{RCL}(\gamma,\bth,\bP(\hat\bth^{L\pr}))|\by]
    = \frac{2}{n}\Big\langle \bz^{\leq}_{\ga},\bp^{>}_{\ga} \Big\rangle
    = \frac{2}{n}\Big\langle \bz^{>}_{\ga},\bp^{\leq}_{\ga} \Big\rangle,
\end{equation}
which can be derived from an application of equation (\ref{eq:rank 01}), 
and where we have defined $\bz^{\leq}_{\ga}$, $\bz^{>}_{\ga}$,
$\bp^{\leq}_{\ga}$ and $\bp^{>}_{\ga}$, 
following the convention introduced in equations (\ref{eq:bz}) and
(\ref{eq:bp}), respectively.

For all simulated scenarios and replicate data sets in the sequel,
when considering the performance of plug-in estimators under TCL, we fixed 
the threshold $C$ to the following value,
\begin{equation}
      C:=\bar{\theta} +
      \lt(\frac{1}{n}\sum_{i=1}^{n}\lt(\theta_{i}-\bar{\theta}\rt)^{2}\rt)^{1/2}. 
      \label{eq:classification threshold}
\end{equation}
with $\bar{\theta}:=n^{-1}\sum_{i=}^{n}\theta_{i}$ being the true
ensemble mean, and $\bth$ denotes the true ensemble
distribution. That is, we evaluated the ability of the target
plug-in estimators to determine which elements in the parameter
ensemble should be classified as farther than one standard deviation
from the true empirical mean. For the RCL, we fixed $\ga=.80$ for all scenarios
and replicates, which corresponded to the identification of the
top $20\%$ of ensemble's elements taking the largest values. These
conventions for the TCL and RCL functions were applied to both the 
non-spatial and spatial simulation experiments. 

\section{Non-spatial Data Simulations}\label{sec:clas data non-spatial}
The full description of the synthetic data sets used in this
simulation experiment have been fully described in chapter
\ref{chap:mrrr}. We are here concerned with two BHMs: (i) the compound
Gaussian and (ii) compound Gamma models. In the next section, we
briefly describe which parameters will be the object of the
classification scheme in each model family and recall the two main
factors, which were manipulated in these simulations. 

\subsection{Parameter Ensembles  and Experimental Factors}
Our choice of simulation models and the specific values given to 
the hyperparameters parallel the first part of the simulation design
conducted by \citet{Shen1998}. As in chapter \ref{chap:mrrr}, 
the same models were used to generate and to fit the data. The parameter
ensembles, denoted $\bth$, in both the compound Gaussian and compound Gamma
models were classified according to the threshold described in
equation (\ref{eq:classification threshold}).
This choice of classification threshold resulted in 
approximately $16\%$ and $10\%$ of the ensemble distribution being
above $C$ for the Normal-Normal and Gamma-Inverse Gamma models,
respectively. This choice of $C$ therefore produced a proportion of
above-threshold ensemble elements of the same order of magnitude as
the one used for the evaluation of the posterior expected RCL, which
was $\ga=.80$. As in chapter \ref{chap:mrrr}, we manipulated the
ratio of the largest to the smallest (RLS) sampling variances in both
models. In addition, we also considered the effect of an increase of the
size of the parameter ensemble on the posterior expected TCL and RCL
functions. The plug-in estimators of interest were the classifiers
constructed on the basis of the ensemble of MLEs, posterior means,
as well as the WRSEL, CB and GR plug-in classifiers. 


\subsection{Plug-in Estimators under TCL}
%
\begin{table}[t] 
 \footnotesize
 \caption{
Posterior regrets based on $\op{TCL}(C,\bth,\bm\delta)$,
with $C=\E[\bth] + \op{sd}[\bth]$, for five plug-in estimators
for the conjugate Gaussian model in equation (\ref{eq:normal-normal})
and the conjugate Gamma model in equation (\ref{eq:gamma-inverse gamma}),
and for 3 different levels of RLS,
and 3 different values of $n$, averaged over 100 replicate data sets.
The posterior expected loss of the optimal estimator is given in the
first column. Entries are scaled by a factor of $10^3$. 
The posterior regrets expressed as percentage of the posterior loss
under the optimal estimator is indicated in parentheses.
\label{tab:non-spatial cmse_table}} 
 \centering
\begin{threeparttable}
 \begin{tabular}{>{\RaggedRight}p{60pt}>{\RaggedLeft}p{30pt}|>{\RaggedLeft}p{25pt}@{}>{\RaggedLeft}p{17pt}>{\RaggedLeft}p{25pt}@{}>{\RaggedLeft}p{10pt}>{\RaggedLeft}p{25pt}@{}>{\RaggedLeft}p{20pt}>{\RaggedLeft}p{25pt}@{}>{\RaggedLeft}p{17pt}>{\RaggedLeft}p{25pt}@{}>{\RaggedLeft}p{17pt}}\hline
\multicolumn{1}{c}{\itshape Scenarios}&
\multicolumn{11}{c}{\itshape Posterior regrets\tnote{a}}
\tabularnewline \cline{3-12}
\multicolumn{1}{>{\RaggedRight}p{60pt}}{}&\multicolumn{1}{c}{TCL}&\multicolumn{2}{c}{MLE}&
\multicolumn{2}{c}{SSEL}&\multicolumn{2}{c}{WRSEL}&\multicolumn{2}{c}{CB}&\multicolumn{2}{c}{GR}\tabularnewline
\hline
{\itshape $\op{RLS}\doteq1$}&&&&&&&&&&&\tabularnewline
\normalfont   N-N, $n=100$ &    $133.6$ &    $49.9$ &    ($37$) &    $0.0$ &    ($0$) &    $ 16.6$ &    ($ 12$) &    $13.3$ &    ($10$) &    $13.3$ &    ($10$)\tabularnewline
\normalfont   N-N, $n=200$ &    $133.1$ &    $49.1$ &    ($37$) &    $0.0$ &    ($0$) &    $ 69.7$ &    ($ 52$) &    $13.2$ &    ($10$) &    $13.8$ &    ($10$)\tabularnewline
\normalfont   N-N, $n=1000$ &    $133.2$ &    $48.8$ &    ($37$) &    $0.0$ &    ($0$) &    $167.3$ &    ($126$) &    $13.9$ &    ($10$) &    $14.1$ &    ($11$)\tabularnewline
\normalfont   G-IG, $n=100$ &    $109.7$ &    $74.0$ &    ($67$) &    $3.5$ &    ($3$) &    $ 35.9$ &    ($ 33$) &    $37.3$ &    ($34$) &    $31.2$ &    ($28$)\tabularnewline
\normalfont   G-IG, $n=200$ &    $102.4$ &    $72.1$ &    ($70$) &    $2.9$ &    ($3$) &    $100.8$ &    ($ 98$) &    $33.7$ &    ($33$) &    $32.1$ &    ($31$)\tabularnewline
\normalfont   G-IG, $n=1000$ &    $ 99.1$ &    $72.4$ &    ($73$) &    $2.6$ &    ($3$) &    $268.0$ &    ($270$) &    $29.4$ &    ($30$) &    $33.3$ &    ($34$)\tabularnewline
\hline
{\itshape $\op{RLS}\doteq20$}&&&&&&&&&&&\tabularnewline
\normalfont   N-N, $n=100$ &    $127.9$ &    $61.3$ &    ($48$) &    $0.0$ &    ($0$) &    $ 15.0$ &    ($ 12$) &    $15.2$ &    ($12$) &    $17.2$ &    ($13$)\tabularnewline
\normalfont   N-N, $n=200$ &    $131.4$ &    $59.5$ &    ($45$) &    $0.0$ &    ($0$) &    $ 63.4$ &    ($ 48$) &    $14.3$ &    ($11$) &    $18.8$ &    ($14$)\tabularnewline
\normalfont   N-N, $n=1000$ &    $128.5$ &    $59.5$ &    ($46$) &    $0.0$ &    ($0$) &    $173.0$ &    ($135$) &    $14.1$ &    ($11$) &    $18.7$ &    ($15$)\tabularnewline
\normalfont   G-IG, $n=100$ &    $102.9$ &    $68.2$ &    ($66$) &    $2.3$ &    ($2$) &    $ 28.6$ &    ($ 28$) &    $30.1$ &    ($29$) &    $30.1$ &    ($29$)\tabularnewline
\normalfont   G-IG, $n=200$ &    $100.4$ &    $67.7$ &    ($67$) &    $2.5$ &    ($3$) &    $ 94.6$ &    ($ 94$) &    $28.1$ &    ($28$) &    $31.5$ &    ($31$)\tabularnewline
\normalfont   G-IG, $n=1000$ &    $ 93.9$ &    $69.0$ &    ($73$) &    $2.2$ &    ($2$) &    $276.2$ &    ($294$) &    $25.0$ &    ($27$) &    $31.8$ &    ($34$)\tabularnewline
\hline
{\itshape $\op{RLS}\doteq100$}&&&&&&&&&&&\tabularnewline
\normalfont   N-N, $n=100$ &    $122.8$ &    $70.5$ &    ($57$) &    $0.0$ &    ($0$) &    $ 12.6$ &    ($ 10$) &    $14.9$ &    ($12$) &    $23.4$ &    ($19$)\tabularnewline
\normalfont   N-N, $n=200$ &    $123.3$ &    $68.4$ &    ($56$) &    $0.0$ &    ($0$) &    $ 53.6$ &    ($ 43$) &    $15.2$ &    ($12$) &    $23.3$ &    ($19$)\tabularnewline
\normalfont   N-N, $n=1000$ &    $123.4$ &    $70.6$ &    ($57$) &    $0.0$ &    ($0$) &    $174.1$ &    ($141$) &    $14.6$ &    ($12$) &    $24.3$ &    ($20$)\tabularnewline
\normalfont   G-IG, $n=100$ &    $ 99.2$ &    $62.8$ &    ($63$) &    $2.5$ &    ($3$) &    $ 24.1$ &    ($ 24$) &    $25.4$ &    ($26$) &    $29.9$ &    ($30$)\tabularnewline
\normalfont   G-IG, $n=200$ &    $ 93.5$ &    $64.4$ &    ($69$) &    $1.9$ &    ($2$) &    $ 78.8$ &    ($ 84$) &    $23.5$ &    ($25$) &    $30.9$ &    ($33$)\tabularnewline
\normalfont   G-IG, $n=1000$ &    $ 93.2$ &    $63.3$ &    ($68$) &    $2.0$ &    ($2$) &    $314.9$ &    ($338$) &    $21.7$ &    ($23$) &    $32.1$ &    ($34$)\tabularnewline
\hline
\end{tabular}
\begin{tablenotes}
   \item[a] Entries for the posterior regrets have been truncated to the closest first
     digit after the decimal point, and entries for the percentage
     regrets have been truncated to the closest integer. 
     For some entries, percentage regrets are smaller than 1 percentage point. 
\end{tablenotes}
\end{threeparttable}
\end{table}

The results of these simulations for the TCL function are summarised
in table \ref{tab:non-spatial cmse_table} on page \pageref{tab:non-spatial cmse_table}. As in
chapter \ref{chap:mrrr}, we have reported these results in terms of
both absolute and percentage posterior regrets. The percentage
regret expresses the posterior regret of a particular plug-in
estimator as a proportion of the posterior expected TCL under the
optimal estimator, $\bth^{\op{MED}}$.

Overall, the ensemble of posterior means exhibited the best
performance over all simulation scenarios considered. The SSEL plug-in
classifiers was found to be substantially better
than all the other plug-in estimators in terms of posterior
regrets. Moreover, since the set of optimal classifiers under the TCL
function is the ensemble of posterior medians, the performance of the
ensemble of posterior means was found to be no different from the one of the
optimal classifiers under the compound Gaussian model. The performance
of the SSEL plug-in estimators also approached optimality under the
compound Gamma BHM, with its percentage posterior regret not exceeding $3\%$.
Of the remaining four plug-in classifiers, the CB and GR ensembles of
point estimates performed best. The performance of these two ensembles
of point estimates was approximately equivalent over the range of
scenarios studied. However, some differences are nonetheless notable.
The CB plug-in classifiers tended to exhibit a better performance than
the triple-goal classifiers under the compound Gaussian model. This
superiority of the CB plug-in estimator over the GR one was
accentuated by an increase in RLS. This
trend was also true under the compound Gamma model, for which the CB
plug-in classifiers modestly outperformed the triple-goal estimator. 

The WRSEL and MLE plug-in classifiers exhibited the worst
performance. Although the WRSEL plug-in classifier tended to
outperform the MLE-based one for small parameter ensembles, the posterior
regret of the WRSEL rapidly increased with $n$. The poor performance
of the WRSEL plug-in estimator can here be explained using the same
arguments that we have discussed in section \ref{sec:plug-in q-sel}
on page \pageref{sec:plug-in q-sel}. The effect of the vector of
weights in the WRSEL function is indeed highly sensitive to the size
of the parameter ensemble, as was described in equation 
(\ref{eq:phi annoying}) and can be observed from figure \ref{fig:hist wrsel}.
In addition, increasing the size of the parameter ensemble resulted in a
modest reduction of the posterior expected TCL under the optimal
estimator for the compound Gamma model, but not for the compound
Gaussian model. Among the plug-in estimators, only the CB classifier
benefited from an increase in $n$, and this improvement was
restricted to the compound Gamma model. For all other plug-in
estimators, no systematic effect of the size of the parameter ensemble
could be identified. We now turn to the RCL simulation results, which
are characterised by a different ordering of the plug-in classifiers.

\subsection{Plug-in Estimators under RCL}\label{sec:rcl non-spatial}
%
\begin{table}[t]
 \footnotesize
 \caption{
Posterior regrets based on $\op{RCL}(\ga,\bth,\bm\delta)$, with $\ga:=.80$, for five plug-in estimators,
for the conjugate Gaussian model in equation (\ref{eq:normal-normal})
and the conjugate Gamma model in equation (\ref{eq:gamma-inverse gamma}),
and for 3 different levels of RLS,
and 3 different values of $n$, averaged over 100 replicate data sets.
The posterior expected loss of the optimal estimator is given in the
first column. Entries are scaled by a factor of $10^3$. 
The posterior regrets expressed as percentage of the posterior loss
under the optimal estimator is indicated in parentheses.
\label{tab:non-spatial rmse_table}} 
 \centering
 \begin{threeparttable}
 \begin{tabular}{>{\RaggedRight}p{70pt}>{\RaggedLeft}p{30pt}|>{\RaggedLeft}p{25pt}@{}>{\RaggedLeft}p{15pt}>{\RaggedLeft}p{25pt}@{}>{\RaggedLeft}p{15pt}>{\RaggedLeft}p{25pt}@{}>{\RaggedLeft}p{15pt}>{\RaggedLeft}p{25pt}@{}>{\RaggedLeft}p{15pt}>{\RaggedLeft}p{25pt}@{}>{\RaggedLeft}p{15pt}}\hline
\multicolumn{1}{c}{\itshape Scenarios}&
\multicolumn{11}{c}{\itshape Posterior regrets\tnote{a}}
\tabularnewline \cline{3-12}
\multicolumn{1}{>{\RaggedRight}p{70pt}}{}&\multicolumn{1}{c}{RCL}&\multicolumn{2}{c}{MLE}&\multicolumn{2}{c}{SSEL}&\multicolumn{2}{c}{WRSEL}&\multicolumn{2}{c}{CB}&\multicolumn{2}{c}{GR}\tabularnewline
\hline
{\itshape $\op{RLS}\doteq1$}&&&&&&&&&&&\tabularnewline
\normalfont   N-N, $n=100$ &    $173.01$ &    $ 0.09$ &    ($ 0$) &    $0.03$ &    ($0$) &    $ 0.06$ &    ($ 0$) &    $0.03$ &    ($0$) &    $0.05$ &    ($0$)\tabularnewline
\normalfont   N-N, $n=200$ &    $172.53$ &    $ 0.10$ &    ($ 0$) &    $0.04$ &    ($0$) &    $ 0.15$ &    ($ 0$) &    $0.04$ &    ($0$) &    $0.04$ &    ($0$)\tabularnewline
\normalfont   N-N, $n=1000$ &    $172.19$ &    $ 0.08$ &    ($ 0$) &    $0.03$ &    ($0$) &    $ 1.41$ &    ($ 1$) &    $0.03$ &    ($0$) &    $0.04$ &    ($0$)\tabularnewline
\normalfont   G-IG, $n=100$ &    $221.72$ &    $ 0.11$ &    ($ 0$) &    $0.06$ &    ($0$) &    $ 0.11$ &    ($ 0$) &    $0.06$ &    ($0$) &    $0.04$ &    ($0$)\tabularnewline
\normalfont   G-IG, $n=200$ &    $221.60$ &    $ 0.11$ &    ($ 0$) &    $0.05$ &    ($0$) &    $ 0.16$ &    ($ 0$) &    $0.05$ &    ($0$) &    $0.04$ &    ($0$)\tabularnewline
\normalfont   G-IG, $n=1000$ &    $221.21$ &    $ 0.12$ &    ($ 0$) &    $0.06$ &    ($0$) &    $ 0.89$ &    ($ 0$) &    $0.06$ &    ($0$) &    $0.05$ &    ($0$)\tabularnewline
\hline
{\itshape $\op{RLS}\doteq20$}&&&&&&&&&&&\tabularnewline
\normalfont   N-N, $n=100$ &    $169.13$ &    $ 9.64$ &    ($ 6$) &    $0.91$ &    ($1$) &    $ 0.18$ &    ($ 0$) &    $0.91$ &    ($1$) &    $2.10$ &    ($1$)\tabularnewline
\normalfont   N-N, $n=200$ &    $172.33$ &    $ 9.07$ &    ($ 5$) &    $0.76$ &    ($0$) &    $ 0.77$ &    ($ 0$) &    $0.76$ &    ($0$) &    $2.05$ &    ($1$)\tabularnewline
\normalfont   N-N, $n=1000$ &    $169.80$ &    $ 8.91$ &    ($ 5$) &    $0.79$ &    ($0$) &    $ 9.88$ &    ($ 6$) &    $0.79$ &    ($0$) &    $2.04$ &    ($1$)\tabularnewline
\normalfont   G-IG, $n=100$ &    $216.42$ &    $ 2.60$ &    ($ 1$) &    $0.08$ &    ($0$) &    $ 0.15$ &    ($ 0$) &    $0.08$ &    ($0$) &    $0.20$ &    ($0$)\tabularnewline
\normalfont   G-IG, $n=200$ &    $215.70$ &    $ 2.03$ &    ($ 1$) &    $0.06$ &    ($0$) &    $ 0.26$ &    ($ 0$) &    $0.06$ &    ($0$) &    $0.22$ &    ($0$)\tabularnewline
\normalfont   G-IG, $n=1000$ &    $215.53$ &    $ 2.08$ &    ($ 1$) &    $0.05$ &    ($0$) &    $ 2.04$ &    ($ 1$) &    $0.05$ &    ($0$) &    $0.22$ &    ($0$)\tabularnewline
\hline
{\itshape $\op{RLS}\doteq100$}&&&&&&&&&&&\tabularnewline
\normalfont   N-N, $n=100$ &    $169.24$ &    $20.68$ &    ($12$) &    $1.95$ &    ($1$) &    $ 0.39$ &    ($ 0$) &    $1.95$ &    ($1$) &    $4.24$ &    ($3$)\tabularnewline
\normalfont   N-N, $n=200$ &    $167.74$ &    $19.83$ &    ($12$) &    $1.93$ &    ($1$) &    $ 1.77$ &    ($ 1$) &    $1.93$ &    ($1$) &    $3.96$ &    ($2$)\tabularnewline
\normalfont   N-N, $n=1000$ &    $167.94$ &    $19.53$ &    ($12$) &    $2.20$ &    ($1$) &    $16.47$ &    ($10$) &    $2.20$ &    ($1$) &    $4.76$ &    ($3$)\tabularnewline
\normalfont   G-IG, $n=100$ &    $209.71$ &    $ 4.07$ &    ($ 2$) &    $0.08$ &    ($0$) &    $ 0.22$ &    ($ 0$) &    $0.08$ &    ($0$) &    $0.58$ &    ($0$)\tabularnewline
\normalfont   G-IG, $n=200$ &    $212.23$ &    $ 3.39$ &    ($ 2$) &    $0.11$ &    ($0$) &    $ 0.68$ &    ($ 0$) &    $0.11$ &    ($0$) &    $0.71$ &    ($0$)\tabularnewline
\normalfont   G-IG, $n=1000$ &    $210.66$ &    $ 3.08$ &    ($ 1$) &    $0.11$ &    ($0$) &    $ 6.81$ &    ($ 3$) &    $0.11$ &    ($0$) &    $0.69$ &    ($0$)\tabularnewline
\hline
\end{tabular}
\begin{tablenotes}
   \item[a] Entries for the posterior regrets have been truncated to the closest second
     digit after the decimal point, and entries for the percentage
     regrets have been truncated to the closest integer. 
     For some entries, percentage regrets are smaller than 1 percentage point. 
\end{tablenotes}
\end{threeparttable}
\end{table}

Table \ref{tab:non-spatial rmse_table} on page \pageref{tab:non-spatial rmse_table}
documents the performance of the different classifiers under posterior
expected RCL. 
Overall, all estimators performed well with posterior regrets
typically within a few percentage points of the optimal estimator.
The ensemble of posterior means and the CB plug-in classifier
performed slightly better than the other plug-in estimators
considered. In fact, these two families of classifiers exhibited
identical performances on all the simulation scenarios. (This can be 
observed by comparing columns 2 and 8 of table \ref{tab:non-spatial
  rmse_table} on page \pageref{tab:non-spatial rmse_table}). 
This is due to the fact that the CB plug-in
estimator is simply a monotonic function of the posterior means, which
preserves the ranking of this ensemble of point estimates. That is,
from equation (\ref{eq:cb}) in chapter \ref{chap:review}, the CB point
estimates are defined as 
\begin{equation}
    \hat\theta^{\op{CB}}_i :=  \ome \hat\theta^{\op{SSEL}}_i +
    (1-\ome)\bar\theta^{\op{SSEL}}, 
\end{equation}
and therefore it follows that 
\begin{equation}
      R_{i}(\hat\bth^{\op{SSEL}}) 
      =\sum_{j=1}^{n} \cI\lt\lb \hat\theta^{\op{SSEL}}_{i}\geq \hat\theta^{\op{SSEL}}_{j}\rt\rb 
      = \sum_{j=1}^{n} \cI\lt\lb \hat\theta^{\op{CB}}_{i}\geq \hat\theta^{\op{CB}}_{j}\rt\rb       
      = R_{i}(\hat\bth^{\op{CB}}),
\end{equation}
which accounts for the identical posterior regrets of these two
plug-in estimators in table \ref{tab:non-spatial rmse_table}.
The GR plug-in classifiers exhibited the third best percentage regret
performance after the SSEL and CB estimators. This is somewhat
surprising since the ranks of the GR point estimates are explicitly
optimised as part of the triple-goal procedure. However, this
optimisation of the ranks is conditional upon the optimisation of the
EDF of the parameter ensemble, and therefore the GR point estimates are 
only expected to produce quasi-optimal ranks. 
The WRSEL plug-in classifier tended to outperform the MLE
one, although the latter appeared to do better for $\op{RLS}\doteq 1$.

Manipulation of the RLS experimental factor resulted in a substantial
increase of the posterior regrets for all plug-in estimators under both
the compound Gaussian and compound Gamma models. This increase in
posterior regret due to higher RLS was especially severe for the
compound Gaussian model. Under this model, all plug-in estimators performed worse as the
RLS increased. This may be explained in terms of the relationship
between the level of the RLS factor and the ranking of the different
elements in a parameter ensemble. That is, the larger the RLS factor,
the greater was the influence of hierarchical shrinkage on the
ordering of the ensemble's elements, since the $\theta_{i}$'s with
larger sampling variance tended to be more shrunk towards the prior
mean.
It was difficult to pinpoint the exact effect of changes in the size
of the parameter ensemble on the performance of the various plug-in
estimators. As was noted before, the posterior regret associated with
the use of the WRSEL plug-in estimator substantially deteriorated as
$n$ increased (see section \ref{sec:plug-in q-sel}). However, no
systematic trend related to an increase of the size of the parameter ensembles
seemed to emerge for the other plug-in estimators. 

In contrast to threshold-based classification, most of the different
plug-in estimators under the RCL function exhibited performance very
close to the one of the optimal estimator.
In particular, one can observe that all plug-in classifiers produced
a percentage posterior regret of 0 in at least one of the simulation
scenarios. This starkly contrasts with the behaviour of the same ensembles
of point estimates under the TCL function. This set of non-spatial simulations therefore
highlights a pressing need for carefully choosing plug-in estimators when
considering threshold-based classification. However, such a choice
appears to be somewhat less consequential when considering rank-based
classification. 

\section{Spatial Simulations} \label{sec:clas data spatial}
We now turn to the set of spatial simulations, which have been
conducted in order to evaluate the performance of the different
plug-in classifiers under both the TCL and RCL functions for spatially
structured parameter ensembles. In addition, we also consider the 
weighted TCL function, as this can be shown to be equivalent to the
decision rule proposed by \citet{Richardson2004}.
The complete description of these synthetic data sets was provided in
section \ref{sec:mrrr spatial sim}. In this section, we solely highlight
the specific aspects of the simulations, which are directly relevant
to the classification problem at hand. 

\subsection{Parameter Ensembles and Experimental Factors}
We investigated four different spatially structured scenarios: SC1,
characterised by one isolated cluster of elevated risk; SC2, composed
of a mixture of isolated areas and isolated clusters of elevated risk;
SC3, where the spatial structure of the level of risk was generated 
by the Mat\`{e}rn function; and SC4, where the risk surface was indexed
by a hidden covariate. In addition, we varied the heterogeneity of the
risk surface in all of these scenarios by modifying the data generating
process, in order to produce a greater dispersion of the true RRs.
We recovered these simulated levels of risk by using two different
models: BYM and the robustified Laplace version of the BYM model,
which is denoted by L1. These two models were fully described in
section \ref{sec:mrrr spatial sim}. Finally, we also varied the level
of the expected counts in order to asses the
sensitivity of the performance of the classification estimators to
modifications in the amount of information available from the data. 
All these analyses were conducted within the WinBUGS 1.4 software \citep{Lunn2000}.
As for the non-spatial data simulations, we evaluated the different
classifiers using two loss criteria: the posterior expected TCL and
the posterior expected RCL functions. Moreover, we considered the
performance of five different plug-in classifiers based on the MLE, SSEL,
WRSEL, CB and triple-goal ensembles of point estimates. These
different plug-in estimators were compared using the posterior
regrets of the TCL and RCL functions with respect to the vector of
RRs, denoted $\bth$ in the models of interest. We discuss the
performance of the different plug-in classifiers under the (i) unweighted
TCL, (ii) weighted TCL and (iii) RCL functions, in turn. Finally, we
evaluate the consequences of scaling the expected counts on the
performance of the various plug-in estimators of interest. 

All the results presented in this section are expressed as posterior
expected losses, which are functions of the joint posterior
distribution of the parameter ensemble of interest. Naturally, this is
highly dependent on our choice of model. For completeness, we therefore
also compared the optimal classifiers with respect to the \ti{true} values
of the simulated data. For the BYM model, the use of the optimal classifier
(i.e. the vector of posterior medians) yielded $8.1\%$ of 
misclassifications, on average over all simulation scenarios. This
misclassifcation rate was slightly higher for the Laplace model at
$8.6\%$.

\subsection{Plug-in Estimators under Unweighted TCL}\label{sec:clas spatial tcl}
%
\begin{table}[t]
 \footnotesize
 \caption{
Posterior regrets based on $\op{TCL}(C,\bth,\bm\delta)$ with
$C:=\E[\bth] + \op{sd}[\bth]$ for five plug-in estimators
and with the posterior expected loss of the optimal estimator in the first column.
Results are presented for 3 different levels of variability and for 4 spatial scenarios:
an isolated cluster (SC1), a set of isolated clusters and isolated areas (SC2),
highly structured spatial heterogeneity (SC3), and spatial structure generated by a hidden covariate (SC4).
Entries were scaled by a factor of $10^3$, with posterior regrets expressed as percentage of the posterior loss
under the optimal estimator indicated in parentheses.\label{tab:spatial_cmse_table}} 
 \centering
 \begin{threeparttable}
 \begin{tabular}{>{\RaggedRight}p{50pt}>{\RaggedLeft}p{25pt}|>{\RaggedLeft}p{25pt}@{}>{\RaggedLeft}p{25pt}>{\RaggedLeft}p{25pt}@{}>{\RaggedLeft}p{15pt}>{\RaggedLeft}p{25pt}@{}>{\RaggedLeft}p{25pt}>{\RaggedLeft}p{25pt}@{}>{\RaggedLeft}p{15pt}>{\RaggedLeft}p{25pt}@{}>{\RaggedLeft}p{15pt}}\hline
\multicolumn{1}{c}{\itshape Scenarios}&
\multicolumn{1}{c}{}&
\multicolumn{10}{c}{\itshape Posterior regrets\tnote{a}}
\tabularnewline \cline{3-12}
\multicolumn{1}{>{\RaggedRight}p{50pt}}{}&\multicolumn{1}{c}{TCL}&
\multicolumn{2}{c}{MLE}&\multicolumn{2}{c}{SSEL}&\multicolumn{2}{c}{WRSEL}&\multicolumn{2}{c}{CB}&\multicolumn{2}{c}{GR}\tabularnewline
\hline
{\itshape  Low Variab.~}&&&&&&&&&&&\tabularnewline
\normalfont   BYM-SC1 &    $71.6$ &    $110.3$ &    ($154$) &    $0.1$ &    ($0$) &    $117.9$ &    ($165$) &    $12.9$ &    ($18$) &    $20.1$ &    ($28$)\tabularnewline
\normalfont   BYM-SC2 &    $49.8$ &    $ 39.2$ &    ($ 79$) &    $0.1$ &    ($0$) &    $151.3$ &    ($304$) &    $10.6$ &    ($21$) &    $ 4.0$ &    ($ 8$)\tabularnewline
\normalfont   BYM-SC3 &    $67.9$ &    $ 26.8$ &    ($ 39$) &    $0.1$ &    ($0$) &    $156.2$ &    ($230$) &    $ 1.7$ &    ($ 3$) &    $ 2.4$ &    ($ 4$)\tabularnewline
\normalfont   BYM-SC4 &    $51.4$ &    $ 33.1$ &    ($ 64$) &    $0.1$ &    ($0$) &    $191.8$ &    ($373$) &    $ 4.4$ &    ($ 9$) &    $ 4.3$ &    ($ 8$)\tabularnewline
\normalfont   L1-SC1 &    $76.1$ &    $107.4$ &    ($141$) &    $0.1$ &    ($0$) &    $115.7$ &    ($152$) &    $14.4$ &    ($19$) &    $23.1$ &    ($30$)\tabularnewline
\normalfont   L1-SC2 &    $52.0$ &    $ 37.8$ &    ($ 73$) &    $0.1$ &    ($0$) &    $159.9$ &    ($308$) &    $12.7$ &    ($24$) &    $ 5.6$ &    ($11$)\tabularnewline
\normalfont   L1-SC3 &    $78.2$ &    $ 20.7$ &    ($ 26$) &    $0.2$ &    ($0$) &    $159.0$ &    ($203$) &    $ 3.1$ &    ($ 4$) &    $ 3.6$ &    ($ 5$)\tabularnewline
\normalfont   L1-SC4 &    $60.4$ &    $ 25.9$ &    ($ 43$) &    $0.1$ &    ($0$) &    $197.0$ &    ($326$) &    $ 5.5$ &    ($ 9$) &    $ 6.5$ &    ($11$)\tabularnewline
\hline
{\itshape  Med.~ Variab.~}&&&&&&&&&&&\tabularnewline
\normalfont   BYM-SC1 &    $26.6$ &    $ 25.2$ &    ($ 95$) &    $0.0$ &    ($0$) &    $ 78.6$ &    ($296$) &    $ 1.1$ &    ($ 4$) &    $ 7.9$ &    ($30$)\tabularnewline
\normalfont   BYM-SC2 &    $44.1$ &    $ 15.2$ &    ($ 34$) &    $0.1$ &    ($0$) &    $104.9$ &    ($238$) &    $ 4.9$ &    ($11$) &    $ 1.8$ &    ($ 4$)\tabularnewline
\normalfont   BYM-SC3 &    $42.6$ &    $ 10.9$ &    ($ 25$) &    $0.1$ &    ($0$) &    $110.9$ &    ($260$) &    $ 0.8$ &    ($ 2$) &    $ 1.6$ &    ($ 4$)\tabularnewline
\normalfont   BYM-SC4 &    $36.8$ &    $ 12.7$ &    ($ 35$) &    $0.0$ &    ($0$) &    $178.2$ &    ($484$) &    $ 1.7$ &    ($ 5$) &    $ 2.2$ &    ($ 6$)\tabularnewline
\normalfont   L1-SC1 &    $31.3$ &    $ 23.3$ &    ($ 74$) &    $0.0$ &    ($0$) &    $ 83.8$ &    ($267$) &    $ 1.4$ &    ($ 4$) &    $10.5$ &    ($33$)\tabularnewline
\normalfont   L1-SC2 &    $44.1$ &    $ 13.8$ &    ($ 31$) &    $0.1$ &    ($0$) &    $113.2$ &    ($257$) &    $ 5.3$ &    ($12$) &    $ 2.0$ &    ($ 5$)\tabularnewline
\normalfont   L1-SC3 &    $48.3$ &    $  7.9$ &    ($ 16$) &    $0.1$ &    ($0$) &    $119.9$ &    ($248$) &    $ 1.6$ &    ($ 3$) &    $ 2.1$ &    ($ 4$)\tabularnewline
\normalfont   L1-SC4 &    $40.3$ &    $  9.7$ &    ($ 24$) &    $0.1$ &    ($0$) &    $191.3$ &    ($475$) &    $ 2.4$ &    ($ 6$) &    $ 2.6$ &    ($ 6$)\tabularnewline
\hline
{\itshape  High Variab.~}&&&&&&&&&&&\tabularnewline
\normalfont   BYM-SC1 &    $ 2.9$ &    $  1.5$ &    ($ 52$) &    $0.0$ &    ($0$) &    $  0.7$ &    ($ 25$) &    $ 0.0$ &    ($ 0$) &    $ 0.2$ &    ($ 6$)\tabularnewline
\normalfont   BYM-SC2 &    $25.5$ &    $  3.9$ &    ($ 15$) &    $0.0$ &    ($0$) &    $106.1$ &    ($416$) &    $ 1.2$ &    ($ 5$) &    $ 3.7$ &    ($15$)\tabularnewline
\normalfont   BYM-SC3 &    $35.9$ &    $  8.6$ &    ($ 24$) &    $0.1$ &    ($0$) &    $107.8$ &    ($300$) &    $ 0.6$ &    ($ 2$) &    $ 1.3$ &    ($ 4$)\tabularnewline
\normalfont   BYM-SC4 &    $27.3$ &    $  6.0$ &    ($ 22$) &    $0.0$ &    ($0$) &    $ 98.7$ &    ($361$) &    $ 0.9$ &    ($ 3$) &    $ 1.2$ &    ($ 5$)\tabularnewline
\normalfont   L1-SC1 &    $ 4.1$ &    $  1.4$ &    ($ 34$) &    $0.0$ &    ($0$) &    $  1.6$ &    ($ 38$) &    $ 0.1$ &    ($ 1$) &    $ 1.0$ &    ($25$)\tabularnewline
\normalfont   L1-SC2 &    $25.0$ &    $  3.7$ &    ($ 15$) &    $0.0$ &    ($0$) &    $107.7$ &    ($430$) &    $ 1.0$ &    ($ 4$) &    $ 4.1$ &    ($16$)\tabularnewline
\normalfont   L1-SC3 &    $40.8$ &    $  5.1$ &    ($ 12$) &    $0.1$ &    ($0$) &    $116.7$ &    ($286$) &    $ 0.9$ &    ($ 2$) &    $ 1.6$ &    ($ 4$)\tabularnewline
\normalfont   L1-SC4 &    $28.7$ &    $  4.5$ &    ($ 16$) &    $0.0$ &    ($0$) &    $115.1$ &    ($401$) &    $ 1.2$ &    ($ 4$) &    $ 1.6$ &    ($ 6$)\tabularnewline
\hline
\end{tabular}
\begin{tablenotes}
   \item[a] Entries for the posterior regrets have been truncated to the closest first
     digit after the decimal point, and entries for the percentage
     regrets have been truncated to the closest integer. 
     For some entries, percentage regrets are smaller than 1 percentage point. 
\end{tablenotes}
\end{threeparttable}
\end{table}

The results of the spatial simulations under the posterior expected TCL
function are reported in tables \ref{tab:spatial_cmse_table}
on page \pageref{tab:spatial_cmse_table}. As before, the percentage regrets
associated with each plug-in estimator is reported in parentheses.
Overall, the SSEL plug-in classifiers were found to be quasi-optimal
for all simulation scenarios. This may be explained by the fact that
in both models of interest, the full conditional distribution of each $\theta_{i}$
was a symmetric distribution. 

The ordering of the remaining plug-in classifiers in terms of
percentage regrets varies depending on the type of spatial scenarios
considered. Overall, the CB classifier was found to be the second best
plug-in estimator after the ensemble of posterior means. This was true
for all spatial scenarios, except under SC2, where the triple-goal
estimator demonstrated better performance than the CB classifier,
except when considering a high degree of variability in the parameter ensemble.
The MLE-based classifiers behaved poorly throughout the entire set of
simulations, and the WRSEL plug-in classifier was consistently outperformed by its
counterparts, except for the SC1 scenario and under a high level of variability of the
ensemble distribution. 

Increasing the heterogeneity of the parameter ensemble tended to
produce easier classification problems for the plug-in estimators. This was especially apparent
when examining the posterior TCL under the optimal classifier, in
table \ref{tab:spatial_cmse_table}, where one can observe that the
posterior expected losses diminish with increased levels of
heterogeneity, in all four spatial scenarios. That is, most plug-in
estimators appeared to benefit from an increase of the parameter ensemble's dispersion. 
In terms of relative performance as measured by percentage posterior
regret, the areas in the SC1 scenario were found to be easier to
classify for all plug-in classifiers when greater levels of
heterogeneity was considered. Under SC2 and to a much lesser extent
under SC3, only the MLE and CB classifiers appeared to benefit from an
increase in the parameter ensemble's dispersion. For SC4, the CB, MLE and to a
lesser extent the GR plug-in classifiers appeared to benefit from
higher variability in the parameter ensemble. 
However, for the plug-in estimators, which failed to estimate
the overall shape of the ensemble distribution, such as the WRSEL
classifier, the \ti{number} of areas above or below the prescribed
threshold was more likely to be erroneous. Hence, we did not observe any
systematic improvement of the performance of the WRSEL classifier when the
parameter ensemble's dispersion increased. 

As was previously noted, the CAR Normal model was found to yield lower
posterior losses than the CAR Laplace model across all the studied spatial
scenarios. 
In chapter \ref{chap:mrrr}, we have emphasised that the simulation of
discrete two-category distributions in the SC1 and
SC2 scenarios had detrimental consequences on the
quantile estimation procedures. We here make a similar caveat for the
assessment of classification estimators. The fact that the simulated ensemble
distributions in SC1 and SC2 take a discrete form should be
taken into account when evaluating the performance of the
classification procedures under scrutiny in this chapter. In
particular, the classifiers, which are dependent on the entire
ensemble distribution will be more heavily penalised.
This dependency was especially detrimental 
to the GR-based classification because each point estimate in this set
of point estimates is dependent on the full ensemble distribution, as
is visible from table \ref{tab:spatial_cmse_table} for SC1 especially.
However, this detrimental effect on the performance of the GR
classifier was attenuated by an increase in the
variability of the parameter ensemble. 

\subsection{Plug-in Estimators under Weighted TCL}\label{sec:clas weighted tcl}
%
\begin{table}[t]
 \footnotesize
 \caption{
Posterior regrets based on $\op{TCL}_{0.8}(C,\bth,\bm\delta)$ with
$C:=1.0$ for five plug-in estimators
and with the posterior expected loss of the optimal estimator in the first column.
Results are presented for 3 different levels of variability and for 4 spatial scenarios:
an isolated cluster (SC1), a set of isolated clusters and isolated areas (SC2),
highly structured spatial heterogeneity (SC3), and spatial structure generated by a hidden covariate (SC4).
Entries are scaled by a factor of $10^3$, with posterior regrets expressed as percentage of the posterior loss
under the optimal estimator in parentheses.\label{tab:spatial_cmse80_table}} 
 \centering
 \begin{threeparttable}
 \begin{tabular}{>{\RaggedRight}p{50pt}>{\RaggedLeft}p{20pt}|>{\RaggedLeft}p{20pt}@{}>{\RaggedLeft}p{18pt}>{\RaggedLeft}p{20pt}@{}>{\RaggedLeft}p{18pt}>{\RaggedLeft}p{20pt}@{}>{\RaggedLeft}p{18pt}>{\RaggedLeft}p{20pt}@{}>{\RaggedLeft}p{18pt}>{\RaggedLeft}p{20pt}@{}>{\RaggedLeft}p{18pt}}\hline
\multicolumn{1}{c}{\itshape Scenarios}&
\multicolumn{1}{c}{}&
\multicolumn{10}{c}{\itshape Posterior regrets\tnote{a}}
\tabularnewline \cline{3-12}
\multicolumn{1}{>{\RaggedRight}p{50pt}}{}&\multicolumn{1}{c}{$\op{TCL}_{0.8}$}&
\multicolumn{2}{c}{MLE}&\multicolumn{2}{c}{SSEL}&\multicolumn{2}{c}{WRSEL}&\multicolumn{2}{c}{CB}&\multicolumn{2}{c}{GR}\tabularnewline
\hline
{\itshape  Low Variab.~}&&&&&&&&&&&\tabularnewline
\normalfont   BYM-SC1 &    $80$ &    $92$ &    ($114$) &    $58$ &    ($72$) &    $48$ &    ($ 59$) &    $60$ &    ($74$) &    $79$ &    ($ 98$)\tabularnewline
\normalfont   BYM-SC2 &    $52$ &    $54$ &    ($103$) &    $28$ &    ($53$) &    $30$ &    ($ 58$) &    $30$ &    ($58$) &    $55$ &    ($105$)\tabularnewline
\normalfont   BYM-SC3 &    $39$ &    $39$ &    ($100$) &    $20$ &    ($51$) &    $36$ &    ($ 91$) &    $20$ &    ($51$) &    $20$ &    ($ 50$)\tabularnewline
\normalfont   BYM-SC4 &    $47$ &    $55$ &    ($117$) &    $24$ &    ($51$) &    $44$ &    ($ 94$) &    $28$ &    ($59$) &    $37$ &    ($ 79$)\tabularnewline
\normalfont   L1-SC1 &    $84$ &    $85$ &    ($101$) &    $67$ &    ($80$) &    $61$ &    ($ 72$) &    $69$ &    ($82$) &    $85$ &    ($101$)\tabularnewline
\normalfont   L1-SC2 &    $57$ &    $46$ &    ($ 80$) &    $33$ &    ($58$) &    $37$ &    ($ 64$) &    $35$ &    ($62$) &    $57$ &    ($100$)\tabularnewline
\normalfont   L1-SC3 &    $48$ &    $36$ &    ($ 75$) &    $27$ &    ($55$) &    $38$ &    ($ 78$) &    $27$ &    ($55$) &    $27$ &    ($ 56$)\tabularnewline
\normalfont   L1-SC4 &    $53$ &    $44$ &    ($ 83$) &    $29$ &    ($56$) &    $46$ &    ($ 88$) &    $32$ &    ($61$) &    $42$ &    ($ 79$)\tabularnewline
\hline
{\itshape  Med.~ Variab.~}&&&&&&&&&&&\tabularnewline
\normalfont   BYM-SC1 &    $68$ &    $74$ &    ($109$) &    $46$ &    ($68$) &    $24$ &    ($ 35$) &    $47$ &    ($69$) &    $68$ &    ($101$)\tabularnewline
\normalfont   BYM-SC2 &    $25$ &    $24$ &    ($ 99$) &    $ 9$ &    ($37$) &    $ 9$ &    ($ 38$) &    $10$ &    ($40$) &    $38$ &    ($156$)\tabularnewline
\normalfont   BYM-SC3 &    $30$ &    $26$ &    ($ 86$) &    $14$ &    ($48$) &    $41$ &    ($134$) &    $14$ &    ($48$) &    $15$ &    ($ 51$)\tabularnewline
\normalfont   BYM-SC4 &    $34$ &    $35$ &    ($103$) &    $17$ &    ($49$) &    $48$ &    ($139$) &    $18$ &    ($54$) &    $24$ &    ($ 70$)\tabularnewline
\normalfont   L1-SC1 &    $70$ &    $67$ &    ($ 96$) &    $49$ &    ($70$) &    $26$ &    ($ 38$) &    $50$ &    ($72$) &    $71$ &    ($102$)\tabularnewline
\normalfont   L1-SC2 &    $28$ &    $19$ &    ($ 70$) &    $11$ &    ($41$) &    $11$ &    ($ 42$) &    $12$ &    ($45$) &    $41$ &    ($148$)\tabularnewline
\normalfont   L1-SC3 &    $37$ &    $24$ &    ($ 64$) &    $18$ &    ($50$) &    $38$ &    ($103$) &    $18$ &    ($50$) &    $21$ &    ($ 56$)\tabularnewline
\normalfont   L1-SC4 &    $38$ &    $30$ &    ($ 78$) &    $19$ &    ($50$) &    $46$ &    ($121$) &    $21$ &    ($55$) &    $28$ &    ($ 74$)\tabularnewline
\hline
{\itshape High Variab.~}&&&&&&&&&&&\tabularnewline
\normalfont   BYM-SC1 &    $52$ &    $55$ &    ($106$) &    $30$ &    ($58$) &    $37$ &    ($ 71$) &    $30$ &    ($58$) &    $61$ &    ($117$)\tabularnewline
\normalfont   BYM-SC2 &    $ 7$ &    $ 5$ &    ($ 66$) &    $ 1$ &    ($18$) &    $ 6$ &    ($ 82$) &    $ 1$ &    ($19$) &    $17$ &    ($236$)\tabularnewline
\normalfont   BYM-SC3 &    $28$ &    $24$ &    ($ 85$) &    $15$ &    ($51$) &    $45$ &    ($159$) &    $15$ &    ($52$) &    $15$ &    ($ 51$)\tabularnewline
\normalfont   BYM-SC4 &    $26$ &    $23$ &    ($ 90$) &    $12$ &    ($46$) &    $56$ &    ($217$) &    $13$ &    ($49$) &    $17$ &    ($ 66$)\tabularnewline
\normalfont   L1-SC1 &    $52$ &    $53$ &    ($102$) &    $31$ &    ($60$) &    $36$ &    ($ 69$) &    $31$ &    ($60$) &    $62$ &    ($119$)\tabularnewline
\normalfont   L1-SC2 &    $ 9$ &    $ 4$ &    ($ 45$) &    $ 2$ &    ($24$) &    $ 6$ &    ($ 67$) &    $ 2$ &    ($25$) &    $19$ &    ($224$)\tabularnewline
\normalfont   L1-SC3 &    $32$ &    $22$ &    ($ 68$) &    $17$ &    ($52$) &    $42$ &    ($128$) &    $17$ &    ($52$) &    $18$ &    ($ 56$)\tabularnewline
\normalfont   L1-SC4 &    $28$ &    $20$ &    ($ 71$) &    $13$ &    ($46$) &    $51$ &    ($179$) &    $14$ &    ($48$) &    $20$ &    ($ 70$)\tabularnewline
\hline
\end{tabular}
\begin{tablenotes}
   \item[a] Entries for both the posterior regrets and percentage 
            regrets have been truncated to the closest integers. 
\end{tablenotes}
\end{threeparttable}
\end{table}

Here, we consider the performance of various plug-in estimators under
a weighted TCL function. This
decision-theoretic framework reproduces the decision rule proposed by
\citet{Richardson2004} for CAR models in the context of spatial
epidemiology. In the notation adopted by \citet{Richardson2004}, the
$i\tth$ area in an ensemble was classified as having an ``increased
risk'' when the following condition was verified, 
\begin{equation}
       \p[\theta_{i} > C|\by]> p.
\end{equation}
For the BYM and BYM-L1 models, these parameters were given the
following values: $p=.80$, and $C=1.0$. This particular decision rule can easily
be seen to be equivalent to a weighted TCL based on $1-p$ as introduced in equation 
\ref{eq:ptcl}, such that 
\begin{equation}
     \op{TCL}_{0.8}(1.0,\bth,\bm\delta) := \frac{1}{n} \sum_{i=1}^{n}
     0.8\op{FP}(1.0,\theta_{i},\delta_{i}) +
     0.2\op{FN}(1.0,\theta_{i},\delta_{i}),
     \label{eq:tcl.20}
\end{equation}
which implies that \citet{Richardson2004} selected a conservative
decision rule, giving a larger penalty to FPs than to FNs. 
As a result, the number of potential false alarms is deprecated. This
rule indeed requires that a large amount of probability mass
($p=0.80$) is situated above threshold before declaring an area to have
``increased risk'' for a given medical condition of interest. Using proposition
\ref{pro:tcl}, the posterior expectation of the loss function in equation
(\ref{eq:tcl.20}) is minimised by the ensemble of posterior
$.20$-quantiles, denoted $\bth^{\op{TCL}}_{(.20)}$. We have evaluated the different
plug-in estimators of interest under this weighted classification
loss and reported the results of these simulations for the CAR Normal
and CAR Laplace models in table \ref{tab:spatial_cmse80_table} on page 
\pageref{tab:spatial_cmse80_table}.

The ordering of the plug-in estimators in terms of percentage
regrets reproduced the findings reported under the weighted TCL. 
Moreover, the performance of the plug-in classifiers was
found to be approximately consistent across the different spatial scenarios 
considered. Overall, the SSEL estimator exhibited the best performance
throughout this simulation study. Under the SC1 and SC2 scenarios, the
SSEL classifier was more or less on a par with the CB estimator,
especially for medium to high heterogeneity. This can be observed 
by considering the first two rows of each of the three sections of
table \ref{tab:spatial_cmse80_table} on page
\pageref{tab:spatial_cmse80_table}. However, the WRSEL classifier was
found to do marginally better than the SSEL and CB plug-in estimators
under SC2 and for a low level of variability in the ensemble distribution. 
Under the SC3 and SC4 scenarios, the SSEL, CB and GR classifiers
exhibited similar behaviour and outperformed their counterparts for
all levels of heterogeneity. This can be observed by comparing the
different columns in the third and fourth rows of each of the three
sections of table \ref{tab:spatial_cmse80_table}. This set of
simulations therefore showed that while the
set of posterior means constitute a good plug-in estimator under the
unweighted TCL, it also outperformed its counterparts 
under a weighted version of the same loss function, when a greater
penalty is given to false positives. In addition, we also conducted
some further simulations under $\op{TCL}_{0.2}$, which gives a greater penalty to false
negatives (results not shown). These simulations yielded a different ordering of the plug-in
classifiers. In that case, the SSEL estimator were
found to be outperformed by some of its counterparts under several
spatial scenarios. These findings are contrasted and discussed in
section \ref{sec:clas conclusion}.

\subsection{Plug-in Estimators under RCL}\label{sec:clas spatial rcl}
%
\begin{table}[t]
 \footnotesize
 \caption{
Posterior regrets based on $\op{RCL}(\ga,\bth,\bm\delta)$ with $\ga=.80$, for five plug-in estimators
and with the posterior expected loss of the optimal estimator in the first column.
Results are presented for 3 different levels of variability and for 4 spatial scenarios:
an isolated cluster (SC1), a set of isolated clusters and isolated areas (SC2),
highly structured spatial heterogeneity (SC3), and spatial structure generated by a hidden covariate (SC4).
Entries were scaled by a factor of $10^3$,
with posterior regrets expressed as percentage of the posterior loss
under the optimal estimator indicated in parentheses.\label{tab:spatial_rmse_table}} 
 \centering
 \begin{threeparttable}
 \begin{tabular}{>{\RaggedRight}p{50pt}>{\RaggedLeft}p{25pt}|>{\RaggedLeft}p{25pt}@{}>{\RaggedLeft}p{20pt}>{\RaggedLeft}p{25pt}@{}>{\RaggedLeft}p{15pt}>{\RaggedLeft}p{25pt}@{}>{\RaggedLeft}p{20pt}>{\RaggedLeft}p{25pt}@{}>{\RaggedLeft}p{15pt}>{\RaggedLeft}p{25pt}@{}>{\RaggedLeft}p{15pt}}\hline
\multicolumn{1}{c}{\itshape Scenarios}&
\multicolumn{1}{c}{}&
\multicolumn{10}{c}{\itshape Posterior regrets\tnote{a}}
\tabularnewline \cline{3-12}
\multicolumn{1}{>{\RaggedRight}p{50pt}}{}&\multicolumn{1}{c}{RCL}&
\multicolumn{2}{c}{MLE}&\multicolumn{2}{c}{SSEL}&\multicolumn{2}{c}{WRSEL}&\multicolumn{2}{c}{CB}&\multicolumn{2}{c}{GR}\tabularnewline
\hline
{\itshape  Low Variab.~}&&&&&&&&&&&\tabularnewline
\normalfont   BYM-SC1 &    $187.82$ &    $26.70$ &    ($14$) &    $0.36$ &    ($0$) &    $ 3.26$ &    ($ 2$) &    $0.36$ &    ($0$) &    $0.66$ &    ($0$)\tabularnewline
\normalfont   BYM-SC2 &    $106.49$ &    $12.03$ &    ($11$) &    $0.11$ &    ($0$) &    $ 9.61$ &    ($ 9$) &    $0.11$ &    ($0$) &    $0.49$ &    ($0$)\tabularnewline
\normalfont   BYM-SC3 &    $ 86.07$ &    $22.13$ &    ($26$) &    $0.07$ &    ($0$) &    $22.44$ &    ($26$) &    $0.07$ &    ($0$) &    $0.22$ &    ($0$)\tabularnewline
\normalfont   BYM-SC4 &    $105.46$ &    $16.45$ &    ($16$) &    $0.05$ &    ($0$) &    $17.01$ &    ($16$) &    $0.05$ &    ($0$) &    $0.36$ &    ($0$)\tabularnewline
\normalfont   L1-SC1 &    $204.99$ &    $11.55$ &    ($ 6$) &    $0.35$ &    ($0$) &    $ 2.97$ &    ($ 1$) &    $0.35$ &    ($0$) &    $0.67$ &    ($0$)\tabularnewline
\normalfont   L1-SC2 &    $120.84$ &    $ 3.81$ &    ($ 3$) &    $0.12$ &    ($0$) &    $ 9.68$ &    ($ 8$) &    $0.12$ &    ($0$) &    $0.49$ &    ($0$)\tabularnewline
\normalfont   L1-SC3 &    $101.70$ &    $10.31$ &    ($10$) &    $0.11$ &    ($0$) &    $29.85$ &    ($29$) &    $0.11$ &    ($0$) &    $0.42$ &    ($0$)\tabularnewline
\normalfont   L1-SC4 &    $118.72$ &    $ 4.79$ &    ($ 4$) &    $0.13$ &    ($0$) &    $18.33$ &    ($15$) &    $0.13$ &    ($0$) &    $0.62$ &    ($1$)\tabularnewline
\hline
{\itshape  Med.~ Variab.~}&&&&&&&&&&&\tabularnewline
\normalfont   BYM-SC1 &    $163.44$ &    $16.05$ &    ($10$) &    $0.33$ &    ($0$) &    $ 5.57$ &    ($ 3$) &    $0.33$ &    ($0$) &    $0.92$ &    ($1$)\tabularnewline
\normalfont   BYM-SC2 &    $ 35.37$ &    $ 2.36$ &    ($ 7$) &    $0.03$ &    ($0$) &    $ 7.75$ &    ($22$) &    $0.03$ &    ($0$) &    $0.11$ &    ($0$)\tabularnewline
\normalfont   BYM-SC3 &    $ 66.60$ &    $12.45$ &    ($19$) &    $0.03$ &    ($0$) &    $22.54$ &    ($34$) &    $0.03$ &    ($0$) &    $0.19$ &    ($0$)\tabularnewline
\normalfont   BYM-SC4 &    $ 75.85$ &    $ 7.33$ &    ($10$) &    $0.03$ &    ($0$) &    $21.10$ &    ($28$) &    $0.03$ &    ($0$) &    $0.30$ &    ($0$)\tabularnewline
\normalfont   L1-SC1 &    $169.40$ &    $ 5.63$ &    ($ 3$) &    $0.42$ &    ($0$) &    $ 5.88$ &    ($ 3$) &    $0.42$ &    ($0$) &    $1.17$ &    ($1$)\tabularnewline
\normalfont   L1-SC2 &    $ 42.21$ &    $ 0.71$ &    ($ 2$) &    $0.05$ &    ($0$) &    $12.85$ &    ($30$) &    $0.05$ &    ($0$) &    $0.15$ &    ($0$)\tabularnewline
\normalfont   L1-SC3 &    $ 78.77$ &    $ 3.64$ &    ($ 5$) &    $0.06$ &    ($0$) &    $32.29$ &    ($41$) &    $0.06$ &    ($0$) &    $0.37$ &    ($0$)\tabularnewline
\normalfont   L1-SC4 &    $ 83.99$ &    $ 2.25$ &    ($ 3$) &    $0.07$ &    ($0$) &    $22.24$ &    ($26$) &    $0.07$ &    ($0$) &    $0.40$ &    ($0$)\tabularnewline
\hline
{\itshape  High Variab.~}&&&&&&&&&&&\tabularnewline
\normalfont   BYM-SC1 &    $152.71$ &    $10.83$ &    ($ 7$) &    $0.42$ &    ($0$) &    $ 9.59$ &    ($ 6$) &    $0.42$ &    ($0$) &    $1.42$ &    ($1$)\tabularnewline
\normalfont   BYM-SC2 &    $  8.88$ &    $ 0.51$ &    ($ 6$) &    $0.05$ &    ($1$) &    $ 4.00$ &    ($45$) &    $0.05$ &    ($1$) &    $0.11$ &    ($1$)\tabularnewline
\normalfont   BYM-SC3 &    $ 61.47$ &    $ 8.96$ &    ($15$) &    $0.03$ &    ($0$) &    $21.32$ &    ($35$) &    $0.03$ &    ($0$) &    $0.16$ &    ($0$)\tabularnewline
\normalfont   BYM-SC4 &    $ 57.71$ &    $ 3.74$ &    ($ 6$) &    $0.02$ &    ($0$) &    $15.73$ &    ($27$) &    $0.02$ &    ($0$) &    $0.18$ &    ($0$)\tabularnewline
\normalfont   L1-SC1 &    $155.96$ &    $ 4.01$ &    ($ 3$) &    $0.39$ &    ($0$) &    $10.28$ &    ($ 7$) &    $0.39$ &    ($0$) &    $1.70$ &    ($1$)\tabularnewline
\normalfont   L1-SC2 &    $ 10.15$ &    $ 0.33$ &    ($ 3$) &    $0.06$ &    ($1$) &    $ 5.07$ &    ($50$) &    $0.06$ &    ($1$) &    $0.13$ &    ($1$)\tabularnewline
\normalfont   L1-SC3 &    $ 70.20$ &    $ 3.15$ &    ($ 4$) &    $0.08$ &    ($0$) &    $29.69$ &    ($42$) &    $0.08$ &    ($0$) &    $0.39$ &    ($1$)\tabularnewline
\normalfont   L1-SC4 &    $ 62.13$ &    $ 0.91$ &    ($ 1$) &    $0.03$ &    ($0$) &    $17.94$ &    ($29$) &    $0.03$ &    ($0$) &    $0.23$ &    ($0$)\tabularnewline
\hline
\end{tabular}
\begin{tablenotes}
   \item[a] Entries for the posterior regrets have been truncated to
     the closest second digit after the decimal point, and entries for the percentage
     regrets have been truncated to the closest integer. 
     For some entries, percentage regrets are smaller than 1 percentage point. 
\end{tablenotes}
\end{threeparttable}
\end{table}

The results of the spatial simulations under the posterior expected RCL are
summarised in table \ref{tab:spatial_rmse_table} on page
\pageref{tab:spatial_rmse_table}.
Overall, the SSEL and CB classifiers were found to outperform their
counterparts under all the spatially structured scenarios
studied. This corroborates our findings for the non-spatial
simulations in section \ref{sec:rcl non-spatial}. As indicated in that
section, the ranks obtained from these two classifiers are identical,
because of the monotonicity of the CB transformation of the posterior
means. The triple-goal estimator was also found to exhibit good
performance, which closely followed the ones of the SSEL and CB
plug-in classifiers. The MLE and WRSEL demonstrated the worse
performance overall, with the use of the MLE classifier yielding
lower percentage regret under the SC3 and SC4 scenarios. By contrast, the
WRSEL plug-in estimator outperformed the MLE classifier on the SC1 and
SC2 spatial scenarios, albeit as the variability of the ensemble
distribution increased, the MLE became better than the WRSEL estimator
under SC2. This discrepancy in performance between the MLE and WRSEL
classifiers may be explained by the extreme cases considered under the
SC1 and SC2 scenarios, for which the WRSEL plug-in estimator was
generally found to outperform the MLE-based classifier. Indeed, as
discussed in chapter \ref{chap:mrrr}, the discrete nature
of the true ensemble distributions in both SC1 and SC2 required a very
high level of countershrinkage, which was better achieved by the
artificial re-weighting of the WRSEL function than by the MLE
classifiers. Note, however, that for the more standard SC3 and SC4
scenarios, the MLE plug-in estimator was found to provide better ranks
than the WRSEL estimator.
%
\begin{table}[htbp]
 \footnotesize
 \caption{
Posterior regrets based on $\op{TCL}(C,\bth,\bm\delta)$ with
$C:=\E[\bth] + \op{sd}[\bth]$ for five plug-in estimators
and with the posterior expected loss of the optimal estimator in the first column.
Results are presented for 3 different levels of variability and for 4 spatial scenarios:
an isolated cluster (SC1), a set of isolated clusters and isolated areas (SC2),
highly structured spatial heterogeneity (SC3), and spatial structure
generated by a hidden covariate (SC4); as well as two different scaling (SF) of the expected
counts. Entries were scaled by a factor of $10^3$
with posterior regrets expressed as percentage of the posterior loss
under the optimal estimator indicated in parentheses.\label{tab:spatial_cmse_table_SF}} 
 \centering
 \begin{threeparttable}
 \begin{tabular}{>{\RaggedRight}p{50pt}>{\RaggedLeft}p{25pt}|>{\RaggedLeft}p{25pt}@{}>{\RaggedLeft}p{25pt}>{\RaggedLeft}p{25pt}@{}>{\RaggedLeft}p{15pt}>{\RaggedLeft}p{25pt}@{}>{\RaggedLeft}p{25pt}>{\RaggedLeft}p{25pt}@{}>{\RaggedLeft}p{15pt}>{\RaggedLeft}p{25pt}@{}>{\RaggedLeft}p{15pt}}\hline
\multicolumn{1}{c}{\itshape Scenarios}&
\multicolumn{1}{c}{}&
\multicolumn{10}{c}{\itshape Posterior regrets\tnote{a}}
\tabularnewline \cline{3-12}
\multicolumn{1}{>{\RaggedRight}p{50pt}}{}&\multicolumn{1}{c}{TCL}&
\multicolumn{2}{c}{MLE}&\multicolumn{2}{c}{SSEL}&\multicolumn{2}{c}{WRSEL}&\multicolumn{2}{c}{CB}&\multicolumn{2}{c}{GR}\tabularnewline
\hline
\multicolumn{5}{l}{\itshape $\op{SF}=0.1$ Low Variab.~}\tabularnewline
\normalfont   BYM-SC1 &    $128.3$ &    $253.5$ &    ($ 198$) &    $0.3$ &    ($0$) &    $ 92.0$ &    ($ 72$) &    $75.5$ &    ($59$) &    $66.4$ &    ($ 52$)\tabularnewline
\normalfont   BYM-SC2 &    $ 11.0$ &    $236.6$ &    ($2153$) &    $0.0$ &    ($0$) &    $ 17.8$ &    ($162$) &    $ 5.7$ &    ($52$) &    $ 6.5$ &    ($ 59$)\tabularnewline
\normalfont   BYM-SC3 &    $119.4$ &    $113.5$ &    ($  95$) &    $0.4$ &    ($0$) &    $ 84.9$ &    ($ 71$) &    $ 4.5$ &    ($ 4$) &    $ 6.3$ &    ($  5$)\tabularnewline
\normalfont   BYM-SC4 &    $ 53.9$ &    $204.3$ &    ($ 379$) &    $0.1$ &    ($0$) &    $ 88.9$ &    ($165$) &    $23.9$ &    ($44$) &    $17.4$ &    ($ 32$)\tabularnewline
\normalfont   L1-SC1 &    $109.5$ &    $270.0$ &    ($ 247$) &    $0.1$ &    ($0$) &    $ 61.8$ &    ($ 56$) &    $80.3$ &    ($73$) &    $61.7$ &    ($ 56$)\tabularnewline
\normalfont   L1-SC2 &    $  7.2$ &    $241.1$ &    ($3346$) &    $0.0$ &    ($0$) &    $  9.2$ &    ($127$) &    $ 6.7$ &    ($93$) &    $ 6.2$ &    ($ 86$)\tabularnewline
\normalfont   L1-SC3 &    $116.0$ &    $128.7$ &    ($ 111$) &    $1.4$ &    ($1$) &    $ 42.4$ &    ($ 37$) &    $23.3$ &    ($20$) &    $15.1$ &    ($ 13$)\tabularnewline
\normalfont   L1-SC4 &    $ 45.9$ &    $218.9$ &    ($ 477$) &    $0.1$ &    ($0$) &    $ 64.6$ &    ($141$) &    $20.4$ &    ($44$) &    $23.3$ &    ($ 51$)\tabularnewline
\hline
\multicolumn{5}{l}{\itshape $\op{SF}=0.1$ Med.~ Variab.~}\tabularnewline
\normalfont   BYM-SC1 &    $ 74.8$ &    $218.5$ &    ($ 292$) &    $0.5$ &    ($1$) &    $ 80.4$ &    ($108$) &    $24.0$ &    ($32$) &    $23.4$ &    ($ 31$)\tabularnewline
\normalfont   BYM-SC2 &    $ 33.5$ &    $116.1$ &    ($ 346$) &    $0.3$ &    ($1$) &    $140.9$ &    ($420$) &    $13.6$ &    ($41$) &    $10.3$ &    ($ 31$)\tabularnewline
\normalfont   BYM-SC3 &    $ 49.2$ &    $ 64.7$ &    ($ 131$) &    $0.3$ &    ($1$) &    $118.9$ &    ($241$) &    $ 2.4$ &    ($ 5$) &    $ 6.2$ &    ($ 13$)\tabularnewline
\normalfont   BYM-SC4 &    $ 55.4$ &    $129.7$ &    ($ 234$) &    $0.3$ &    ($0$) &    $125.4$ &    ($226$) &    $16.4$ &    ($30$) &    $15.8$ &    ($ 29$)\tabularnewline
\normalfont   L1-SC1 &    $ 48.3$ &    $252.5$ &    ($ 523$) &    $0.2$ &    ($0$) &    $ 48.8$ &    ($101$) &    $25.1$ &    ($52$) &    $23.2$ &    ($ 48$)\tabularnewline
\normalfont   L1-SC2 &    $ 35.1$ &    $116.6$ &    ($ 332$) &    $0.2$ &    ($1$) &    $146.0$ &    ($416$) &    $19.4$ &    ($55$) &    $13.6$ &    ($ 39$)\tabularnewline
\normalfont   L1-SC3 &    $ 67.3$ &    $ 61.2$ &    ($  91$) &    $1.1$ &    ($2$) &    $ 84.5$ &    ($125$) &    $12.5$ &    ($19$) &    $ 9.0$ &    ($ 13$)\tabularnewline
\normalfont   L1-SC4 &    $ 58.7$ &    $130.9$ &    ($ 223$) &    $0.1$ &    ($0$) &    $124.7$ &    ($213$) &    $19.8$ &    ($34$) &    $23.5$ &    ($ 40$)\tabularnewline
\hline
\multicolumn{5}{l}{\itshape $\op{SF}=0.1$ High Variab.~}\tabularnewline
\normalfont   BYM-SC1 &    $ 50.3$ &    $110.6$ &    ($ 220$) &    $0.1$ &    ($0$) &    $101.5$ &    ($202$) &    $ 2.2$ &    ($ 4$) &    $17.6$ &    ($ 35$)\tabularnewline
\normalfont   BYM-SC2 &    $ 30.0$ &    $ 37.0$ &    ($ 124$) &    $0.1$ &    ($0$) &    $163.8$ &    ($547$) &    $ 5.4$ &    ($18$) &    $ 4.1$ &    ($ 14$)\tabularnewline
\normalfont   BYM-SC3 &    $ 65.2$ &    $ 63.9$ &    ($  98$) &    $0.4$ &    ($1$) &    $106.5$ &    ($163$) &    $ 3.5$ &    ($ 5$) &    $ 4.2$ &    ($  6$)\tabularnewline
\normalfont   BYM-SC4 &    $ 51.1$ &    $ 62.9$ &    ($ 123$) &    $0.2$ &    ($0$) &    $188.8$ &    ($370$) &    $ 9.2$ &    ($18$) &    $ 7.4$ &    ($ 14$)\tabularnewline
\normalfont   L1-SC1 &    $ 83.5$ &    $ 82.0$ &    ($  98$) &    $0.2$ &    ($0$) &    $132.4$ &    ($159$) &    $11.4$ &    ($14$) &    $29.7$ &    ($ 36$)\tabularnewline
\normalfont   L1-SC2 &    $ 29.7$ &    $ 36.9$ &    ($ 124$) &    $0.2$ &    ($1$) &    $166.1$ &    ($560$) &    $ 6.6$ &    ($22$) &    $ 5.2$ &    ($ 18$)\tabularnewline
\normalfont   L1-SC3 &    $ 83.2$ &    $ 61.2$ &    ($  74$) &    $0.6$ &    ($1$) &    $ 92.2$ &    ($111$) &    $ 9.8$ &    ($12$) &    $ 9.9$ &    ($ 12$)\tabularnewline
\normalfont   L1-SC4 &    $ 62.1$ &    $ 57.2$ &    ($  92$) &    $0.3$ &    ($1$) &    $161.7$ &    ($260$) &    $12.8$ &    ($21$) &    $15.1$ &    ($ 24$)\tabularnewline
\hline
\multicolumn{5}{l}{\itshape $\op{SF}=2$ Low Variab.~}\tabularnewline
\normalfont   BYM-SC1 &    $ 36.4$ &    $ 66.1$ &    ($ 182$) &    $0.0$ &    ($0$) &    $ 78.3$ &    ($215$) &    $ 2.0$ &    ($ 6$) &    $16.0$ &    ($ 44$)\tabularnewline
\normalfont   BYM-SC2 &    $ 54.3$ &    $ 30.0$ &    ($  55$) &    $0.0$ &    ($0$) &    $130.9$ &    ($241$) &    $ 9.6$ &    ($18$) &    $ 3.9$ &    ($  7$)\tabularnewline
\normalfont   BYM-SC3 &    $ 49.6$ &    $ 13.3$ &    ($  27$) &    $0.1$ &    ($0$) &    $151.8$ &    ($306$) &    $ 0.7$ &    ($ 1$) &    $ 2.8$ &    ($  6$)\tabularnewline
\normalfont   BYM-SC4 &    $ 45.4$ &    $ 12.2$ &    ($  27$) &    $0.0$ &    ($0$) &    $192.5$ &    ($424$) &    $ 1.1$ &    ($ 2$) &    $ 1.7$ &    ($  4$)\tabularnewline
\normalfont   L1-SC1 &    $ 46.5$ &    $ 59.4$ &    ($ 128$) &    $0.0$ &    ($0$) &    $ 86.1$ &    ($185$) &    $ 1.8$ &    ($ 4$) &    $19.8$ &    ($ 43$)\tabularnewline
\normalfont   L1-SC2 &    $ 56.0$ &    $ 26.4$ &    ($  47$) &    $0.0$ &    ($0$) &    $128.1$ &    ($229$) &    $ 9.4$ &    ($17$) &    $ 2.4$ &    ($  4$)\tabularnewline
\normalfont   L1-SC3 &    $ 55.7$ &    $  9.2$ &    ($  16$) &    $0.0$ &    ($0$) &    $165.1$ &    ($296$) &    $ 2.0$ &    ($ 4$) &    $ 3.2$ &    ($  6$)\tabularnewline
\normalfont   L1-SC4 &    $ 50.0$ &    $  9.5$ &    ($  19$) &    $0.0$ &    ($0$) &    $201.2$ &    ($402$) &    $ 1.5$ &    ($ 3$) &    $ 2.4$ &    ($  5$)\tabularnewline
\hline
\multicolumn{5}{l}{\itshape $\op{SF}=2$ Med.~ Variab.~}\tabularnewline
\normalfont   BYM-SC1 &    $  4.7$ &    $  8.8$ &    ($ 186$) &    $0.0$ &    ($0$) &    $  2.5$ &    ($ 53$) &    $ 0.0$ &    ($ 0$) &    $ 5.0$ &    ($105$)\tabularnewline
\normalfont   BYM-SC2 &    $ 48.1$ &    $ 11.8$ &    ($  25$) &    $0.1$ &    ($0$) &    $ 99.8$ &    ($208$) &    $ 3.0$ &    ($ 6$) &    $ 2.1$ &    ($  4$)\tabularnewline
\normalfont   BYM-SC3 &    $ 30.6$ &    $  6.3$ &    ($  21$) &    $0.0$ &    ($0$) &    $ 91.1$ &    ($298$) &    $ 0.2$ &    ($ 1$) &    $ 0.3$ &    ($  1$)\tabularnewline
\normalfont   BYM-SC4 &    $ 28.0$ &    $  7.4$ &    ($  26$) &    $0.0$ &    ($0$) &    $109.4$ &    ($391$) &    $ 0.4$ &    ($ 2$) &    $ 1.7$ &    ($  6$)\tabularnewline
\normalfont   L1-SC1 &    $  9.3$ &    $  6.7$ &    ($  72$) &    $0.0$ &    ($0$) &    $ 25.0$ &    ($269$) &    $ 0.2$ &    ($ 3$) &    $ 4.4$ &    ($ 48$)\tabularnewline
\normalfont   L1-SC2 &    $ 49.4$ &    $ 10.0$ &    ($  20$) &    $0.0$ &    ($0$) &    $103.3$ &    ($209$) &    $ 4.2$ &    ($ 8$) &    $ 2.2$ &    ($  5$)\tabularnewline
\normalfont   L1-SC3 &    $ 34.0$ &    $  2.8$ &    ($   8$) &    $0.0$ &    ($0$) &    $ 96.2$ &    ($283$) &    $ 0.5$ &    ($ 2$) &    $ 0.8$ &    ($  2$)\tabularnewline
\normalfont   L1-SC4 &    $ 29.8$ &    $  5.4$ &    ($  18$) &    $0.0$ &    ($0$) &    $124.2$ &    ($417$) &    $ 0.6$ &    ($ 2$) &    $ 1.4$ &    ($  5$)\tabularnewline
\hline
\multicolumn{5}{l}{\itshape $\op{SF}=2$ High Variab.~}\tabularnewline
\normalfont   BYM-SC1 &    $  0.1$ &    $  0.0$ &    ($   0$) &    $0.0$ &    ($0$) &    $  0.0$ &    ($  0$) &    $ 0.0$ &    ($ 0$) &    $ 0.0$ &    ($  0$)\tabularnewline
\normalfont   BYM-SC2 &    $ 21.2$ &    $  1.1$ &    ($   5$) &    $0.0$ &    ($0$) &    $ 95.8$ &    ($452$) &    $ 0.4$ &    ($ 2$) &    $ 4.4$ &    ($ 21$)\tabularnewline
\normalfont   BYM-SC3 &    $ 15.3$ &    $  4.1$ &    ($  27$) &    $0.0$ &    ($0$) &    $ 16.5$ &    ($108$) &    $ 0.2$ &    ($ 1$) &    $ 3.5$ &    ($ 23$)\tabularnewline
\normalfont   BYM-SC4 &    $ 23.2$ &    $  2.7$ &    ($  12$) &    $0.0$ &    ($0$) &    $ 58.1$ &    ($251$) &    $ 0.4$ &    ($ 2$) &    $ 0.8$ &    ($  4$)\tabularnewline
\normalfont   L1-SC1 &    $  0.3$ &    $  0.0$ &    ($   0$) &    $0.0$ &    ($0$) &    $  0.0$ &    ($  0$) &    $ 0.0$ &    ($ 0$) &    $ 0.0$ &    ($  0$)\tabularnewline
\normalfont   L1-SC2 &    $ 21.0$ &    $  1.6$ &    ($   7$) &    $0.0$ &    ($0$) &    $100.2$ &    ($478$) &    $ 0.4$ &    ($ 2$) &    $ 4.7$ &    ($ 22$)\tabularnewline
\normalfont   L1-SC3 &    $ 17.8$ &    $  1.8$ &    ($  10$) &    $0.0$ &    ($0$) &    $ 21.3$ &    ($119$) &    $ 0.3$ &    ($ 1$) &    $ 3.6$ &    ($ 20$)\tabularnewline
\normalfont   L1-SC4 &    $ 23.8$ &    $  2.1$ &    ($   9$) &    $0.0$ &    ($0$) &    $ 66.6$ &    ($279$) &    $ 0.7$ &    ($ 3$) &    $ 1.0$ &    ($  4$)\tabularnewline
\hline
\end{tabular}
\begin{tablenotes}
   \item[a] Entries for the posterior regrets have been truncated to the closest first
     digit after the decimal point, and entries for the percentage
     regrets have been truncated to the closest integer. 
     For some entries, percentage regrets are smaller than 1 percentage point. 
\end{tablenotes}
\end{threeparttable}
\end{table}

%
\begin{table}[htbp]
 \footnotesize
 \caption{
Posterior regrets based on $\op{RCL}(\ga,\bth,\bm\delta)$ with
$\ga=.80$, for five plug-in estimators
and with the posterior expected loss of the optimal estimator in the first column.
Results are presented for 3 different levels of variability and for 4
spatial scenarios: an isolated cluster (SC1), a set of isolated
clusters and isolated areas (SC2),
highly structured spatial heterogeneity (SC3), and spatial structure
generated by a hidden covariate (SC4); as well as two different
scaling (SF) of the expected counts. Entries were scaled by a factor of $10^3$
with posterior regrets expressed as percentage of the posterior loss
under the optimal estimator indicated in parentheses.\label{tab:spatial_rmse_table_SF}} 
 \centering
 \begin{threeparttable}
 \begin{tabular}{>{\RaggedRight}p{50pt}>{\RaggedLeft}p{25pt}|>{\RaggedLeft}p{25pt}@{}>{\RaggedLeft}p{20pt}>{\RaggedLeft}p{25pt}@{}>{\RaggedLeft}p{15pt}>{\RaggedLeft}p{25pt}@{}>{\RaggedLeft}p{20pt}>{\RaggedLeft}p{25pt}@{}>{\RaggedLeft}p{15pt}>{\RaggedLeft}p{25pt}@{}>{\RaggedLeft}p{15pt}}\hline
\multicolumn{1}{c}{\itshape Scenarios}&
\multicolumn{1}{c}{}&
\multicolumn{10}{c}{\itshape Posterior regrets\tnote{a}}
\tabularnewline \cline{3-12}
\multicolumn{1}{>{\RaggedRight}p{50pt}}{}&\multicolumn{1}{c}{RCL}&
\multicolumn{2}{c}{MLE}&\multicolumn{2}{c}{SSEL}&\multicolumn{2}{c}{WRSEL}&\multicolumn{2}{c}{CB}&\multicolumn{2}{c}{GR}\tabularnewline
\hline
\multicolumn{5}{l}{\itshape $\op{SF}=0.1$ Low Variab.~}\tabularnewline
\normalfont   BYM-SC1 &    $277.38$ &    $15.82$ &    ($ 6$) &    $1.71$ &    ($1$) &    $ 4.32$ &    ($ 2$) &    $1.71$ &    ($1$) &    $1.77$ &    ($1$)\tabularnewline
\normalfont   BYM-SC2 &    $280.36$ &    $10.29$ &    ($ 4$) &    $1.21$ &    ($0$) &    $ 3.71$ &    ($ 1$) &    $1.21$ &    ($0$) &    $1.43$ &    ($1$)\tabularnewline
\normalfont   BYM-SC3 &    $176.56$ &    $64.78$ &    ($37$) &    $0.30$ &    ($0$) &    $23.14$ &    ($13$) &    $0.30$ &    ($0$) &    $0.83$ &    ($0$)\tabularnewline
\normalfont   BYM-SC4 &    $213.14$ &    $46.40$ &    ($22$) &    $0.50$ &    ($0$) &    $ 4.11$ &    ($ 2$) &    $0.50$ &    ($0$) &    $0.72$ &    ($0$)\tabularnewline
\normalfont   L1-SC1 &    $288.44$ &    $ 7.11$ &    ($ 2$) &    $0.84$ &    ($0$) &    $ 3.79$ &    ($ 1$) &    $0.84$ &    ($0$) &    $0.92$ &    ($0$)\tabularnewline
\normalfont   L1-SC2 &    $293.29$ &    $ 2.99$ &    ($ 1$) &    $0.84$ &    ($0$) &    $ 2.87$ &    ($ 1$) &    $0.84$ &    ($0$) &    $0.83$ &    ($0$)\tabularnewline
\normalfont   L1-SC3 &    $248.41$ &    $13.79$ &    ($ 6$) &    $0.34$ &    ($0$) &    $ 5.96$ &    ($ 2$) &    $0.34$ &    ($0$) &    $0.40$ &    ($0$)\tabularnewline
\normalfont   L1-SC4 &    $266.64$ &    $ 7.98$ &    ($ 3$) &    $0.44$ &    ($0$) &    $ 2.27$ &    ($ 1$) &    $0.44$ &    ($0$) &    $0.54$ &    ($0$)\tabularnewline
\hline
\multicolumn{5}{l}{\itshape $\op{SF}=0.1$ Med.~ Variab.~}\tabularnewline
\normalfont   BYM-SC1 &    $213.78$ &    $48.11$ &    ($23$) &    $0.83$ &    ($0$) &    $ 4.26$ &    ($ 2$) &    $0.83$ &    ($0$) &    $1.05$ &    ($0$)\tabularnewline
\normalfont   BYM-SC2 &    $196.60$ &    $18.45$ &    ($ 9$) &    $0.24$ &    ($0$) &    $ 5.29$ &    ($ 3$) &    $0.24$ &    ($0$) &    $0.55$ &    ($0$)\tabularnewline
\normalfont   BYM-SC3 &    $122.31$ &    $60.99$ &    ($50$) &    $0.13$ &    ($0$) &    $22.57$ &    ($18$) &    $0.13$ &    ($0$) &    $0.30$ &    ($0$)\tabularnewline
\normalfont   BYM-SC4 &    $191.33$ &    $34.96$ &    ($18$) &    $0.17$ &    ($0$) &    $ 4.76$ &    ($ 2$) &    $0.17$ &    ($0$) &    $0.32$ &    ($0$)\tabularnewline
\normalfont   L1-SC1 &    $280.59$ &    $ 4.85$ &    ($ 2$) &    $0.79$ &    ($0$) &    $ 3.17$ &    ($ 1$) &    $0.79$ &    ($0$) &    $0.75$ &    ($0$)\tabularnewline
\normalfont   L1-SC2 &    $218.70$ &    $ 6.35$ &    ($ 3$) &    $0.24$ &    ($0$) &    $ 3.53$ &    ($ 2$) &    $0.24$ &    ($0$) &    $0.45$ &    ($0$)\tabularnewline
\normalfont   L1-SC3 &    $208.18$ &    $ 6.12$ &    ($ 3$) &    $0.23$ &    ($0$) &    $ 3.23$ &    ($ 2$) &    $0.23$ &    ($0$) &    $0.31$ &    ($0$)\tabularnewline
\normalfont   L1-SC4 &    $230.52$ &    $ 6.08$ &    ($ 3$) &    $0.48$ &    ($0$) &    $ 2.55$ &    ($ 1$) &    $0.48$ &    ($0$) &    $0.61$ &    ($0$)\tabularnewline
\hline
\multicolumn{5}{l}{\itshape $\op{SF}=0.1$ High Variab.~}\tabularnewline
\normalfont   BYM-SC1 &    $182.54$ &    $44.78$ &    ($25$) &    $0.47$ &    ($0$) &    $ 4.02$ &    ($ 2$) &    $0.47$ &    ($0$) &    $1.04$ &    ($1$)\tabularnewline
\normalfont   BYM-SC2 &    $117.85$ &    $15.45$ &    ($13$) &    $0.07$ &    ($0$) &    $12.03$ &    ($10$) &    $0.07$ &    ($0$) &    $0.60$ &    ($1$)\tabularnewline
\normalfont   BYM-SC3 &    $112.30$ &    $49.51$ &    ($44$) &    $0.14$ &    ($0$) &    $15.31$ &    ($14$) &    $0.14$ &    ($0$) &    $0.42$ &    ($0$)\tabularnewline
\normalfont   BYM-SC4 &    $143.06$ &    $34.08$ &    ($24$) &    $0.09$ &    ($0$) &    $12.02$ &    ($ 8$) &    $0.09$ &    ($0$) &    $0.40$ &    ($0$)\tabularnewline
\normalfont   L1-SC1 &    $207.60$ &    $ 6.85$ &    ($ 3$) &    $0.19$ &    ($0$) &    $ 2.61$ &    ($ 1$) &    $0.19$ &    ($0$) &    $0.61$ &    ($0$)\tabularnewline
\normalfont   L1-SC2 &    $135.69$ &    $ 4.25$ &    ($ 3$) &    $0.05$ &    ($0$) &    $ 7.35$ &    ($ 5$) &    $0.05$ &    ($0$) &    $0.34$ &    ($0$)\tabularnewline
\normalfont   L1-SC3 &    $155.10$ &    $23.98$ &    ($15$) &    $0.19$ &    ($0$) &    $ 5.43$ &    ($ 4$) &    $0.19$ &    ($0$) &    $0.75$ &    ($0$)\tabularnewline
\normalfont   L1-SC4 &    $179.98$ &    $ 1.97$ &    ($ 1$) &    $0.10$ &    ($0$) &    $ 3.94$ &    ($ 2$) &    $0.10$ &    ($0$) &    $0.36$ &    ($0$)\tabularnewline
\hline
\multicolumn{5}{l}{\itshape $\op{SF}=2$ Low Variab.~}\tabularnewline
\normalfont   BYM-SC1 &    $176.52$ &    $30.11$ &    ($17$) &    $0.36$ &    ($0$) &    $ 4.98$ &    ($ 3$) &    $0.36$ &    ($0$) &    $1.25$ &    ($1$)\tabularnewline
\normalfont   BYM-SC2 &    $ 66.23$ &    $ 7.20$ &    ($11$) &    $0.01$ &    ($0$) &    $17.47$ &    ($26$) &    $0.01$ &    ($0$) &    $0.12$ &    ($0$)\tabularnewline
\normalfont   BYM-SC3 &    $ 82.39$ &    $13.60$ &    ($17$) &    $0.05$ &    ($0$) &    $21.08$ &    ($26$) &    $0.05$ &    ($0$) &    $0.38$ &    ($0$)\tabularnewline
\normalfont   BYM-SC4 &    $ 79.16$ &    $ 5.55$ &    ($ 7$) &    $0.08$ &    ($0$) &    $21.89$ &    ($28$) &    $0.08$ &    ($0$) &    $0.21$ &    ($0$)\tabularnewline
\normalfont   L1-SC1 &    $186.73$ &    $16.80$ &    ($ 9$) &    $0.76$ &    ($0$) &    $ 4.61$ &    ($ 2$) &    $0.76$ &    ($0$) &    $1.14$ &    ($1$)\tabularnewline
\normalfont   L1-SC2 &    $ 73.88$ &    $ 2.39$ &    ($ 3$) &    $0.03$ &    ($0$) &    $21.69$ &    ($29$) &    $0.03$ &    ($0$) &    $0.33$ &    ($0$)\tabularnewline
\normalfont   L1-SC3 &    $ 90.87$ &    $ 5.99$ &    ($ 7$) &    $0.12$ &    ($0$) &    $36.52$ &    ($40$) &    $0.12$ &    ($0$) &    $0.57$ &    ($1$)\tabularnewline
\normalfont   L1-SC4 &    $ 85.05$ &    $ 1.00$ &    ($ 1$) &    $0.04$ &    ($0$) &    $23.86$ &    ($28$) &    $0.04$ &    ($0$) &    $0.31$ &    ($0$)\tabularnewline
\hline
\multicolumn{5}{l}{\itshape $\op{SF}=2$ Med.~ Variab.~}\tabularnewline
\normalfont   BYM-SC1 &    $168.41$ &    $19.20$ &    ($11$) &    $0.51$ &    ($0$) &    $ 7.57$ &    ($ 4$) &    $0.51$ &    ($0$) &    $1.17$ &    ($1$)\tabularnewline
\normalfont   BYM-SC2 &    $ 26.11$ &    $ 2.37$ &    ($ 9$) &    $0.31$ &    ($1$) &    $14.22$ &    ($54$) &    $0.31$ &    ($1$) &    $0.40$ &    ($2$)\tabularnewline
\normalfont   BYM-SC3 &    $ 46.24$ &    $ 8.51$ &    ($18$) &    $0.01$ &    ($0$) &    $19.38$ &    ($42$) &    $0.01$ &    ($0$) &    $0.01$ &    ($0$)\tabularnewline
\normalfont   BYM-SC4 &    $ 54.57$ &    $ 2.33$ &    ($ 4$) &    $0.00$ &    ($0$) &    $14.51$ &    ($27$) &    $0.00$ &    ($0$) &    $0.17$ &    ($0$)\tabularnewline
\normalfont   L1-SC1 &    $172.85$ &    $ 7.70$ &    ($ 4$) &    $0.45$ &    ($0$) &    $ 9.23$ &    ($ 5$) &    $0.45$ &    ($0$) &    $1.67$ &    ($1$)\tabularnewline
\normalfont   L1-SC2 &    $ 25.80$ &    $ 2.41$ &    ($ 9$) &    $0.08$ &    ($0$) &    $15.72$ &    ($61$) &    $0.08$ &    ($0$) &    $0.36$ &    ($1$)\tabularnewline
\normalfont   L1-SC3 &    $ 52.17$ &    $ 3.37$ &    ($ 6$) &    $0.04$ &    ($0$) &    $23.93$ &    ($46$) &    $0.04$ &    ($0$) &    $0.11$ &    ($0$)\tabularnewline
\normalfont   L1-SC4 &    $ 58.02$ &    $ 0.77$ &    ($ 1$) &    $0.03$ &    ($0$) &    $23.41$ &    ($40$) &    $0.03$ &    ($0$) &    $0.31$ &    ($1$)\tabularnewline
\hline
\multicolumn{5}{l}{\itshape $\op{SF}=2$ High Variab.~}\tabularnewline
\normalfont   BYM-SC1 &    $163.70$ &    $ 9.72$ &    ($ 6$) &    $1.17$ &    ($1$) &    $ 9.28$ &    ($ 6$) &    $1.17$ &    ($1$) &    $1.91$ &    ($1$)\tabularnewline
\normalfont   BYM-SC2 &    $ 19.77$ &    $ 0.50$ &    ($ 3$) &    $0.06$ &    ($0$) &    $15.13$ &    ($77$) &    $0.06$ &    ($0$) &    $0.38$ &    ($2$)\tabularnewline
\normalfont   BYM-SC3 &    $ 41.49$ &    $ 4.46$ &    ($11$) &    $0.05$ &    ($0$) &    $15.40$ &    ($37$) &    $0.05$ &    ($0$) &    $0.07$ &    ($0$)\tabularnewline
\normalfont   BYM-SC4 &    $ 45.21$ &    $ 2.15$ &    ($ 5$) &    $0.00$ &    ($0$) &    $15.45$ &    ($34$) &    $0.00$ &    ($0$) &    $0.00$ &    ($0$)\tabularnewline
\normalfont   L1-SC1 &    $167.87$ &    $ 5.38$ &    ($ 3$) &    $0.37$ &    ($0$) &    $ 8.27$ &    ($ 5$) &    $0.37$ &    ($0$) &    $1.92$ &    ($1$)\tabularnewline
\normalfont   L1-SC2 &    $ 19.99$ &    $ 0.50$ &    ($ 2$) &    $0.18$ &    ($1$) &    $15.29$ &    ($76$) &    $0.18$ &    ($1$) &    $0.35$ &    ($2$)\tabularnewline
\normalfont   L1-SC3 &    $ 47.61$ &    $ 1.58$ &    ($ 3$) &    $0.07$ &    ($0$) &    $22.15$ &    ($47$) &    $0.07$ &    ($0$) &    $0.38$ &    ($1$)\tabularnewline
\normalfont   L1-SC4 &    $ 46.52$ &    $ 0.45$ &    ($ 1$) &    $0.05$ &    ($0$) &    $17.94$ &    ($39$) &    $0.05$ &    ($0$) &    $0.21$ &    ($0$)\tabularnewline
\hline
\end{tabular}
\begin{tablenotes}
   \item[a] Entries for the posterior regrets have been truncated to the closest second
     digit after the decimal point, and entries for the percentage
     regrets have been truncated to the closest integer. 
     For some entries, percentage regrets are smaller than 1 percentage point. 
\end{tablenotes}
\end{threeparttable}
\end{table}

Increasing the amount of heterogeneity present in the
ensemble distribution tended to systematically diminish the posterior
expected loss under the optimal estimator. As noted in section 
\ref{sec:clas spatial tcl} when considering the TCL function, greater
variability in the ensemble distribution tended to attenuate the effect of 
hierarchical shrinkage on rank estimation for all plug-in estimators. 
The performance of the SSEL, CB and GR plug-in classifiers was
approximately stable
under different levels of heterogeneity. The MLE-based estimator, however,
substantially benefited from increasing the dispersion of the
parameter ensemble. For the WRSEL classifier, the effect of 
increasing the parameter ensemble's heterogeneity was more difficult
to evaluate, and tended to vary with the spatial scenario considered. 

As for all other simulations, the posterior expected RCL when using
the optimal estimator was lower under the CAR Normal model than under
the CAR Laplace for all scenarios considered. 
In comparison to the TCL function, one can also observe that, as
discussed in section \ref{sec:rcl non-spatial}, 
the overall size of the posterior and percentage
regrets under the RCL function was found to be substantially lower than under the
TCL function. This indicated that the choice of a particular plug-in
classifier is more consequential under the TCL decision-theoretic
framework than under the RCL function. 

\subsection{Consequences of Scaling the Expected Counts}\label{sec:clas sf}
Tables \ref{tab:spatial_cmse_table_SF} and \ref{tab:spatial_rmse_table_SF} on pages 
\pageref{tab:spatial_cmse_table_SF} and \pageref{tab:spatial_rmse_table_SF} document the posterior and
percentage regrets of the different plug-in estimators of interest
under the TCL and RCL functions, respectively, for two different scaling of the
expected counts with $\op{SF}\in\lb 0.1, 2.0\rb$.
Overall, the posterior expected losses under both the TCL and RCL
functions were found to benefit from an increase of the level of the
expected counts. The percentage regrets of the different plug-in
classifiers, however, did not necessarily diminish with an increase in
SF. 

Under the TCL function, our results showed that the SSEL classifier tended to do
better for higher levels of expected counts. Although the differences
in posterior and percentage regrets tended to be small since the SSEL
plug-in estimator is close to optimal under the TCL function, a
systematic trend is nonetheless notable, whereby the SSEL classifier
showed consistent improvement for higher expected counts. Similarly,
the MLE and CB plug-in classifiers benefited from an increase in SF. Although
this tended to be also true for the triple-goal classifier, its
performance in terms of percentage regrets was sometimes worse with
larger levels of expected counts. This can be observed by considering
the percentage regrets in parentheses in the eleventh column of table 
\ref{tab:spatial_cmse_table_SF} on page
\pageref{tab:spatial_cmse_table_SF}, and comparing the first part of
the table for which $\op{SF}=0.1$ with the second part of the table for
which $\op{SF}=2.0$. A more confusing picture emerged for
the WRSEL classifier. In this latter case, the type of spatial
scenario considered and the level of heterogeneity of the simulated ensembles
both played a substantial role in determining whether the WRSEL
plug-in estimator benefited from an increase in SF. The ordering of
the different plug-in estimators, however, was left unchanged by the
use of different levels of expected counts. 

A similar trend can be observed for the RCL function in table 
\ref{tab:spatial_rmse_table_SF} on page
\pageref{tab:spatial_rmse_table_SF}. Here, changes in SF did not, in
general, modify the ordering of the plug-in estimators reported in 
section \ref{sec:clas spatial rcl}. The WRSEL classifier, however, was
found to be detrimentally affected by an increase of the
level of the expected counts under most spatial scenarios
considered. This appeared to be true both in terms of posterior and
percentage regrets. For the remaining plug-in estimators, the
differences in percentage regrets associated an increase in SF was too small to
allow the detection of any systematic trend, except perhaps for the MLE classifier, whose
performance notably improved when specifying larger expected counts,
albeit these improvements were restricted to the medium and high
heterogeneity simulations. 
The different classification procedures considered in this chapter
have also been applied to a real data example, which we describe in
the next section. 

\section{MRSA Prevalence in UK NHS Trusts}\label{sec:mrsa}
\begin{table}[t]
 \footnotesize
 \caption{Number of hospitals above threshold for the MRSA data set
   with three different choices of threshold.
   The number of hospitals above threshold using the optimal
   classifier for the TCL function is reported in the first column. 
   Classifications using plug-in estimators are reported as departures
   from the number of hospitals classified above threshold using the
   vector of posterior medians. For each plug-in estimator, the percentage
   departure has been indicated in parentheses.
   \label{tab:mrsa_departure}} 
 \centering
\begin{threeparttable}
\begin{tabular}{>{\RaggedRight}p{60pt}c|
>{\RaggedLeft}p{20pt}@{}>{\RaggedLeft}p{18pt}
>{\RaggedLeft}p{20pt}@{}>{\RaggedLeft}p{18pt}
>{\RaggedLeft}p{20pt}@{}>{\RaggedLeft}p{20pt}
>{\RaggedLeft}p{20pt}@{}>{\RaggedLeft}p{18pt}
>{\RaggedLeft}p{20pt}@{}>{\RaggedLeft}p{18pt}}
\hline
\multicolumn{1}{l}{\itshape Thresholds}
&\multicolumn{1}{c}{\itshape Optimal}
&\multicolumn{10}{c}{\itshape Number of Hospitals above Threshold\tnote{a}}
\tabularnewline 
\cmidrule(l){3-12} 
\multicolumn{1}{>{\RaggedRight}p{10pt}}{}
&\multicolumn{1}{c}{TCL}
&\multicolumn{2}{c}{MLE}
&\multicolumn{2}{c}{SSEL}
&\multicolumn{2}{c}{WRSEL}
&\multicolumn{2}{c}{CB}
&\multicolumn{2}{c}{GR}
\tabularnewline
\hline
\tabularnewline
$C=1/1.3$&$56$&$ 5$&($ 9$)&$-1$&($-2$)&$19$&($34$)&$2$&($ 4$)&$3$&($ 5$)\tabularnewline
$C=1.0$&$ 15$&$ 2$&($13$)&$0$&($ 0$)&$ 10$&($ 67$)&$ 1$&($ 7$)&$ 0$&($ 0$)\tabularnewline
$C=1.3$&$ 22$&$11$&($50$)&$3$&($14$)&$ 21$&($ 95$)&$ 8$&($36$)&$ 6$&($27$)\tabularnewline
\tabularnewline
\hline
\end{tabular}
\begin{tablenotes}
   \item[a] Entries for the percentage departure have been truncated to
     the closest integer. For some entries, the percentage departure
     is smaller than 1 percentage point. 
\end{tablenotes}
\end{threeparttable}
\end{table}

The prevalence of MRSA in UK hospitals has been under scrutiny for the
past 10 years, and
surveillance data on MRSA prevalence has been made publicly available. 
The classification of National Health Services (NHS) trusts on the
basis of MRSA risk may be of interest to medical practitioners and patients
wishing to minimise their probability to contract the
condition. Several statistical issues have been raised and discussed
on the basis of this data set, including the evaluation of hospital
performances over time \citep{Spiegelhalter2005}, and the monitoring
of over-dispersed MRSA counts \citep{Grigg2009}. 
In this section, we use this data set to illustrate the implementation of a decision-theoretic
approach to the classification of elements in a parameter ensemble.
Here, the NHS hospitals constitute the parameter ensemble of interest
and our main objective is to identify which hospitals have a level of
risk for MRSA above a given threshold. 

\subsection{Data Pre-processing}
The full data set was retrieved from the archives of the UK's
Department of Health (\rm www.dh.gov.UK). This data set documents MRSA
prevalence in all NHS trusts in the UK. Hospitals are classified in three categories: (i) general
acute trusts, (ii) specialist trusts and (iii) single speciality
trusts. For each hospital, the data set covers four distinct time
periods. In this thesis, we will focus on a subset of this longitudinal data,
corresponding to the period running from April 2003 to March 2004, and
consider how to classify NHS trusts according to the number of cases
in that particular year. The Department of Health provided the
prevalence of MRSA for each hospital and the number of bed days over that period. 
(Bed days are the number of inpatient hospital days per 1000 members of a health plan.)
The NHS trusts with missing data or where no cases of MRSA was
observed were eliminated, due to the impossibility of retrieving the
number of bed days for these hospitals. That is, since the Department
of Health only provided yearly numbers of observed cases and the rates
of MRSA per bed days, the number of bed days could not be computed for
the hospitals with no observed cases, since the MRSA rates was zero
for these trusts. Data from the following seven hospitals were therefore discarded: 
Birmingham Women's Hospital,
Moorfields Eye Hospital, 
Liverpool Women's Hospital, 
Medway Hospital, 
Royal National Hospital for Rheumatic Diseases, 
Sheffield Children's Hospital, 
and the Royal Orthopaedic Hospital. 
The final sample size was composed of 166 trusts with MRSA prevalence
for the year 2003-2004.
In previous statistical publications, this data set has been utilised
in order to illustrate the inherent difficulties
associated with the monitoring of changes in levels of risk over
several time points. Here, by contrast, we are mainly concerned with the
classification of hospitals according to their levels of risks, as opposed to a
consideration of the evolution of trust-specific risk over several years.

\subsection{Fitted Model} \label{sec:mrsa model}
\begin{table}[t]
 \footnotesize
 \caption{
Posterior regrets based on $\op{TCL}(C,\bth,\bm\delta)$
and $\op{RCL}(\gamma,\bth,\bm\delta)$ for the MRSA data set with three
different choices of thresholds in each case. 
The posterior expected loss of the optimal estimator is given in the
first column. In parentheses, posterior regrets are expressed as
percentage of the posterior loss under the optimal estimator. Entries
have been all scaled by a factor of $10^{3}$.
\label{tab:mrsa_regret}}
 \centering
 \begin{threeparttable}
 \begin{tabular}{>{\RaggedRight}p{58pt}>{\RaggedLeft}p{30pt}|
>{\RaggedLeft}p{22pt}@{}>{\RaggedLeft}p{18pt}
>{\RaggedLeft}p{18pt}@{}>{\RaggedLeft}p{12pt}
>{\RaggedLeft}p{18pt}@{}>{\RaggedLeft}p{18pt}
>{\RaggedLeft}p{18pt}@{}>{\RaggedLeft}p{12pt}
>{\RaggedLeft}p{18pt}@{}>{\RaggedLeft}p{12pt}}\hline
\multicolumn{1}{c}{\itshape Loss functions \& Thresholds}&
\multicolumn{1}{c}{}&
\multicolumn{10}{c}{\itshape Posterior regrets\tnote{a}}
\tabularnewline \cline{3-12}
\multicolumn{1}{>{\RaggedRight}p{58pt}}{}&
\multicolumn{1}{c}{Post.~ Loss}&
\multicolumn{2}{c}{MLE}&\multicolumn{2}{c}{SSEL}&\multicolumn{2}{c}{WRSEL}&\multicolumn{2}{c}{CB}&\multicolumn{2}{c}{GR}
\tabularnewline
\hline
\tabularnewline
{\tb{TCL}}&&&&&&&&&&&\tabularnewline
$C=1/1.3$&$121.69$&$0.80$&($ 1$)&$0.13$&($0$)&$27.93$&($23$)&$0.16$&($0$)&$0.33$&($0$)\tabularnewline
$C=1.0$ &$125.14$&$0.15$&($ 0$)&$0.15$&($0$)&$14.78$&($12$)&$0.15$&($0$)&$0.15$&($0$)\tabularnewline
$C=1.3$ &$ 65.36$&$8.18$&($13$)&$0.15$&($0$)&$45.62$&($70$)&$3.85$&($6$)&$3.77$&($6$)\tabularnewline
\hline
\tabularnewline
{\tb{RCL}}&&&&&&&&&&&\tabularnewline
$\gamma=.60$ & $131.58$&$0.40$&($0$)&$0.40$&($0$)&$0.32$&($0$)&$0.40$&($0$)&$0.40$&($0$) \tabularnewline 
$\gamma=.75$ & $105.63$&$1.43$&($1$)&$0.63$&($1$)&$5.48$&($5$)&$0.63$&($1$)&$1.55$&($1$) \tabularnewline
$\gamma=.90$ & $ 67.65$&$4.14$&($6$)&$1.13$&($2$)&$0.86$&($1$)&$1.13$&($2$)&$1.13$&($2$) \tabularnewline
\hline
\end{tabular}
\begin{tablenotes}
   \item[a] Entries for the posterior regrets have been truncated to
     the closest second digit after the decimal point, and the
     percentage regrets were truncated to the closest integer.
     For some entries, percentage regrets are smaller than 1 percentage point. 
\end{tablenotes}
\end{threeparttable}
\end{table}

We represent observed cases of MRSA by $y_{i}$ for each
hospital, with $i=1,\ldots,n$, and $n=166$. The expected counts for each
NHS trust were computed using the MRSA
prevalence per thousand bed days and the number of inpatient bed days 
in each trust. That is, 
\begin{equation}
     E_{i}:= p_{\op{MRSA}}\times\op{BD}_{i},
\end{equation}
where $p_{\op{MRSA}}:=\sum_{i=1}^{n}y_{i}/\sum_{i=1}^{n}\op{BD}_{i}$
is the population MRSA prevalence, and we have assumed this rate to be
constant across all UK hospitals. The $\op{BD}_{i}$'s denote the number
of bed days in the $i\tth$ NHS trust in thousands. 

A two-level hierarchical model with Poisson likelihood and a lognormal
exchangeable prior on the RRs was used to fit the 
data. This assumes that the counts are Poisson, but with potential
over-dispersion, due to clustering of cases caused by the
infectiousness of the disease or other unmeasured risk factors. 
The full model had therefore the following structure,
\begin{equation}
  \begin{aligned}
    y_i &\stack{\ind}{\sim} \poi(\theta_iE_{i}) \qq i=1,\ldots,n, \\ 
    \log\theta_i & = \alpha + v_i\\
     v_i &\stack{\iid}{\sim} N(0,\sig^2), \label{eq:mrsa model}
  \end{aligned}
\end{equation}
where the inverse of the variance parameter was given a `diffuse' gamma distribution,
$\sig^{-2} \sim \dgam(0.5, 0.0005)$, while a flat Normal distribution
was specified for the intercept, $\alpha \sim N(0.0,10^{-6})$. For
each hospital, we have $\op{RR}_{i}:=\theta_{i}$. Since the joint posterior
distribution of the $\theta_{i}$'s is not available in closed-form,
an MCMC algorithm was used to estimate this model. The MRSA data was
fitted using WinBUGS 1.4 \citep{Lunn2000}. The codes used for the estimation of the
model have been reproduced in Appendix \ref{app:winbugs}.

Five different ensembles of point estimates were derived from the
joint posterior distribution of the $\theta_{i}$'s, including the
MLEs, the posterior means and medians, and the WRSEL, CB and
triple-goal point estimates. Given the size of the parameter
ensemble in this data set, the vector of weights in the WRSEL function
was specified using $a_{1}=a_{2}=0.1$. (The sensitivity of the
performance of the WRSEL plug-in estimators to the choice of $a_{1}$
and $a_{2}$ will be discussed in chapter \ref{chap:discussion}).
We considered the performance of these plug-in estimators under the
TCL and RCL functions with three different choices of thresholds in
each case. Here, we selected $C\in\lb 0.77, 1.0, 1.3\rb$, where the
first and third thresholds are equidistant from $1.0$ on the
logarithmic scale --that is, $1/1.3\doteq 0.77$. For $C=1/1.3$, we
evaluated the number of hospitals classified \ti{below}
threshold. That is, we were interested in identifying the NHS trusts
characterised by a substantially lower level of risk for MRSA. 
A choice of a threshold of 1.3, in this study, implies that hospitals above this threshold
have an MRSA rate, which is $30\%$ higher than the national
average. A threshold lower than 1.0 (i.e. $C=1/1.3$) was also selected in order to
estimate which hospitals can be confidently classified as having 
substantially lower MRSA rates than the remaining trusts. Note that in
that case, the definitions of the false positives and false negatives
in equations (\ref{eq:abba1}) and (\ref{eq:abba2}) become
inverted. That is, for $C<1.0$, we now have
\begin{equation}
    \op{FP}(C,\theta,\delta) := \cI\lt\lb \theta > C, \delta \leq C \rt\rb,
\end{equation}
and
\begin{equation}
    \op{FN}(C,\theta,\delta) :=  \cI\lt\lb  \theta \leq C, \delta > C \rt\rb.
\end{equation}
However, since we will solely be interested in unweighted classification
loss, in this section, this inversion does not affect the computation
of the posterior and percentage regrets of the plug-in
estimators. The results are presented for the TCL and RCL functions,
in turn.

\begin{figure}[htbp]
    \centering
    \includegraphics[width=15cm]{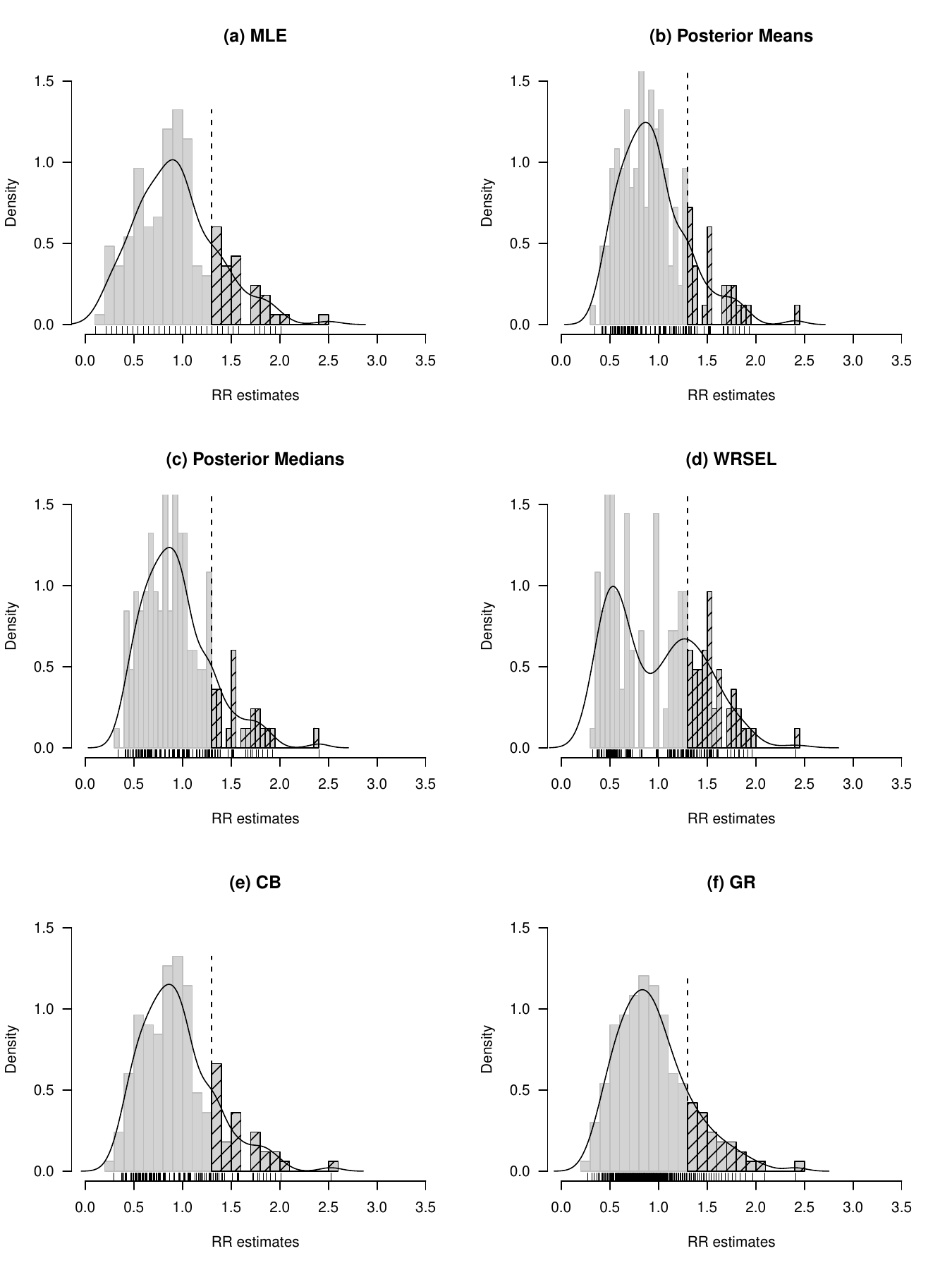}
    \caption{Ensembles of point estimates of MRSA RRs
    under different loss functions for 
    166 NHS trusts during the 2003--2004 period. The panels correspond to the (a)
    MLEs, (b) posterior means, (c) posterior medians, (d)
    point estimates under WRSEL, (e) constrained Bayes point estimates
    and (f) triple-goal point estimates. Classification of these point
    estimates is conducted with respect to a threshold taken to be $C=1.3$
    (dashed line). A smoothed version of the histograms has also been
    superimposed. 
    \label{fig:mrsa his}}
\end{figure}
\begin{figure}[htbp]
    \centering
    \includegraphics[width=15cm]{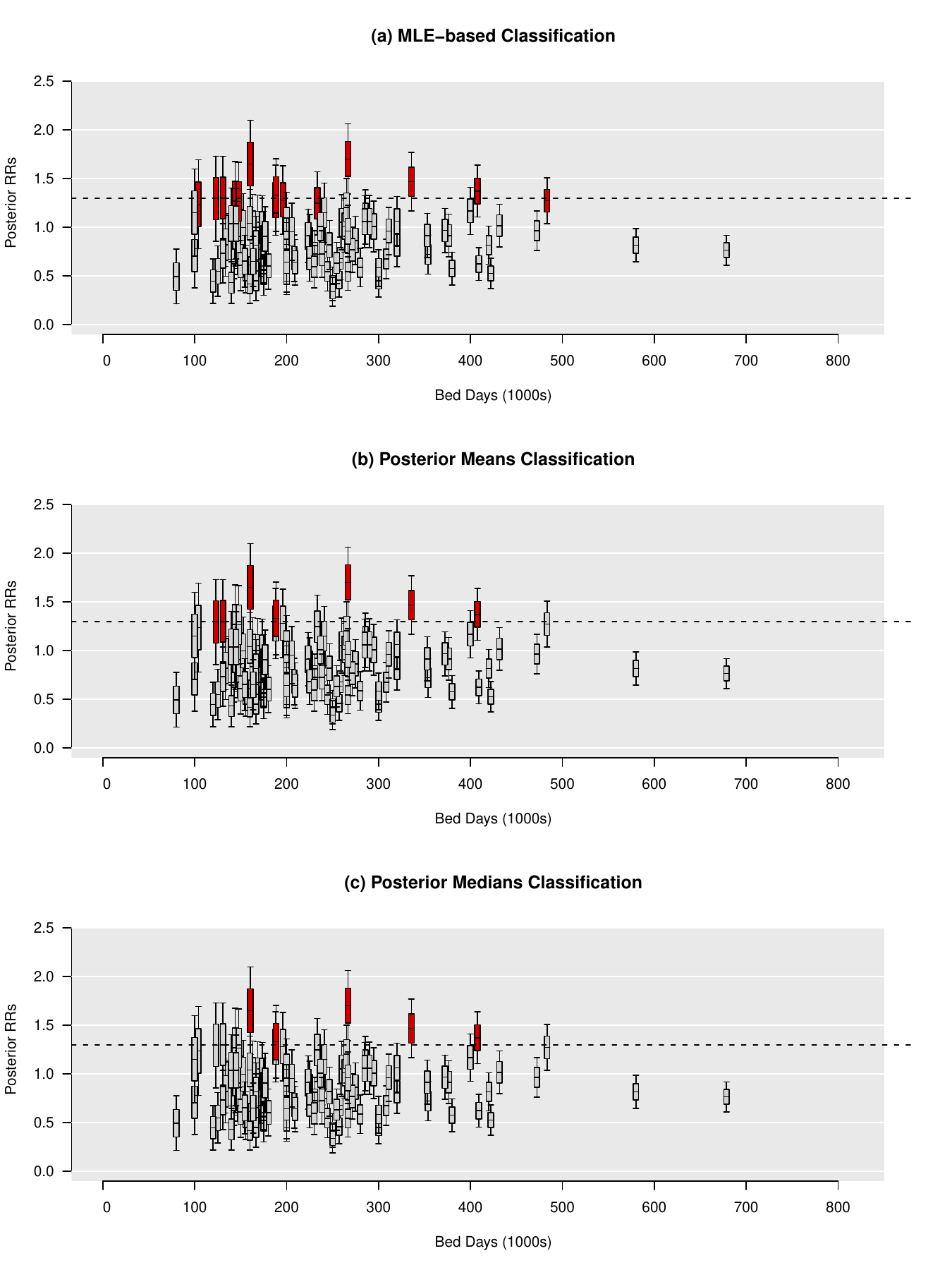}
    \caption{Classification of individual `general acute' NHS trusts ($n=110$)
      during the year 2003--2004, on the basis of three
      different families of point estimates: (a) MLEs, (b) posterior
      means, and (c) posterior medians. The marginal posterior
      distributions of trust-specific RRs for MRSA are represented by
      box plots (Median, $\pm \op{sd}$,$\pm 2\op{sd}$). In each panel, the trusts
      classified above threshold, $C=1.3$ (dashed line), are indicated in red.
      \label{fig:mrsaGA}}
\end{figure}
\begin{figure}[htbp]
    \centering
    \includegraphics[width=15cm]{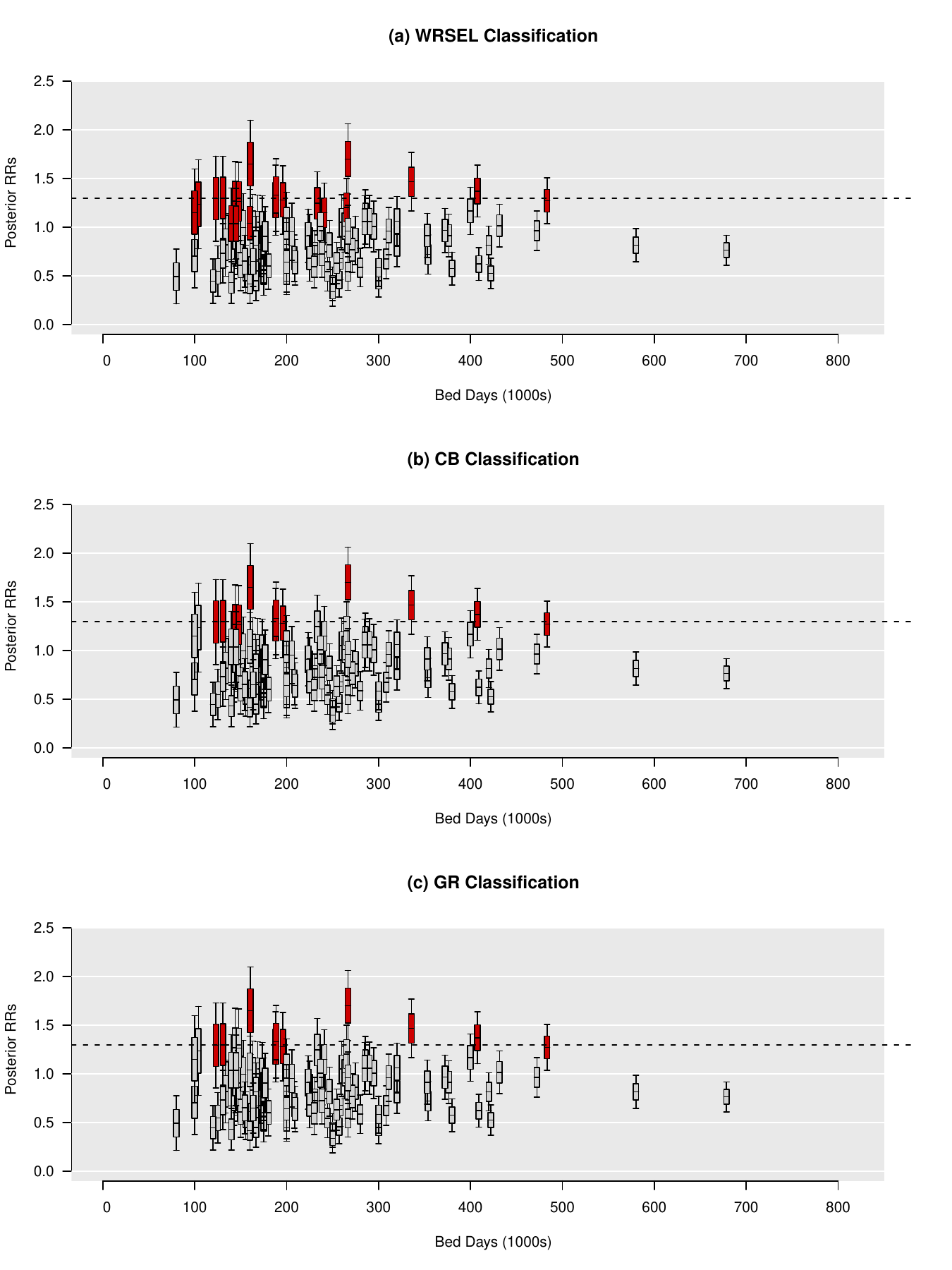}
    \caption{Classification of individual `general acute' NHS trusts ($n=110$)
      during the year 2003--2004, on the basis of three
      different families of point estimates: (a) WRSEL, (b) CB
      and (c) triple-goal estimates. The marginal posterior
      distributions of trust-specific RRs for MRSA are represented by
      box plots (Median, $\pm \op{sd}$,$\pm 2\op{sd}$). In each panel, the trusts
      classified above threshold, $C=1.3$ (dashed line), are indicated
      in red. Note the particular poor performance of the WRSEL
      plug-in classifier in this case. 
      \label{fig:mrsaGA2}}
\end{figure}
\begin{figure}[htbp]
    \centering
    \includegraphics[width=15cm]{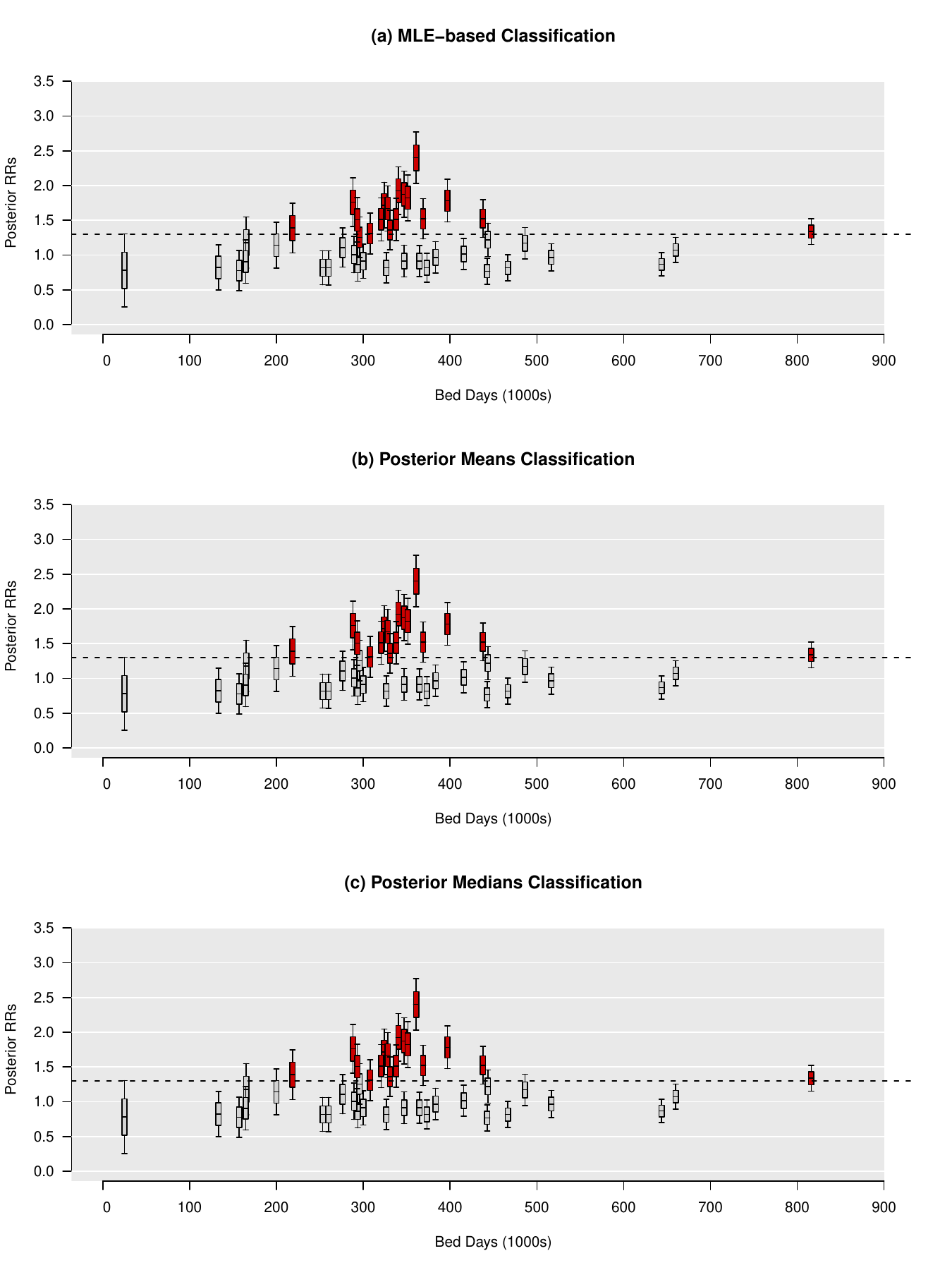}
    \caption{Classification of individual `specialist' NHS trusts ($n=43$)
      during the year 2003--2004, on the basis of three
      different families of point estimates: (a) MLEs, (b) posterior
      means, and (c) posterior medians. The marginal posterior
      distributions of trust-specific
      RRs are represented by box plots (Median,
      $\pm \op{sd}$,$\pm 2\op{sd}$). In each panel, the trusts
      classified above threshold, $C=1.3$ (dashed line), are indicated in red.
      \label{fig:mrsaSp}}
\end{figure}
\begin{figure}[htbp]
    \centering
    \includegraphics[width=15cm]{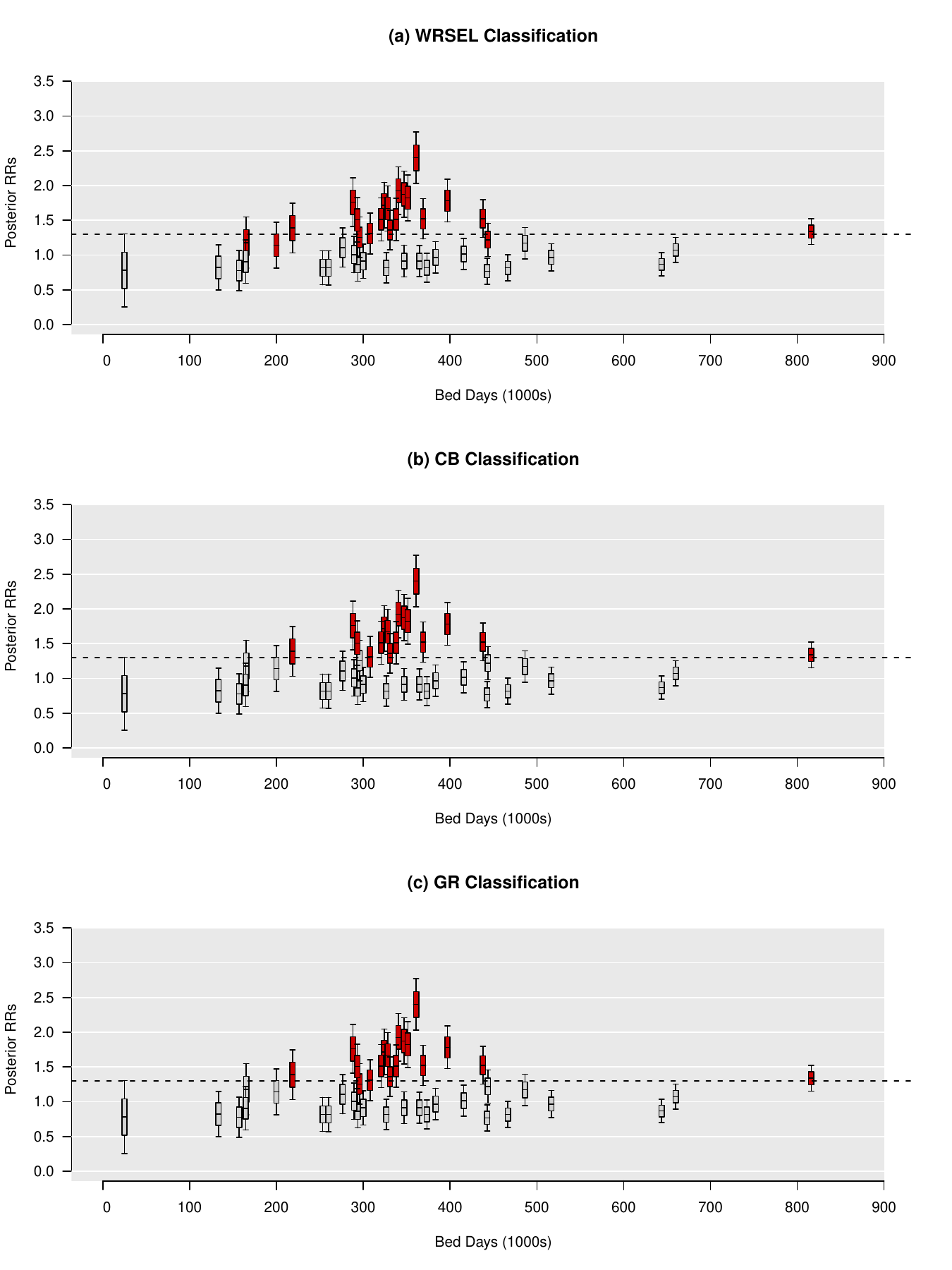}
    \caption{Classification of individual `specialist' NHS trusts ($n=43$)
      during the year 2003--2004, on the basis of three
      different families of point estimates: (a) WRSEL, (b) CB
      and (c) triple-goal estimates. The marginal posterior
      distributions of trust-specific
      RRs are represented by box plots (Median,
      $\pm \op{sd}$,$\pm 2\op{sd}$). In each panel, the trusts
      classified above threshold, $C=1.3$ (dashed line), are indicated in red.
      \label{fig:mrsaSp2}}
\end{figure}

\subsection{TCL Classification}
In table \ref{tab:mrsa_departure} on page
\pageref{tab:mrsa_departure}, we have reported the departure of each
plug-in classifier from the optimal TCL estimator in terms of number of
NHS hospitals classified above --or below, if $C=1/1.3$--
threshold. Remarkably, for the three thresholds ($C\in\lb 1/1.3, 1.0, 1.3\rb$), 
almost all plug-in classifiers were found to yield a greater number of NHS
trusts classified above (respectively below) threshold than when using
the set of posterior medians. The sole exception to this trend was for
the SSEL-based classifiers, which were modestly more conservative than the posterior
medians under $C=1/1.3$. That is, the ensemble of posterior means
produced a smaller set of hospitals considered to have RRs lower than $1/1.3$.
This indicates therefore that the optimal TCL classifier generally tends to produce more
conservative classifications than its plug-in counterparts. 

Moreover, we also note that, for this particular data set, the
GR-based classification was found to be closer to the optimal
classification than the one based on the CB point estimates. These
results should be contrasted with our spatial and non-spatial data
simulations, which have shown that the CB and GR plug-in estimators
tended to behave similarly under the TCL function. Not surprisingly,
however, the ensemble of posterior means was found to provide the
smallest number of misclassified areas in comparison to the optimal
classifier. For $C=1.0$, one can observe that both the SSEL and GR
plug-in classifiers and the posterior medians yielded an identical
number of NHS trusts above threshold.

In addition, the performance of the different sets of plug-in estimators in terms
of posterior regret is also reported in table \ref{tab:mrsa_regret} on page
\pageref{tab:mrsa_regret} for the TCL and RCL functions. As
aforementioned, one can see that the smallest posterior regret is achieved by the SSEL
classifier across all three thresholds considered. In line with the
results of our synthetic simulations, the
triple-goal and CB plug-in estimators were found to exhibit the second
best performance. Both of these classifiers tended to behave
similarly under both the TCL and RCL functions. These plug-in estimators were followed 
by the MLE and WRSEL classifiers in increasing order of posterior regret.
Figure \ref{fig:mrsa his} on page \pageref{fig:mrsa his} shows the
empirical ensemble distributions of point estimates for the six
ensembles of point estimates that we have considered. In that figure, the RRs
for MRSA for each NHS trust are classified above and below $C=1.3$. The pathological behaviour
of the WRSEL point estimates first discussed in section
\ref{sec:plug-in q-sel} is here particularly visible in panel (d), where 
the bimodality of the WRSEL ensemble of point estimates can be seen to 
emerge. The classification of NHS trusts using the six different types
of classifiers can also be examined in
greater detail in figures \ref{fig:mrsaGA} and \ref{fig:mrsaGA2} on pages \pageref{fig:mrsaGA} and 
\pageref{fig:mrsaGA2}, which portray the classifications of the general acute 
NHS trusts. Similarly, the classification of specialist NHS trusts
under these different ensembles of point estimates can be compared in figures
\ref{fig:mrsaSp} and \ref{fig:mrsaSp2} on pages \pageref{fig:mrsaSp} and 
\pageref{fig:mrsaSp2}, respectively. These box plots provide a summary
of the hospital-specific marginal posterior distributions of the level of risk
for MRSA. Specifically, one can see that for general NHS
trusts, the set of hospitals classified above an RR of 1.3
is lower when we use the set of posterior medians to classify
them. This contrasts with the adoption of the ensemble of MLEs and
posterior means or other sets of point estimates for the same task. A
similar pattern is visible for the specialist NHS trusts in figures
\ref{fig:mrsaSp} and \ref{fig:mrsaSp2}, albeit to a lesser degree, since both the posterior
means and posterior medians produce identical classifications for this 
particular class of hospitals. 

\subsection{RCL Classification}
Results on the use of plug-in estimators under the RCL function, in
this data set, are reported in table \ref{tab:mrsa_regret} on page
\pageref{tab:mrsa_regret}. The different plug-in classifiers can here be
compared in terms of posterior and percentage regrets. We have
computed the performance of these plug-in estimators using three different
proportions, $\ga\in\lb .60,.75,.90\rb$. As pointed out in section
\ref{sec:rcl non-spatial}, the SSEL and CB classifications under RCL
are necessarily identical. In addition, the performance of these two
classifiers was found to be very close to the one of the GR
plug-in estimator. These three plug-in estimators outperformed the MLE
and WRSEL classifiers, as expected from our spatial and non-spatial
simulation results. 
As previously mentioned, it is important to note that the differences
in posterior regrets between different choices of plug-in classifiers
were found to be small in comparison to the differences in posterior
regrets between different choices of classifiers under the TCL
function, thereby indicating that the choice of plug-in estimator is
more important under the latter function than when considering
rank-based classification. 

\section{Conclusion}\label{sec:clas conclusion}
In this chapter, we have investigated a standard classification
question, which often arises in both epidemiology and spatial
epidemiology. This problem centers on
the allocation of a subset of elements in a parameter ensemble of interest to 
an elevated-risk cluster according to the estimated RRs for
each of the elements in that ensemble.
We have showed that such a problem can be reformulated
in a decision-theoretic framework, where standard machinery can be
implemented to optimise the solution of the problem. Doing so, we
showed that the posterior expected TCL function is minimised by the
ensemble of posterior medians. In
addition, we have also considered the RCL function, whose properties
and minimiser have already been documented by \citet{Lin2006}.

As in chapter \ref{chap:mrrr}, our main focus in this chapter has been on
evaluating the performance of various plug-in estimators under both the TCL and
RCL functions. Overall, the ensemble of posterior means was found to
perform close to optimality under both decision-theoretic paradigms,
thus indicating that this standard choice of point estimates could
constitute a good candidate for routine epidemiological
investigations. In particular, the posterior means have the
non-negligible advantage of familiarity, which may aid the
reporting of such point estimates in public health settings. The good
performance of the SSEL point estimates under the TCL function can be explained
by noting that, in most applications, 
the marginal posterior distributions of the parameters of
interest will be asymptotically normally distributed. In such cases, 
as we gather more information about each parameter in the ensemble, 
the set of posterior means converges towards the set of posterior
medians. 

In section \ref{sec:clas weighted tcl}, we noted that 
the ensemble of SSEL point estimates also performed well
under weighted TCL. In that section, we reported the results of a
set of simulations based on $\op{TCL}_{0.8}$, which gives a greater
penalty to false positives. However, in a separate 
set of simulations (results not shown), we also evaluated the
performance of the different plug-in estimators under
$\op{TCL}_{0.2}$, which gives a greater penalty to false negatives. 
In that case, we have found that the set of posterior means
was outperformed by other plug-in classifiers under several
experimental conditions. These discrepancies between
different weighted TCLs can be explained in terms of hierarchical
shrinkage. The posterior means tend to be shrunk towards the prior
mean, and will therefore constitute a naturally conservative choice
that is likely to produce less false positives, but a greater amount of false
negatives. 

The GR and CB classifiers produced good performance on the TCL and RCL
functions, as measured by the posterior and percentage
regrets. These plug-in estimators were followed by the MLE and WRSEL
classifiers. However, the performance differences between the candidate estimators
under the TCL decision-theoretic paradigm were much larger than the
differences in posterior regrets between the plug-in estimators under
the RCL function. These discrepancies may be explained in terms of the
substantial differences between these two loss functions. As described
in section \ref{sec:rcl} on page \pageref{sec:rcl}, the optimisation
of the RCL function only requires an adequate ordering of the elements
in the ensemble under scrutiny. For the TCL function, however,
minimisation necessitates not only a good ordering of the elements in
the ensemble, but also a reasonably precise estimation of the values
of each of these elements. It appears that the combination of
these two desiderata makes the task of optimising the posterior TCL
more difficult than the one of minimising the RCL one, thereby
yielding greater discrepancies between the different candidate plug-in
estimators. 

In spatial epidemiology, it is often of interest to consider weighted
classification loss functions, which may, for instance, privilege
true positives over false positives, or the converse. This is an issue,
which has been addressed by \citet{Richardson2004}, who discussed different
re-weighting of the false positive and false negative rates. They used
a numerical minimisation procedure based on synthetic data sets, which 
showed that an area should be classified above a threshold if
approximately $80\%$ of the mass of the posterior distribution of the
level of risk in that area is above the threshold of interest (for the BYM and
BYM-L1 models). In section \ref{sec:clas weighted tcl}, we have seen
that this choice of decision rule is equivalent to the specification
of a weighted $\op{TCL}_{p}$ with $p=.80$. Thus, the decision rule
proposed by \citet{Richardson2004} gives a greater penalty to
false positives. This represents a 
natural choice of classification framework in the context of spatial
epidemiology, which deprecates the number of potential false alarms. 
The adoption of a conservative approach to the identification
of institutions or geographical areas as characterised by
``elevated risk'' is amply justified by the sensitive nature of
public health studies and their extensive media coverage.

In spite of our emphasis on the use of suboptimal classifiers, we also
note that our analysis of the MRSA data set has shown that
the use of the set of posterior medians yield the most conservative
classification of the NHS trusts. In an area as sensitive as the level
of risk for contracting MRSA, one may prefer to err on the side of
caution by classifying as few hospitals as possible as part of an
elevated-risk cluster. From a public health perspective, 
\citet{Grigg2009} among others have argued in favour of the utilisation of
conservative communication strategies when reporting surveillance
data. When communicating such epidemiological
findings to a general audience, the use of the optimal estimator under
the TCL function may therefore be usually preferred in order to
attenuate the potential detrimental effects of too liberal a
classification. 

\pagebreak[4]
{\singlespacing
\section{Proof of TCL Minimisation}\label{sec:clas proof}
\sub{Proof of proposition \ref{pro:tcl} on page \pageref{pro:tcl}.}
Let $\rho_{p}(C,\bth,\bth^{\op{est}})$ denote $\E[\op{TCL}_{p}(C,\bth,\bth^{\op{est}})|\by]$.
We prove the result by exhaustion over three cases. In order to prove that 
\begin{equation}
   \rho_{p}(C,\bth,\bth^{(1-p)})\leq\rho_{p}(C,\bth,\bth^{\op{est}}), 
\end{equation}
for any $\bth^{\op{est}}\in\bTh$ with $\theta_{i}^{(1-p)}:=Q_{\theta_{i}|\by}(1-p)$, it suffices to show that 
$\rho_{p}(C,\theta_{i},\theta_{i}^{(1-p)})\leq\rho_{p}(C,\theta_{i},\theta_{i}^{\op{est}})$
holds, for every $i=1,\ldots,n$. Expanding these unit-specific risks,
\begin{equation}
  \begin{aligned}
    p\cI\lb&\theta_{i}^{(1-p)}>C\rb\p\lt[\theta_{i}\leq
    C|\by\rt] + (1-p)\cI\lb\theta_{i}^{(1-p)}\leq
    C\rb\p\lt[\theta_{i}> C|\by\rt] \\
    &\leq\,
    p\cI\lb\theta_{i}^{\op{est}}>C\rb\p\lt[\theta_{i}\leq
    C|\by\rt] + (1-p)\cI\lb\theta_{i}^{\op{est}}\leq
    C\rb\p\lt[\theta_{i}> C|\by\rt].
   \label{eq:expansion}
  \end{aligned}
\end{equation}

Now, fix $C$ and $p\in[0,1]$ to arbitrary values. Then, for any
point estimate $\theta_{i}^{\op{est}}$, we have
\begin{equation}
  \rho_{p}(C,\theta_{i},\theta_{i}^{\op{est}}) = 
  \begin{cases}
    p\p[\theta_{i}\leq C|\by], & \te{if }\theta_{i}^{\op{est}}> C,\\
    (1-p)\p[\theta_{i}> C|\by], & \te{if }\theta_{i}^{\op{est}}\leq C.
  \end{cases}
\end{equation}
The optimality of $\theta^{(1-p)}_{i}$ over
$\theta^{\op{est}}_{i}$ as a point estimate is therefore directly
dependent on the relationships between $\theta_{i}^{(1-p)}$ and $C$,
and between $\theta_{i}^{\op{est}}$ and $C$. This determines the
following three cases:
\begin{description}
  \item[i.] If $\theta_{i}^{(1-p)}$ and $\theta_{i}^{\op{est}}$ are
    on the same side of $C$, then clearly, 
    \begin{equation}
    \rho_{p}(C,\theta_{i},\theta_{i}^{(1-p)}) = \rho_{p}(C,\theta_{i},\theta_{i}^{\op{est}}),
    \label{eq:case1}
    \end{equation}
  \item[ii.] If $\theta_{i}^{(1-p)} \leq C$ and $\theta_{i}^{\op{est}}> C$, then,
    \begin{equation}
    \rho_{p}(C,\theta_{i},\theta_{i}^{(1-p)}) 
    = (1-p)\p[\theta_{i}> C|\by]
    \,\leq\,
    p\p[\theta_{i}\leq C|\by] =
    \rho_{p}(C,\theta_{i},\theta_{i}^{\op{est}}),
    \label{eq:case3}
    \end{equation}
  \item[iii.] If $\theta_{i}^{(1-p)} > C$ and $\theta_{i}^{\op{est}}\leq C$, then, 
    \begin{equation}
    \rho_{p}(C,\theta_{i},\theta_{i}^{(1-p)}) 
    = p\p[\theta_{i}\leq C|\by]
    \,<\,
    (1-p)\p[\theta_{i}> C|\by] =
    \rho_{p}(C,\theta_{i},\theta_{i}^{\op{est}}),     
    \label{eq:case4}
    \end{equation} 
\end{description}
Equation (\ref{eq:case1}) follows directly from
an application of the result in (\ref{eq:expansion}), and cases two
and three follow from consideration of the following
relationship:
\begin{equation}
  p\p[\theta_{i}\leq C|\by] \gtreqless (1-p)\p[\theta_{i}>
  C|\by],
  \label{eq:grteqless1}
\end{equation}
where $\gtreqless$ means either $<$, $=$ or $>$. 
Using $\p[\theta_{i}> C|\by]=1-\p[\theta_{i}\leq C|\by]$, this gives
\begin{equation}
    \p[\theta_{i}\leq C|\by] = F_{\theta_{i}|\by}(C) \gtreqless 1-p.
     \label{eq:grteqless3}
\end{equation}
Here, $F_{\theta_{i}|\by}$ is the posterior CDF
of $\theta_{i}$. Therefore, we have
\begin{equation}
     C \gtreqless F^{-1}_{\theta_{i}|\by}(1-p) =:
     Q_{\theta_{i}|\by}(1-p) :=: \theta_{i}^{(1-p)},
     \label{eq:grteqless2}
\end{equation}
where $\gtreqless$ takes the same value in equations
(\ref{eq:grteqless1}), (\ref{eq:grteqless3}) and (\ref{eq:grteqless2}).

This proves the optimality of $\bth^{(1-p)}$. Moreover, since one can
construct a vector of point estimates $\theta_{i}^{\op{est}}$
satisfying $\theta_{i}^{\op{est}}\gtreqless C$, whenever 
$\theta^{(1-p)}_{i}\gtreqless C$, for every $i$, it
then follows that $\bth^{(1-p)}$ is not unique.
}

%
%
\chapter{Discussion}\label{chap:discussion}
In this final chapter, we review the main findings of the thesis, and
explore possible extensions of this work in several directions. In
particular, we consider whether the loss functions used in this thesis
can be generalised in order to take into account more sophisticated
inferential requirements. Moreover, we also discuss how different
modelling assumptions may better serve the decision-theoretic
questions addressed in the thesis. 

\section{Main Findings}
In the present thesis, we have adopted a formal decision-theoretic approach
to two inferential issues that naturally arise in epidemiology and
spatial epidemiology. Firstly, we have considered how the estimation
of the heterogeneity of a parameter ensemble can be
optimised. Secondly, we have derived the optimal estimator for the
classification of the elements of a parameter ensemble above or below
a given threshold. For consistency, epidemiologists are generally
under pressure to report a single set of point estimates when
presenting their research findings. We have therefore also explored
the utilisation of various plug-in estimators based on more
commonly used sets of point estimates in order to identify the ensembles
of point estimates that may simultaneously satisfy several of these
inferential objectives. 

Overall, our experimentation and real data analyses have shown that 
the GR plug-in estimator is very close to optimality when evaluating
the dispersion of a parameter ensemble, as quantified by the posterior
QR. By contrast, the best quasi-optimal plug-in estimator under both the TCL and
RCL functions was found to be the one based on the SSEL point
estimates. Taken together, these results confirm the inherent
difficulties associated with attempting to satisfy several inferential
desiderata using a single set of point estimates. Ultimately, the two objectives
considered in this thesis have been found to be optimised by two
different sets of point estimates. Nonetheless, further research could
investigate the formulation of nested decision-theoretic procedures in 
the spirit of the triple-goal methods introduced by \citet{Shen1998},
which could sub-optimally satisfy these two goals using a single set
of point estimates. Specifically, one could construct a
decision-theoretic framework where the two inferential objectives are
weighted, hence allowing different modellers to express their
preference for one goal over another in a formal manner. 

Another of the consistent findings in this thesis has been the relatively poor
performance of the WRSEL introduced by \citet{Wright2003}. In most of
the spatial and non-spatial synthetic simulations considered, this
plug-in estimator tended to perform on a par with the MLE plug-in
estimator. The main objective of the WRSEL function is to
counteract the effect of hierarchical shrinkage in the models
studied. Point estimates based on WRSEL have proved to be useful in
improving the estimation of the `spread' of parameter ensembles. 
However, the WRSEL estimation procedure was found to lack portability
from one data set to another. Specifically, the main difficulty with
this approach resides in the necessary 
pre-specification of the set of weights that recalibrate the loss
function. We have seen that these $\phi_{i}$'s are determined by both
the number of elements in the ensemble and the values taken by two
tuning parameters, $a_{1}$ and $a_{2}$, which control the symmetry and the
magnitude of the re-weighting effected by the WRSEL function. A step
towards the automatisation of the use of the WRSEL for any data set
may, for instance, include the specification of both $a_{1}$ and
$a_{2}$ as functions of the size of the ensemble. 

Our work on classification loss was originally motivated by the paper
of \citet{Richardson2004} on the identification of geographical areas
with ``elevated risk''. Our exploration of various classification loss
functions in chapter \ref{chap:clas} has demonstrated that the
decision rule proposed by \citet{Richardson2004} is equivalent to the
utilisation of a weighted TCL function based on
$p=.80$. This particular choice of decision rule can be shown to give
a greater penalty to false positives than to false negatives, thereby
safeguarding the public and decision makers against potential 
false alarms. 

\section{Generalised Classification Losses}
In chapter \ref{chap:clas}, we have described and used a 
classification loss function based on the ranking of the elements in a
parameter ensemble. The RCL function was originally
introduced by \citet{Lin2006} under a variety of guises. These authors
suggested the use of several generalised RCL functions that combined 
penalties for misclassifications with explicit penalties for the
distance of the point estimates from the threshold of interest. Their
decision-theoretic framework was developed on the basis of rank
classification but could easily be adapted to the problem of classifying
elements in a parameter ensemble with respect to a threshold on the
natural scale of these elements. The basic principles underlying the
TCL function introduced in this thesis, could
therefore be extended following the lead of \citet{Lin2006}, in the
following three directions. Here, each loss function takes three
parameters, $p,q,b\geq 0$, 
\begin{equation}
\begin{aligned}
  L^{|\cdot|}_{C}(p,q,b) & := \frac{1}{n} \sum_{i=1}^{n} 
  |C-\theta^{\op{est}}_{i}|^{p} 
  \op{FP}(C,\theta_{i},\theta_{i}^{\op{est}}) +
  b|C-\theta^{\op{est}}_{i}|^{q}
  \op{FN}(C,\theta_{i},\theta_{i}^{\op{est}}), \\
  L^{\dagger}_{C}(p,q,b) & := \frac{1}{n} \sum_{i=1}^{n} 
  |C-\theta_{i}|^{p} 
  \op{FP}(C,\theta_{i},\theta_{i}^{\op{est}}) +
  b|C-\theta_{i}|^{q}
  \op{FN}(C,\theta_{i},\theta_{i}^{\op{est}}), \\
  L^{\ddagger}_{C}(p,q,b) & := \frac{1}{n} \sum_{i=1}^{n} 
  |\theta^{\op{est}}_{i}-\theta_{i}|^{p} \op{FP}(C,\theta_{i},\theta_{i}^{\op{est}}) +
  b|\theta^{\op{est}}-\theta_{i}|^{q}\op{FN}(C,\theta_{i},\theta_{i}^{\op{est}}); \\
\end{aligned}
\label{eq:tcl generalisations}
\end{equation}
where FP and FN are defined as in equations (\ref{eq:abba1}) and
(\ref{eq:abba2}), respectively, on page \pageref{eq:abba1}.
As for the TCL, none of these families of loss functions produce a penalty if the
true parameter  $\theta_{i}$ and the point estimate
$\theta^{\op{est}}_{i}$ are both above or below the cut-off point
$C$. (Note, however, that this is not true for the \ti{expected}
TCL, where integration over the parameter space will necessarily yield
a loss greater than zero, regardless of the correctness of the
classification). If $\theta_{i}$ and $\theta_{i}^{\op{est}}$ are not on the same side
of $C$, $L^{|\cdot|}_{C}$ penalises the estimation procedure by
an amount which is proportional to the distance of $\theta^{\op{est}}_{i}$
from $C$; $L^{\dagger}_{C}$ penalises by an amount which is
proportional to the distance of $\theta_{i}$ from $C$; and
$L^{\ddagger}_{C}$ penalises by an amount which is proportional to
the distance between $\theta^{\op{est}}_{i}$ and $\theta_{i}$. The parameters
$p$ and $q$ permits to vary the strength of the penalties associated
with the false positives and the false negatives, respectively. 
Moreover, a final parameter, $b$, allows to further adjust these two types of
penalties by directly re-weighting their relative importance. 
Although $L^{\dagger}_{C}$ 
and $L^{\ddagger}_{C}$ are of theoretical interest, 
it should be clear that these particular loss families
would be difficult to implement, in practice, because they 
rely on the unknown quantities, $\theta_{i}$'s.
 
The three generalisations of the TCL family in equation (\ref{eq:tcl
  generalisations}) will require specific optimisation procedures as
was demonstrated by \citet{Lin2006}, who considered similar
generalisations with respect to rank-based classification. Since the
minimisation of the expected TCL and RCL functions were found to
differ substantially, it is not expected that the work of \citet{Lin2006}
on generalised RCLs will necessarily be relevant to the minimisation
of the generalised TCLs described in equation (\ref{eq:tcl
  generalisations}). Further research will therefore be needed to explore
these issues and investigate the behaviour of various plug-in
estimators under these generalised versions of the TCL function. 

\section{Modelling Assumptions}
Since the choice of a particular loss function only affects the
post-processing of the joint posterior distribution of the parameter
ensemble of interest, such methods are only as good as the statistical
model on which they are based. The use of different models may
therefore lead to substantial improvements in either estimating the
true heterogeneity of the ensemble or in classifying the elements of
that ensemble. This is a concern, which has already been voiced by
several authors. \citet{Shen1999}, for instance, have noted the dependence of their
triple-goal estimator's performance on the quality of the ensemble prior
distribution specified for the parameters of interest. It is this type
of concern that led these authors to propose the Smoothing by
Roughening (SBR) algorithm in order to produce an Empirical Bayes (EB)
prior distribution for the $\theta_{i}$'s \citep{Shen1999}. For
spatial models, the critical influence of modelling assumptions on the validity of 
the subsequent estimation of the level of heterogeneity in a parameter ensemble has
been highlighted by \citet{Lawson2009} and \citet{Ugarte2009}.

In chapter \ref{chap:review}, we have classified the models studied
throughout the thesis in three categories depending on the choice of
second-level priors. This included (i) proper conjugate iid priors, (ii)
non-conjugate proper iid priors and (iii) non-conjugate improper
non-idd priors on the $\theta_{i}$'s. It may be of interest to
consider non-parametric extensions of these models, such as ones based
on the Dirichlet process prior \citep{MacEachern1994,MacEachern1998}. 
The specification of such priors in BHMs tends to
give more flexibility to the model. In our case, this may
particularly help to evaluate the shape of the ensemble distribution
with greater accuracy, which could yield better estimates of the
heterogeneity of the ensemble. The combination of a parametric
likelihood function with a non-parametric hierarchical prior is often
referred to as a semi-parametric model. When combined with the Q-SEL or QR-SEL
functions, such semi-parametric Bayesian models may be particularly
well-suited to the type of inferential problems encountered in
epidemiology, where the estimation of the properties of the empirical
distributions of parameter ensembles is especially important. 

In terms of the classification of the elements of an ensemble, 
\citet{Ugarte2009a} have compared the performance of a HBM
with an EB strategy, and have concluded that the full Bayesian
framework may be preferable for the classification of elements in a parameter
ensemble. In addition, a natural modelling approach to this problem
may be the use of mixture models. 
One of the important limitations of the procedure developed in chapter
\ref{chap:clas} is that we are imposing the number of categories and
fixing a particular cut-off point. When
one is simply interested in identifying the elements of an ensemble,
which are characterised by a level of risk greater than a certain
threshold, this classification procedure could be sufficient. However,
if one is interested in the number of categories per se, a more sophisticated
modelling strategy may be adopted, where inference about the
number of categories can be directly conducted. 
\citet{Richardson2004}, for instance, have considered the use of the
reversible-jump MCMC algorithm in conjunction with a
decision-theoretic framework in order to classify the elements of an
ensemble \citep[see also][]{Green1995,Richardson1997}.
As explained in the introduction of chapter \ref{chap:clas}, 
\citet{Richardson2004} have calibrated their choice of classification
thresholds differently, when utilising different models. While we have
shown that the use of the posterior median is optimal under the
posterior TCL function, it may nonetheless be of interest to
investigate how specific modelling outputs, such as the distribution
of the number of mixture components in a semi-parametric mixture
models could be used in the context of classifying the elements of an
ensemble in terms of levels of risk. More specifically, the posterior
distribution of the number of mixture components could aid with the
determination of the ideal cut-off value upon which the classification
exercise is conducted.

\appendix 
\chapter{Non-Spatial Simulations}\label{app:non-spatial}
In this appendix, we provide some descriptive statistics 
of the non-spatial simulated data sets used in chapters 
\ref{chap:mrrr} and \ref{chap:clas}. 
%
\begin{table}[t]
 \small
 \caption{
\tb{Descriptive statistics for the simulated $y_{i}$'s.}
Presented for the Normal-Normal model in equation
(\ref{eq:normal-normal}) and the Gammma-Inverse gamma model in equation (\ref{eq:gamma-inverse gamma}), and for 3 different levels of RLS (ratio of the largest to the smallest $\sig_{i}$), and 3 different values for $n$, averaged over 100 synthetic data sets.\label{tab Y_table}} 
 \centering
 \begin{tabular}{>{\RaggedRight}p{100pt}>{\centering}p{50pt}>{\centering}p{50pt}>{\centering}p{50pt}>{\centering}p{50pt}>{\centering}p{50pt}}\hline
\multicolumn{1}{c}{\itshape Models \& Shrinkage}&
\multicolumn{5}{c}{\itshape Descriptive Statistics}
\tabularnewline \cline{2-6}
\multicolumn{1}{>{\RaggedRight}p{100pt}}{}&\multicolumn{1}{c}{Mean}&\multicolumn{1}{c}{SD}&\multicolumn{1}{c}{$2.5\tth$ Qua.}&\multicolumn{1}{c}{Median}&\multicolumn{1}{c}{$97.5\tth$ Qua.}\tabularnewline
\hline
{\itshape $\op{RLS}\doteq1$}&&&&&\tabularnewline
\normalfont   N-N, $n=100$ &    -0.002 &    1.414 &    -2.682 &    -0.006 &     2.641\tabularnewline
\normalfont   N-N, $n=200$ &     0.001 &    1.410 &    -2.709 &    -0.001 &     2.707\tabularnewline
\normalfont   N-N, $n=1000$ &     0.004 &    1.419 &    -2.766 &     0.006 &     2.775\tabularnewline
\normalfont   G-IG, $n=100$ &     1.000 &    1.253 &     0.027 &     0.588 &     4.252\tabularnewline
\normalfont   G-IG, $n=200$ &     0.992 &    1.325 &     0.024 &     0.575 &     4.308\tabularnewline
\normalfont   G-IG, $n=1000$ &     0.997 &    1.396 &     0.020 &     0.572 &     4.478\tabularnewline
\hline
{\itshape $\op{RLS}\doteq20$}&&&&&\tabularnewline
\normalfont   N-N, $n=100$ &     0.014 &    1.569 &    -2.984 &     0.016 &     3.014\tabularnewline
\normalfont   N-N, $n=200$ &     0.022 &    1.554 &    -2.998 &     0.026 &     3.053\tabularnewline
\normalfont   N-N, $n=1000$ &     0.005 &    1.558 &    -3.101 &     0.007 &     3.090\tabularnewline
\normalfont   G-IG, $n=100$ &     1.408 &    2.099 &     0.001 &     0.673 &     6.688\tabularnewline
\normalfont   G-IG, $n=200$ &     1.413 &    2.190 &     0.001 &     0.667 &     6.955\tabularnewline
\normalfont   G-IG, $n=1000$ &     1.412 &    2.235 &     0.000 &     0.659 &     7.071\tabularnewline
\hline
{\itshape $\op{RLS}\doteq100$}&&&&&\tabularnewline
\normalfont   N-N, $n=100$ &     0.023 &    1.786 &    -3.419 &    -0.001 &     3.647\tabularnewline
\normalfont   N-N, $n=200$ &    -0.003 &    1.767 &    -3.562 &    -0.010 &     3.563\tabularnewline
\normalfont   N-N, $n=1000$ &     0.006 &    1.773 &    -3.609 &     0.002 &     3.635\tabularnewline
\normalfont   G-IG, $n=100$ &     2.191 &    3.807 &     0.000 &     0.733 &    11.791\tabularnewline
\normalfont   G-IG, $n=200$ &     2.101 &    3.690 &     0.000 &     0.712 &    11.589\tabularnewline
\normalfont   G-IG, $n=1000$ &     2.130 &    3.749 &     0.000 &     0.704 &    11.885\tabularnewline
\hline
\end{tabular}

\end{table}

%
\begin{table}[t]
 \small
 \caption{
\tb{Descriptive statistics for the simulated $\sig_{i}$'s.}           
Presented for the Normal-Normal model in equation
(\ref{eq:normal-normal}) and the Gammma-Inverse gamma model in equation (\ref{eq:gamma-inverse gamma}), and for 3 different levels of RLS (ratio of the largest to the smallest $\sig_{i}$), and 3 different values for $n$, averaged over 100 synthetic data sets.\label{tab Sig_table}} 
 \centering
 \begin{tabular}{>{\RaggedRight}p{100pt}>{\centering}p{50pt}>{\centering}p{50pt}>{\centering}p{50pt}>{\centering}p{50pt}>{\centering}p{50pt}}\hline
\multicolumn{1}{c}{\itshape Models \& Shrinkage}&
\multicolumn{5}{c}{\itshape Descriptive Statistics}
\tabularnewline \cline{2-6}
\multicolumn{1}{>{\RaggedRight}p{100pt}}{}&\multicolumn{1}{c}{Mean}&\multicolumn{1}{c}{SD}&\multicolumn{1}{c}{$2.5\tth$ Qua.}&\multicolumn{1}{c}{Median}&\multicolumn{1}{c}{$97.5\tth$ Qua.}\tabularnewline
\hline
{\itshape $\op{RLS}\doteq1$}&&&&&\tabularnewline
\normalfont   $n=100$ &    1.000 &    0.006 &    0.991 &    1.000 &    1.009\tabularnewline
\normalfont   $n=200$ &    1.000 &    0.006 &    0.991 &    1.000 &    1.009\tabularnewline
\normalfont   $n=1000$ &    1.000 &    0.006 &    0.991 &    1.000 &    1.010\tabularnewline
\hline
{\itshape $\op{RLS}\doteq20$}&&&&&\tabularnewline
\normalfont   $n=100$ &    1.398 &    1.135 &    0.246 &    0.993 &    4.027\tabularnewline
\normalfont   $n=200$ &    1.426 &    1.152 &    0.244 &    1.017 &    4.125\tabularnewline
\normalfont   $n=1000$ &    1.415 &    1.147 &    0.241 &    1.001 &    4.147\tabularnewline
\hline
{\itshape $\op{RLS}\doteq100$}&&&&&\tabularnewline
\normalfont   $n=100$ &    2.158 &    2.503 &    0.117 &    1.022 &    8.627\tabularnewline
\normalfont   $n=200$ &    2.119 &    2.464 &    0.115 &    1.006 &    8.663\tabularnewline
\normalfont   $n=1000$ &    2.139 &    2.487 &    0.113 &    0.999 &    8.839\tabularnewline
\hline
\end{tabular}

\end{table}

\chapter{Spatial Simulations}\label{app:spatial}
In this appendix, we provide some descriptive statistics of
the spatial simulated data sets used in chapters 
\ref{chap:mrrr} and \ref{chap:clas}. 
%
\begin{table}[t]
 \small
 \caption{
            \tb{Descriptive statistics for the $E_{i}$'s used in the
            synthetic data simulations.} 
            These are here presented with respect to 
            three different values of the scaling factor (SF), controlling the
            level of the expected counts. The $E_{i}$'s reported here
            correspond to the expected counts for lung cancer in West
            Sussex adjusted for age only, occurring among males
            between 1989 and 2003. These data points have been
            extracted from the Thames Cancer Registry (TCR).
            \label{tab:E_table}} 
 \centering
 \begin{tabular}{>{\RaggedRight}p{100pt}>{\centering}p{50pt}>{\centering}p{50pt}>{\centering}p{50pt}>{\centering}p{50pt}>{\centering}p{50pt}}\hline
\multicolumn{1}{c}{\itshape Scaling Factor}&
\multicolumn{5}{c}{\itshape Descriptive Statistics}
\tabularnewline \cline{2-6}
\multicolumn{1}{>{\RaggedRight}p{100pt}}{}&\multicolumn{1}{c}{Mean}&\multicolumn{1}{c}{SD}&\multicolumn{1}{c}{$2.5\tth$ Qua.}&\multicolumn{1}{c}{Median}&\multicolumn{1}{c}{$97.5\tth$ Qua.}\tabularnewline
\hline
{\itshape }&&&&&\tabularnewline
\normalfont   $SF=0.1$ &     17 &      9 &      5 &     15 &     39\tabularnewline
\normalfont   $SF=1$ &    170 &     88 &     52 &    149 &    389\tabularnewline
\normalfont   $SF=2$ &    340 &    175 &    105 &    298 &    778\tabularnewline
\hline
\end{tabular}

\end{table}

%
\begin{table}[t]
 \small
 \caption{
\tb{Descriptive statistics for the Relative Risks (RRs) in the
  synthetic spatial simulations.}
Presented under three different levels of variability (low, medium and
high), where the parameters controlling the levels of variability
are modified according to the spatial scenario (i.e. SC1, SC2, SC3 and
SC4) considered. Note that while the RRs were simulated from specific
statistical models in SC3 and SC4, they were set to specific values
in SC1 and SC2. \label{tab:RR_table}} 
 \centering
 \begin{tabular}{>{\RaggedRight}p{100pt}>{\centering}p{50pt}>{\centering}p{50pt}>{\centering}p{50pt}>{\centering}p{50pt}>{\centering}p{50pt}}\hline
\multicolumn{1}{c}{\itshape Scenarios \& Variability}&
\multicolumn{5}{c}{\itshape Descriptive Statistics}
\tabularnewline \cline{2-6}
\multicolumn{1}{>{\RaggedRight}p{100pt}}{}&\multicolumn{1}{c}{Mean}&\multicolumn{1}{c}{SD}&\multicolumn{1}{c}{$2.5\tth$ Qua.}&\multicolumn{1}{c}{Median}&\multicolumn{1}{c}{$97.5\tth$ Qua.}\tabularnewline
\hline
{\itshape Low Variab.~}&&&&&\tabularnewline
\normalfont   SC1 ($LR=1.5$) &    1.018 &    0.093 &    1.000 &    1.000 &    1.500\tabularnewline
\normalfont   SC2 ($LR=1.5$) &    1.107 &    0.206 &    1.000 &    1.000 &    1.500\tabularnewline
\normalfont   SC3 ($\sig=0.1$) &    1.027 &    0.283 &    0.604 &    0.992 &    1.622\tabularnewline
\normalfont   SC4 ($\beta=0.2$) &    1.025 &    0.243 &    0.696 &    0.971 &    1.590\tabularnewline
\hline
{\itshape Medium Variab.~}&&&&&\tabularnewline
\normalfont   SC1 ($LR=2$) &    1.036 &    0.187 &    1.000 &    1.000 &    2.000\tabularnewline
\normalfont   SC2 ($LR=2$) &    1.215 &    0.412 &    1.000 &    1.000 &    2.000\tabularnewline
\normalfont   SC3 ($\sig=0.2$) &    1.194 &    0.513 &    0.535 &    1.077 &    2.408\tabularnewline
\normalfont   SC4 ($\beta=0.3$) &    1.055 &    0.367 &    0.612 &    0.954 &    1.938\tabularnewline
\hline
{\itshape High Variab.~}&&&&&\tabularnewline
\normalfont   SC1 ($LR=3$) &    1.072 &    0.373 &    1.000 &    1.000 &    3.000\tabularnewline
\normalfont   SC2 ($LR=3$) &    1.430 &    0.824 &    1.000 &    1.000 &    3.000\tabularnewline
\normalfont   SC3 ($\sig=0.3$) &    1.193 &    0.598 &    0.465 &    1.049 &    2.641\tabularnewline
\normalfont   SC4 ($\beta=0.4$) &    1.095 &    0.515 &    0.531 &    0.935 &    2.358\tabularnewline
\hline
\end{tabular}

\end{table}

%
\begin{table}[t]
 \small
 \caption{
            \tb{Descriptive statistics for the simulated $y_{i}$'s.}
            Presented for three choices of the Scaling Factor (SF) of the
            expected counts, three different levels of variability (low, medium and high), 
            and four different spatial scenarios (i.e. SC1, SC2, SC3
            and SC4).\label{tab:Y_table}} 
 \centering
 \begin{tabular}{>{\RaggedRight}p{110pt}>{\centering}p{50pt}>{\centering}p{50pt}>{\centering}p{50pt}>{\centering}p{50pt}>{\centering}p{50pt}}\hline
\multicolumn{1}{c}{\itshape Scenarios \& Variability}&
\multicolumn{5}{c}{\itshape Descriptive Statistics}
\tabularnewline \cline{2-6}
\multicolumn{1}{>{\RaggedRight}p{110pt}}{}&\multicolumn{1}{c}{Mean}&\multicolumn{1}{c}{SD}&\multicolumn{1}{c}{$2.5\tth$ Qua.}&\multicolumn{1}{c}{Median}&\multicolumn{1}{c}{$97.5\tth$ Qua.}\tabularnewline
\hline
{\itshape $SF=0.1$, Low Variab.~}&&&&&\tabularnewline
\normalfont   SC1 ($LR=1.5$) &     4.300 &     3.061 &     1.000 &     3.775 &     12.144\tabularnewline
\normalfont   SC2 ($LR=1.5$) &     4.312 &     3.084 &     1.000 &     3.825 &     11.725\tabularnewline
\normalfont   SC3 ($\sig=0.1$) &     4.314 &     3.182 &     1.000 &     3.650 &     12.219\tabularnewline
\normalfont   SC4 ($\beta=0.2$) &     4.325 &     3.291 &     1.000 &     3.550 &     12.463\tabularnewline
\hline
{\itshape $SF=0.1$, Medium Variab.~}&&&&&\tabularnewline
\normalfont   SC1 ($LR=2$) &     4.311 &     3.382 &     1.000 &     3.625 &     12.494\tabularnewline
\normalfont   SC2 ($LR=2$) &     4.325 &     3.458 &     1.000 &     3.275 &     13.350\tabularnewline
\normalfont   SC3 ($\sig=0.2$) &     4.331 &     3.529 &     1.000 &     3.300 &     13.344\tabularnewline
\normalfont   SC4 ($\beta=0.3$) &     4.328 &     3.666 &     1.000 &     3.100 &     13.669\tabularnewline
\hline
{\itshape $SF=0.1$, High Variab.~}&&&&&\tabularnewline
\normalfont   SC1 ($LR=3$) &     4.304 &     3.982 &     1.000 &     3.050 &     14.475\tabularnewline
\normalfont   SC2 ($LR=3$) &     4.345 &     4.177 &     1.000 &     3.000 &     15.331\tabularnewline
\normalfont   SC3 ($\sig=0.3$) &     4.341 &     3.649 &     1.000 &     3.300 &     13.619\tabularnewline
\normalfont   SC4 ($\beta=0.4$) &     4.355 &     4.210 &     1.000 &     3.050 &     15.456\tabularnewline
\hline
{\itshape $SF=1$, Low Variab.~}&&&&&\tabularnewline
\normalfont   SC1 ($LR=1.5$) &    42.440 &    24.587 &    10.294 &    36.850 &    104.000\tabularnewline
\normalfont   SC2 ($LR=1.5$) &    42.440 &    25.761 &     9.931 &    36.125 &    104.469\tabularnewline
\normalfont   SC3 ($\sig=0.1$) &    42.440 &    25.681 &     9.531 &    36.475 &    106.369\tabularnewline
\normalfont   SC4 ($\beta=0.2$) &    42.440 &    28.559 &     8.575 &    36.650 &    114.831\tabularnewline
\hline
{\itshape $SF=1$, Medium Variab.~}&&&&&\tabularnewline
\normalfont   SC1 ($LR=2$) &    42.440 &    27.544 &    10.350 &    36.000 &    122.906\tabularnewline
\normalfont   SC2 ($LR=2$) &    42.440 &    30.502 &     8.750 &    33.225 &    122.244\tabularnewline
\normalfont   SC3 ($\sig=0.2$) &    42.440 &    30.418 &     7.381 &    34.625 &    117.656\tabularnewline
\normalfont   SC4 ($\beta=0.3$) &    42.440 &    32.439 &     6.813 &    35.300 &    125.587\tabularnewline
\hline
{\itshape $SF=1$, High Variab.~}&&&&&\tabularnewline
\normalfont   SC1 ($LR=3$) &    42.440 &    34.910 &     9.638 &    34.100 &    175.525\tabularnewline
\normalfont   SC2 ($LR=3$) &    42.440 &    38.826 &     6.725 &    29.275 &    151.644\tabularnewline
\normalfont   SC3 ($\sig=0.3$) &    42.440 &    31.448 &     7.744 &    33.500 &    123.769\tabularnewline
\normalfont   SC4 ($\beta=0.4$) &    42.440 &    38.143 &     6.081 &    32.700 &    135.731\tabularnewline
\hline
{\itshape $SF=2$, Low Variab.~}&&&&&\tabularnewline
\normalfont   SC1 ($LR=1.5$) &    84.886 &    48.263 &    22.631 &    73.575 &    200.469\tabularnewline
\normalfont   SC2 ($LR=1.5$) &    84.886 &    51.058 &    21.744 &    71.700 &    212.281\tabularnewline
\normalfont   SC3 ($\sig=0.1$) &    84.886 &    53.573 &    18.131 &    72.200 &    220.244\tabularnewline
\normalfont   SC4 ($\beta=0.2$) &    84.886 &    56.077 &    17.500 &    74.275 &    227.494\tabularnewline
\hline
{\itshape $SF=2$, Medium Variab.~}&&&&&\tabularnewline
\normalfont   SC1 ($LR=2$) &    84.886 &    53.721 &    22.181 &    72.175 &    216.163\tabularnewline
\normalfont   SC2 ($LR=2$) &    84.886 &    59.674 &    20.056 &    66.075 &    248.344\tabularnewline
\normalfont   SC3 ($\sig=0.2$) &    84.886 &    56.546 &    18.444 &    69.750 &    231.025\tabularnewline
\normalfont   SC4 ($\beta=0.3$) &    84.886 &    65.128 &    14.856 &    70.825 &    250.919\tabularnewline
\hline
{\itshape $SF=2$, High Variab.~}&&&&&\tabularnewline
\normalfont   SC1 ($LR=3$) &    84.886 &    68.050 &    20.881 &    69.400 &    268.369\tabularnewline
\normalfont   SC2 ($LR=3$) &    84.886 &    77.636 &    16.775 &    57.625 &    316.769\tabularnewline
\normalfont   SC3 ($\sig=0.3$) &    84.886 &    64.130 &    15.100 &    67.325 &    244.894\tabularnewline
\normalfont   SC4 ($\beta=0.4$) &    84.886 &    75.035 &    12.675 &    66.125 &    271.725\tabularnewline
\hline
\end{tabular}

\end{table}

\chapter{WinBUGS Codes}\label{app:winbugs}
\section{CAR Normal (BYM) Model}
\begin{Verbatim}[baselinestretch=1.15, xrightmargin=2pt,
    samepage=true,frame=lines,label={[CAR Normal]}]

# CAR (convolution) Normal.
#######################################
model{
for (i in 1:N){
    y[i] ~ dpois(lambda[i])
    log(lambda[i]) <- log(E[i]) + theta[i]
    theta[i] <- alpha + v[i] + u[i]
    RR[i]    <- exp(theta[i])
    v[i] ~ dnorm(0,tau2_v)
}

###############
# CAR prior:
u[1:N] ~ car.normal(adj[],weights[],num[],tau2_u)
for (j in 1:sumNumNeigh){
    weights[j] <- 1.0
}

###############
# Hyperpriors.
alpha ~ dflat()
# Scaling Parameters.
tau2_v ~ dgamma(0.5,0.0005)
tau2_u ~ dgamma(0.5,0.0005)
sig_v <- sqrt(1/tau2_v)
sig_u <- sqrt(1/tau2_u)

}# EoF
\end{Verbatim}
\pagebreak[4]

\section{CAR Laplace (L1) Model}
\begin{Verbatim}[baselinestretch=1.15, xrightmargin=2pt,
    samepage=true,frame=lines,label={[CAR Laplace]}]

# CAR (convolution) L1.

##################################################
model{
for (i in 1:N){
    y[i] ~ dpois(lambda[i])
    log(lambda[i]) <- log(E[i]) + theta[i]
    theta[i]  <- alpha + v[i] + u[i]
    RR[i] <- exp(theta[i])
    v[i] ~ dnorm(0,tau2_v)
}

###############
# CAR prior
u[1:N] ~ car.l1(adj[],weights[],num[],tau2_u)
for (j in 1:sumNumNeigh){
    weights[j] <- 1.0
}

###############
# Hyperpriors.
alpha ~ dflat()

# Scaling Parameters.
tau2_v ~ dgamma(0.5,0.0005)
tau2_u ~ dgamma(0.5,0.0005)
sig_v <- sqrt(1/tau2_v)
sig_u <- sqrt(1/tau2_u)

}# EoF
\end{Verbatim}

\pagebreak[4]
\section{MRSA Model}
\begin{Verbatim}[baselinestretch=1.15, xrightmargin=2pt,
    samepage=true,frame=lines,label={[MRSA Model]}]

# Loglinear Model.
##################################################
model{
for (i in 1:n){
    y[i] ~ dpois(lambda[i])
    log(lambda[i]) <- log(E[i]) + theta[i]
    theta[i] <- alpha + v[i] 
    RR[i]    <- exp(theta[i])
    v[i] ~ dnorm(0,tau2_v)
}

###############
# Hyperpriors.
alpha ~ dnorm(0.0,1.0E-6)

# Scaling Parameters.
tau2_v ~ dgamma(0.5,0.0005)
sig_v <- sqrt(1/tau2_v)

}# EoF
\end{Verbatim}


\addcontentsline{toc}{section}{References}
\bibliography{/home/cgineste/ref/bibtex/Statistics}
\bibliographystyle{/home/cgineste/ref/style/oupced3}


\addcontentsline{toc}{section}{Index}

\end{document}